\documentclass[12pt,twoside]{book}
\usepackage{a4}
\usepackage{amssymb}
\usepackage{amsmath}
\usepackage{amsthm}
\usepackage{amstext}
\usepackage{amscd}
\usepackage{latexsym}
\usepackage{mathabx}
\usepackage{graphics}
\usepackage{graphicx}
\usepackage{color}
\usepackage{pdfpages}

\usepackage{verbatim}
\usepackage{amsgen, amsfonts, euscript, enumerate, url, calc}
\usepackage{color}

\usepackage[cmtip, all]{xy}
\usepackage{MnSymbol}

\usepackage[colorlinks, pagebackref]{hyperref}
\hypersetup{
  colorlinks=true,
  citecolor=blue,
  linkcolor=blue,
  urlcolor=blue}

\usepackage{emptypage} 

\usepackage{fancyhdr}
\usepackage{makeidx}
\makeindex 

\pagebreak
\newtheorem{thm}{Theorem}[section]
\newtheorem{prop}[thm]{Proposition}
\newtheorem{lem}[thm]{Lemma}
\newtheorem{cor}[thm]{Corollary}

\theoremstyle{definition}

\newtheorem{defn}[thm]{Definition}
\theoremstyle{remark}
\newtheorem{remk}[thm]{Remark}
\newtheorem{remks}[thm]{Remarks}

\newtheorem{exm}[thm]{Example}
\newtheorem{exms}[thm]{Examples}
\newtheorem{notat}[thm]{Notation}
\numberwithin{equation}{section}

\newcommand{\Tab}{{\mathbf {Tab}}}
\newcommand{\cf}{{\rm cf}}
\newcommand{\fr}{{\rm tor}}
\newcommand{\pf}{{\rm pf}}

{\hfill$\square$\end{defn}}
\newenvironment{rem}{\begin{remk}}%
{\hfill$\square$\end{remk}}
{\hfill$\square$\end{remks}}
{\hfill$\square$\end{exm}}
{\hfill$\square$\end{exms}}
{\hfill$\square$\end{notat}}

\newcommand{\thmref}{Theorem~\ref}
\newcommand{\propref}{Proposition~\ref}
\newcommand{\corref}{Corollary~\ref}
\newcommand{\defref}{Definition~\ref}
\newcommand{\lemref}{Lemma~\ref}

\newcommand{\remref}{Remark~\ref}

\newcommand{\secref}{Section~\ref}

\newcommand{\sA}{{\mathcal A}}
\newcommand{\sB}{{\mathcal B}}
\newcommand{\sC}{{\mathcal C}}
\newcommand{\sD}{{\mathcal D}}
\newcommand{\sE}{{\mathcal E}}
\newcommand{\sF}{{\mathcal F}}
\newcommand{\sG}{{\mathcal G}}
\newcommand{\sH}{{\mathcal H}}
\newcommand{\sI}{{\mathcal I}}

\newcommand{\sK}{{\mathcal K}}
\newcommand{\sL}{{\mathcal L}}
\newcommand{\sM}{{\mathcal M}}
\newcommand{\sN}{{\mathcal N}}
\newcommand{\sO}{{\mathcal O}}

\newcommand{\sV}{{\mathcal V}}
\newcommand{\sW}{{\mathcal W}}

\newcommand{\A}{{\mathbb A}}

\newcommand{\F}{{\mathbb F}}
\newcommand{\G}{{\mathbb G}}
\renewcommand{\H}{{\mathbb H}}

\newcommand{\N}{{\mathbb N}}
\renewcommand{\P}{{\mathbb P}}
\newcommand{\Q}{{\mathbb Q}}

\newcommand{\T}{{\mathbb T}}

\newcommand{\Z}{{\mathbb Z}}

\newcommand{\fm}{{\mathfrak m}}

\newcommand{\fM}{{\mathfrak M}}

\newcommand{\ff}{{\mathfrak f}}

\newcommand{\divv}{{\rm div}}

\newcommand{\Ker}{{\rm Ker}}
\newcommand{\gr}{{\rm gr}}

\newcommand{\surj}{\twoheadrightarrow}
\newcommand{\inj}{\hookrightarrow}
\newcommand{\red}{{\rm red}}

\newcommand{\rank}{{\rm rank}}
\newcommand{\Pic}{{\rm Pic}}
\newcommand{\Div}{{\rm Div}}
\newcommand{\Hom}{{\rm Hom}}

\newcommand{\Spec}{{\rm Spec \,}}

\newcommand{\bk}{{\rm bk}}
\newcommand{\Tr}{{\rm Tr}}

\newcommand{\ab}{\rm ab}

\newcommand{\Gal}{{\rm Gal}}

\newcommand{\supp}{{\rm supp}\,}
\newcommand{\sHom}{{\mathcal{H}{om}}}

\newcommand{\n}{{\un{n}}}

\newcommand{\id}{{\operatorname{id}}}
\newcommand{\Zar}{{\text{\rm Zar}}}

\newcommand{\pfd}{{\operatorname{\mathbf{Pfd}}}} 
\newcommand{\Sch}{{\operatorname{\mathbf{Sch}}}}

\newcommand{\<}{\langle}
\renewcommand{\>}{\rangle}

\newcommand{\Picc}{{\mathbf{Pic}}}

\newcommand{\Ab}{{\mathbf{Ab}}}

\newcommand{\Wedge}{{\Lambda}}

\newcommand{\et}{{\text{\'et}}}

\newcommand{\res}{{\operatorname{res}}}

\renewcommand{\log}{{\operatorname{log}}}

\newcommand{\Tor}{{\operatorname{Tor}}}

\newcommand{\Br}{{\operatorname{Br}}}

\newcommand{\un}{\underline}
\newcommand{\ov}{\overline}

\newcommand{\tuborg}{\left\{\begin{array}{ll}}
\newcommand{\sluttuborg}{\end{array}\right.}

\newcommand{\zar}{{\rm zar}}
\newcommand{\nis}{{\rm nis}}

\newcommand{\tor}{{\rm tor}}

\newcommand{\dlog}{{\rm dlog}}
\newcommand{\ft}{{\rm ft}}

\newcommand{\trv}{{\rm triv}}
\newcommand{\Et}{{\rm {\bf{Et}}}}

\newcommand{\wt}{\widetilde}
\newcommand{\wh}{\widehat}

\newcommand{\Aut}{{\rm Aut}}
\newcommand{\coker}{{\rm Coker}}
\newcommand{\Fil}{{\rm fil}}

\newcommand{\dR}{{\rm dR}}

\newcommand{\etl}{{\acute{e}t}}

\newcounter{elno}

\newcounter{elno-abc}   

\newcounter{elno-abc-prime}

\begin{document}

\begin{titlepage}

\thispagestyle{empty}
    \begin{center}
 
{\LARGE\bf {On Kato's ramification filtration}} \\ 
\vspace{0.5 in}
{\large {Subhadip Majumder}}\\
\vspace{0.5 in}
{\large\bf {Abstract}}
\end{center}
For a Henselian discrete valued field $K$ of characteristic $p>0$, Kato \cite{Kato-89} defined a ramification filtration $\{\Fil_nH^q(K,\Q_p/\Z_p(q-1))\}_{n \ge 0}$ on $H^q(K,\Q_p/\Z_p(q-1))$.  One can also define a ramification filtration on $H^q(U,\Z/p^m(q-1))$ using the local Kato-filtration, where $U$ is the complement of a simple normal crossing divisor in a regular scheme $X$ of characteristic $p>0$. The main objective of this thesis is to provide a cohomological description of these filtrations using de Rham-Witt sheaves and present several applications.

To achieve our goal, we study a theory of the filtered de Rham-Witt complex of $F$-finite regular schemes of characteristic $p>0$ and prove several properties which are well known for the classical de Rham-Witt complex of regular schemes. As applications, we prove a refined version of Jannsen-Saito-Zhao's duality \cite{JSZ} over finite fields, and a similar duality for smooth projective curves over local fields. As another application, we  prove a Lefschetz theorem for unramified and ramified Brauer group (with modulus) of smooth projective $F$-finite schemes over a field of characteristic $p>0$. Further applications are given in \cite{KM-classfield} and \cite{KM-crys}.

\noindent
\\
\textit{This work is a slightly modified version of the author's Ph.D. thesis, submitted in July 2024 at TIFR, Mumbai.}
 \pagenumbering{roman} 
    
\end{titlepage}

\setcounter{page}{3}

\thispagestyle{empty}

\thispagestyle{empty}

\setcounter{tocdepth}{3}
\setcounter{chapter}{0}

\thispagestyle{empty}

\pdfbookmark[0]{Contents}{Contents}
\tableofcontents

\thispagestyle{empty}

\setlength{\headheight}{27.5pt}

\pagestyle{fancy}

\fancyhf{}

\fancyhead[LE,RO]{\thepage}

\renewcommand{\headrulewidth}{0pt}
\fancypagestyle{plain}{ %
\fancyhf{} 
\renewcommand{\headrulewidth}{0pt} 
\renewcommand{\footrulewidth}{0pt}}

\cleardoublepage
\chapter*{Notations}\label{sec:Notn}
\addcontentsline{toc}{chapter}{Notations}
\thispagestyle{empty}
We shall work over a field $k$ of characteristic $p > 0$ throughout this
thesis, unless otherwise stated. We let $\Sch^{\ft}_k$ denote the category of separated and
finite type $k$-schemes. We let $\Sch_k$ denote the category of 
separated Noetherian $k$-schemes. The product $X \times_{\Spec(k)} Y$
in $\Sch_k$ will be written as $X \times Y$. We let
$X^{(q)}$ (resp. $X_{(q)}$) denote the set of points on $X$ having codimension
(resp. dimension) $q$.

Let $X$ be a scheme and let $\sM_X$ be a sheaf of monoids on $X$. We shall denote a log scheme $(X, \sM_X, \alpha\colon \sM_X \to \sO_X)$ on a scheme
$X$ by $X_\log$. If $X = \Spec(A)$, we shall also write these as
$(A, \alpha \colon M_A \to A)$ and $A_\log$. By $X_\trv$, we shall mean the
(trivial) log structure given by the usual inclusion
$\alpha \colon \sO^\times_X \inj \sO_X$. For a scheme $X$, we let $\Sch_X$ denote
the category of schemes over $X$. All rings are assumed to be Noetherian commutative rings with unity. For a reduced ring $A$, we shall let $Q(A)$ denote
the ring of total quotients of $A$.

We let $\Sch_{k/\zar}$ (resp. $\Sch_{k/\nis}$, resp. $\Sch_{k/ \etl})$
denote the Zariski (resp. Nisnevich, resp. {\'e}tale) site of $\Sch_{k}$.
We let $\epsilon \colon
\Sch_{k/ \et} \to \Sch_{k/ \nis}$ denote the canonical morphism of sites.
If $\sF$ is a sheaf on $\Sch_{k/\nis}$, 
we shall denote $\epsilon^* \sF$ also by $\sF$ as long as the usage of
the {\'e}tale topology is clear in a context.
For $X \in \Sch_k$, we shall let $F \colon X \to X$ denote the
absolute Frobenius morphism.
For a scheme $X$ endowed with a topology $\tau \in \{\zar, \nis, \et\}$, $X_\tau$ denotes the small $\tau$-site of $X$. For
a closed subset $Y \subset X$, and a
$\tau$-sheaf $\sF$, we let $H^*_{\tau, Y}(X, \sF)$ denote the $\tau$-cohomology
groups of $\sF$ with support in $Y$. If $\tau$ is fixed in a section, we shall
drop it from the notations of the cohomology groups. We shall let $D(X_\tau)$ denote the bounded derived category of sheaves on $X_\tau$.
If $X=\Spec A$, then the cohomology group $H^i_\tau(X,\sF)$ may also be written as $H^i_\tau(A,\sF)$. 

For any field extension $k'$ over $k$, $X(k')$ will denote the set of $k'$-rational points of $X$. $\Div(X)$ will denote the set of all Weil divisors of $X$. For a reduced effective Weil divisor $E$, $\Div_E(X)$ will denote the set of all Weil divisors, whose supports are inside $E$.

$\Z$ will denote the set of all integers and $\N$ (resp. $\N_0$) will denote the set of all positive (resp. non negative) integers. $\lfloor . \rfloor$ (resp. $\lceil .\rceil$) will denote the greatest (resp. least) integer function. For $\n=(n_1,\ldots,n_r)\in \Z^r$ and $p \in \N$, we define $\n/p:=(\lfloor n_1/p \rfloor,\ldots, \lfloor n_r/p \rfloor)$.

We shall let $\Ab$ denote the category of abelian groups. For an abelian group $A$, we shall write $\Tor^1_{\Z}(A, {\Z}/n)$ as
$_{n}A$ and $A/{nA}$ as $A/n$.
For $A, B \in \Ab$, the tensor product $A \otimes_{\Z} B$ will be written as $A \otimes B$.
We shall let $A\{p\}$ (resp. $A\{p'\}$) denote the
subgroup of elements of $A$ which are annihilated by a power of (resp.
an integer prime to) $p$. We let ${\bf Ab_p}$ be the category of $p$-primary torsion abelian groups. We define $A_\tor:= A\{p\} \cup A\{p'\}$.  We let $\Tab$ denote the category of topological abelian groups
with continuous homomorphisms. A locally compact Hausdorff topological group $G$ is called ``torsion-by-profinite" if there is an exact sequence
\begin{equation*}
  0 \to G_\pf \to G \to G_{\text{dt}} \to 0,
\end{equation*}
where $G_\pf$  is an open and profinite subgroup of $G$ and $G_{\text{dt}}$
is a torsion group with the quotient topology (necessarily discrete). The category of ``torsion-by-profinite" will be denoted as $\pfd$.

By an $n$-local field, we mean a complete discrete valued field whose residue field is an $(n-1)$-local field, where the $0$-local fields are defined to be finite fields. The $1$-local fields will also be mentioned as local fields.

 We shall write $H^q_\et(X, {\Q}/{\Z}(q-1))$ 
(resp. $H^q_\et(A, {\Q}/{\Z}(q-1))$) as $H^q(X)$ (resp. $H^q(A)$) for
any $q \ge 0$ (see \cite[\S~3.2, Defn.~1]{Kato80-2}).

For a Noetherian scheme $X$, we shall let $\pi^{\ab}_1(X)$ 
denote the abelianized {\'e}tale fundamental group of $X$ and consider it as a
topological group with its profinite topology. We shall write $\pi^{\ab}_1(\Spec(F))$
interchangeably as $G_F$, where the latter is the abelianized absolute Galois
group of a field $F$. We shall let $\Br(X)$ denote the cohomological Brauer group of $X$, which is defined as the {\'e}tale cohomology group $H^2_\et(X, \G_m)$. 

RLR is abbreviated as regular local ring. SNCD is abbreviated as simple normal crossing divisor.

\chapter*{Introduction}
\addcontentsline{toc}{chapter}{Introduction}
\pagenumbering{arabic}
\pagestyle{fancy}

\fancyhf{}

\fancyhead[LE,RO]{\thepage}
\fancyhead[RE]{}
\fancyhead[LO]{}

The {\'e}tale fundamental group and the Brauer group are two of the most extensively studied objects in algebraic and arithmetic geometry. The first one was introduced by Grothendieck in the 1960's in \cite{SGA1}. Given a connected scheme $X$ and a geometric point $\ov{x}$, the {\'e}tale fundamental group $\pi^{\et}_1(X,\ov{x})$ is defined as the inverse limit of the automorphism groups $\Aut_X(X_\alpha(\ov x))$, where $X_{\alpha}$ varies over all finite {\'e}tale covers of $X$ and $X_\alpha(\ov x)$ be the geometric fiber over $\ov x$. The {\'e}tale fundamental group of a field $k$ is the absolute Galois group $\Gal(\ov k /k)$. So the theory of {\'e}tale fundamental group can be considered as a generalization of the Galois theory of fields to the case of schemes. The maximal abelian quotient of $\pi^{\et}_1(X,\ov{x})$ is denoted by $\pi^{ab}_1(X,\ov{x})$.

On the other hand, the Brauer group of a field $k$, $\Br (k)$ was introduced by R. Brauer. It is defined as the group of equivalent classes of central simple algebras over $k$, with addition given by the tensor product of algebras. Another important interpretation of the Brauer group of a field $k$ is that it classifies the projective varieties over $k$ that become isomorphic to projective space when base changed to $\ov k$, the algebraic closure of $k$. Such a variety is called a Severi-Brauer variety. Brauer group was generalized from field to any commutative rings by Auslander and Goldman in \cite{AG-Brauer} and later to arbitrary schemes by Grothendieck, where one uses Azumaya algebras instead of central simple algebras and projective bundles which are {\'e}tale locally trivial instead of projective space in the previous definition of Brauer group of a field. This group is often known as the Brauer-Azumaya group denoted by $\Br_{Az}(X)$, for a scheme $X$ (cf. \cite[Chapter~3]{CTS}). On the other hand, another generalisation of the Brauer group of a field to a scheme is the cohomological Brauer group $\Br(X):= H^2_\et(X,\G_m)$. It is well known (cf. \cite[Thm.~4.2.1]{CTS}) that $\Br_{Az}(X)\inj \Br(X)_{\tor}$; and it is an isomorphism if $X$ is a quasi-projective scheme. Also, $\Br(X)$ is torsion if $X$ is a regular and integral scheme by \cite[Thm.~3.5.5]{CTS}.

 The goal of this thesis is to study the wild ramification part of $\pi^{ab}_1(-)$,
$\Br(-)$ and more generally of the $p$-adic {\'e}tale cohomology of a scheme which is a complement of a simple normal crossing divisor in a proper smooth variety of characteristic $p>0$. In an attempt to do so, we study a theory of filtered de Rham-Witt complex and its role in Kato's local ramification theory \cite{Kato-89}.
 
\subsection*{Ramification theory} 
Let $k$ be a field and $X$ be a proper smooth $k$ scheme. The group $\pi^{ab}_1(X)$ ( or $\Br(X)$) is often known as the unramified (abelian) {\'e}tale fundamental group (resp unramified Brauer group). However, if $U \inj X$ is an open subscheme of $X$ whose complement is of codimension $1$ (if codimension $>1$, then $\pi^{ab}_1(-) \text{ and } \Br(-)$ are unchanged by purity), then to study $\pi^{ab}_1(U)$ (resp. $\Br(U)$), one needs to consider all the abelian {\'e}tale covers (resp Brauer classes) of $U$ with ramifications along its complement. This is where the Ramification theory comes into the picture.

 Let $K$ be a Henselian discrete valued field. If $K$ is a local field, the ramification theory of $G_K$ is classical (\cite{Serre-LF}). If residue field of $K$ is not necessarily finite, then the ramification theory for $H^1_\et(K,\Q/\Z)=\Hom_{cts}(G_K, \Q/\Z)$ was developed by Brylinski \cite{Brylinski}, Kato \cite{Kato-89}, Matsuda \cite{Matsuda} (see also \cite{Abbes-Saito}). As in the classical case, these authors defined a filtration $\{\Fil_nH^1_\et(K,\Q/\Z)\}_n$, that classifies all abelian {\'e}tale covers of $K$ whose ramification at the closed point of $\sO_K$ is bounded by $n \ge 0$. More generally, for any $q \ge 0$, Kato  \cite{Kato-89} also defined a ramification filration of $H^{q+1}(K)$ (which we write as $\Fil^{bk}_nH^{q+1}(K)$), where, in modern language, $H^{q+1}(K)$ is same as $ H^{q+1}(K,\Q/\Z(q))$, the {\'e}tale motivic cohomology of Suslin and Voevodsky. These filtrations are usually difficult to study and also unlike the unramified case, they don't have any cohomological interpretation in general.

 Now to study global ramification theory, one can use Kato's local ramification theory to define a filtration of $H^1_\et(U, \Q/\Z)$, namely $\{\Fil_DH^1_\et(U,\Q/\Z)\}_D$ so that $$H^1_\et(U, \Q/\Z)=\varinjlim\limits_{D}\Fil_DH^1_\et(U,\Q/\Z),$$
 where $D$ varies over all effective divisors whose support is contained in $X\setminus U$. The profinite group $\pi^{ab}_1(X,D):=(\Fil_DH^1_\et(U,\Q/\Z))^\vee$ classifies all the abelian {\'e}tale covers  of $U$ whose ramifications along its complement is bounded by $D$. Here $(-)^\vee$ means the discrete dual. One can also define a ramification filtration for $\Br(U)$ (see \defref{defn:2}) in a similar way.

 In \cite{Kerz-Saito-ANT},  Kerz and Saito constructed a two term complex (which we write as $W_m\sF^{0,\bullet}_D$) in $\sD^b(X_\et,\Z/p^m)$, such that $\H^1_\et(X, W_m\sF^{0,\bullet}_D)\cong \Fil_DH^1_\et(U,\Z/p^m)= \ _{p^m}(\Fil_DH^1_\et(U,\Q/\Z))$, when $E= X \setminus U$ is a simple normal crossing divisor and $X$ is smooth over a finite field. This allowed them to get a Lefschetz-type result for $\pi^{ab}_1(X,D)$ and the higher dimensional ramified class field theory over finite fields in \cite{Kerz-Saito-Duke}.

 One of the objectives of this thesis is to study the ramification theory of a regular $F$-finite scheme of characteristic $p>0$ using Kato's filtration and give several applications. We recall that a scheme $X$ is called $F$-finite, if the absolute Frobenius map $F: X \to X$ is a finite morphism. Let $X$ be a regular $F$-finite scheme of dim $N$ and $E=\stackrel{r}{\underset{i =1}\cup}E_i$ be a simple normal crossing divisor, where $E_i$'s are the irreducible components. Let $U=X \setminus E$ and $j: U \inj X$. We fix an $m \ge 1$. Let $D=\sum\limits_{1\le i \le r}n_iE_i$ with $n_i \ge 0$ and $K_i= Frac(\sO_{X,x}^h)$ if $\ov{\{x\}}=E_i$. Also we set  $H^{q+1}_m(K_i)=H^{q+1}(K_i,\Z/p^m(q))$. We define $$\Fil_D^{\log}H^{q+1}(U, \Z/p^m(q)):= \Ker \left(H^{q+1}(U,\Z/p^m(q)) \to \bigoplus \limits_{1\le i \le r} \frac{H^{q+1}_m(K_i)}{\Fil^{\bk}_{n_i}H^{q+1}(K_i) \cap {H^{q+1}_m(K_i)}}\right).$$ 
  The arrow on the right hand side is induced by canonical restrictions.
 
 The first main result of this thesis is the following theorem whose $q=0$ case is the Kerz-Saito's result mentioned above :
 \begin{thm}\label{thm:1}(\thmref{thm:H^1-fil})
     For any $q \ge 0$, there exists a two term complex $W_m\sF^{q,\bullet}_D$ of {\'e}tale sheaves such that 
     $$\H^1_\et(X, W_m\sF^{q,\bullet}_D) \cong \Fil_DH^{q+1}(U, \Z/p^m(q)).$$
 \end{thm}
 
 \begin{defn}\label{defn:2}
     We let $\Br^\divv(X|D)$ denote the subgroup of $\Br(U)$ consisting of elements $\chi$
  such that for every $x \in D^{(0)}$, the image $\chi_{x}$ of
$\chi$ under the canonical map $\Br(U) \to\Br(K_{x})$ 
lies in $\Fil_{n_{x}} \Br(K_{x})$, where $n_x$ is multiplicity of $D$ at $x$. 
\end{defn}
As an immediate corollary of \thmref{thm:1}, we have the following. 
\begin{cor}\label{cor:Br}(\corref{cor:F^q})
 We have an exact sequence    
\begin{equation*}
        0 \to \Pic(U)/p^m \to \H^1_\et(X, W_m\sF^{1,\bullet}_{D}) \to \ _{p^m}\Br^\divv(X|D) \to 0.
    \end{equation*}
\end{cor}
The above result can be thought of as the ramified version of the following well-known exact sequence :
$$0 \to \Pic(X)/p^m \to H^1(X, W_m\Omega^1_{X, \log}) \to \ _{p^m}\Br(X) \to 0,$$
where $\ _{p^m}A:=\Ker (A \xrightarrow{p^m} A)$. \corref{cor:Br} will have applications to prove a Lefschetz theorem for $\Br^\divv(X|D)$ and also a ramified duality with modulus for curves over local fields in the later sections of this thesis. 

The following theorem in local ramification theory is the key result for proving \thmref{thm:1}.
\begin{thm}(\thmref{thm:Kato-fil-10})
    Let $K$ be an $F$-finite Henselian discrete valued field of characteristic $p>0$. Then, for all $n, q \ge 0$ there exists a two term complex $W_m\sF^{q,\bullet}_n$ in $\sD^b(K_\et,\Z/p^m)$ such that 
    $$\H^1_\et(K, W_m\sF^{q,\bullet}_n)\cong \ _{p^m}\Fil_n^{bk}H^{q+1}(K). $$
\end{thm}

\subsection*{Filtered de Rham-Witt complex}
To construct the two term complex $W_m\sF^{q,\bullet}_D$, we study a theory of filtered de Rham-Witt complex of $X$. We continue with the same notations. Also, we allow $n_i$ in the definition of $D$ to be any element in $\Z$.

For any $q \ge 0$, we first define a quasi-coherent sheaf of $W_m\sO_X$ modules,   $\Fil_DW_m\Omega^q_U$ $\subset j_*W_m\Omega^q_U$ (see \defref{defn:Log-fil-3}), which satisfies the following properties (\lemref{lem:Log-fil-4}):
\begin{enumerate}
           \item
    $d(\Fil_D W_m \Omega^q_U) \subset \Fil_D W_m \Omega^{q+1}_U$,
 \item
    $F(\Fil_D W_{m+1} \Omega^q_U) \subset \Fil_D W_m \Omega^{q}_U$,
  \item
    $V(\Fil_D W_{m} \Omega^q_U) \subset \Fil_D W_{m+1} \Omega^{q}_U$,
  \item
    $R(\Fil_D W_{m+1} \Omega^q_U) = \Fil_{D/p} W_m \Omega^{q}_U$, where $D/p= \sum\limits_i\lfloor n_i/p\rfloor E_i$,
   \item    $\varinjlim_n \Fil_{nE} W_m \Omega^q_U \xrightarrow{\cong} j_*W_m\Omega^{q}_U$. 
      \end{enumerate}

Also we have $\Fil_D W_{m} \Omega^q_U \subset \Fil_{D'} W_{m} \Omega^q_U$, if $D' \ge D$ and $\supp(D') \subset E$. Moreover, $\Fil_D W_{m} \Omega^q_U \subset  W_{m} \Omega^q_X$ if $D < 0$ (i.e, all its components have multiplicity $<0$). We also remark that, for $D=$ empty divisor, $\Fil_DW_m\Omega^\bullet_U =W_m\Omega^\bullet_X(\log E)$, the de Rham-Witt sheaf for the canonical log structure given by $E$ (cf. \cite{Hyodo-Kato}). However, if $U=X$, then $\Fil_DW_m\Omega^\bullet_U = W_m\Omega^\bullet_X$.

At this point, we refer the reader to \cite{JSZ} and \cite{GK-Duality}, where the authors studied a similar filtered de Rham-Witt complex tailored to their respective objectives. See also \cite{Ren-Rulling}, where a version of relative de Rham-Witt complex with modulus has been studied from the perspective of motives with modulus.

The following technical result plays an important role in studying $\Fil_DW_m\Omega^q_U$.
\begin{thm}(\thmref{thm:Global-version}(2))
    We have a short exact sequence of Zariski sheaves of $W_m\sO_X$-modules
    \[
    0 \to V^{m-1}(\Fil_D\Omega^q_U) + dV^{m-1}(\Fil_D\Omega^{q-1}_U) \to
    \Fil_DW_m\Omega^q_U \xrightarrow{R}  \Fil_{D/p}W_{m-1}\Omega^q_U \to 0.
    \]
\end{thm}
The above theorem is an analogue of \cite[I.3.2]{Illusie} and \cite[1.16]{Lorenzon}.

Recall that $Z_1W_m\Omega^q_U= F(W_{m+1}\Omega^q_U)= \Ker (F^{m-1}d:W_m\Omega^q_U \to \Omega^{q+1}_U)$. We define 
$$Z_1 \Fil_DW_m\Omega^q_U:=  \Fil_DW_m\Omega^q_U \cap j_*Z_1W_m\Omega^q_U.$$
The following key proposition, which is analogous to \cite[Lemma~4.1.3]{Kato-Duality}, allows one to define the two term complex $W_m\sF^{q,\bullet}_D$. Recall that there is a higher Cartier operator $C: Z_1W_m\Omega^q_U \to W_m\Omega^q_U$.
\begin{prop}(\lemref{lem:Complete-6})
    The above higher Cartier operator restricts to a map
    $$C:Z_1 \Fil_DW_m\Omega^q_U \to  \Fil_{D/p}W_m\Omega^q_U.$$ 
    Moreover, there are isomorphisms
    \[
    \ov{F} \colon \Fil_{{D}/p}W_{m}\Omega^q_U \xrightarrow{\cong}
    \frac{Z_1\Fil_{D}W_m\Omega^q_U}{dV^{m-1}(\Fil_{D}\Omega^{q-1}_U)}; \ \
    \ov{C} \colon  \frac{Z_1\Fil_{D}W_m\Omega^q_U}{dV^{m-1}(\Fil_{D}\Omega^{q-1}_U)}
     \xrightarrow{\cong} \Fil_{{D}/p}W_{m}\Omega^q_U.
    \]

\end{prop}
As a result, we have
\begin{defn}(\defref{defn:1-c complex}) 
For $D \ge -E_1-E_2- \cdots -E_r$, we define
    $$W_m\sF^{q,\bullet}_{D} :=
\left(Z_1\Fil_{D}W_m\Omega^q_U \xrightarrow{1-C} \Fil_{D}W_m\Omega^q_U\right).$$
\end{defn}

The complex $W_m\sF^{q,\bullet}_{D}$ can be considered as a ramified analogue of $W_m\Omega^q_{X,\log}$. Here, $W_m\Omega^q_{X,\log}$ is the logarithmic Hodge-Witt sheaf on $X_\et$, defined as the image of the map 
$$\dlog:\sK^M_{q,X} \to W_m\Omega^q_{X}$$
$$\dlog(u_,,\cdots,u_q)=\dlog[u_1]_m \wedge \cdots \wedge \dlog[u_q]_m,$$
where $\sK^M_{q,X}$ be the Milnor K-theory sheaf. The following result is the analogue of the classical $p-R$ exact sequence for $W_m\Omega^q_{X, \log}$ (and also \cite[Thm~1.1.6]{JSZ}).
\begin{thm}(\thmref{thm:Global-version}(11))
    There exists a distinguished triangle in the derived category of Zariski sheaves of abelian groups on $X$
    
    \[
     W_1\sF^{q, \bullet}_{D} \xrightarrow{V^{m-1}} W_m\sF^{q,\bullet}_{D}
    \xrightarrow{R} W_{m-1}\sF^{q,\bullet}_{D/p} \xrightarrow{+}  W_1\sF^{q, \bullet}_{D} [1].
    \]

\end{thm}

\subsection*{Relation with Jannsen-Saito-Zhao's sheaf}
In this section, we continue with the additional assumption that $X$ is smooth and finite-type over an $F$-finite field $k$.

Let's continue with the notations of the previous subsections. Moreover, we assume $D >0$. Then $D_\red = E$. Jannsen, Saito and Zhao in \cite{JSZ} have defined a subsheaf $W_m\Omega^q_{X|D, \log} \subset W_m\Omega^q_{X,\log}$ in $X_\et$ and shown an isomorphism
$$\dlog:\sK^M_{q,X|D}/(p^m\sK^M_{q,X}\cap \sK^M_{q,X|D}) \xrightarrow{\cong} W_m\Omega^q_{X|D,\log}.$$
Here,  $\sK^M_{q,X|D}$ is the relative Milnor K-theory sheaf in $X_\et$ defined in \cite{Rulling-Saito} as the image of the cup product
$$\Ker(\sO_X^\times \to \sO_E^\times) \otimes j_*\sK^M_{q-1,U} \to j_*\sK^M_{q,U}, \ \ \text{where }j:U \inj X.$$
This log-version of relative Milnor K-theory sheaf was also considered in \cite{Kato80-1}, \cite{Kato80-2}, \cite{Kato-89}, \cite{Kato-Inv}, \cite{Kato83} to establish Higher dimensional local class field theory and abelianized ramification theory. Part (1) of our next result relates Jannsen-Saito-Zhao's sheaf with our filtration in the following way :
\begin{lem}\label{lem:log}(\lemref{lem:Complex-6})
    We have exact sequences
    \begin{enumerate}
        \item \ \ $0 \to W_m\Omega^q_{X|D,\log} \to \Fil_{-D}W_m\Omega^q_U
    \xrightarrow{1 -\ov{F}}
    \frac{\Fil_{-D}W_m\Omega^q_U}{dV^{m-1}(\Fil_{-D}\Omega^{q-1}_U)}$,
      \item \ \ $0 \to j_*(W_m\Omega^q_{U,\log}) \to Z_1W_m\Omega^q_X(\log E)
    \xrightarrow{1 - C}
    W_m\Omega^q_X(\log E) \to 0$.
    \end{enumerate}
\end{lem}
The item (1) in the above lemma is already known by \cite{JSZ} if $m=1$ or $D$ is of the form $p^{m-1}D'$. Moreover, it is also right exact in the first case.

Next, we present some applications of the above results.

\subsection*{First application: Ramified duality over finite fields with fixed modulus}
We get the following refined version of Jannsen-Saito-Zhao's ramified duality for smooth schemes over finite fields. We write  $D_{\n}= \sum\limits_{1\le i \le r}n_iE_i$, for $\n=(n_1,\ldots,n_r) \in \Z^r$. We say $\n \ge \un{0}$ if $n_i \ge 0$ for all $i$.
\begin{thm}(\thmref{thm:Perfect-finite},\corref{cor:Duality-log})
   Let $X$ be a connected smooth projective variety over a finite field $k$ and dim $X=N$. Then for all $\n \ge \un{0}$, we have a perfect pairing of finite abelian groups 
    \begin{equation}\label{eqn:JSZ-2}
 \H^i_\et(X, W_{m}\sF^{q, \bullet}_{D_{\un{n}}}) \times
    H^{N+1-i}_\et(X, W_{m}\Omega^{N-q}_{X|D_{\un{n}+1}, \log})  \xrightarrow{\cup}
    {\Z}/{p^m}
    \end{equation}
\[
  H^i_\et(X, j_*(W_m\Omega^q_{U, \log})) \times H^{N+1-i}_\et(X, W_m\Omega^{N-q}_{X|E,\log})
  \xrightarrow{\cup} {\Z}/{p^m}.\]
\end{thm}
In view of \lemref{lem:log}(2), the second duality of the above theorem can be considered as the Milne's duality for the log scheme $X$, whose log structure is given by $E$ (cf. \cite{Hyodo-Kato}). Now by taking limits over $\n$ in the top pairing of \eqref{eqn:JSZ-2}, we recover the ramified duality of Jannsen-Saito-Zhao : 
\begin{equation}\label{eqn:JSZ}
    H^i_\et(U,W_m\Omega^q_{U,\log}) \times \varprojlim\limits_{D}H^{N+1-i}_\et(X,W_m\Omega^{N-q}_{X|D,\log}) \to \Z/p^m.
\end{equation}

\subsection*{Second application: Duality for open curves over local fields}
Now we proceed to our next main result which is the duality theorem for open curves over local fields of characteristic $p>0$. Since there is no duality theorem for $p$-adic {\'e}tale motivic cohomology for a smooth projective variety $X$ over local fields in the literature, except dim $X=1$ (\cite[Prop~4]{Kato-Saito-Ann}, \cite[Thm~4.7]{KRS}), we consider only dim $1$ case. In this case, the group $H^i_\et(X, W_m\Omega^q_{X|D_{\n+1}, \log})$ is not finite in general. Moreover, it might have non-discrete topology induced from the adic topology of the base field $k$.  Considering these non trivial topologies, we have the following theorem.

\begin{thm}\label{thm:curve}(\thmref{thm:Duality-curve})
     Let $X$ be a smooth projective geometrically connected curve over a local field $k$ of characteristic $p>0$. Then for all $\n \ge \un{0}$, we have perfect pairings of topological abelian groups
   \begin{enumerate}
       \item \hspace{2cm}  $H^i_\et(X, W_m\Omega^2_{X|D_{\n+1}, \log})^\pf \times
    \H^{2-i}_\et(X, W_m\sF^{0,\bullet}_{D_{\n}}) \to {\Z}/{p^m},$ \\
    where the group $H^i_\et(X, W_m\Omega^2_{X|D_{\n+1}, \log})$ is Hausdorff for all $i$ (also profinite if $i \neq 1$).

    \item  \hspace{2cm}$H^i_\et(X,W_m\Omega^1_{X|D_{\n+1},\log}) \times \H^{2-i}_\et(X,W_m\sF^{1,\bullet}_{D_{\n}}) \to \Z/p^m, \ i \neq 1,$ \\
    where $H^i_\et(X,W_m\Omega^1_{X|D_{\n+1},\log})$ has profinite topology if $i=0$ and discrete topology if $i=2$,

    \item \hspace{2cm} $H^1_\et(X,W_m\Omega^1_{X|D_{\n+1},\log})^\wedge \times \H^1_\et(X,W_m\sF^{1,\bullet}_{D_{\n}}) \to \Z/p^m,$ \\ 
    where $H^1_\et(X,W_m\Omega^1_{X|D_{\n+1},\log})$ is an extension of a discrete group by an adic topological group and $\H^1_\et(X,W_m\sF^{1,\bullet}_{D_{\n}})$ is an extension of a discrete group by a profinite group and $(H^1_\et(X,W_m\Omega^1_{X|D_{\n+1},\log}))^\wedge$ is the ``discrete by profinite" completion of $H^1_\et(X,W_m\Omega^1_{X|D_{\n+1},\log})$.
   \end{enumerate}
   
\end{thm}
Here, for a topological group $G$, $G^\pf$ means the topological profinite completion of $G$, (i.e. $G^\pf=\varprojlim\limits_{\substack{\{U :\ [G:U] <\infty}\}}G/U$, where $U$ is an open subgroup of $G$).
   Now, by taking limits over $\n$, we have the following corollary.
   \begin{cor}(\thmref{thm:Duality-curve}(4))
       Under the same hypothesis of \thmref{thm:curve}, we have perfect pairing of topological abelian groups
       \begin{enumerate}
           \item \hspace{1cm}  $\varprojlim\limits_\n H^i_\et(X, W_m\Omega^2_{X|D_{\n+1}, \log})^\pf \times
    H^{2-i}_\et(U,\Z/p^m) \to {\Z}/{p^m},$
             \item  \hspace{1cm}$\varprojlim\limits_\n H^i_\et(X,W_m\Omega^1_{X|D_{\n+1},\log}) \times H^{2-i}_\et(U,W_m\Omega^1_{U,\log}) \to \Z/p^m, \ i \neq 1,$ 
             \item  \hspace{1cm}$\varprojlim\limits_\n H^1_\et(X,W_m\Omega^1_{X|D_{\n+1},\log})^\wedge \times H^{1}_\et(U,W_m\Omega^1_{U,\log}) \to \Z/p^m.$ 
       \end{enumerate}
   \end{cor}
We note that the groups in the pairing (1) and (2) above are zero if $i=0$.
As an application of the above theorem, we can deduce the $p$-primary part of the class field theory for $U$. This is proved in \cite{CFT}. We also remark that the class field theory for $X$ was established by Bloch \cite{Bloch-CFT}, S. Saito \cite{Saito-JNT}, Yoshida \cite{Yoshida03} and the prime to $p$-part of the class field theory for $U$ was established by T. Hiranouchi in \cite{Hiranouchi-2}.

\subsection*{Third application: Lefschetz hypersurface theorem}

Finally, we give one more application of the theory of filtered de Rham-Witt complex, namely, the Lefschetz hypersurface theorem.

\begin{thm}\label{thm:Lef}(\thmref{thm:Lefschetz})
    Let $X \inj \P^L_k$ be a smooth projective variety over an $F$-finite field $k$ of dim $N$ and $D$ as before. Let $Y$ be a smooth hypersurface section of $X$ that intersects $D$ transversally. Let $D'=D \times_X Y$. If $Y$ is sufficiently ample (w.r.t $D$ in (1) and (2)), then for all $m \ge 1$, the group homomorphisms
 \begin{enumerate}
     \item  \hspace{.5cm} $\H^i_\et(X,W_m\sF^{q,\bullet}_{D,U}) \to \H^i_\et(Y,W_m\sF^{q,\bullet}_{D',U'})$

    \item \hspace{.5cm}
    $H^i_\et(X,W_m\Omega^q_{X|D,\log}) \to H^i_\et(Y,W_m\Omega^q_{Y|D',\log})$

    \item \hspace{.5cm}
    $H^i_\et(X,W_m\Omega^q_{X,\log}) \to H^i_\et(Y,W_m\Omega^q_{Y,\log})$
    
 \end{enumerate}
   are isomorphisms if $i+q \le N-2$ and injective if $i+q=N-1$.
\end{thm}

Recall that the Lefschetz theorem for the {\'e}tale fundamental groups of smooth projective schemes over a
field was proven by Grothendieck (see \cite[Expos´e XII, Corollaire 3.5]{SGA2}). Kerz-Saito \cite{Kerz-Saito-ANT} also proved a similar Lefschetz theorem for $\pi^{ab}_1(X,D)$, the ramified {\'e}tale fundamental groups (with modulus) of smooth projective scheme $X$. To the best of our knowledge, a similar Lefschetz theorem for the Brauer groups is not known. As an application of the above theorem, we get the following Lefschetz theorem for $\Br^\divv(X|D)$ and $\Br(X)$.

\begin{thm}(\thmref{thm:Br-lef})
Let's assume $N \ge 4$ in \thmref{thm:Lef}.
  Then, for sufficiently ample $Y$ (w.r.t $D$ in second case), we have $$\Br(X)\{p\} \xrightarrow{\cong} \Br(Y)\{p\}, \ \ \ \Br^\divv(X|D)\{p\} \xrightarrow{\cong} \Br^\divv(Y|D')\{p\}.$$
       Moreover, if $k$ is either a  finite field, local field, or algebraically closed field, then we have
       $$\Br(X) \xrightarrow{\cong} \Br(Y), \ \ \ \Br^\divv(X|D) \xrightarrow{\cong} \Br^\divv(Y|D').$$
\end{thm}
As a corollary, we get the following result whose prime to $p$-part is already known by \cite[Cor~5.5.4]{CTS}.
\begin{cor}(\corref{cor:Br-ci})
    Let $k$ be any field and $X$ be any smooth complete intersection in $\P^N_k$ of dim $d >2$. Then, we have $\Br(k) \xrightarrow{\cong}\Br(X)$.
\end{cor}
\section*{Further applications}
The filtered de Rham-Witt complex studied here has further applications in understanding the log-crystalline cohomology of $U$, a topic that will be discussed in \cite{KM-crys}. Additionally, the duality results established in this thesis can also be extended to curves over higher local fields. These duality results can, in turn, be used to prove class field theory for such curves, a topic that will be explored in \cite{KM-classfield}.  
\section*{Organization of the thesis}
The thesis contains nine chapters, which can be divided into four parts. We give an outline of the content here.

Part 1 is the Chapter 1, which recalls the preliminaries required in the later Chapters. We recall some basic properties of the de Rham-Witt complex and the logarithmic Hodge-Witt sheaves for regular schemes in \secref{sec:de-rham-witt}. In \secref{sec:LogDR}, we discuss a few properties of the de Rham-Witt complex of Hyodo-Kato for log schemes whose log structure is given by a simple normal crossing divisor (SNCD).

Part 2 comprises Chapters 2, 3, and 4. Chapter 2 studies the basic definitions and properties of the theory of filtered de Rham-Witt complex. There we define the sheaf $\Fil_DW_m\Omega^q_U$, where $U$ is the complement of a simple normal crossing divisor $E$ in a regular scheme $X$. 

In Chapter 3, we study the structure of the above sheaf and prove various results when $X$ is spectrum of a ring of the form $S[[\pi]]$ ($\pi$ is a variable). 

In Chapter 4, we study the structure of the above sheaf for regular local rings. We also extend various results of the classical de Rham-Witt complex to the filtered setup when $X$ is spectrum of an $F$-finite regular local ring, and finally to $F$-finite regular schemes.

Part 3 consists of Chapter 5, which establishes a cohomological description of Kato's local and global ramification filtration using the filtered de Rham-Witt complex studied in Chapter 4. \S~\ref{sec:loc-Kato} deals with the local Kato filtration and \S~\ref{sec:the group} deals with the global Kato filtration.

Part 4 is composed of Chapters 6-9, which consist of applications of the concepts developed in Parts 2 and 3.

In Chapters 6-8, we prove duality theorems for logarithmic Hodge-Witt sheaves with modulus over finite fields and 1-local fields. In \S~\ref{sec:Complexes}, we prove a relation between Jannsen-Saito-Zhao's sheaf $W_m\Omega^q_{X|D,\log}$ and $\Fil_{-D}W_m\Omega^q_U$, define various complexes. In \S~\ref{sec:Pairing}, we define pairings between various complexes, which induce the pairings in the duality theorem. In \S~\ref{sec:DT}, we prove the duality theorems over finite fields and in \S~\ref{chap:duality-localfields}, we prove duality theorems over 1-local fields. 

Finally, in Chapter 9, we prove a Lefschetz hypersurface theorem for the logarithmic Hodge-Witt cohomology with modulus. As an application, we prove a Lefschetz theorem for unramified Brauer group and ramified Brauer group (with modulus) in \S~\ref{sec:Br-lef}.

\pagestyle{fancy}

\fancyhf{}

\fancyhead[LE,RO]{\thepage}
\fancyhead[RE]{\nouppercase{\leftmark}}
\fancyhead[LO]{\nouppercase{\rightmark}}
\renewcommand{\headrulewidth}{0pt}
\fancypagestyle{plain}{ %
\fancyhf{} 
\renewcommand{\headrulewidth}{0pt} 
\renewcommand{\footrulewidth}{0pt}}

\chapter{Preliminaries}
\section{de Rham-Witt Complex}\label{sec:de-rham-witt}
In this section, we shall recall definitions of the de Rham-Witt complex and logarithmic Hodge-Witt sheaf for a regular scheme over a field of characteristic $p>0$. We shall also review some basic properties of these sheaves, which are classically known in the case of smooth schemes over perfect fields. These properties were later extended to arbitrary regular schemes in \cite{Shiho} using the following two theorems.

\begin{thm}\label{thm:Popescu}(cf. \cite[Thm~2.5]{Popescu})
   Let $u: A \to A'$ be a regular morphism of noetherian rings. Then $A'$ is a filtered inductive limit of smooth $A$ algebras. 
\end{thm}
Recall that a morphism of schemes is called regular if it is flat and fibers are geometrically regular. For example the canonical map $k[x_1,\ldots,x_n] \to k[[x_1,\ldots, x_n]]$ is regular. More generally, the map $A \to \wh A$ is regular for any excellent local ring. As a corollary of ~\thmref{thm:Popescu}, we get
\begin{cor}
    Any regular local ring $A$ of characteristic $p$ is filtering inductive limit $\varinjlim\limits_{\lambda} A(\lambda)$ of finitely generated smooth $\F_p$ algebras. 
\end{cor}

\begin{thm}\label{thm:lim-cohomology}(cf. \cite[VII, Thm~5.7]{SGA4}, \cite[Thm~6.6]{Panin})
    Let $\{X_i\}_{i \in I}$ be a filtering projective system of Noetherian schemes such that the transition maps are affine and $\varprojlim\limits_{i}X_i=X$ is noetherian. Let $\{F_i\}_i$ be a compatible system of sheaves on $\{X_{i,\tau}\}$ ($\tau = \zar \text{ or } \et$) with limit sheaf $F$ on $X$. Then we have 
    $$H^n(X_\tau, F) \cong \varinjlim\limits_{i}H^n(X_{i,\tau},F_i).$$
\end{thm}

Recall from \cite{Illusie}, the de Rham-Witt complex $W_m\Omega^{\bullet}_X$ over a scheme $X$ of characteristic $p$ is a strictly anti-commutative differential graded complex :
$$W_m\Omega^0_X= W_m\sO_X \xrightarrow{d}W_m\Omega^1_X \xrightarrow{d} \cdots \xrightarrow{d} W_m\Omega^i_X \xrightarrow{d}\cdots .$$
$W_m\sO_X$ is called the sheaf of Witt vectors of length $m$ and $W_mX= (|X|,W_m\sO_X)$ be the corresponding Witt scheme. Let $[-]_m:\sO_X\to W_m\sO_X$ be the multiplicative Teichm{\"u}llar function given by $[x]_m=(x,0,\dots, 0)$ for any $x \in \sO_X$. Let $F :W_m\Omega^q_X \to W_{m-1}\Omega^q_X$, $V: W_m\Omega^q_X \to W_{m+1}\Omega^q_X$ and $R:W_m\Omega^q_X \to W_{m-1}\Omega^q_X$ be the Frobenius, Verschiebung and projection, respectively (cf. \cite{Illusie}). The homomorphisms $F, V$ and $R$, satisfy the following properties.
\begin{enumerate}
    \item $FV=VF=p$,
    \item $FdV=d$, $dF=pFd$, $pdV=Vd$,
    \item $Fd[a]_m=[a]_m^{p-1}d[a]$, for $a \in \sO_X$,
    \item $V(F(x)y)=xV(y)$, $V(xdy)=VxdVy$ for $x, y \in W_*\Omega^\bullet_X$,
    \item $Rd=dR$, $RV=VR$, $RF=FR$.
\end{enumerate}
Let $H^q_\dR(X)$ denote the $q$-th cohomology of the de Rham complex of $X$.
Let $C^{-1}: \Omega^q_X \to \un{H}^q_\dR(X)$ be the inverse Cartier homomorphism characterized by the following properties.
\begin{enumerate}
    \item $C^{-1}(x\omega) = x^p \omega, \text{ for } x \in \sO_X, \ \omega \in \Omega^q_X ,$
    \item $C^{-1}(dx_1 \cdots dx_q)=(x_1\cdots x_q)^{p-1}dx_1 \cdots dx_q, \text{ for } x_i\in \sO_X.$
\end{enumerate}
We let $Z_1\Omega^q_X = \Ker(d \colon \Omega^q_X \to \Omega^{q+1}_X)$
  and $B_1\Omega^q_X = {\rm Image}(d \colon \Omega^{q-1}_X \to
  \Omega^{q}_X)$.
\begin{prop}\label{prop:cartier-1}(\cite[Prop~2.5]{Shiho})
 If $X$ is a regular scheme over $\F_p$, then $C^{-1}$ is isomorphism.   
\end{prop}
\begin{proof}
    If $X$ is smooth scheme, then this is well known. To prove for general case, we may assume $X$ is strictly Henselian local. Now the proposition follows from ~\thmref{thm:Popescu} and \thmref{thm:lim-cohomology}.  
\end{proof}

We write $C$ for the composite map $Z_1\Omega^q_X \surj  \un{H}^q_\dR(X) \xrightarrow{(C^{-1})^{-1}}\Omega^q_X$.  Assuming the $\sO_X$-modules
$Z_i\Omega^q_X$ and $B_i\Omega^q_X$ are defined for some $i \ge 1$, we let
$Z_{i+1}\Omega^q_X$ (resp. $B_{i+1}\Omega^q_{X}$) be the inverse image of $Z_i\Omega^q_X$ (resp. $B_{i}\Omega^q_{X}$) under the
composite map $Z_1\Omega^q_X \surj  H^q_\dR(X)
\xrightarrow{C}\Omega^q_{X}$ such that there is an induced isomorphism $C^{-i}:\Omega^q_X \xrightarrow{\cong}Z_i\Omega^q_X/B_i\Omega^q_X$, for $i \ge 1$. As a result, we get the following corollary.

\begin{cor}\label{cor:cartier-1}
    We have the exact sequences
    \begin{equation}\label{eqn:cartier-1.1}
        0 \to Z_{m+1}\Omega^q_X \to Z_m\Omega^q_X \xrightarrow{dC^m} B_1\Omega^{q+1}_X \to 0,
    \end{equation}
    \begin{equation}\label{eqn:cartier-1.2}
        0 \to B_{m}\Omega^q_X \to B_{m+1}\Omega^q_X \xrightarrow{C^m} B_1\Omega^{q}_X \to 0,
    \end{equation}
    for $m \ge 1, q \ge 0$. 
\end{cor}
\begin{proof}
    Using the Cartier isomorphism, the proof is same as given in \cite[Lem~1.14]{Lorenzon}. So we omit the details.
\end{proof}

\begin{prop}\label{prop:basics}(cf. \cite{Illusie})
    Let $X$ be a regular scheme over $\F_p$. Fix $m \ge 0$ and $q \ge 0$. Then we have
    \begin{enumerate}
        \item $0 \to V^nW_{m-n}\Omega^q_X+dV^nW_{m-n}\Omega^{q-1}_X \to W_m\Omega^q_X \xrightarrow{R^{m-n}} W_n\Omega^q_X \to 0$ is exact for $m \ge n \ge 0$;
        \item $\Ker\left( p^n : W_m\Omega^q_X \to W_m\Omega^q_X \right)=\Ker \left(R^n : W_m\Omega^q_X \to W_{m-n}\Omega^q_X \right)$, for $n \ge 1$;
        \item  $\Ker \left( V^m:\Omega^q_X \to W_{m+1}\Omega^q_X \right) = B_m\Omega^q_X$;
        \item $\Ker \left( V^m:\Omega^q_X \to W_{m+1}\Omega^q_X/dV^m\Omega^{q-1}_X \right) = B_{m+1}\Omega^q_X = F^{m}dW_{m+1}\Omega^{q-1}_X $;
        \item $\Ker \left( dV^m:\Omega^{q-1}_X \to W_{m+1}\Omega^q_X \right) = Z_{m+1}\Omega^{q-1}_X $;
        \item $\Ker \left( dV^m:\Omega^{q-1}_X \to W_{m+1}\Omega^q_X /V^m\Omega^q_X\right) = Z_m\Omega^{q-1}_X = F^mW_{m+1}\Omega^{q-1}_X$;
        \item $\Ker(F^l \colon W_{m+l}\Omega^q_X \to  W_{m}\Omega^q_X)
= V^m(W_{l}\Omega^q_K)$ for all $l \ge 1$;
        \item $\Ker(V \colon W_{m}\Omega^q_X \to W_{m+1}\Omega^{q}_X)
    = dV^{m-1}(\Omega^{q-1}_X)$.
    \end{enumerate}
\end{prop}
\begin{proof}
    All these statements are true if $X$ is smooth. See  \cite[Propositions~I.3.2, I.3.4]{Illusie} for (1) and (2), respectively. For (3)-(6), use Theorem~I.3.8 of loc. cit.  together with (I.3.21.1.4) and (I.3.21.1.3) of loc. cit. for the second equality of (4) and (6), respectively. For (7) and (8), use (I.3.21.1.2) and (I.3.21.1.4), respectively.  For the general case, we may assume $X$ is strictly local and reduces to the smooth case by ~\thmref{thm:Popescu} and \thmref{thm:lim-cohomology}.
\end{proof}
\begin{defn}
     We let $R^q_{m-1,X} := \Ker \left( (V^{m-1},dV^{m-1}) : \Omega^q_X \oplus \Omega^{q-1}_X \to W_m\Omega^q_X\right)$.
\end{defn}
 We note that $R^q_{m-1,X} \subset B_{m}\Omega^q_X \oplus Z_{m-1}\Omega^{q-1}_X$ by \propref{prop:basics}(4) and (6).
\begin{cor}\label{cor:R-m-q}
   For a regular scheme $X$ over $\F_p$, we have
    $$0 \to R^q_{m-1,X} \to B_{m}\Omega^q_X \oplus Z_{m-1}\Omega^{q-1}_X \xrightarrow{C^{m-1}, \ dC^{m-1}} \  B_1\Omega^q_X \to 0.$$
\end{cor}
\begin{proof}
   First note that $C^{m-1} :B_m\Omega^q_X \to B_1\Omega^q_X $ (resp. $dC^{m-1} : Z_{m-1}\Omega^{q-1}_X \to B_1\Omega^q_X$) is surjective by \corref{cor:cartier-1}. Now if $(a,b) \in \Ker (C^{m-1},dC^{m-1})$, by ~\propref{prop:basics} $(4)$ and $(6)$, we can write $(a,b)= (F^{m-1}da',F^{m-1}b')$, for $a',b' \in W_m\Omega^{q-1}_X$, such that $dR^{m-1}(a'+b')=0$ (noting $CF=R$). Now $V^{m-1}(a)+dV^{m-1}(b)= p^{m-1}d(a'+b') =0$ by $(2)$ of ~\propref{prop:basics}. This implies $\Ker (C^{m-1}.dC^{m-1}) \subset R^q_{m-1,X}$.
   
   Now, for the other direction, take any $(a,b) \in R^q_{m-1,X}$. Then there exists $b'\in Z_{m-1}\Omega^{q-1}_X$ such that $(a,b)-(0,b')\in \Ker (C^{m-1},dC^{m-1})$. Since $\Ker (C^{m-1}.dC^{m-1}) \subset R^q_{m-1,X}$, we get that $(0,b')\in R^q_{m-1,X}$. Now by \propref{prop:basics}(5), we get $b' \in Z_{m}\Omega^{q-1}_X$ and hence $(0,b') \in \Ker (C^{m-1}.dC^{m-1})$ by \eqref{eqn:cartier-1.1}. This implies $(a,b)\in \Ker (C^{m-1},dC^{m-1})$. This finishes the proof.
\end{proof}

\propref{prop:basics} and the relations between $F,d,V$ listed above immediately imply:
\begin{lem}
    The Frobenius map $F : W_{m+1}\Omega^q_X \to W_{m}\Omega^q_X$ induces the homomorphism $\ov{F}: W_{m}\Omega^q_X \to W_{m}\Omega^q_X/dV^{m-1}\Omega^{q-1}_X.$
\end{lem}
Now we recall the definition of logarithmic Hodge-Witt sheaf.
\begin{defn}
    For a scheme $X$ over $\F_p$, and $m,q \ge 1$, we define the logarithmic Hodge-Witt sheaf $W_m\Omega^q_{X,\log,\tau}$ by ${\rm Im}(\dlog: (\sO_X^\times)^{\otimes q} \to W_m\Omega^q_X)$, where $\dlog$ is defined by 
    $$\dlog (a_1\otimes \cdots \otimes a_q)= \dlog[a_1]_m \cdots \dlog[a_q]_m,$$
    where the image is taken $\tau = \zar, \nis \text{ or }\et$ale locally.
\end{defn}
The above map induces a map $\dlog \colon {\sK^M_{i,X}}/{p^m} \to
W_m\Omega^i_{X, \log}$, $\tau$-locally, which is known to be an isomorphism
when $X$ is regular (cf. \cite[Thm.~1.2, Cor.~4.2]{Morrow-ENS}).

The following proposition is known for smooth case and can be proved for regular case by using Theorems~\ref{thm:Popescu} and \ref{thm:lim-cohomology}
\begin{prop}\label{prop:usual}(\cite[Prop~2.8, 2.10, 2.12]{Shiho}; cf. \cite{Illusie},\cite[Lem~2,3]{CTSS})
    Let $X$ be a regular scheme over $\F_p$. Then we have the following exact sequences in $X_\et$.
    \begin{enumerate}
        \item $0 \to W_n\Omega^q_{X,\log} \xrightarrow{\un{p}^m} W_{m+n}\Omega^q_{X,\log} \xrightarrow{R^n} W_m\Omega^q_{X,\log} \to 0$.
        \item $0 \to W_m\Omega^q_{X,\log} \to W_m\Omega^q_{X} \xrightarrow{1-\ov{F}}W_{m}\Omega^q_X/dV^{m-1}\Omega^{q-1}_X \to 0$.
        \item $0\to W_m\Omega^q_{X,\log} \to Z_1W_m\Omega^q_{X} \xrightarrow{1-C}W_{m}\Omega^q_X \to 0$.
    \end{enumerate}
\end{prop}

\section{The log de Rham-Witt complex}\label{sec:LogDR}
In this section, we shall recall the definitions and properties of log de Rham-Witt complex over a log scheme. The log structure, we shall be interested in is the one induced by a SNCD.
 \subsection{Definitions}
 We recall the definition of log structure on a scheme \cite{Kato-log} (see also \cite[Chap~2]{Shiho-2}). 

\begin{defn}
  Let $Y$ be a scheme. A pre-log structure on $Y$ is a pair $(\sM_Y, \alpha)$, where $\sM_Y$ is a sheaf of commutative monoids on $Y_\et$ (or $Y_\zar$) and $\alpha: \sM_Y \to (\sO_Y, \cdot)$ is a morphism of sheaf of monoids on $Y$, with respect to multiplication on $\sO_Y$. 
The pre-log structure $(\sM_Y, \alpha)$ is called a log structure if $\alpha^{-1}(\sO_Y^\times)\xrightarrow{\alpha}\sO_Y^\times$ is an isomorphism. A log scheme is scheme with a log structure on it.
\end{defn}

\begin{defn}
    Let $(\sM_Y, \alpha)$ be a pre-log structure on $Y$. Let $\sM^a_Y$ be the push-out of the following diagram in the category of sheaves of monoid on $Y$.
    $$\sO_Y^\times \xleftarrow{\alpha}\alpha^{-1}(\sO_Y^\times)\xrightarrow{}\sM_Y.$$
    That is, $\sM^a_Y= (\sO_Y^\times \bigoplus \sM_Y) /\sim $, where, locally, $(a,b)\sim (a',b')$ if there exists $c_1,c_2 \in \alpha^{-1}(\sO_Y^\times)$ such that $a\alpha(c_1)=a'\alpha(c_2)$ and $bc_2=b'c_1$. One can check $(\sM_Y^a,\alpha^a)$ gives a log structure on $Y$, where $\alpha^a:\sM_Y^a\to \sO_Y$ is given by $\alpha^a(u,v)=u\alpha(v)$. $(\sM_Y^a,\alpha^a)$ is called the log structure associated to $(\sM_Y,\alpha)$.   
\end{defn}
Given a pre-log structure $(\sM_Y,\alpha)$ on $Y$, we shall denote the associated log scheme as $(Y, \alpha:\sM_Y\to \sO_Y)$, if no confusion is possible.

A morphism between two log schemes $(Y, \alpha:\sM_Y\to \sO_Y)$ and $(Y', \alpha':\sM'_{Y'}\to \sO_{Y'})$ is a morphism of schemes $f: Y\to Y'$ and a homomorphism $h: f^{-1}(\sM'_{Y'})\to \sM_Y$, such that the following diagram commutes.
\begin{equation}
    \xymatrix{
f^{-1}(\sM'_{Y'}) \ar[r]^{h} \ar[d]_-{f^{-1}\alpha'} & \sM_Y \ar[d]^-{\alpha} \\
f^{-1}(\sO_{Y'}) \ar[r] & \sO_Y.
}
\end{equation}

\begin{defn}
    Let $f:Y_1 \to Y_2$ be a morphism of schemes and $(\sM_{Y_2},\alpha)$ be a log structure on $Y_2$. Then the pull back log structure $f^*\sM_{Y_2}$ on $Y_1$ is defined to be the log structure associated to the pre-log structure $f^{-1}\sM_{Y_2} \xrightarrow{f^{-1}(\alpha)}f^{-1}\sO_{Y_2}\to \sO_{Y_1}$ on $Y_1$. 
\end{defn}

At this point, we refer the reader to \cite{Kato-log} (also \cite[Chap~2]{Shiho-2}) for a comprehensive discussion on log schemes and the definitions of various types of morphisms (e.g. integral, saturated, smooth, Cartier type etc).

Now we recall the definition of log Witt complex associated to a log structure (cf. \cite[\S~1]{Geisser-Hesselholt-JAMS}, \cite[\S~3]{HM-Annals}). Let $A_\log$ denote a log ring $(A, \alpha : M_A\to (A,\cdot))$ of characteristic $P>0$ (i.e. char $A=p$). This induces a log structure $(M_A, \ \alpha_m:M_A \to W_m(A))$ on $W_m(A)$, where $\alpha_m :=[\cdot]_m \circ \alpha$ and
$[\cdot]_m \colon A \to W_m(A)$ is the Teichm{\"u}ller map. 

\begin{defn}
    A log differential graded ring $(E^\star, M_E)$ over a log ring $(E^0, \ M_E, \ \alpha:M_E\to E^0)$ is a differential graded ring $E^\star$, together with a map of monoids, $\dlog: M_E \to (E^1,+)$ such that $d\circ \dlog=0$ and $d\alpha(a)=\alpha(a)\dlog(a)$ for all $a \in M_E$.
\end{defn}
For any log ring $A_\log$, a universal log differential graded ring always exists (see \cite[Section~2.2, P~33]{HM-Annals}). This is called the de Rham complex of the log ring $A_\log$ and denoted by $\Omega^\star_{A_\log}$.
\begin{defn}
    A log Witt complex over $A_\log$ consists of :
    \begin{enumerate}
        \item a pro-log differential graded ring $(E_\textbf{.}^\star, M_E)$ and a strict map of pro-log rings $\lambda:(W_\textbf{.}(A),M_A)\to (E_\textbf{.}^0, M_E)$;
        \item a strict map of pro-log graded rings
     $$F:E^\star_m \to E^\star_{m-1},$$
     such that $\lambda F=F \lambda$, $F \dlog_m\lambda(a)=\dlog_{m-1}\lambda(a)$ for all $a \in M_A$ and $Fd\lambda([a]_m)=\lambda([a]_{m-1})^{p-1}d\lambda([a]_{m-1})$ for all $a\in A$; where $\dlog_m: M_E\to E^1_{m+1}$;
     
     \item and a strict map of pro graded modules over the pro graded rings $E^\star_\cdot$
     $$V: F_\star E^\star_m \to E^\star_{m+1}$$
     such that $V \lambda=\lambda V$, $FV=p$ and $FdV=d$, where $F_\star E^\star_m$ indicates that $E^\star_m$ is considered as $E_{m+1}^0$-module via $F: E^0_{m+1}\to E^0_m$.

    \end{enumerate}
\end{defn}
We write $R$ for the structure map of the pro-system $E^\star_.$ and call it the restriction map. 
For any log ring $A_\log$, a universal log Witt complex always exists (cf. \cite[Prop.~3.2.2]{HM-Annals}). It is called the de Rham-Witt complex of the log ring $A_\log$ and denoted by $W_\textbf{.}\Omega^\star_{A_\log}$. Furthermore, $W_1\Omega^\bullet_{A_\log}$ coincides with the de Rham complex
$\Omega^\bullet_{A_\log}$ of the log ring $A_\log$. One can also extend these definitions and results to any log scheme $Y$ of characteristic p in a natural way.

\subsection{Log structure associated to SNCD}
Let $X$ be a connected regular Noetherian scheme of characteristic $p > 0$.
We let $K$ denote the function field of $X$.
Let $E \subset X$ be a simple normal crossing divisor (SNCD) with irreducible components
$E_1, \ldots , E_r$. Recall that this means that $E_i$ is a regular divisor on $X$
for each $i$ and the scheme theoretic intersection $E_{i_1} \cap \cdots \cap E_{i_s}$
($1 \le i_1 < \cdots < i_s \le r$) is regular of codimension $s$ in $X$.
Let $j \colon U \inj X$ be the inclusion of the open complement of $E$.

We let $X_\log$ denote the log scheme $(X, \sM_X, \alpha_X)$, where
$\alpha_X \colon \sM_X = \sO_X \cap j_*\sO^\times_U \inj \sO_X$ is the usual
inclusion. Here the intersection is taken inside the constant sheaf given by $K$, the function field of $X$. For $m \ge 1$, let $W_m(X)_\log$ be the log scheme structure on
$W_m(X)$ associated to the pre-log structure
$(\sM_X, \alpha_m \colon \sM_X \to W_m\sO_X)$, where $\alpha_m := [\cdot]_m \circ \alpha$ and
$[\cdot]_m \colon \sO_X \to W_m\sO_X$ is the Teichm{\"u}ller map.

We let $W_m\Omega^\bullet_{X_\log}$ denote the
de Rham-Witt complex of the log scheme $X_\log$.
Note that $W_1\Omega^\bullet_{X_\log}$ coincides with the de Rham complex
$\Omega^\bullet_{X_\log}$ of the log scheme $X_\log$.
When $X = \Spec(A)$, we shall write $W_m\Omega^q_{X_\log}$ as $W_m\Omega^q_{A_\log}$, for $q \ge 0$.

Recall that
the canonical map of de Rham-Witt complexes $W_m\Omega^\bullet_X \to
j_*(W_m\Omega^\bullet_U)$ is injective (because $W_m\Omega^\bullet_X \to
\eta_*(W_m\Omega^\bullet_K)$ is injective, where $\eta \colon \Spec(K) \inj X$ is the
inclusion (cf. \cite[Prop.~2.8]{KP-Comp})). This fact will be implicitly used throughout
this thesis.

If $X = \Spec(A)$, where $A$ is a regular local $\F_p$-algebra and
$E_i = \Spec({A}/{(x_i)})$, we let $\pi = \stackrel{r}{\underset{i =1}\prod} x_i$
and let $A_\pi = A[\pi^{-1}]$.
We let $S_0 = \{0\}$ and for $1 \le j \le r$, let $S_j$ be the collection of ordered
subsets $J = \{i_1 < \cdots < i_j\}$ of $\{1, \ldots , r\}$.
For $J = \{i_1 < \cdots < i_j\} \in S_j$, we let $W_m\Omega^q_A(\log\pi)_J =
W_m\Omega^{q-j}_A\dlog[x_{i_1}]_m \wedge \cdots \wedge \dlog[x_{i_j}]_m$
as a $W_m(A)$-submodule of $W_m\Omega^q_{A_\pi}$. For $J = S_0$, we let
$W_m\Omega^q_A(\log\pi)_J = W_m\Omega^q_A$. We let $W_m\Omega^q_{X}(\log E)$ be the
image $W_m\Omega^q_A(\log \pi)$ of the canonical map
$\stackrel{min\{q,r\}}{\underset{j=0}\bigoplus} {\underset{J \in S_j}\bigoplus}
  W_m\Omega^q_A(\log\pi)_J$ $\to W_m\Omega^q_{A_\pi}$.
  If $X$ is not necessarily local, we let $W_m\Omega^q_X(\log E)$ be the Zariski
  subsheaf of $j_*W_m\Omega^q_U$ whose stalk at a point $x \in X$ is
  $W_m\Omega^q_{X_x}(\log \ E_x)$, where $X_x = \Spec(\sO_{X,x})$ and $E_x= X_x \times_X E$.

 Purpose of this section is to identify $W_.\Omega^\bullet_{X_\log}$ with $W_\textbf{.}\Omega^\bullet_X(\log E)$ (cf. \thmref{thm:Log-DRW}).

\begin{lem}\label{lem:LWC-1}
  $W_\textbf{.}\Omega^\bullet_X(\log E)$ is a log Witt complex over $X_\log$. In particular,
  there is a canonical (strict) map of log Witt complexes
  $\theta_X \colon W_\textbf{.}\Omega^\bullet_{X_\log} \to W_\textbf{.}\Omega^\bullet_X(\log E)$.
  \end{lem}
\begin{proof}
  To endow  $W_\textbf{.}\Omega^\bullet_X(\log E)$ with the structure of a log Witt complex
 over $X_\log$, we let $\dlog \circ [\cdot]_m \colon \sM_X \to W_m\Omega^1_X(\log E)$ be
  the map locally defined by $a \mapsto \dlog[a]_m$.
We take $\lambda \colon (W_\textbf{.} \sO_X, \sM_X) \to (W_\textbf{.}\Omega^0_X(\log E), \sM_X)$ to
be the identity map. Since $j_* W_\textbf{.}\Omega^\bullet_U$ is a Witt complex over
$X$ and $W_\textbf{.}\Omega^\bullet_{X}(\log E)$ is its subcomplex, we only need to verify
that the latter is preserved by the Frobenius and Verchiebung operators of
$j_*W_\textbf{.}\Omega^\bullet_U$.
Since $F$ is multiplicative and $F(\dlog[a]_m) = \dlog[a]_{m-1}$,
the claim is clear for $F$.
Since $V(w\dlog[a]_m) = V(wF(\dlog[a]_{m+1})) = V(w)\dlog[a]_{m+1}$, we 
see that $V$ also preserves $W_\textbf{.}\Omega^\bullet_{X}(\log E)$.
\end{proof}

Recall that an $\F_p$-scheme $Y$ is called $F$-finite if the absolute Frobenius
map $F \colon Y \to Y$ is a finite morphism. In the sequel, we shall use that
every essentially of finite type scheme over an $F$-finite field of characteristic $p$
as well as its Henselizations and completions is $F$-finite (cf. \cite[\S~1.1]{Morrow-ENS}).

\begin{lem}\label{lem:LWC-3}
     Let $A$ be a regular local $F$-finite $\F_p$-algebra of dim $N$ with maximal
  ideal $\fm = (x_1, \ldots , x_N)$ and residue field $\ff$. Then $\Omega^1_A$ is free $A$-module of finite rank. Moreover, $\{dx_1,\ldots , dx_N\}$ is a   part of basis of $\Omega^1_A$.
\end{lem}
\begin{proof}
    Since $A$ is regular and $F-$finite, $\Omega^1_A$ is free (by \cite[Lem~2.1.2]{Zhao}) and finitely generated  over $A$ (by \cite[Prop.~1]{Tyc}, taking $k=\F_p$). Now consider the exact sequence
    \begin{equation}\label{eqn:2nd-funda}
        0 \to \fm/\fm^2 \xrightarrow{d \otimes 1} \Omega^1_A \otimes_A \ff \to \Omega^1_{\ff} \to 0.
    \end{equation}
    Here, $d \otimes1$ is injective because $\ff$ is $0-$smooth (separable) over $\F_p$ (see \cite[Thm~25.2, 26.9]{Matsumura}). Also, it follows from \eqref{eqn:2nd-funda} that $\Omega^1_\ff$ has finite rank. Let $y_1,\ldots,y_s \in A$ such that image of $dy_1,\dots, dy_s$ forms a basis of $\Omega^1_\ff$. Then by \eqref{eqn:2nd-funda}, we get that $d\ov{x}_1,\dots,d\ov{x}_N,d\ov{y}_1,\\ \dots, d\ov{y}_s$ forms a basis of $\Omega^1_A \otimes_A \ff$. Since $\Omega^1_A$ is finitely generated flat $A$ module, by Nakayama lemma, $dx_1,\dots,dx_N,dy_1,\dots,dy_s$ forms a basis of $\Omega^1_A$. 
\end{proof}

\begin{lem}\label{lem:LWC-2}
  Let $A$ be as in \lemref{lem:LWC-3}. Let
  $\pi = x_1 \cdots x_r, \ r \le N$. Let $A_\log$ be the log ring whose log structure on $A$ is induced by $\pi$. Then the canonical maps of log
  differential graded $A$-algebras
  $ \colon \Omega^\bullet_{A_\log} \xrightarrow{\theta_A}\Omega^\bullet_A(\log \pi) \xleftarrow{\psi_A}\Wedge^\bullet_A (\Omega^1_A(\log\pi)) $
  are isomorphisms.  Moreover, each $\Omega^q_A(\log\pi)$ is a free $A$-module of finite rank.
\end{lem}
\begin{proof}
  We continue with the same notations of the proof of \lemref{lem:LWC-3}. We let $\wt{\Omega}^1_A(\log\pi) = (\underset{1\le j \le s}{\oplus}A dy_j ) \bigoplus
  ({\underset{1 \le i \le r}\oplus} A\dlog(x_i)) \bigoplus 
  ({\underset{r+1 \le i \le N}\oplus} Adx_i)$.
  Since $\Omega^\bullet_{A_\log}$ is a DG (differential graded) complex over $A$, there is a unique
     morphism of differential graded $A$-algebras $\phi_A \colon \Omega^\bullet_A \to
     \Omega^\bullet_{A_\log}$.
     As $\Omega^\bullet_{A_\log}$ has a log DG complex structure,
     we also have a commutative diagram of monoids
\begin{equation}\label{eqn:LWC-3-1.0}
    \xymatrix@C.8pc{
      M_A \ar[r]^-{\dlog} \ar[dr]_-{\dlog} &  \Omega^1_{A_\log} \ar[d]^-{\rm can} \\
      & \Omega^1_{A_\pi},}
    \end{equation}
where $M_A = \{ux_1^{i_1} \cdots x_r^{i_r}| u \in A^\times \ \mbox{and} \ i_j \in \N \
\forall \ 1\le j \le r\}$.

Since $\wt{\Omega}^1_A$ is a free $A$-module, we have a well-defined $A$-linear map
  $\theta'_A \colon \wt{\Omega}^1_A(\log\pi) \to \Omega^1_{A_\log}$ such that $\theta'_A(a\dlog(x_i)) = a \dlog(x_i)$ for all $a \in A$. Hence, we
  get a morphism of differential graded $A$-algebras 
  $\theta'_A \colon \Wedge^\bullet_A (\wt{\Omega}^1_A(\log\pi)) \to
  \Omega^\bullet_{A_\log}$. It is clear that there is a monoid homomorphism
  $\dlog \colon M_A \to \wt{\Omega}^1_A(\log\pi)$
  which is compatible with $\theta'_A$. In particular, the latter is a morphism of
  log differential graded $A$-algebras which is identity in degree zero.
  We conclude from the universal property of $\Omega^\bullet_{A_\log}$
  that $$\theta'_A \colon \Wedge^\bullet_A (\wt{\Omega}^1_A(\log\pi)) \xrightarrow{\cong}
  \Omega^\bullet_{A_\log}.$$

  On the other hand,
  as $\wt{\Omega}^1_A(\log\pi)$ is a free $A$-module, the canonical map  $$\Wedge^\bullet_A(\wt{\Omega}^1_A(\log\pi)) \to \Wedge^\bullet_A(\wt{\Omega}^1_A(\log\pi)) \otimes_A A_\pi = \Wedge^\bullet_A(\Omega^1_{A_\pi})= \Omega^\bullet_{A_\pi}$$
  is injective and its image lies in $\Omega^\bullet_A(\log\pi)$. We thus get a commutative diagram 
  \begin{equation}\label{eqn:LWC-3-1}
    \xymatrix@C.8pc{
      \Wedge^\bullet_A(\wt{\Omega}^1_A(\log\pi)) \ar@{.>}[rr]^-{\wt{\psi}^{\bullet}_A} \ar@{^{(}->}[dr] & & \Omega^\bullet_A(\log\pi)
      \ar@{^{(}->}[dl] \\
      & \Omega^\bullet_{A_\pi}. &}
    \end{equation}

  It is clear from the definition of $\Omega^\bullet_A(\log\pi)$ and \lemref{lem:LWC-3} that $\wt\psi^\bullet_A$ is surjective and hence it is an isomorphism. This implies second part of the lemma. In particular $\wt\psi^1_A: \wt{\Omega}^1_A(\log\pi) \to \Omega^1_A(\log \pi)$ is isomorphism. This forces $\psi_A:  \Wedge^\bullet_A(\Omega^1_A(\log\pi)) \to \Omega^\bullet_A(\log \pi)$ to be isomorphism. Also we have, $$\theta_A= \wt \psi^\bullet_Ao(\theta'_A)^{-1} : \Omega^\bullet_{A_\log} \xrightarrow{\cong} \Omega^\bullet_A(\log \pi).$$ 
 This achieves the first part of the lemma.
\end{proof}

We let $A$ be as in \lemref{lem:LWC-2} and let $\F_p^{[n]} = \F_p[X_1, \ldots , X_n]$, for $n \in \N$.
There are homomorphisms of log rings
\begin{equation}\label{eqn:Log-smooth}
\F_{p,\log} \to (\F_p^{[N]})_\log \xrightarrow{\gamma} A_\log, \ \mbox{where} \
\gamma(X_i) = x_i \ \forall \ 1 \le i \le N.
\end{equation}
The log structure of $\F_p$ is trivial while 
that of $\F_p^{[N]}$ (resp. $A$) is given by the simple normal crossing
divisor $V((X))$ (resp. $V((\pi))$) on $\A^N_{\F_p}$ (resp. $\Spec(A)$), where
$X = X_1\cdots X_r$. The first arrow in ~\eqref{eqn:Log-smooth} is log smooth
while the second arrow is a morphism of log rings such that the underline morphism of rings is regular. To see the latter claim, first note that $A$ is formally $\fm$-smooth over $(\F_p^{[N]})_{(X_1,\cdots,X_N)}$, as $\F_p$ is perfect and $A$ is regular (cf.
\cite[Lem.~1.3]{Tanimoto-2}). It follows now from Andr{\'e}'s theorem (cf. \cite{Andre}) that $\F_p^{[N]} \to A$ is regular. We let $H_{\dR}^i(A_\log)$ denote the $i$-th cohomology
of the log de Rham complex $\Omega^\bullet_{A_\log}$.

  \begin{lem}\label{lem:Log-Cartier}
    For $q \ge 0$, there is an isomorphism of $A$-modules
    $C^{-1}_q \colon \Omega^q_{A_\log} \xrightarrow{\cong} H_{\dR}^q(A_\log)$.
  \end{lem}
  \begin{proof}
    Since the map $\gamma \colon \F_p^{[N]} \to A$ is regular, \thmref{thm:Popescu}
    implies that there is a direct system $\{A(\lambda)\}$ of 
    smooth (in particular, finite type)
    $\F_p^{[N]}$-algebras such that $A = \varinjlim_\lambda A(\lambda)$.
    We endow each $A(\lambda)$ with the log structure given by the pull-back of the
    log structure of $(\F_p^{[N]})_\log$. Since the log structure of $A$ also has this
    property, it follows that we have a direct system of
    log smooth $(\F_p^{[N]})_{\log}$-algebras $\{A(\lambda)_\log\}$ such that
    $A_\log = \varinjlim_\lambda A(\lambda)_\log$ in the category of log rings.
    As $\F_{p,\log} \to (\F_p^{[N]})_\log$ is log smooth, we get that each $A(\lambda)_\log$ is
    log smooth over $\F_{p,\log}$.

    Since the homomorphism of log rings $\F_{p,\log} \to (\F_p^{[N]})_\log$ is clearly integral, and
    so is the homomorphism $(\F_p^{[N]})_\log \to A(\lambda)_\log$ by the definition of
    latter's log structure (\cite[Cor~4.4(i)]{Kato-log}), it follows that $\F_{p,\log} \to A(\lambda)_\log$ is integral.
    By a similar reason, this map is also saturated. It follows from Kato's criterion
    that each $\F_{p,\log} \to A(\lambda)_\log$ is of Cartier type (cf.
    \cite[\S~2.12]{Hyodo-Kato}). 
    We conclude from loc. cit. that there is an isomorphism
    $C^{-1}_q \colon \Omega^q_{A(\lambda)_\log} \xrightarrow{\cong}
    H^q_\dR(A(\lambda)_{\log})$ given by $C^{-1}_q(a\dlog(b_1)\wedge \cdots \wedge
    \dlog(b_q)) = a^p\dlog(b_1)\wedge \cdots \wedge \dlog(b_q)$.
Since this is clearly natural with respect to the homomorphisms of log rings, we
    get after passing to the limit that there is a similar inverse Cartier isomorphism
    for $A$ as well. It is clear from the description that $C^{-1}_q$ is $A$-linear
    if we endow $\Ker(d)$ with the $A$-module structure via latter's Frobenius
    endomorphism.
   \end{proof} 

  We let $Z_1\Omega^q_{A_\log} = \Ker(d \colon \Omega^q_{A_\log} \to \Omega^{q+1}_{A_\log})$
  and $B_1\Omega^q_{A_\log} = {\rm Image}(d \colon \Omega^{q-1}_{A_\log} \to
  \Omega^{q}_{A_\log})$. Taking the inverse of $C^{-1}_\bullet$, we get the following.
  
\begin{cor}\label{cor:Log-Cartier-0}
  For $q \ge 0$, there is an $A$-linear Cartier homomorphism
  $C_q \colon Z_1\Omega^q_{A_\log} \to \Omega^q_{A_\log}$ which induces an
  isomorphism of $A$-modules $C_q \colon H^q_\dR({A_\log}) \xrightarrow{\cong}
  \Omega^q_{A_\log}$.
\end{cor}

Now $Z_i\Omega^q_{A_\log}$ and $B_i\Omega^q_{A_\log}$ can be defined like the usual case of regular schemes (cf. \secref{sec:de-rham-witt}). We thus get a filtration
\begin{equation*}
  0 = B_0\Omega^q_{A_\log} \subseteq B_1\Omega^q_{A_\log} \subseteq \cdots \subseteq
 \bigcup_i B_i\Omega^q_{A_\log} \subseteq \bigcap_i Z_{i}\Omega^q_{A_\log}  \subseteq \cdots \subseteq 
 Z_{1}\Omega^q_{A_\log} \subseteq Z_{0}\Omega^q_{A_\log} = \Omega^q_{A_\log}
\end{equation*}
of $\Omega^q_{A_\log}$ together with an isomorphism
$C^{-i}_q \colon \Omega^q_{A_\log} \xrightarrow{\cong}
{Z_{i}\Omega^q_{A_\log}}/{B_i\Omega^q_{A_\log}}$ for all $i \ge 0$. Using this isomorphism, proof of the following corollary is analogous to that of \corref{cor:cartier-1}.
\begin{cor}\label{cor:Log-cartier-1}
    We have the exact sequences
    \begin{equation}\label{eqn:log-cartier-1.1}
        0 \to Z_{m+1}\Omega^q_{A_\log} \to Z_m\Omega^q_{A_\log} \xrightarrow{dC^m} B_1\Omega^{q+1}_{A_\log} \to 0,
    \end{equation}
    \begin{equation}\label{eqn:log-cartier-1.2}
        0 \to B_{m}\Omega^q_{A_\log} \to B_{m+1}\Omega^q_{A_\log} \xrightarrow{C^m} B_1\Omega^{q}_{A_\log} \to 0,
    \end{equation}
    for $m \ge 1, q \ge 0$. 
\end{cor}
\begin{lem}\label{lem:LWC-4}
  For every $m \ge 1$, the canonical map $W_m\Omega^q_{A_\log} \to W_m\Omega^q_{A_\pi}$ is
  injective.
\end{lem}
\begin{proof}
  We shall prove the lemma by induction on $m$. The base case $m =1$ follows from
  \lemref{lem:LWC-2}. 
For $q, r \ge 0$ and $m \ge 1$, we let
$\Fil^r_V W_m\Omega^q_{A_\log} = V^r(W_{m-r}\Omega^q_{A_\log}) +
dV^r(W_{m-r}\Omega^{q-1}_{A_\log}) \subset W_m\Omega^q_{A_\log}$.
This yields the (decreasing) V-filtration on $W_m\Omega^q_{A_\log}$ as a
$W_m(A)$-module. Furthermore, we have an exact sequence
(cf. \cite[Lem.~3.2.4]{HM-Annals}) 
\begin{equation}\label{eqn:LWC-4-0}
  0 \to \Fil^r_V W_m\Omega^q_{A_\log} \to W_m\Omega^q_{A_\log}
  \xrightarrow{R^{m-r}} W_r\Omega^q_{A_\log} \to 0,
  \end{equation}
which is natural in $A_\log$.
We let $R^q_{r}(A_\log) = \Ker((V^{r}, dV^{r}): \Omega^q_{A_\log} \oplus
\Omega^{q-1}_{A_\log} \to W_{r+1}\Omega^q_{A_\log})$.
By comparing the above exact sequence with the analogous exact sequence for $A_\pi$
and for $m \ge 2, r =m-1$, and using induction on $m$, it suffices to show that
the square
\begin{equation}\label{eqn:LWC-4-1}
  \xymatrix@C1pc{
    R^q_{m-1}(A_\log) \ar[r] \ar[d] & \Omega^q_{A_\log} \oplus \Omega^{q-1}_{A_\log}
    \ar[d] \\
    R^q_{m-1}(A_{\pi}) \ar[r]  & \Omega^q_{A_\pi} \oplus \Omega^{q-1}_{A_\pi}}
  \end{equation}
is Cartesian.

Suppose now that there is an element $\alpha \in
(\Omega^q_{A_\log} \oplus \Omega^{q-1}_{A_\log}) \bigcap R^q_{m-1}(A_{\pi})
\subset (\Omega^q_{A_\pi} \oplus \Omega^{q-1}_{A_\pi})$. Then
$\alpha \in  B_{m}\Omega^q_{A_\pi} \oplus Z_{m-1}\Omega^{q-1}_{A_\pi}$
by \corref{cor:R-m-q}. We write $\alpha = (\alpha_0, \alpha_1)$ with
$\alpha_0 \in B_{m}\Omega^q_{A_\pi} \bigcap \Omega^q_{A_\log}$ and
$\alpha_1 \in  Z_{m-1}\Omega^{q-1}_{A_\pi} \bigcap \Omega^{q-1}_{A_\log}$.
We claim that $\alpha \in B_{m}\Omega^1_{A_\log} \oplus Z_{m-1}\Omega^1_{A_\log}$.

To prove the claim, we first note that $B_1\Omega^q_{A_\log} \subset \Omega^q_{A_\log} \inj
\Omega^q_{A_\pi}$  by \lemref{lem:LWC-2}, and this implies that
the map $B_1\Omega^q_{A_\log} \to B_1\Omega^q_{A_\pi}$ is injective. 
We next note that as $\Omega^{\bullet}_{A_\log} \inj \Omega^{\bullet}_{A_\pi}$,
our assumption already implies that
$\alpha \in Z_1\Omega^q_{A_\log} \oplus Z_{1}\Omega^{q-1}_{A_\log}$.
Using \eqref{eqn:cartier-1.1},  \eqref{eqn:log-cartier-1.1} and induction on $m$, we conclude that
$\alpha_1 \in Z_{m-1}\Omega^1_{A_\log}$. Since $\alpha_0 \in B_{m}\Omega^q_{A_\pi}
\cap \Omega^q_{A_\log} \subset Z_{i}\Omega^q_{A_\pi}$
for every $i \ge 1$, the same argument shows that
$\alpha_0 \in Z_m\Omega^q_{A_\log}$. Using the commutative diagram
of short exact sequences
\begin{equation}\label{eqn:LWC-4-2}
  \xymatrix@C.8pc{
    0 \ar[r] & B_{m}\Omega^q_{A_\log} \ar[r] \ar[d] & Z_m\Omega^q_{A_\log} \ar[r] \ar[d] &
    H^q_\dR(A_\log) \ar[r] \ar@{^{(}->}[d] & 0 \\
     0 \ar[r] & B_{m}\Omega^q_{A_\pi} \ar[r] & Z_m\Omega^q_{A_\pi} \ar[r] &
    H^q_\dR(A_\pi) \ar[r] & 0,}
\end{equation}
we conclude that $\alpha_0 \in B_{m}\Omega^q_{A_\log}$. This proves the claim.

We now use \cite[(1.15.3)]{Lorenzon} (with a limit argument as in \lemref{lem:Log-Cartier}) and \corref{cor:R-m-q}, which give us a commutative diagram
of short exact sequences
\begin{equation}\label{eqn:LWC-4-3}
  \xymatrix@C.8pc{
    0 \ar[r] & R^q_{m-1}(A_\log) \ar[r] \ar[d] &
    B_{m}\Omega^1_{A_\log} \oplus Z_{m-1}\Omega^1_{A_\log} \ar[r] \ar[d] & B_1\Omega^q_{A_\log}
    \ar[r] \ar[d] & 0 \\
   0 \ar[r] & R^q_{m-1}(A_\pi) \ar[r] &
    B_{m}\Omega^1_{A_\pi} \oplus Z_{m-1}\Omega^1_{A_\pi} \ar[r] & B_1\Omega^q_{A_\pi}
    \ar[r] & 0.}
  \end{equation}
Since the right vertical arrow is injective, we conclude from the above claim 
and a diagram chase that $\alpha \in R^q_{m-1}(A_\log)$. This proves that
~\eqref{eqn:LWC-4-1} is Cartesian and completes the proof of the lemma.
\end{proof}

We can now prove the main result of this section.

\begin{thm}\label{thm:Log-DRW}
Let $X$ be a regular $F$-finite $\F_p$-scheme and let 
$E \subset X$ be a simple normal crossing divisor. Then the canonical strict map
of log Witt complexes $\theta_X \colon W_\textbf{.}\Omega^\bullet_{X_\log} \to
W_\textbf{.}\Omega^\bullet(\log E)$ is an isomorphism.
\end{thm}
\begin{proof}
  We can assume $X = \Spec(A)$ with $A$ as in \lemref{lem:LWC-2}.
  We now look at the commutative diagram of log Witt complexes over $A_\log$:
  \begin{equation}\label{eqn:Log-DRW-0}
    \xymatrix@C1pc{
      W_m\Omega^\bullet_A(\log\pi) \ar@{.>}[rr] \ar@{^{(}->}[dr] & &
      W_m\Omega^\bullet_{A_\log} \ar@{^{(}->}[dl] \\
      & W_m\Omega^\bullet_{A_\pi}. &}
    \end{equation}

  The diagonal arrow on the right is injective by \lemref{lem:LWC-4}.
  It is clear from the definition of $W_m\Omega^\bullet_A(\log\pi)$ that it lies in
  $W_m\Omega^\bullet_{A_\log}$ as a log Witt subcomplex of $W_m\Omega^\bullet_{A_\pi}$.
  In particular, the inclusion
  $W_m\Omega^\bullet_A(\log\pi) \inj W_m\Omega^\bullet_{A_\pi}$ uniquely factors through
  an inclusion of log Witt complexes
  $\psi_A \colon W_m\Omega^\bullet_A(\log\pi) \to W_m\Omega^\bullet_{A_\log}$.
  Since $W_m\Omega^\bullet_{A_\log}$ is the universal log Witt-complex over $A_\log$,
  this forces $\psi_A$ and $\theta_A$ to be isomorphisms and inverses of each other.
\end{proof}

\begin{cor}\label{cor:Log-DRW-0}
  Let $X$ be as in \thmref{thm:Log-DRW}. Then for every $q, r,s \ge 0$ with
  $r+s = m \ge 1$, there is a short exact sequence
  of $W_m\sO_X$-modules
  \begin{equation}\label{eqn:Log-DRW-0-0}
    0 \to V^r(W_{s}\Omega^q_{X}(\log E)) + dV^r(W_{s}\Omega^{q-1}_X(\log E)) \to
    W_m\Omega^q_X(\log E) \xrightarrow{R^{s}} W_r\Omega^q_X(\log E) \to 0.
  \end{equation}
\end{cor}
\begin{proof}
  Combine \cite[\S~3]{HM-Annals} and \thmref{thm:Log-DRW}.
  \end{proof}

\begin{cor}\label{cor:LWC-0-1}
    Let $S$ be a regular $F$-finite $\F_p$ algebra. Let $A=S[[X]]]$ (where $X$ is a variable).
    Then we have a canonical isomorphism of
  $W_m(S)$-modules 
  $$W_m\Omega^q_A \oplus
  W_m\Omega^{q-1}_{S} \dlog[X]_m \xrightarrow{\cong} W_m\Omega^q_A(\log X)$$ for every $q \ge 0,m\ge 1$
\end{cor}  

\begin{proof}
    Note that $\{ W_m\Omega^q_A \oplus
  W_m\Omega^{q-1}_{S} \dlog[X]_m \}_{m,q}$ is already a log de Rham-Witt complex over $A_\log$ and there is a canonical map $ W_m\Omega^q_A \oplus
  W_m\Omega^{q-1}_{S} \dlog[X]_m \to W_m\Omega^q_A(\log X)$ which is compatible with the following monoid maps.
  $$\alpha_m: M_A \to W_m\Omega^1_A(\log X) ; \hspace{1.5cm} \alpha'_m: M_A \to W_m\Omega^1_A \oplus
  W_m(S) \dlog[X]_m,$$
  where $M_A=\{uX^n : u \in A^\times, n \in \Z\}$ and $\alpha'_m(uX^n)=(\dlog[u]_m, n\dlog[X]_m)$. Now corollary follows from \thmref{thm:Log-DRW} and the universal property of $W_m\Omega^\bullet_{A_\log}$. 

\end{proof}

\cleardoublepage

\chapter{The filtered de Rham-Witt complex}\label{chap:FilDR}
Let $X$ be a connected
Noetherian regular $F$-finite $\F_p$-scheme of Krull dimension $N$.
Let $\Div(X)$ denote the group of Weil divisors on $X$.
For any $D  = \stackrel{r}{\underset{i =1}\sum} n_i D_i \in \Div(X)$, we let $|D|$
denote the support of $D$ and
let $D/p = \stackrel{r}{\underset{i =1}\sum} \lfloor{{n_i}/p}\rfloor D_i$.
For any $\F_p$-scheme $Y$ and integer $m \ge 1$, we let $W_m(Y) = W_m(H^0_\et(Y, \sO_Y))$.
The elements of $W_m(Y)$ will be denoted by $m$-tuples
$\un{a} = (a_{m-1}, \dots , a_0)$. In this chapter, we shall define the filtered de Rham-Witt complex of $X$ and prove some basic properties.

\section{Definition and properties}\label{sec:Defn+prop}

\begin{defn}\label{defn:Log-fil-0}
  We let $D = \stackrel{r}{\underset{i =1}\sum} n_i D_i \in \Div(X)$ and let
  $j \colon U \inj X$ be the inclusion of the open complement of $D$.
  For a morphism $u \colon \wt{X} \to X$ of schemes, we let $\wt{U} = U \times_X \wt{X}$ and
      $\wt{D} = D\times_X \wt{X}$. We let
      \[
        \Fil_D W_m(\wt{U}) = \{\underline{a}=(a_{m-1},\ldots, a_0) \in W_m(\wt{U}): 
        a_i^{p^i} \in H^0_\et(\wt{X}, \sO_{\wt{X}}(\wt{D})) \ \forall \ i\}.
      \]
      We let $\Fil_D W_m \sO_U$ be the presheaf on $\Sch_X$ such that
      $\Fil_D W_m \sO_U(\wt{X}) = \Fil_D W_m(\wt{U})$.
\end{defn}

One easily checks that $\Fil_D\sO_U=\sO_X(D)$ and $\Fil_D W_m \sO_U$ actually restricts
      to a sheaf on $X_\zar$ and we have the inclusions of {\'e}tale sheaves
      $\Fil_D W_m \sO_U \subset \Fil_{D'} W_m \sO_W \subset h_* W_m \sO_W$
      whenever $D \le D'$ and $h \colon W = X \setminus |D'| \inj X$ is the
      inclusion.
      Note also that the canonical map $j^*(\Fil_D W_m \sO_U) \to W_m \sO_U$ is an
      isomorphism.
If $P \in X$ is a point and $x_i$ is a generator of the ideal of
$D_i$ in $\sO_{X,P}$, then one checks that $a_i^{p^i} \in (\sO_{X}(D))_P=\sO_{X,P}.\prod_j
          x^{-n_j}_j$ iff $a_i^{p^i} \prod_j
          x^{n_j}_j \in \sO_{X,P}$. Hence, we obtain 
\begin{equation}\label{eqn:Log-fil-1}
  (\Fil_D W_m \sO_U)_P
          =
          \{\un{a} \in W_m(\sO_{X,P}[\pi^{-1}]): a_i^{p^i} \prod_j
          x^{n_j}_j \in \sO_{X,P} \ \forall \ i\},
\end{equation}
where $\pi = \prod_i x_i$.

If $X = \Spec(A)$, where $A$ is a regular local ring and $\{x_1, \ldots , x_r\}$ are
generators of the ideals of the irreducible components of $D$, we shall write
$\Fil_D W_m \sO_U$ as $\Fil_I W_m(A_\pi)$ or $\Fil_{\un{n}} W_m(A_\pi)$ or $\Fil_I W_m(A_\pi)$,
where $\pi = \prod_i x_i$ and
$I \subset Q(A)$ is the invertible ideal $(x)$, where $x = \prod_i x^{n_i}_i$ and $\un{n}=(n_1,\ldots,n_r) \in \Z^r$.
We shall write $\sO_X(D)$ as $A{x^{-1}}$.

\begin{lem}\label{lem:Log-fil-1}
    $\Fil_DW_m\sO_U$ is a sheaf of abelian groups. Moreover, we have an exact sequence
   \begin{equation}\label{eqn:Log-fil-1.1}
       0 \to \Fil_D\sO_U \xrightarrow{V^{m-1}} \Fil_DW_m\sO_U \xrightarrow{R} \Fil_{D/p}W_{m-1}\sO_U\to 0.
   \end{equation}
\end{lem}
\begin{proof}
    We may assume $X= \Spec A$ and following the same notation as above, let $D$ is given by the ideal $(x)$, where $x= \prod_i x^{n_i}_i$.
   We only need to show that the map $R$ is well-defined and surjective and $\Fil_DW_m\sO_U$ is a sheaf of abelian groups. Rest of the part follows from \eqref{eqn:Log-fil-1}.  We let $\un{a} \in \Fil_IW_{m}(A_\pi)$ and let
$I^{1/p}= (\prod_i x^{\lfloor{{n_i}/{p}}\rfloor}_i)$.
Then $R(\un{a}) = \un{b} = (b_{m-1}, \ldots , b_0)$, where $b_i = a_{i+1}$. We can write (because $A$ is UFD) $a_{i+1} = a'_{i+1}\prod_jx^{t_j}_j$, where $t_j \in \Z$ for every $j$
and $a'_{i+1} \in A$ is not divisible by any $x_j$.
By our hypothesis, $(a_{i+1})^{p^{i+1}} x\in A$. This implies, $p^{i+1}t_j+n_j \ge 0$ which is equivalent to $p^{i}t_j + \lfloor n_j/p \rfloor \ge 0$. So we get $(b_i)^{p^i}x' \in A$ for
$0 \le i \le m-1$, where $x' = \prod_j x^{\lfloor{{n_j}/{p}}\rfloor}_j$. This implies that
$\un{b} \in \Fil_{I^{1/p}}W_{m}(A_\pi)$. Hence $R$ is well defined. Also if $b =(b_{m-2},\ldots,b_0) \in \Fil_{I^{1/p}}W_{m}(A_\pi)$ then by reversing the above argument we get $(b_{m-2},\ldots,b_0,0) \in \Fil_IW_m(A_\pi)$. So $R$ is surjective and hence we win.  

Now we show $\Fil_DW_m\sO_U$ is a sheaf of abelian groups by induction on $m$. If $m=1$, then $\Fil_D\sO_U=\sO_X(D)$ is a sheaf of abelian group. If $m >1$, let $\un{a}=(a_{m-1}, \ldots,a_0) \text{ and }\un{a}'=(a'_{m-1},\ldots,a'_0) \text{ are in } \Fil_DW_m\sO_U$. Enough to show $\un{a}-\un{a}' \in \Fil_DW_m\sO_U$. Since $R$ is additive, by induction and the above argument we have $R(\un{a}-\un{a}') \in \Fil_{D/p}W_{m-1}\sO_U$. So, we can write $R(\un{a}-\un{a}')= (c_{m-2},\ldots,c_0)$ satisfying the relation \eqref{eqn:Log-fil-1}. Then the previous argument shows $(c_{m-2},\ldots,c_0,0) \in \Fil_DW_m\sO_U$. Now note that $\un{a}-\un{a}'= (c_{m-2},\ldots,c_0, a_0-a'_0)$ which is clearly in $\Fil_DW_m\sO_U$ (because, $a_0-a'_0 \in \Fil_D\sO_U$).
\end{proof}

\begin{lem}\label{lem:Log-fil-2}
For two Weil divisors $D$ and $D'$ with complements $U$ and $U'$, respectively, we have
    $\Fil_D W_m \sO_{U} + \Fil_{D'} W_m \sO_{U'} \subset \Fil_{D+D'} W_m \sO_{U\cap U'}$ and $\Fil_D W_m \sO_{U} \cdot \Fil_{D'} W_m \sO_{U'} \subset \Fil_{D+D'} W_m \sO_{U\cap U'}$ under the
  addition and multiplication of $W_m\sO_\eta$.
  In particular, $\Fil_D W_m \sO_U$ is a subsheaf of
    $W_m\sO_X$-modules inside $W_m\sO_\eta$.
\end{lem}
\begin{proof}
  Since we can check this locally, we assume that $X = \Spec(A)$, where $A$ is
  a local ring. We let $\{x_1, \ldots , x_r\}$ and $\{y_1, \ldots , y_s\}$ be the
  generators of the ideals of the irreducible components of $D$ and $D'$, respectively.
  We let $x = \prod_i x^{m_i}_i$ and $y = {\prod}_j y^{n_j}_j$, where
  $D = \sum_i m_iD_i$ and $D' = \sum_j n_jD'_j$. Write $I = (x)$ and $I' = (y)$.
  Suppose $\un{a} \in \Fil_IW_m(A_x)$ and $\un{b} \in \Fil_{I'}W_m(A_y)$.
  We can write $\un{a} = \stackrel{m-1}{\underset{i =0}\sum}V^{i}([a_{m-1-i}])$ and
  $\un{b} = \stackrel{m-1}{\underset{i = 0}\sum}V^i([b_{m-1-i}])$.
  This yields  ${\un{a}}\cdot {\un{b}} =
  \stackrel{m-1}{\underset{i,j = 0}\sum}V^i([a_{m-1-i}])V^j([b_{m-1-j}])$.

  For $0 \le i,j \le m-1$, using the identity $xV(y)=V(Fx.y)$ we have 
  \begin{equation}\label{eqn:Log-fil-2-0}
    V^i([a_{m-1-i}])V^j([b_{m-1-j}])=\left\{ \begin{array}{ll}
                            0 & \mbox{if $i+j \ge m$} \\
                            V^{i+j}([a_{m-1-i}^{p^j}b_{m-1-j}^{p^i}]) & \mbox{if $i+j <m$}.
                          \end{array}
                        \right.
    \end{equation}

By ~\eqref{eqn:Log-fil-1}, the condition $\un{a} \in \Fil_IW_m(A_x)$ and
  $\un{b} \in \Fil_{I'} W_m(A_y)$ means that
$xa^{p^i}_i \in A$ and $yb^{p^i}_i \in A$ for each $i$.
From this, we get for $i+j<m$ that
$ (a_{m-1-i}^{p^j}b_{m-1-j}^{p^i})^{p^{m-1-i-j}} xy= a^{p^{m-1-i}}_{m-1-i}b^{p^{m-1-j}}_{m-1-j}xy\in A$.
This implies that  $V^i([a_{m-1-i}])V^j([b_{m-1-j}])\in \Fil_{II'}W_m(A_{xy})$. By \lemref{lem:Log-fil-1} we conclude that ${\un{a}}\cdot {\un{b}} \in \Fil_{II'}W_m(A_{xy})$. Also, the claim about the addition immediately follows from \lemref{lem:Log-fil-1}, observing that $\Fil_D W_m \sO_{U} \  (\text{and } \Fil_{D'} W_m \sO_{U'}) \subset \Fil_{D+D'} W_m \sO_{U\cap U'}$.  
This concludes the proof of the first part of the lemma.
The second part follows from the first by taking $D' = 0$ and noting that
$\Fil_0W_m\sO_X = W_m\sO_X$.
\end{proof}

\begin{notat}\label{not:twist}
    For $D$ as in Definition~\ref{defn:Log-fil-0}, let $W_m \sO_X(D)$ be the 
sheaf on $X_\et$ which is locally defined as
$W_m \sO_X(D)  = [x^{-1}]_m \cdot W_m \sO_X \subset j_* W_m\sO_U$,
where $[-]_m \colon j_*\sO_U \to j_*W_m\sO_U$ is the Teichm{\"u}ller homomorphism,
and $D$ is locally defined by the invertible ideal
$(x) \subset K(X)=$ the fraction field of $X$, where $x = \prod_i x^{m_i}_i$. One knows that $W_m \sO_X(D)$
is a sheaf of invertible $W_m\sO_X$-modules. For a sheaf of $W_m\sO_X$-modules
$\sM$, we let $\sM(D) = \sM \otimes_{W_m\sO_X} W_m\sO_X(D)$. If $A$ is a local ring
and $D$ is defined by $I = (x)$, we also write $\sM(D)$ as $M(x^{-1})$ when $\sM$ is
defined by a $W_m(A)$-module $M$.
\end{notat}

\begin{lem}\label{lem:Log-fil-2-1}
  One has the identity $\Fil_{p^{m-1}D}W_m\sO_U = W_m\sO_X(D)$.
\end{lem}
\begin{proof}
Since we can check this locally, we assume that $X = \Spec(A)$ is as in the proof of
  \lemref{lem:Log-fil-2}. We let $I = (x^{p^{m-1}})$.
 Following our notations, we see that
 $\un{a} = (a_{m-1}, \ldots, a_0) \in \Fil_{I}W_m(A_x)$ if and only if $a_i^{p^i}x^{p^{m-1}} \in A $ for all $0\le i\le m-1$. Since $A$ is UFD, this is equivalent to $a_i x^{p^{m-1-i}} \in A$ for all $0\le i\le m-1$, or equivalently $\un{a}. [x]_m=(a_{m-1}, \dots, a_ix^{p^{m-1-i}}, \dots, a_0x^{m-1}) \in W_m(A)$.
This finishes the proof.
\end{proof}

Let us now assume that $E \subset X$ is a simple normal crossing divisor
with irreducible components $\{E_1, \ldots , E_r\}$. We let $j \colon U \inj X$ be the
inclusion of the complement of $E$.
Let $\Div_E(X) \cong \Z^r$ be the group of Weil divisors of $X$
whose support lie in $E$.
\begin{defn}\label{defn:Log-fil-3}
  For $D = \stackrel{r}{\underset{i =1}\sum} m_i[E_i] \in \Div_E(X)$
  and an integer $q \ge 0$, we let
  \[
  \Fil_D W_m \Omega^q_U =  \Fil_D W_m \sO_U \cdot W_m\Omega^q_X(\log E) +
  d(\Fil_D W_m \sO_U) \cdot W_m\Omega^{q-1}_X(\log E)
  \]
considered as a subsheaf of $j_*W_m\Omega^q_U$. We shall write $\Fil_D W_1 \Omega^q_U$ as
  $\Fil_D \Omega^q_U$.
  
Throughout our discussion, we shall treat $\Fil_D W_m \Omega^q_U$ as a
subsheaf of $j_*W_m\Omega^q_U$. If $X = \Spec(A)$, where $A$ is a regular local ring
with maximal ideal \\ $\fm = (x_1, \ldots x_r, y_1, \ldots , y_s)$ and
$\{x_1, \ldots , x_r\}$ are generators of the ideals of the irreducible components of
$E$, then we shall write $\Fil_D W_m \Omega^q_U$ as
$\Fil_I W_m\Omega^q_{A_\pi}$ or $\Fil_I W_m\Omega^q_{A_x}$, where $\pi = \prod_i x_i$ and
$I = (\prod_i x^{m_i}_i)$ and $x=\prod_i x^{m_i}_i$. 
\end{defn}

\begin{lem}\label{lem:Log-fil-4}
  We have the following.
  \begin{enumerate}
  \item
    $\Fil_D W_m \Omega^q_U$ is a sheaf of finite type $W_m\sO_X$-modules.
  \item
    $\Fil_D\Omega^q_U = \Omega^q_X(\log E)(D)$. In particular, $\Fil_{-E}\Omega^d_U =
    \Omega^d_X$, where $d$ is the rank of $\Omega^1_X$.
    \item
    $\varinjlim_n \Fil_{nE} W_m \Omega^q_U \xrightarrow{\cong} j_*W_m\Omega^{q}_U$.
  \item
    $d(\Fil_D W_m \Omega^q_U) \subset \Fil_D W_m \Omega^{q+1}_U$.
 \item
    $F(\Fil_D W_{m+1} \Omega^q_U) \subset \Fil_D W_m \Omega^{q}_U$.
  \item
    $V(\Fil_D W_{m} \Omega^q_U) \subset \Fil_D W_{m+1} \Omega^{q}_U$.
  \item
    $R(\Fil_D W_{m+1} \Omega^q_U) = \Fil_{D/p} W_m \Omega^{q}_U$.
    \item
    $\Fil_D W_m \Omega^q_U \cdot \Fil_{D'} W_m \Omega^{q'}_U \subset
      \Fil_{D+D'} W_m \Omega^{q+q'}_U$ if $|D'| \subset E$.
    \item
      $\Fil_{p^{m-1}D} W_m \Omega^q_U \xrightarrow{\cong} W_m\Omega^q_X(\log E)(D)$.
    \item
      $\Fil_0W_m\Omega^q_U = W_m\Omega^q_X(\log E)$ and
      $\Fil_DW_m\Omega^q_U \subset W_m\Omega^q_X$ if $m_i <0$ for every $1 \le i \le r$.
    \end{enumerate}
\end{lem}
\begin{proof}
  Since all claims of the lemma can be checked locally, 
  we shall assume (whenever necessary in the proof) 
  that $X = \Spec(A)$ and $D \subset X$ are as in \defref{defn:Log-fil-3}. The first part of item (1) is clear.
  To prove that $\Fil_D W_m \Omega^q_U$ is a sheaf of finite type $W_m\sO_X$-modules,
  we can find an effective Weil divisor $D' \in \Div_E(X)$ such that $D'':= D + D'$ is effective.
  In particular, $\Fil_D W_m \Omega^q_U \subset \Fil_{D''} W_m \Omega^q_U \subset
  \Fil_{p^{m-1}D''} W_m \Omega^q_U$. Since $W_m\sO_X$ is a sheaf of Noetherian rings
  (cf. \cite[Lem.~2.9]{Morrow-ENS}), it suffices to show that
  $\Fil_{p^{m-1}D''} W_m \Omega^q_U$ is a sheaf of finite type $W_m\sO_X$-modules.
  But this follows from item (9) because $W_m\Omega^q_X(\log E)$ is a sheaf of
  finite type $W_m\sO_X$-modules (by loc. cit.) and $W_m\sO_X(D'')$ is a sheaf of
  invertible $W_m\sO_X$-modules.
  
When $q = 0$, we have $\Fil_D\Omega^q_U =\Fil_D\sO_X = \sO_X(D) =
  \Omega^0_X(\log E)(D)$. It follows that
  $\Fil_D \sO_U \cdot \Omega^q_X(\log E) = \sO_X(D) \cdot 
  \Omega^q_X(\log E) = \Omega^q_X(\log E)(D)$ when $q > 0$. On the other hand,
  \[
  \begin{array}{lll}
  d(\Fil_D \sO_U) \cdot \Omega^{q-1}_X(\log E) & = & d(\sO_X(D)) \cdot
  \Omega^{q-1}_X(\log E) \\
  & \subset & \sO_X(D)\Omega^1_X(\log E) \cdot \Omega^{q-1}_X(\log E) \\
  & \subset & \Omega^q_X(\log E)(D).
  \end{array}
  \]
  This proves (2).

  For (3), we recall that $j_*W_m\Omega^{q}_U \cong j_*W_m\sO_U \otimes_{W_m\sO_X}
  W_m\Omega^q_X = j_*W_m\sO_U \cdot W_m\Omega^q_X$.
Meanwhile, one checks directly from ~\eqref{eqn:Log-fil-1} that
$j_*W_m\sO_U  = \varinjlim_n \Fil_{nE} W_m\sO_U$. This proves the desired isomorphism.

Next, we note that
  $d(\Fil_D W_m \sO_U \cdot W_m\Omega^q_X(\log E))
  \subset d(\Fil_D W_m \sO_U) \cdot W_m\Omega^q_X(\log E) + \Fil_D W_m \sO_U \cdot
  W_m\Omega^{q+1}_X(\log E) = \Fil_D W_m\Omega^{q+1}_U$.
  We also have
  $$d(d(\Fil_D W_m \sO_U) \cdot W_m\Omega^{q-1}_X(\log E)) \subset
  d(\Fil_D W_m \sO_U) \cdot W_m\Omega^{q}_X(\log E) \subset
  \Fil_D W_m\Omega^{q+1}_U.$$ 
  This proves (4).

We now prove (5). When $q = 0$, the claim is clear using
  ~\eqref{eqn:Log-fil-1} and the fact that
  $F = R \circ \ov{F}$, where $\ov{F}(\un{a}) = (a^{p}_i)$. We now
  assume $q > 0$ and let
  $a \in \Fil_IW_{m+1}(A_\pi), \ w \in W_{m+1}\Omega^q_A(\log\pi)$. We then have
  $F(aw) = F(a)F(w) \in \Fil_I W_m(A_\pi)W_{m}\Omega^q_A(\log\pi)$ by the
  $q = 0$ case and \thmref{thm:Log-DRW}. 

  Suppose next that $\un{a} \in \Fil_I W_{m+1}(A_\pi)$ and
  $w \in W_{m+1}\Omega^{q-1}_A(\log\pi)$. Letting
  $\un{a} = \stackrel{m}{\underset{i =0}\sum}V^{i}([a_{m-i}])$,
  we get
  $Fd(\un{a}) = Fd[a_m] + \stackrel{m-1}{\underset{i=0}\sum} FdV^{i+1}[a_{m-1-i}] =
  Fd[a_m] + d\un{a}'$, where $\un{a}' = (a_{m-1}, \ldots , a_0) \in \Fil_IW_m(A_\pi)$.
  In particular, $d\un{a}' \cdot F(w) \in d(\Fil_IW_m(A_\pi)) \cdot
  W_m\Omega^{q-1}_A(\log\pi) \subset \Fil_I W_m\Omega^{q}_{A_\pi}$.
  We also have
  \begin{equation}\label{eqn:Log-fil-4-0}
    \begin{array}{lll}
    \Fil_I W_m\Omega^q_{A_\pi}\cdot W_m\Omega^{q'}_A(\log\pi)
    & \subset & \Fil_IW_m(A_\pi)W_m\Omega^{q+q'}_A(\log\pi) \\
    & & + d(\Fil_I W_m(A_\pi))W_m\Omega^{q+q'-1}_A(\log\pi) \\
    & = & \Fil_I W_m\Omega^{q+q'}_{A_\pi}.
    \end{array}
    \end{equation}
Since $F(d[a_m] w) = Fd[a_m] \cdot F(w)$
  and $F(w) \in W_m\Omega^{q-1}_A(\log\pi)$
  by \thmref{thm:Log-DRW}, it remains to show
  using ~\eqref{eqn:Log-fil-4-0} that $Fd[a_m] \in \Fil_IW_m\Omega^1_{A_\pi}$.

  We write $a_m = a'\prod_i x^{n_i}_i$, where $n_i \in \Z$ for every $i$
  and $a' \in A$ is not divisible by any $x_i$.
  Then the condition $\un{a} \in \Fil_I W_{m+1}(A_\pi)$ implies
  that $[\prod_i x^{n_i}_i]_{m+1} \in \Fil_I W_{m+1}(A_\pi)$. Letting $\gamma = \prod_i x^{n_i}_i$, we get
  $$ Fd[a_m]_{m+1} = Fd[a']_{m+1} \cdot F([\gamma]_{m+1}) + F([a_m]_{m+1})\dlog([\gamma]_m).$$
 Note that $\dlog[\gamma]_m= \sum\limits_{i}n_i.\dlog[x_i]_m \in W_m\Omega^1_A(\log \pi)$ (cf. \thmref{thm:Log-DRW}). 
Hence, $ F([a_m]_{m+1})\dlog([\gamma]_m) \in  \Fil_I W_m\Omega^1_{A_\pi}$. On the other hand, we have $Fd[a']_{m+1} \cdot F([\gamma]_{m+1}) \in \Fil_I W_m(A_\pi) \cdot
  W_m\Omega^1_A \subset \Fil_I W_m\Omega^1_{A_\pi}$.
 This proves (5).

We next show (6). As the $q = 0$ case is clear using
~\eqref{eqn:Log-fil-1}, we shall assume $q > 0$.
We let $\un{a} \in \Fil_IW_m(A_\pi)$ and
$w \in W_m\Omega^q_A(\log\pi)$. We can write $w$ as the sum of elements of the form
$V^{j_0}([b_0])dV^{j_1}([b_1]) \cdots dV^{j_t}([b_t])\dlog[x_{i_1}]\cdots \dlog[x_{i_s}]$,
where $1 \le i_1 < \cdots < i_s \le r, \ t = q-s$ and $b_i \in A$. Letting
$\un{b} = \un{a}V^{j_0}[b_0]$ and $w' = dV^{j_1}([b_1]) \cdots dV^{j_t}([b_t])$, we get
$$V(\un{b}w') = V(\un{b})dV^{j_1+1}[b_1]\cdots dV^{j_t+1}([b_t])
\in \Fil_I W_{m+1}(A_\pi) \cdot W_{m+1}\Omega^t_A.$$
This yields
\[
\begin{array}{lll}
V(\un{a}w) & = & V(\un{b}w'\dlog[x_{i_1}]\cdots \dlog[x_{i_s}]) = 
V(\un{b}w')\cdot \dlog[x_{i_1}]\cdots \dlog[x_{i_s}] \\
& \in &
\Fil_I W_{m+1}(A_\pi) \cdot W_{m+1}\Omega^t_A \cdot W_{m+1}\Omega^s_A(\log\pi) \\
& \subset & \Fil_IW_{m+1}\Omega^t_A(\log\pi) \cdot  W_{m+1}\Omega^s_A(\log\pi) \\
& \subset & \Fil_{m+1}W_{m+1}\Omega^q_{A_\pi}.
\end{array}
\]
If $\un{a} \in \Fil_I W_m(A_\pi)$ and $w \in W_m\Omega^{q-1}_A(\log\pi)$, then
$V(d\un{a} \cdot w) = dV(\un{a})\cdot V(w) \in d(\Fil_IW_{m+1}(A_\pi))\cdot
W_{m+1}\Omega^{q-1}_A(\log\pi) \subset \Fil_I W_{m+1}\Omega^{1}_{A_\pi}
\cdot W_{m+1}\Omega^{q-1}_A(\log\pi) \subset \Fil_I W_{m+1}\Omega^{q}_{A_\pi}$.
This proves (6).

To prove (7) for $q = 0$ follows from \corref{lem:Log-fil-1}. The $q > 0$ case now easily follows from $q=0$ case and \corref{cor:Log-DRW-0} because
$R$ is multiplicative and commutes with $d$.

For (8), using \lemref{lem:Log-fil-2} and an elementary computation, it suffices to
check that $\Fil_IW_m(A_\pi)\cdot d(\Fil_{I'}W_m(A_\pi)) \subset
\Fil_{II'}W_m\Omega^1_{A_\pi}$, where $x = \prod_ix^{m_i}_i$ (resp. $y = \prod_i x^{n_i}_i$)
and $I = (x)$ (resp. $I' = (y)$).
We let $\un{a} = \stackrel{m-1}{\underset{i =0}\sum}V^{i}([a_{m-1-i}])
\in \Fil_I W_m(A_\pi)$ and
$\un{b} = \stackrel{m-1}{\underset{i = 0}\sum}V^i([b_{m-1-i}]) \in
\Fil_{I'} W_m(A_\pi)$. This yields ${\un{a}}\cdot d{\un{b}} =
\sum_{i,j} V^i([a_{m-1-i}])dV^j([b_{m-1-j}])$.
It suffices now to show that $V^i([a_{m-1-i}])dV^j([b_{m-1-j}]) \in
\Fil_{II'}W_m\Omega^1_{A_\pi}$ for every $0 \le i,j \le m-1$.

If we write $b_{m-1-i} =b'\prod_l x^{r_l}_l$, where $b' \in A, \ r_l \in \Z$ and no
$x_l$ divides $b'$, then we get that $\prod_l [x^{r_l}_l] \in \Fil_{I'}W_m({A_\pi})$. We write
$\gamma' = \prod_l x^{r_l}_l$ .

We first assume $i \ge j$ and write 
$$V^i([a_{m-1-i}])dV^j([b_{m-1-j}]) =
V^j(V^{i-j}[a_{m-1-i}] d([b_{m-1-j}])).$$
It is enough to show using (6) that
$V^{i-j}[a_{m-1-i}] d([b_{m-1-j}]) \in \Fil_{II'}W_{m-j}\Omega^1_{A_\pi}$.
We now compute
\[
d([b_{m-1-j}])V^{i-j}([a_{m-1-i}]) = d([b'])[\gamma']V^{i-j}([a_{m-1-i}]) + [b'\gamma']
V^{i-j}([a_{m-1-i}]) \dlog([\gamma']).
\]
The terms $[\gamma']V^{i-j}([a_{m-1-i}])$ and $[b'\gamma']
V^{i-j}([a_{m-1-i}])$ lie $\Fil_{II'}W_m({A_\pi})$ by \lemref{lem:Log-fil-2}.
Since $d([b']) \in W_m\Omega^1_{A}$ and $\dlog([\gamma']) \in W_m\Omega^1_A(\log\pi)$
(cf. \thmref{thm:Log-DRW}),
we conclude that $d([b_{m-1-j}])V^{i-j}([a_{m-1-i}]) \in \Fil_{II'}W_m\Omega^1_{A_\pi}$.

If $i \le j$, then by Leibniz rule, \lemref{lem:Log-fil-2} and item (4) it reduces to the case $i\ge j$. Hence we are done.

 To prove (9), we write $I' = I^{p^{m-1}}$ and note that
 $\Fil_{I'}W_m\Omega^q_{A_\pi} = \Fil_{I'}W_m(A_\pi) \cdot W_m\Omega^q_{A}(\log\pi) +
 d(\Fil_{I'}W_m(A_\pi))\cdot W_m\Omega^{q-1}_A(\log\pi) =
 W_m\Omega^q_A(\log\pi)(x^{-1}) + d([x^{-1}]_mW_m(A))\cdot W_m\Omega^{q-1}_A(\log\pi)$
 by \lemref{lem:Log-fil-2-1}. On the other hand,

 \[
 \begin{array}{lll}
   d([x^{-1}]_m W_m(A))\cdot
 W_m\Omega^{q-1}_A(\log\pi) & \subset & [x^{-1}]_md(W_m(A))\cdot
 W_m\Omega^{q-1}_A(\log\pi) \\
 & & +  [x^{-1}]_m W_m(A)\dlog([x]_m)\cdot
 W_m\Omega^{q-1}_A(\log\pi) \\
 & \subset & W_m\Omega^{q}_A(\log\pi)(x^{-1}).
 \end{array}
 \]
 The desired isomorphism now follows.

 Finally, to prove (10), we only need to check its second part as the first part
 is clear. To prove its second part, it suffices to show that
$\Fil_{-E}W_m\Omega^q_U \subset W_m\Omega^q_X$.
To show this, we let $\un{a} = (a_{m-1}, \ldots , a_0) \in \Fil_IW_m(A_\pi)$
 and $w \in W_m\Omega^q_A(\log\pi)$. Letting $\un{a} = \sum_i V^i([a_{m-1-i}]_{m-i})$,
 we have that $a_i^{p^i}\pi^{-1} \in A$.
 This implies that $\pi \mid a_i^{p^i}$. Equivalently, $\pi \mid a_i$ (because $A$ is a
 UFD). We write $a_i = \pi \alpha_i$, where $\alpha_i \in A$.
 We can write $w$ as the sum of elements of the form
$$V^{j_0}([b_0])dV^{j_1}([b_1]) \cdots dV^{j_t}([b_t])\dlog([x_{i_1}]_m)\cdots \dlog([x_{i_s}]_m),$$
where $1 \le i_1 < \cdots < i_s \le r, \ t = q-s$ and $b_i \in A$.
We have
\[
V^i([a_{m-1-i}]_{m-i})\dlog([x_{i_1}]_m)\cdots \dlog([x_{i_s}]_m) = \hspace{4cm}
\]
\[
\hspace{4cm} V^i([\alpha_{m-i-1}]_{m-i} [\pi]_{m-i} \dlog([x_{i_1}]_{m-i}) \cdots \dlog([x_{i_s}]_{m-i})).
\]
But this last term clearly lies in $W_m\Omega^q_A$.
Taking the sum over $i \in \{0, \ldots , m-1\}$,
we get that $\un{a}\dlog([x_{i_1}]_m)\cdots \dlog([x_{i_s}]_m)
  \in W_m\Omega^q_A$. But this implies that $\un{a}w \in W_m\Omega^q_A$.
  We have thus shown that
$\Fil_IW_m(A_\pi)W_m\Omega^q_{A}(\log\pi) \subset W_m\Omega^q_A$.

If $\un{a} \in \Fil_IW_m(A_\pi)$ and $w \in W_m\Omega^{q-1}_A(\log\pi)$, then
$d(\un{a})w = d(\un{a}w) - \un{a}d(w)$. We showed above that
$\un{a}w \in  W_m\Omega^{q-1}_A$ and $\un{a}d(w) \in W_m\Omega^q_A$.
It follows that $d(\un{a})w \in W_m\Omega^q_A$.
We have thus shown that $\Fil_IW_m(A_\pi)\subset W_m\Omega^q_A$.
\end{proof}

\chapter{Structure of the filtered de Rham-Witt complex I}\label{chap:F_p}
The main purpose of this chapter is to provide a concrete description of the filtered de Rham-Witt complex defined in \S~\ref{chap:FilDR}, when $X$ is spectrum of a ring of the form $S[[\pi]]$ and the SNCD is given by the ideal $(\pi)$. Using this description, we also prove some results that are well-known in the case of de Rham-Witt complex of the regular schemes.

\section{Filtered de Rham-Witt complex of complete local \texorpdfstring{$\F_p$}{Fp}-algebras}
\label{sec:Complete-local}
We let $S$ be an $F$-finite Noetherian regular local $\F_p$-algebra. Let
$A = S[[\pi]]$ and $K = A_\pi$. For $n \in \Z$ and
the invertible ideal $I = (\pi^n) \subset K$,
we shall denote $\Fil_IW_m\Omega^q_{K}$ by $\Fil_nW_m\Omega^q_K$ and write
$M(\pi^n)$ as $M(n)$ (cf. \ref{not:twist}) for any $W_m(A)$-module $M$.
We let $\gr_nW_m\Omega^q_K = \frac{\Fil_nW_m\Omega^q_K}{\Fil_{n-1}W_m\Omega^q_K}$.

The goal of this section is
to prove some key results about the structure of $\Fil_nW_m\Omega^q_K$.

We shall use the following theorem of
Geisser-Hesselholt \cite[Thm.~B]{Geisser-Hesselholt-Top} for proving our main results.
Let $I_p$ denote the set of positive integers prime to $p$.

\begin{thm}$($Geisser-Hesselholt$)$\label{thm:GH-Top}
  Every element $\omega \in W_m\Omega^q_A$ can be uniquely written as an infinite
  series
  \[
  w = {\underset{i \ge 0}\sum} a^{(m)}_{0,i} [\pi]^i_m +
  {\underset{i \ge 0}\sum} b^{(m)}_{0,i} [\pi]^i_m d[\pi]_m +
  {\underset{s \ge 1}\sum}{\underset{j \in I_p}\sum}
  \left(V^s(a^{(m-s)}_{s,j}[\pi]^j_{m-s}) + dV^s(b^{(m-s)}_{s,j}[\pi]^{j}_{m-s})\right),
  \]
  where the components $a^{(n)}_{s,j} \in W_n\Omega^q_S$ and
    $b^{(n)}_{s,j} \in W_n\Omega^{q-1}_S$.
\end{thm}

For $q, s \ge 0, m \ge 1$ and $i \in \Z$, we let
\begin{equation}\label{eqn:GH-0}
  A^{m,q}_i(S) = W_m\Omega^q_S [\pi]^i_m+ W_m\Omega^q_S [\pi]^i_m\dlog([\pi]_m);
\end{equation}
\[
B^{m,q}_{i,s}(S) = V^s\left(W_{m-s}\Omega^q_S [\pi]^i_{m-s}\right) +
dV^s\left(W_{m-s}\Omega^{q-1}_S [\pi]^i_{m-s}\right).
\]
Note that $A^{m,q}_i(S), B^{m,q}_{i,s}(S) \subset W_m\Omega^q_K$.

\begin{defn}\label{defn:GH-1}
  For $q \ge 0, \ m \ge 1$ and $n \in \Z$, we let
  \[
  F^{m,q}_n(S) =
 \left\{\begin{array}{ll} A^{m,q}_{{n}/{p^{m-1}}}(S) &  \mbox{if $p^{m-1} \mid n$} \\
    B^{m,q}_{{n}/{p^{m-1-s}}, s}(S) & \mbox{if $s = {\rm min}\{j > 0|p^{m-1-j} \mid n\}$}.
  \end{array}\right.
  \]
\end{defn}

\begin{cor}\label{cor:GH-2}
  For every $m \ge 1$ and $q \ge 0$, there is a canonical decomposition
  \begin{equation}\label{eqn:GH-2-0}
    \theta^{m,q}_A \colon {\underset{n \ge 0}\prod} F^{m,q}_n(S) \xrightarrow{\cong} W_m\Omega^q_A(\log\pi) 
          .
  \end{equation}
  \end{cor}

\begin{proof}
   We note from \thmref{thm:GH-Top} that the canonical (once the uniformizer $\pi$ is fixed) maps from $W_m\Omega^q_S,\ F^{m,q}_n(S)$ to $W_m\Omega^q_A$ (for $n \ge 1$) induce an isomorphism $$\ov \theta_A^{m,q}:  W_m\Omega^q_S \bigoplus {\underset{n \ge 1}\prod} F^{m,q}_n(S) \xrightarrow{\cong} W_m\Omega^q_A.$$
    Then $\ov \theta_A^{m,q}$ can be uniquely extended to 
    $$\theta_A^{m,q}: {\underset{n \ge 0}\prod} F^{m,q}_n(S) = W_m\Omega^{q-1}_S \dlog[\pi]_m \bigoplus W_m\Omega^q_S \bigoplus {\underset{n \ge 1}\prod} F^{m,q}_n(S) $$ 
    $$\xrightarrow{Id \ \bigoplus \ \ov \theta_A^{m,q}}W_m\Omega^{q-1}_S \dlog[\pi]_m \bigoplus W_m\Omega^q_A \xrightarrow{\cong}  W_m\Omega^q_A(\log \pi)$$
    where the last isomorphism follows from \lemref{cor:LWC-0-1}.
\end{proof}

In order to extend this decomposition to $W_m\Omega^q_K$, we need a relation  between the groups $\{F^{m,q}_{n}(S)\}$ for
positive and negative values of $n \in \Z$. This can be achieved by the following lemma.

\begin{lem}\label{lem:Positive-F-2}
  For all $n,l \in \Z, m\ge 1,q \ge 0$, 
  we have
  \[
[\pi]_m^{l}F^{m,q}_n(S) = F^{m,q}_{n+ p^{m-1} l}(S).\]
\end{lem}
\begin{proof}
 We first consider the
  case $n = p^{m-1}i$ with $i \in \Z$. In this case, one checks that
  $F^{m,q}_n(S) = [\pi]^i_mF^{m,q}_0(S)$. This implies that
  $[\pi]^{l}_m F^{m,q}_n(S) =  [\pi]^{i+l}_mF^{m,q}_0(S) = F^{m,q}_{p^{m-1}(i+l)}(S)=F^{m,q}_{n+p^{m-1}l}$. 
  
  Suppose now that $n = p^{m-1-s}i$ with $0 <s \le m-1$ and $|i| \in I_p$.
  Let $t = n+p^{m-1}l = p^{m-1-s}i + p^{m-1}l = p^{m-1-s}(i +p^sl) = p^{m-1-s}i'$,
  where we let $i' = i+p^sl \in I_p$.
We let $r = m-s$. We then have
  \[
  \begin{array}{lllr}
    [\pi]^{l}_mF^{m,q}_n(S) & = & [\pi]^{l}_m(V^s(W_r\Omega^q_S[\pi]^i_{r}) +
    dV^s(W_r\Omega^{q-1}_S[\pi]^i_r)) \\
    & = & V^s(W_r\Omega^q_S[\pi]^{i'}_r) +  [\pi]^{l}_mdV^s(W_r\Omega^{q-1}_S[\pi]^i_r) 
    \\
    & \subset_\dagger & V^s(W_r\Omega^q_S[\pi]^{i'}_r) + dV^s(W_r\Omega^{q-1}_S[\pi]^{i'}_r) \\
    && +  V^s(W_r\Omega^{q-1}_S[\pi]^i_rF^sd[\pi]^{l}_r)\\
    & \subset & F^{m,q}_{t}(S) +  V^s(W_r\Omega^{q-1}_S[\pi]^i_rF^sd[\pi]^{l}_r).
  \end{array}
  \]
  To obtain $\subset_\dagger$, we applied Leibniz rule ($xdy=d(xy)-ydx$) on $[\pi]^{l}_mdV^s(W_r\Omega^{q-1}_S[\pi]^i_r)$. We also used the relations $aV(b)=V(F(a).b)$ and $F(d[t]_r)=[t]_{r-1}^{p-1}d[t]_{r-1}$, for $a, b \in W_r\Omega^q_A$ and $t\in A$ in other equalities.
  Meanwhile, we have
  \[
  \begin{array}{lll}
    V^s(W_r\Omega^{q-1}_S[\pi]^i_rF^sd[\pi]^{l}_r) & \subset &
    V^s(W_r\Omega^{q-1}_Sd([\pi]^{i'}_r)) \\
    & \subset & V^sd(W_r\Omega^{q-1}_S[\pi]^{i'}_r) + V^s(d(W_r\Omega^{q-1}_S)[\pi]^{i'}_r) \\
    & \subset & dV^s(W_r\Omega^{q-1}_S[\pi]^{i'}_r) + V^s(W_r\Omega^q_S[\pi]^{i'}_r) \\
    & = & F^{m,q}_{t}(S).
  \end{array}
  \]
  This shows that $[\pi]^l_mF^{m,q}_n(S) \subset F^{m,q}_{n+p^{m-1}l}(S)$ for all $l,n \in \Z$.
  But this also implies $[\pi]^{-l}_mF^{m,q}_{n+p^{m-1}l}(S) \subset F^{m,q}_{n}(S)$ and hence we are done. 
\end{proof}

Combining \corref{cor:GH-2} and \lemref{lem:Positive-F-2}, we get

\begin{cor}\label{cor:GH-3}
  For every $m \ge 1$ and $q \ge 0$, there is a canonical decomposition
  \[
  \theta^{m,q}_K \colon
  \left(\bigoplus_{n <0} F^{m,q}_{n}(S)\right)
 \bigoplus \left(\prod_{n\ge 0} F^{m,q}_n(S)\right) \xrightarrow{\cong}  W_m\Omega^q_K.
  \]
\end{cor}
\begin{proof}
We first claim that the canonical map
  \begin{equation}\label{eqn:Base-change-0}
  W_m\Omega^q_A(\log\pi) \otimes_{W_m(A)} W_m(K) \to W_m\Omega^q_K
  \end{equation}
  is an isomorphism of $W_m(K)$-modules.
  Indeed, we know by \cite[0.1.5.3]{Illusie}  that $W_m(K) = W_m(A)[u^{-1}]$, where the
 latter is the localization of $W_m(A)$ obtained by inverting
  $u = [\pi]_m$. This implies that $M \otimes_{W_m(A)} W_m(K) \cong M [u^{-1}]$
  for $M \in \{W_m\Omega^q_A, W_m\Omega^q_A(\log\pi)\}$.
  Since $W_m\Omega^q_A \subset W_m\Omega^q_A(\log\pi) \subset W_m\Omega^q_K$, we get
$W_m\Omega^q_A[u^{-1}] \subset W_m\Omega^q_A(\log\pi)[u^{-1}] \subset
  W_m\Omega^q_K$. The claim now follows because the composite inclusion is an
  isomorphism by \cite[Thm.~C]{Hesselholt-Acta}.

By ~\eqref{eqn:Base-change-0}, we now have
  \[
  \begin{array}{lll}
    W_m\Omega^q_K & \cong & W_m\Omega^q_A(\log\pi)[u^{-1}] \ \ \cong \ \
    {\varinjlim}_{l \ge 0}
      [\pi]^{-l}_m W_m\Omega^q_A(\log\pi) \\
      & {\cong}^1 &  {\varinjlim}_{l \ge 0} [\pi]^{-l}_m \prod_{n \ge 0} F^{m,q}_n(S) 
      \ \ {\cong} \ \ {\varinjlim}_{l \ge 0}  \prod_{n \ge 0} [\pi]^{-l}_m  F^{m,q}_n(S) \\ 
      & {\cong}^2 & {\varinjlim}_{l \ge 0} \prod_{n \ge 0} F^{m,q}_{n-p^{m-1}l}(S)
      \ \ \cong \ \ {\varinjlim}_{l \ge 0} \prod_{n \ge -p^{m-1}l}  F^{m,q}_n(S) \\
      & \cong &  \left[{\varinjlim}_{l \ge 0}
      \left({\bigoplus}_{-p^{m-1}l \le n \le -1} F^{m,q}_n(S)\right)\right]
      \bigoplus \left[{\varinjlim}_{l \ge 0}
        \left(\prod_{n \ge 0} F^{m,q}_n(S)\right)\right] \\
      & \cong & \left[{\varinjlim}_{l \ge 0}
      \left({\bigoplus}_{-l \le n \le -1} F^{m,q}_n(S)\right)\right] \bigoplus \left(\prod_{n \ge 0} F^{m,q}_n(S)\right)
       \\
      & \cong & \left(\bigoplus_{n < 0} F^{m,q}_n(S)\right) \bigoplus \left(\prod_{n \ge 0} F^{m,q}_n(S)\right),
\end{array}
  \]
  where the isomorphisms ${\cong}^1$ and ${\cong}^2$ hold by \corref{cor:GH-2} (induced by $(\theta^{m,q}_A)^{-1}$) and
  \lemref{lem:Positive-F-2}, respectively. 
\end{proof}

Our task now is to describe the image of $\Fil_nW_m\Omega^q_K$ under $\theta^{m,q}_K$.
We shall achieve this in several steps.
We begin by checking some further properties of the groups $F^{m,q}_n(S)$.

\begin{lem}\label{lem:F-group-1}
  For any $n \in \Z, m \ge 1, q \ge 0$, we have the following.
  \begin{enumerate}
    \item
      $d(F^{m,q}_n(S)) \subset F^{m,q+1}_n(S)$.
    \item
      $V(F^{m,q}_n(S)) \subset F^{m+1,q}_n(S)$.
    \item
      $F(F^{m,q}_n(S)) \subset F^{m-1,q}_n(S)$.
    \item
      $R(F^{m,q}_n(S)) = \left\{\begin{array}{ll}
F^{m-1,q}_{n/p}(S) & \mbox{if $p|n$} \\
0  & \mbox{if $p\nmid n$}.
\end{array}\right.$

      \end{enumerate}
\end{lem}
\begin{proof}
  We divide the proof into two cases: when $p^{m-1}|n$ and when $p^{m-1}\nmid n$.
  \\
  $Case~1:$ $n = p^{m-1}i$. \\
  In this case, 
  $F^{m,q}_n(S) = (W_m\Omega^q_S + W_m\Omega^{q-1}_S\dlog([\pi]_m))[\pi]^i_m$.
  Since
  $d(W_m\Omega^q_S[\pi]^i_m) \subset  W_m\Omega^{q+1}_S[\pi]^i_m +
    W_m\Omega^q_S[\pi]^i_m \dlog([\pi]_m) = F^{m,q+1}_n(S)$,
    and, $d(W_m\Omega^{q-1}_S[\pi]^i_m\dlog([\pi]_m)) \subset
    W_m\Omega^q_S[\pi]^i_m\dlog([\pi]_m)
     \subset F^{m,q+1}_n(S)$, the claim (1) follows.

    For (2), note that if $i \in I_p$, we have
    $V(W_m\Omega^q_S[\pi]^i_m) \in F^{m+1,q}_n(S)$ by latter's definition.
    Next, we have
    \[
    \begin{array}{lll}
      V(W_m\Omega^{q-1}_S\dlog([\pi]_m)[\pi]^i_m) & = &
      V(W_m\Omega^{q-1}_S d([\pi]^i_m)) \\
      & \subset & Vd(W_m\Omega^{q-1}_S [\pi]^i_m) + V(W_{m}\Omega^{q}_S[\pi]^i_m) \\
      & \subset & F^{m+1,q}_n(S).
    \end{array}
    \]
    If $i = pi'$, then $n = p^mi'$ and
    $V(W_m\Omega^q_S[\pi]^i_m)
    \subset W_{m+1}\Omega^q_S[\pi]^{i'}_{m+1} \subset F^{m+1,q}_n(S)$.
   Also, we have $V(W_m\Omega^{q-1}_S\dlog[\pi]_m[\pi]^i_m)
    \subset  W_{m+1}\Omega^{q-1}_S\dlog[\pi]_{m+1}[\pi]^{i'}_{m+1} \subset F^{m+1,q}_n(S)$.
    
    For (3), write $n = p^{m-2}i'$, where $i' = pi$. We then get
    $F(W_m\Omega^q_S[\pi]^i_m) \subset W_{m-1}\Omega^q_S[\pi]^{i'}_{m-1}$ 
    which lies in $F^{m-1,q}_n(S)$. Similarly,
    \[
    F(W_m\Omega^{q-1}_S\dlog([\pi]_m)[\pi]^i_m) \subset
    W_{m-1}\Omega^{q-1}_S\dlog([\pi]_{m-1})[\pi]^{i'}_{m-1} \subset F^{m-1,q}_n(S).
    \]
   For (4), write $n' = n/p = p^{m-2}i$ so that
    $R(W_m\Omega^q_S[\pi]^i_m) = W_{m-1}\Omega^q_S[\pi]^{i}_{m-1}$. Similarly,
    $R(W_m\Omega^{q-1}_S\dlog[\pi]_m[\pi]^i_m)
    = W_{m-1}\Omega^{q-1}_S\dlog[\pi]_{m-1}[\pi]^i_{m-1}$. Hence $R(F^{m,q}_{n}(S))=F^{m-1,q}_{n'}(S)$.
    \\
    \noindent
    \\
    $Case~2:$ $n = p^{m-1-s}i$ with $s \ge 1$ and $|i| \in I_p$. \\
    In this case, we have
    $F^{m,q}_n(S) = V^s(W_{r}\Omega^q_S[\pi]^i_r) + dV^s(W_r\Omega^{q-1}_S[\pi]^i_r)$,
    where $r = m-s$ $= (m+1)-(s+1)$.  But this implies that
    $d(F^{m,q}_n(S)) = dV^s(W_{r}\Omega^q_S[\pi]^i_r) \subset F^{m,q+1}_n(S)$.
    This proves (1). 
    
    For (2), we note that
    $V(F^{m,q}_n(S)) = V^{s'}(W_{r}\Omega^q_S[\pi]^i_r) +
    VdV^s(W_r\Omega^{q-1}_S[\pi]^i_r) \subset$ \\
    $V^{s'}(W_{r}\Omega^q_S[\pi]^i_r) + dV^{s'}(W_r\Omega^{q-1}_S[\pi]^i_r) =
    F^{m+1,q}_n(S)$ if we let $s' = s+1$.
    
    For (3), suppose first that $s = 1$ so that $n = p^{m-2}i$.
    Then
    \[
    \begin{array}{lll}
     F(F^{m,q}_n(S)) & = & FV(W_{m-1}\Omega^q_S[\pi]^i_{m-1}) +
     FdV(W_{m-1}\Omega^{q-1}_S[\pi]^i_{m-1}) \\
     & \subset & W_{m-1}\Omega^q_S[\pi]^i_{m-1} + d(W_{m-1}\Omega^{q-1}_S[\pi]^i_{m-1}) \\
     & \subset & (W_{m-1}\Omega^q_S + 
     W_{m-1}\Omega^{q-1}_S\dlog([\pi]_{m-1}))[\pi]^i_{m-1} = F^{m-1,q}_n(S).
\end{array}
    \]
    If $s \ge 2$, we let $s' = s-1 \ge 1$ so that $n = p^{m-2-s'}i$ and
    $r = (m-1) - s'$.
    This yields
    \[
    \begin{array}{lll}
      F(F^{m,q}_n(S)) & = & FV^s(W_{r}\Omega^q_S[\pi]^i_{r}) +
      FdV^s(W_{r}\Omega^{q-1}_S[\pi]^i_{r}) \\
      & \subset &
    V^{s'}(W_{r}\Omega^q_S[\pi]^i_{r}) +
    dV^{s'}(W_{r}\Omega^{q-1}_S[\pi]^i_{r}) = F^{m-1,q}_n(S).
    \end{array}
    \]
    Finally, for (4), note that
    $$RV^s(W_r\Omega^q_S[\pi]^i_r) + RdV^s(W_r\Omega^{q-1}_S[\pi]^i_r)
    = \hspace{3cm} $$
    $$ \hspace{4cm} V^s(W_{r-1}\Omega^q_S[\pi]^{{n'}/{p^{m-2-s}}}_{r-1})+ dV^s(W_{r-1}\Omega^{q-1}_S
    [\pi]^{{n'}/{p^{m-2-s}}}_{r-1}),$$ where $n' = n/p$. The latter group is in $F^{m-1,q}_{n'}(S)$
    if $s < m-1$ and is zero if $s = m-1$. This concludes the proof.
\end{proof}

\begin{lem}\label{lem:F-group-0}
  We have $F^{m,0}_n(S) \cdot F^{m,q}_{n'}(S) \subset F^{m, q}_{n+n'}(S)$ for $q, n, n' \ge 0, m \ge 1$. 
\end{lem}
\begin{proof}
First note that $W_m(S)\cdot F^{m,q}_{n'}(S) \subset F^{m,q}_{n'}(S)$. Indeed, if $p^{m-1} \mid n'$ then this is obvious. If $p^{m-1} \nmid n'$ then only need to show $W_m(S)\cdot dV^s(W_r(S)[\pi]^i_r) \subset F^{m,q}_{n'}(S)$, where $n'=p^{m-1-s}i$, $|i| \in I_p$. But this follows easily by Leibniz rule.
  
  We break the proof into two cases. We shall assume
  $1 \le s \le m-1$ and write $r = m-s$. \\
 $Case~1:$ $n = p^{m-1}i$. \\
  In this case, $F^{m,0}_n(S) = W_m(S)[\pi]^i_m$.
  Hence, we get 
  $$F^{m,0}_n(S) \cdot F^{m,q}_{n'}(S) =
  W_m(S)[\pi]^i_m \cdot F^{m,q}_{n'}(S)
  \subset F^{m, q}_{n+n'}(S),$$
  by the above observation and \lemref{lem:Positive-F-2}.
  \\
  \noindent
\\
$Case~2:$ $n = p^{m-1-s}i$, where $|i| \in I_p$. \\
In this case, $F^{m,0}_n(S) = V^s(W_{r}(S)[\pi]^i_{m-s})$. Thus, we get
\[
\begin{array}{lll}
  F^{m,0}_n(S) \cdot F^{m,q}_{n'}(S) & = & V^s(W_{r}(S)[\pi]^i_{r}) \cdot 
 F^{m,q}_{n'}(S) \\
  &\subset^1 & V^s([\pi]^i_{r}.F^{m-s,q}_{n'}(S)) \subset^2 V^s(F^{m-s,q}_{n'+n}(S))\\
  & \subset^3 & F^{m-s,q}_{n'+n}(S), 
\end{array}
\]
where the equality $\subset^1$ and $\subset^3$ uses \lemref{lem:F-group-1}, item (3) and (2), respectively and also the above observation, while $\subset^2$ uses \lemref{lem:Positive-F-2}. 
\end{proof}

\begin{lem}\label{lem:F-group-6}
  $F^{m,q}_n(S) \hookrightarrow \Fil_{-n}W_m\Omega^q_K$ (by $\theta^{m,q}_A$) for $n, q \ge 0, m \ge 1$.
\end{lem}
\begin{proof} 
First, by \thmref{thm:GH-Top}, for all $q, i \ge 0$, we have $$W_{m}\Omega^q_S[\pi]^i_m \inj W_{m}\Omega^q_A[\pi]^i_m \text{ and } W_{m}\Omega^q_S[\pi]^i_m \dlog([\pi]_m) \inj W_{m}\Omega^q_A[\pi]^i_m\dlog([\pi]_m).$$ 
Now we consider two cases.
  
  If $n = p^{m-1}i$ and $i \ge 0$, we get
  \[
  \begin{array}{lll}
    F^{m,q}_n(S) & = & A^{m,q}_i(S) =
    (W_m\Omega^q_S+ W_m\Omega^{q-1}_S\dlog([\pi]_m))[\pi]^i_m \\
    & \inj & (W_m\Omega^q_A+ W_m\Omega^{q-1}_A\dlog([\pi]_m))[\pi]^i_m \\
    & = & \Fil_{-n}W_m\Omega^q_K,
  \end{array}
  \]
  where the last equality holds by \lemref{lem:Log-fil-4}(9).

  If $n = p^{m-1-s}i$ with $1 \le s \le m-1$ and $|i| \in I_p$, then
   \[
  \begin{array}{lll}
    F^{m,q}_n(S) & = & B^{m,q}_{i,s} = V^s(W_{r}\Omega^q_S[\pi]^i_r) +
    dV^s(W_r\Omega^{q-1}_S[\pi]^i_r) \\
    & {\inj} & V^s(\Fil_{-p^{r-1}i}W_r\Omega^q_K) +
    dV^s(\Fil_{-p^{r-1}i}W_r\Omega^{q-1}_K)
    \\
    & = & V^s(\Fil_{-n}W_r\Omega^q_K) + dV^s(\Fil_{-n}W_r\Omega^{q-1}_K) \\
    & {\subset} &  \Fil_{-n}W_m\Omega^q_K,
   \end{array}
  \]
  where $r = m-s$ and the last inclusion holds by \lemref{lem:Log-fil-4}(4), (6).
\end{proof}

In order to describe $\Fil_nW_m\Omega^q_K$ in terms of $\{F^{m,q}_n\}$, we need to 
dispose of the case $q = 0$ first. We do this in what follows next.
One deduces from \lemref{lem:F-group-6} for $q=0$ that $F^{m,0}_n(S) \subset \Fil_{-n}W_m(K)$ for $n \ge 0$.

\begin{lem}\label{lem:F-group-2}
  For $n \in \Z$, the composite map $\lambda^{m,0}_n \colon
  F^{m,0}_{n}(S) \inj \Fil_{-n}W_m(K) \surj \gr_{-n}W_m(K)$ is an isomorphism.
\end{lem}
\begin{proof}
 The assertion of the lemma is obvious if $m =1$, because $F^{1,0}_n(S)= S \pi^n$ and $\Fil_{-n}W_1(K)= \pi^n A$. 
 
 For $m > 1$,
 \lemref{lem:Log-fil-1} implies that the sequence
  \begin{equation}\label{eqn:F-group-2-0}
0 \to \Fil_{n}W_1(K) \xrightarrow{V^{m-1}}  \Fil_{n}W_m(K) \xrightarrow{R}
\Fil_{\lfloor{n}/p\rfloor}W_{m-1}(K) \to 0
  \end{equation}
  is exact for every $n \in \Z$. Comparing this sequence for $-n$ and $-n-1$ and noting that $\Fil_{-n-1}W_m(K)\subset\Fil_{-n}W_m(K)$, for all $m \ge 1, n \in \Z$,
  we get an exact sequence
\begin{equation}\label{eqn:F-group-2-1}  
0 \to \gr_{-n}W_1(K) \xrightarrow{V^{m-1}} \gr_{-n}W_m(K) \xrightarrow{R} 
\frac{\Fil_{-\lceil{n/p}\rceil}W_{m-1}(K)}{\Fil_{-\lceil{{n+1}/p}\rceil}W_{m-1}(K)} \to 0.
\end{equation}
If we write $n+1 = pt + q$ with $0 \le q < p$, we see that $V^{m-1}$ is an isomorphism
if $q \neq 1$. Since $s = m-1$ (cf. \defref{defn:GH-1}) in this case, we find also that the image of
$V^{m-1}$ is $F^{m,0}_n(S)$.

If $q =1$, we can rewrite ~\eqref{eqn:F-group-2-1} as
\[
0 \to \gr_{-n}W_1(K) \xrightarrow{V^{m-1}} \gr_{-n}W_m(K) \xrightarrow{R} 
\gr_{-t}W_{m-1}(K) \to 0.
\]
We now look at the diagram
\begin{equation}\label{eqn:F-group-2-2} 
  \xymatrix@C.8pc{
    0 \ar[r] & F^{1,0}_n(S) \ar[r]^-{V^{m-1}} \ar[d]_-{\lambda^{1,0}_n} &
    F^{m,0}_n(S) \ar[r]^-{R} \ar[d]^-{\lambda^{1,0}_n} & F^{m-1,0}_t(S) \ar[r]
    \ar[d]^-{\lambda^{m-1,0}_t} & 0 \\
    0 \ar[r] & \gr_{-n}W_1(K) \ar[r]^-{V^{m-1}} & \gr_{-n}W_m(K) \ar[r]^-{R} &
    \gr_{-t}W_{m-1}(K) \ar[r] & 0.}
\end{equation}
It is easy to check that this diagram is commutative and the top row is exact.
Now, the left and the right vertical arrows are bijective by induction.
We conclude that the same is true for the middle vertical arrow too.
\end{proof}

Before we proceed further, we compare the topologies of $W_m\Omega^q_A(\log\pi)$ induced by various
filtrations. We let $\tau_1$ be the topology on $W_m\Omega^q_A(\log\pi)$ induced by the
filtration $\{W_m\Omega^q_A(\log\pi)(n)\}_{n\ge 0}$ (cf. \ref{not:twist}, \S~\ref{sec:Complete-local}). We let $\tau_2$ be the topology on
$W_m\Omega^q_A(\log\pi)$ induced by the filtration $\{W_m((\pi^n))W_m\Omega^q_A(\log\pi)\}_{n \ge 0}$. We let $\tau_3$ be the
topology on $W_m\Omega^q_A(\log\pi)$ induced by the filtration $\{I^nW_m\Omega^q_A(\log\pi)\}_{n \ge 0}$, where $I = W_m((\pi))$.
 Finally, we let $\tau_4$ be the
topology on $W_m\Omega^q_A(\log\pi)$ induced by the filtration 
$\{\Fil_{-n}W_m\Omega^q_K\}_{n \ge 0}$.

\begin{lem}\label{lem:Topology-0}
  The topologies defined above all agree.
\end{lem}
\begin{proof}
  The topologies $\tau_1, \ \tau_2$ and $\tau_3$ will
  agree on $W_m\Omega^q_A(\log\pi)$ if
  they do so on $W_m(A)$. The agreement between $\tau_1$ and $\tau_2$ in this special
  case is an elementary exercise and the the same between $\tau_2$ and $\tau_3$ follows
  from \cite[Cor.~2.3]{Geisser-Hesselholt-Top}. The agreement between
  $\tau_1$ and $\tau_4$ on $W_m\Omega^q_A(\log\pi)$ follows from \lemref{lem:Log-fil-4}(9).
  \end{proof}

\begin{cor}\label{cor:F-group-3}
For $n \ge 0$, the map $\theta^{m,q}_A$ of ~\eqref{eqn:GH-2-0}
induces an isomorphism
\[
\theta^{m,0}_n \colon {\underset{n' \ge n}\prod}
F^{m,0}_{n'}(S) \xrightarrow{\cong} \Fil_{-n}W_m(K).
\]
\end{cor}
\begin{proof}
We first conclude from \thmref{thm:GH-Top} (or \cite[Cor~2.3]{Geisser-Hesselholt-Top}) that $W_m(A)$ is complete with respect to its
$I$-adic topology. Since $A$ is $F$-finite, the ring $W_m(A)$ is Noetherian
(cf. \cite[Lem.~2.9]{Morrow-ENS}). Furthermore, \lemref{lem:Log-fil-4}(1) implies that
$\Fil_{-n}W_m(K)$ is an ideal of $W_m(A)$. It follows that the $\Fil_{-n}W_m(K)$ is $I$-adically
complete, and hence, by \lemref{lem:Topology-0}, complete with respect to $\tau_4$ topology.
That is, the canonical map
$\Fil_{-n}W_m(K) \to {\varprojlim}_{n' \ge 0} \frac{\Fil_{-n}W_m(K)}{\Fil_{-n-n'}W_m(K)}$
is an isomorphism of $W_m(A)$-modules. Meanwhile, one deduces from
\lemref{lem:F-group-2} that the canonical map
$\stackrel{n'-1}{\underset{i = 0}\oplus} F^{m,0}_{n+i}(S) \to
\frac{\Fil_{-n}W_m(K)}{\Fil_{-n-n'}W_m(K)}$ is an isomorphism.
Passing to the limit as $n' \to \infty$, we get an isomorphism
\[
 {\underset{n' \ge n}\prod}
F^{m,0}_{n'}(S) \xrightarrow{\cong} \Fil_{-n}W_m(K),
\]
which is induced by the canonical map $F^{m,0}_{n'}(S) \to W_m(A)$ (i.e. by $\theta^{m,0}_A$). Hence we achieve the proof of the lemma.  
\end{proof}

Our next task is to generalize \corref{cor:F-group-3} to the case when $q > 0$. For $n \ge 0$, we let $\Fil'_{-n}W_m\Omega^q_K = \theta^{m,q}_A({\underset{j \ge n}\prod}
F^{m,q}_{j}(S))$. By \corref{cor:GH-2}, this is an exhaustive descending filtration of
$W_m\Omega^q_A(\log\pi) = \Fil'_{0}W_m\Omega^q_K$.

\begin{lem}\label{lem:fil'-mod}
   For $n, q \ge 0, m \ge 1$, $\Fil'_{-n}W_m\Omega^q_K \subset W_m\Omega^q_A(\log \pi)$ is a $W_m(A)$-submodule.
\end{lem}
\begin{proof}
    We only need to show that $\Fil'_{-n}W_m\Omega^q_K$ is closed under the action of $W_m(A)$. We let $\omega = \sum_{j \ge n} \alpha_j \in \Fil'_{-n}W_m\Omega^q_K$, where $\alpha_j \in \theta^{m,q}_A(
F^{m,q}_{j}(S))$. Write
  $\omega_i = \stackrel{i}{\underset{j = n}\sum} \alpha_j$. Since $W_m\Omega^q_A(\log \pi)$ is finitely generated $W_m(A)$-module (cf. \lemref{lem:Log-fil-4}(1)), it is complete with respect to $I$-adic topology.
  So, we can write $\omega = \lim_{i \to \infty} \omega_i$, where the limit is taken with respect to
  $\tau_3$ topology. Now let $\gamma = \sum_{j \ge 0} \beta_j \in W_m(A)$, where $\beta_j \in \theta^{m,0}_A(
F^{m,0}_{j}(S))$. Similarly to above, we can write $\gamma = \lim_{i \to \infty} \gamma_i$, where $\gamma_i = \stackrel{i}{\underset{j = 0}\sum} \beta_j$ and the limit is taken in $\tau_3$. Then $\lim_{i \to \infty} \gamma_i \cdot \omega_i = \gamma \cdot \omega$. But $\gamma_i \cdot \omega_i \subset \theta^{m,q}_A(\bigoplus\limits_{j=n}^{2i} F^{m,q}_j(S))$ by \lemref{lem:F-group-0}. Hence $\gamma \cdot \omega \in \sum_{j \ge n}\theta^{m,q}_A ( F^{m,q}_j(S))= \Fil'_{-n}W_m\Omega^q_K$.
\end{proof}

\begin{lem}\label{lem:Topology-1}
  $d(\Fil'_{-n}W_m\Omega^q_K) \subset \Fil'_{-n}W_m\Omega^{q+1}_K$ for every $n, q \ge 0$ and $m \ge 1$.
\end{lem}
\begin{proof}
  We continue with the same notations of the proof of \lemref{lem:fil'-mod}. For $\omega \in \Fil'_{-n}W_m\Omega^q_K$, we can write $\omega = \lim_{i \to \infty} \omega_i$, where the limit is taken with respect to
  $\tau_3$, and hence with respect to $\tau_4$ by \lemref{lem:Topology-0}.
  We now let $n' \ge 0$ be any integer. We can then find $i \gg 0$ such that
  $\omega - \omega_i \in \Fil_{-n'}W_m\Omega^q_K$ for all $i \gg n$. 
  This implies that $d(\omega) - d(\omega_i) = d(\omega- \omega_i) \in
  \Fil_{-n'}W_m\Omega^{q+1}_K$ for all $i \gg n$ by \lemref{lem:Log-fil-4}(4). 
  But this means that $d(\omega) = \lim_{i \to \infty} d(\omega_i) =
  \sum_{j \ge n} d(\alpha_j)$, where $d(\alpha_j)\subset d(\theta^{m,q}_A(
F^{m,q}_{j}(S)))=^1 \theta^{m,q+1}_A(
d(F^{m,q}_{j}(S))) \subset^2 \theta^{m,q+1}_A(
F^{m,q+1}_{j}(S))$. Here the equality $=^1$ holds because the map $\theta^{m,q}_A$ is induced by the natural map $F^{m,q}_{j}(S) \to W_m\Omega^q_A(\log \pi)$, which commutes with $d$. On the other hand, the inclusion $\subset^2$ holds by \lemref{lem:F-group-1}(1).
\end{proof}

\begin{lem}\label{lem:F-group-4}
  We have $\Fil_{-n}W_m\Omega^q_K \subseteq \Fil'_{-n}W_m\Omega^q_K$ for every $n, q \ge 0$ and $m \ge 1$.
\end{lem}
\begin{proof}
  By definition of $\Fil_{-n}W_m\Omega^q_K$ and Corollaries ~\ref{cor:GH-2} and
  ~\ref{cor:F-group-3}, we have
  \[
  \begin{array}{lll}
  \Fil_{-n}W_m\Omega^q_K & = & \Fil_{-n}W_m(K)W_m\Omega^q_A(\log\pi) +
  d(\Fil_{-n}W_m(K))W_m\Omega^{q-1}_A(\log\pi) \\
  & = &  \Fil'_{-n}W_m(K)W_m\Omega^q_A(\log\pi) +
  d(\Fil'_{-n}W_m(K))W_m\Omega^{q-1}_A(\log\pi) \\
  & \subset &  \Fil'_{-n}W_m(K)W_m\Omega^q_A(\log\pi) +
  d(\Fil'_{-n}W_m(K)W_m\Omega^{q-1}_A(\log\pi)) \\
  & & + \Fil'_{-n}W_m(K)d(W_m\Omega^{q-1}_A(\log\pi)).
  \end{array}
  \]

  Next, note that every element of $\Fil'_{-n}W_m(K)W_m\Omega^q_A(\log\pi)$ is finite sum of elements of the form $\gamma \cdot \omega$, where $\gamma \in \Fil'_{-n}W_m(K)$ and $\omega \in W_m\Omega^q_A(\log\pi)= \Fil'_0W_m\Omega^q_K$. Recalling the notations from the proof of \lemref{lem:fil'-mod}, we write  $\omega = \lim_{i \to \infty} \omega_i$, where $\omega_i = \stackrel{i}{\underset{j = 0}\sum} \alpha_j$ and $\gamma = \lim_{i \to \infty} \gamma_i$, where $\gamma_i = \stackrel{i}{\underset{j = n}\sum} \beta_j$ (note the change in indices of $\alpha_j$ and $\beta_j$). The limits are taken in $\tau_3$. As in the proof of \lemref{lem:fil'-mod}, $\gamma \cdot \omega = \lim_{i \to \infty} \gamma_i \cdot \omega_i \subset \Fil'_{-n}W_m\Omega^q_K$. This implies $\Fil'_{-n}W_m(K)W_m\Omega^q_A(\log\pi) \subset \Fil'_{-n}W_m\Omega^q_K$.
  Similarly,
  $$\Fil'_{-n}W_m(K)d(W_m\Omega^{q-1}_A(\log\pi)) \subset 
\Fil'_{-n}W_m(K)W_m\Omega^{q}_A(\log\pi) \subset \Fil'_{-n}W_m\Omega^q_K.$$
Finally, $d(\Fil'_{-n}W_m(K)W_m\Omega^{q-1}_A(\log\pi)) \subset
d(\theta^{m,q-1}_A({\prod}_{j \ge n} F^{m,q-1}_{j}(S))) = d(\Fil'_{-n}W_m\Omega^{q-1}_K)$, and the
latter group lies in $\Fil'_{-n}W_m\Omega^{q}_K$ by \lemref{lem:Topology-1}.
This finishes the proof.
\end{proof}

\begin{cor}\label{cor:F-group-5}
  For $n \ge 0$, $\Fil'_{-n}W_m\Omega^q_K$ is open (hence closed) in
  $W_m\Omega^q_A(\log\pi)$ with
  respect to the topology $\tau_i$ for $1 \le i \le 4$.
\end{cor}
\begin{proof}
 Combining Lemmas~\ref{lem:fil'-mod} and ~\ref{lem:F-group-4}, we get $\Fil'_{-n}W_m\Omega^q_K$ is $\tau_4$-open. Now the corollary follows from \lemref{lem:Topology-0}. 
\end{proof}

\begin{lem}\label{lem:F-group-7}
We have $\Fil'_{-n}W_m\Omega^q_K = \Fil_{-n}W_m\Omega^q_K$ for every $n \ge 0$.
\end{lem}
\begin{proof}
  In view of \lemref{lem:F-group-4}, we only need to show that
  $\Fil'_{-n}W_m\Omega^q_K \subset \Fil_{-n}W_m\Omega^q_K$.
  The latter assertion is obvious if $n =0$ because $\Fil_{-n}W_m\Omega^q_K =
  W_m\Omega^q_A(\log\pi) = \Fil'_{-n}W_m\Omega^q_K$ by
  \corref{cor:GH-2}. We assume therefore that $n \ge 1$.
  In particular, $\Fil'_{-n}W_m\Omega^q_K \subset W_m\Omega^q_A$.
 We now let $\omega \in \Fil'_{-n}W_m\Omega^q_K$. Recalling the notations from the proof of \lemref{lem:fil'-mod} we have
    that $\omega = \lim_{i \to \infty} \omega_i$, where the limit is taken
    with respect to $\tau_3$
    (cf. \cite[Proof of Thm.~B, p.~488]{Geisser-Hesselholt-Top}).
    Meanwhile, \lemref{lem:F-group-6} says that $\omega_i \in \Fil_{-n}W_m\Omega^q_K$
    for each $i \ge n$. Since $\Fil_{-n}W_m\Omega^q_K$ is $\tau_4$-open (and hence
    closed) in $W_m\Omega^q_A(\log\pi)$, by \corref{cor:F-group-5}; we conclude that $\omega \in
    \Fil_{-n}W_m\Omega^q_K$.
    \end{proof}

We shall now describe $\Fil_nW_m\Omega^q_K$ for arbitrary
$n \in \Z$. For any $l \in \Z$, we let $t = n-p^{m-1}l $.

\begin{lem}\label{lem:Positive-F-0}
  We have $\Fil_tW_m(K) = [\pi]^{l}_m \Fil_{n}W_m(K)$ for every $n \in  \Z$.
\end{lem}
\begin{proof}
  Let $x = (a_{m-1}, \ldots , a_0) \in \Fil_nW_m(K)$, that is,
  $a_i^{p^i} \pi^n \in A$ for each $i$.
  Then $[\pi]^{l}_mx = (a_{m-1}\pi^{l}, a_{m-2}\pi^{pl}, \ldots , a_0\pi^{p^{m-1}l})$.
  Letting $b_i = a_i\pi^{p^{m-1-i}l}$, we get $(b_i)^{p^i}\pi^{t}
  = a_i^{p^i}\pi^{p^{m-1}l+t} = a_i^{p^i}\pi^n \in A$. This implies that
  $[\pi]^{l}_m\Fil_nW_m(K) \subset \Fil_{t}W_m(K)$ for all $l,n \in \Z$. Then, also, we have $[\pi]^{-l}_m\Fil_tW_m(K) \subset \Fil_{n}W_m(K)$ and hence the lemma.
 
\end{proof}

\begin{lem}\label{lem:Positive-F-1}
  We have $\Fil_tW_m\Omega^q_K = [\pi]^{l}_m \Fil_{n}W_m\Omega^q_K$ for every $n \in \Z$.
\end{lem}
\begin{proof}
By definition, $\Fil_tW_m\Omega^q_K = \Fil_{t}W_m(K)W_m\Omega^q_A(\log\pi) +
d(\Fil_{t}W_m(K))W_m\Omega^{q-1}_A(\log\pi)$.
By \lemref{lem:Positive-F-0}, $\Fil_{t}W_m(K)W_m\Omega^q_A(\log\pi) 
=  [\pi]^{l}_m \Fil_{n}W_m(K)W_m\Omega^q_A(\log\pi)$ and
\[
\begin{array}{lll}
d(\Fil_{t}W_m(K))W_m\Omega^{q-1}_A(\log\pi) & = &
d([\pi]^{l}_m \Fil_{n}W_m(K))W_m\Omega^{q-1}_A(\log\pi) \\
& \subset & [\pi]^{l}_m d(\Fil_{n}W_m(K)) W_m\Omega^{q-1}_A(\log\pi) \\
& & + \Fil_{n}W_m(K)[\pi]^{l}_mW_m\Omega^{q-1}_A(\log\pi) \dlog([\pi]_m) \\
& \subset &  [\pi]^{l}_m \big( d(\Fil_{n}W_m(K)) W_m\Omega^{q-1}_A(\log\pi)  \\
& & +\Fil_{n}W_m(K)W_m\Omega^{q-1}_A \dlog([\pi]_m)\big) \\ 
& = & [\pi]^{l}_m\Fil_{n}W_m\Omega^q_K.
\end{array}
\]
This shows that $\Fil_tW_m\Omega^q_K \subseteq [\pi]^{l}_m \Fil_{n}W_m\Omega^q_K$ for all $l \in \Z$.
The reverse inclusion also follows like before (cf. \lemref{lem:Positive-F-0}).
\end{proof}

We can now describe the image of $\Fil_\bullet W_m\Omega^q_K$ under the isomorphism
$\theta^{m,q}_K$ of \corref{cor:GH-3} as follows.

\begin{prop}\label{prop:Fil-decom}
  For $n \in \Z$, there is a canonical decomposition
\[
\theta^{m,q}_K \colon
 \prod_{j \ge -n} F^{m,q}_j(S) \xrightarrow{\cong} \Fil_nW_m\Omega^q_K.
  \]
  \end{prop}
\begin{proof}
  For $n \le 0$, this follows from \lemref{lem:F-group-7}.
Suppose now $n \ge 0$. Choose $l >>0$ such that $t =n- p^{m-1}l <0$.
 \lemref{lem:Positive-F-1} then says that
  $\Fil_nW_m\Omega^q_K = [\pi]^{-l}_m\Fil_{t}W_m\Omega^q_K$.
  On the other hand, we have
  \[
    [\pi]^{-l}_m\Fil_{t}W_m\Omega^q_K \ \cong \ 
    [\pi]^{-l}_m\left(\prod_{j \ge- t}F^{m,q}_j(S)\right) \ \cong \
    \prod_{j \ge -t} [\pi]^{-l}_mF^{m,q}_j(S) \ \cong \ \prod_{j \ge -t} F^{m,q}_{j - p^{m-1}l}(S),
    \]
where the first and the last isomorphisms are by \lemref{lem:F-group-7} (induced by $(\theta^{m,q}_A)^{-1}$) and \lemref{lem:Positive-F-2}, respectively.  
By a change of variable $j'=j-p^{m-1}l$, we get
\[
\prod_{j \ge -t} F^{m,q}_{j - p^{m-1}l}(S) = \prod_{j' \ge -n} F^{m,q}_{j'}(S).
\]
Hence, we get an isomorphism $\Fil_nW_m\Omega^q_K \cong \prod_{j' \ge -n} F^{m,q}_{j'}(S)$, whose inverse is induced by the canonical map $F^{m,q}_{j'}(S) \to \Fil_nW_m\Omega^q_K$, for $j' \ge -n$ (because $\cong^1$ is induced by $(\theta^{m,q}_A)^{-1}$). We conclude that the resulting (inverse) isomorphism is $\theta^{m,q}_K$.
\end{proof}

Our first application of \propref{prop:Fil-decom} is the following.

\begin{prop}\label{prop:V-R-Fil}
  For every $n \in \Z$, there is a short exact sequence
  \[
  0 \to V^{m-1}(\Fil_n\Omega^q_K) + dV^{m-1}(\Fil_n\Omega^{q-1}_K) \to
  \Fil_nW_m\Omega^q_K \xrightarrow{R} \Fil_{\lfloor{n/p}\rfloor}W_{m-1}\Omega^q_K \to 0.
  \]
\end{prop}
\begin{proof}
  Surjectivity of $R$ is already shown in \lemref{lem:Log-fil-4}(7). To prove the exactness in the middle, we let
  $A = V^{m-1}(\Fil_n\Omega^q_K) + dV^{m-1}(\Fil_n\Omega^{q-1}_K)$ and
  $B = \Ker(R)$. From the exact sequence (cf. \propref{prop:basics}(1))
  \begin{equation}\label{eqn:V-R-Fil-0}
    0 \to  V^{m-1}(\Omega^q_K) + dV^{m-1}(\Omega^{q-1}_K) \to
    W_m\Omega^q_K \xrightarrow{R} W_{m-1}\Omega^q_K \to 0,
  \end{equation}
  and \lemref{lem:Log-fil-4}(4) and (6), we see that $A \subset B = (V^{m-1}(\Omega^q_K) + dV^{m-1}(\Omega^{q-1}_K)) \bigcap
  \Fil_nW_m\Omega^q_K$.

  We now let $x \in B$. By \propref{prop:Fil-decom}, we can uniquely write
  $x = \sum_{i \ge -n} a_i$ with $a_i \in F^{m,q}_i(S)$.
  By \corref{cor:GH-3}, we can also write 
  $x = V^{m-1}(\sum_{i \ge -l} b_i) + dV^{m-1}(\sum_{i \ge -l} c_i)$ for some $l >> |n|$
and $b_i \in F^{1,q}_i(S), \ c_i \in F^{1,q-1}_i(S)$.
Hence, we get $\sum_{i \ge -n} a_i = \sum_{i \ge -l}(V^{m-1}(b_i) + dV^{m-1}(c_i))$.
Letting $a'_i = (V^{m-1}(b_i) + dV^{m-1}(c_i))$, we get from \lemref{lem:F-group-1}
that $a'_i \in F^{m,q}_i(S)$. We conclude from \corref{cor:GH-3} that
$a_i = a'_i$ for each $i \ge -l$. In particular, $a'_i = 0$ for $i < -n$.

We thus get 
$$x = \sum_{i \ge -n} V^{m-1}(b_i) + \sum_{i \ge -n} dV^{m-1}(c_i)
= V^{m-1}(\sum_{i \ge -n} b_i) + dV^{m-1}(\sum_{i \ge -n} c_i).$$
Note here that $V$ and $d$ are continuous with respect to the topology $\tau_2$
by virtue of Lemmas~\ref{lem:Log-fil-4} and ~\ref{lem:Topology-0}.
Since $b_i \in F^{1,q}_i(S)$ and $c_i \in F^{1,q-1}_i(S)$, \propref{prop:Fil-decom}
implies that $\sum_{i \ge -n} b_i \in \Fil_{n}\Omega^q_K$ and
$\sum_{i \ge -n}c_i \in \Fil_{n}\Omega^{q-1}_K$. Letting
$b = \sum_{i \ge -n} b_i$ and $c = \sum_{i \ge -n}c_i$, we get that
$x = V^{m-1}(b) + dV^{m-1}(c)$, where $b \in  \Fil_{n}\Omega^q_K$ and
$c \in \Fil_{n}\Omega^{q-1}_K$. It follows that $x \in A$.
\end{proof}

\section{Cartier homomorphism I}\label{sec:Cartier}
Recall from \cite{Kato-Duality} (see also \cite[\S~10]{GK-Duality})
that for any regular $\F_p$-scheme $X$, there
is a unique Cartier homomorphism $C \colon Z_1W_m\Omega^q_X \to W_m\Omega^q_X$
such that $W_{m+1}\Omega^q_X \xrightarrow{R} W_m\Omega^q_X$ is the composition
$W_{m+1}\Omega^q_X \stackrel{F}{\surj} Z_1W_m\Omega^q_X  \xrightarrow{C} W_m\Omega^q_X$
and $Z_1W_m\Omega^q_X := {\rm Image}(F) = \Ker(W_{m}\Omega^q_X \xrightarrow{F^{m-1}d}
\Omega^{q+1}_X)$. Our next goal is to explain the behavior of the Cartier homomorphism
when restricted to $\Fil_nW_m\Omega^q_K$.
We first need to prove the following lemma. 
We let $Z_1\Fil_nW_m\Omega^q_K := \Fil_nW_m\Omega^q_K \bigcap Z_1W_m\Omega^q_K$
for $n \in \Z$.

\begin{lem}\label{lem:Cartier-fil-0}
  $Z_1\Fil_nW_m\Omega^q_K = F(\Fil_nW_{m+1}\Omega^q_K)$. In particular the Cartier homomorphism restricts to a map
  $$C : Z_1\Fil_nW_m\Omega^q_K \to \Fil_{\lfloor n/p \rfloor}W_m\Omega^q_K.$$
\end{lem}
\begin{proof}
  We only need to show that
  $\Fil_nW_m\Omega^q_K \bigcap Z_1W_m\Omega^q_K \subset F(\Fil_nW_{m+1}\Omega^q_K)$.
We now let $x \in \Fil_nW_m\Omega^q_K \bigcap Z_1W_m\Omega^q_K = 
\Fil_nW_m\Omega^q_K \bigcap F(W_{m+1}\Omega^q_K)$.
By \corref{cor:GH-3}, \lemref{lem:F-group-1} and \propref{prop:Fil-decom},
we can write $x = \sum_{i \ge -n} a_i = \sum_{i \ge -l} a'_i = F( \sum_{i \ge -l} b_i)$
with $a_i \in F^{m,q}_i(S)$
and $a'_i = F(b_i) \in F^{m,q}_i(S)$ for some $b_i \in F^{m+1,q}_i(S)$ and for some
$l >> |n|$. By the uniqueness of the decomposition in \corref{cor:GH-3},
we must have $a_i = a'_i$ for every $i \ge -l$. In particular, $a'_i = 0$ for
$i < -n$. This implies that $x = \sum_{i \ge -n} a_i = F(\sum_{i \ge -n} b_i)$.
Since $\sum_{i \ge -n} b_i \in \Fil_nW_{m+1}\Omega^q_K$, the desired claim follows.
\end{proof}
\begin{rem}
    In the next chapter, we shall also prove a Cartier isomorphism (cf. \lemref{lem:Complete-4}), which is analogous to the classical case. 
\end{rem}

The following key proposition will be used later in \S~\ref{chap:Kato-coh} to study local Kato-filtration.

\begin{prop}\label{prop:Cartier-fil-1}
  Let $n \ge 0$ and $\omega \in Z_1W_m\Omega^q_K$.
 Then $\omega \in Z_1\Fil_nW_m\Omega^q_K$ if and only if
    $(1-C)(\omega) \in \Fil_nW_m\Omega^q_K$.
\end{prop}
\begin{proof}
  If $\omega \in  Z_1\Fil_nW_m\Omega^q_K$, then $C(\omega) = CF(\omega') =
  R(\omega')$ for some $\omega' \in \Fil_nW_{m+1}\Omega^q_K$
  by \lemref{lem:Cartier-fil-0}.
  It follows from \lemref{lem:F-group-1} that $C(\omega) \in
  \Fil_{\lfloor{n/p}\rfloor}W_m\Omega^q_K \subset \Fil_nW_m\Omega^q_K$.  
  In particular, the map
  $Z_1\Fil_nW_m\Omega^q_K \xrightarrow{1-C} \Fil_nW_m\Omega^q_K$ is defined.

  Suppose now that $\omega \in Z_1W_m\Omega^q_K$ is such that
  $(1-C)(\omega) \in \Fil_nW_m\Omega^q_K$. We can write $\omega = F(\alpha)$ for
  some $\alpha \in W_{m+1}\Omega^q_K$. This yields $(1-C)(\omega) = (1-C)F(\alpha)
  = F(\alpha) - R(\alpha)$. By \corref{cor:GH-3}, we can write $\alpha$ uniquely
  as $\alpha = \sum_{i \ge -l} a_i$ for some $a_i \in F^{m+1,q}_i(S)$ and $l \gg 0$.
  By \lemref{lem:F-group-1}, we get
  \begin{equation}\label{eqn:Cartier-fil-1-0}
    F(\alpha) - R(\alpha) = \sum_{i \ge -l} F(a_i) - \sum_{i \ge -l}R(a_i)
  \in \Fil_nW_m\Omega^q_K.
  \end{equation}

  If $l \le n$, then $-n \le -l$ and hence $\alpha \in \Fil_nW_{m+1}\Omega^q_K$
  so that $\omega \in Z_1\Fil_nW_m\Omega^q_K$.
Let $l > n$. \lemref{lem:F-group-1} says that
  $R(a_{-l}) \in F^{m,q}_{-l/p}(S)$ if $p|l$ and $R(a_{-l}) = 0$ if $|l| \in I_p$. Since $F(a_{-l}) \in F^{m,q}_{-l}$ and $ \sum_{i \ge -l+1} F(a_i) - \sum_{i \ge -l}R(a_i) \in \Fil_{l-1}W_m\Omega^q_K$, it follows from ~\eqref{eqn:Cartier-fil-1-0} that
  $F(a_{-l}) = 0$.  An induction argument says that
  $F(a_i) = 0$ for all $i < -n$. Thus, we get $\omega = F(\alpha) =
  \sum_{i \ge -n} F(a_i) \in \Fil_nW_m\Omega^q_K$. 
  \end{proof}

\chapter{Structure of the filtered de Rham-Witt complex II}\label{chap:F_p-2}
The main purpose of this chapter is to provide a description of the filtered de Rham-Witt complex of a regular local ring, based on the description given in the previous chapter. Finally, we extend various results of the de Rham-Witt complex of regular schemes to the filtered setup.

\section{de Rham complex of multivariate power series ring}
\label{sec:GH-multi-var}
We let $S$ be an $F$-finite Noetherian regular local $\F_p$-algebra and 
let $A = S[[x_1,\ldots , x_d]]$. We let $ 1\le r \le d$ and $K = A_\pi$, where
$\pi = \prod_{1 \le i \le r} x_i$. For $\un{n} = (n_1, \ldots , n_r) \in \Z^r$, and the
invertible ideal $I = (x^{n_1}_1\cdots x^{n_r}_r) \subset K$, we shall
denote $\Fil_I W_m\Omega^q_K$ by $\Fil_{\un{n}}W_m\Omega^q_K$ and write
$M(x^{-n_1}_1\cdots x^{-n_r}_r)$ as $M(\un{n})$ for any $W_m(A)$-module $M$ (cf. Notation~\ref{not:twist}). Also $\un{n}/p$ will denote $(\lfloor n_1/p \rfloor, \ldots ,\lfloor n_r/p \rfloor) \in \Z^r$.
We let $q \ge 0$ and let
\begin{equation}\label{eqn:Multi-0}
  F^{1,q}_0(S) =
{\underset{r+1 \le i_{s+1} < \cdots < i_{q-i_0} \le d}
  {\underset{0 \le i_0,
 1 \le i_1 < \cdots < i_s \le r}\bigoplus}} \Omega^{i_0}_S \dlog(x_{i_1})
\cdots \dlog(x_{i_s}) d(x_{i_{s+1}}) \cdots d(x_{i_{q-i_0}}).
\end{equation}
Note that $ F^{1,q}_0(S) \subset \Omega^q_K$. For $\un{m} = (m_1, \ldots , m_d) \in \Z^d$,
we let $F^{1,q}_{\un{m}}(S) = (\prod_i x^{m_i}_i)F^{1,q}_0(S)$.

\begin{lem}\label{lem:F-function-1}
  For any $\un{n} = (n_1, \ldots , n_r) \in \Z^r$, there is a canonical isomorphism
  of $A$-modules
  \[
  \Fil_{\un{n}}\Omega^q_K \xrightarrow{\cong} {\underset{\un{m} \in \N_0^d}\prod}
  (x^{-n_1}_1\cdots x^{-n_r}_r)F^{1,q}_{\un{m}}(S).
  \]
\end{lem}
\begin{proof}
  From the known description of $\Omega^q_{S[x_1, \ldots , x_d]}$
  and \cite[Lem. 2.9, 2.11]{Morrow-ENS} for $q=0$, we get that $\Omega^q_A(\log\pi) \cong
  {\underset{\un{m} \in \N_0^d}\prod} F^{1,q}_{\un{m}}(S)$.
  The asserted expression for $\Fil_{\un{n}}\Omega^q_K$ now follows from
  its definition and item (2) of \lemref{lem:Log-fil-4}.
  \end{proof}

For $\un{m} = ((m_1, \ldots , m_r), (m_{r+1}, \ldots , m_d)) \in \Z^r \times \N_0^{d-r}$,
we let $x^{\un{m}} = \prod_{1 \le i \le d} x^{m_i}_i$. We say $\un{m} \le \un{m'}$
if $m_i \le m'_i$ for each $1 \le i \le d$. We define $\un{m} < \un{m'}$
similarly. Taking the colimit of
$\Fil_{\un{n}} \Omega^q_K$ as $\un{n} \to \infty$, we get
\begin{cor}\label{cor:F-function-2}
  Every element $\omega \in \Omega^q_K$ can be written uniquely as an infinite series
  $\omega = {\underset{\un{m} \in \Z^r \times \N_0^{d-r}}\sum} a_{\un{m}}x^{\un{m}}$,
  where $a_{\un{m}} \in F^{1,q}_0(S)$ for all $\un{m}$ and $a_{\un{m}} = 0$ if
  $m_i \ll 0$ for some $1 \le i \le r$. Furthermore, $\omega \in \Omega^q_A(\log\pi)$ if
  and only if $a_{\un{m}} = 0$ for every $\un{m}$ such that $m_i < 0$ for some $1 \le i \le r$ and $\omega \in \Fil_{\un{n}}\Omega^q_K$
  if and only if $a_{\un{m}} =0 $ for every $\un{m}$ such that
  $m_i < -n_i$ for some $1 \le i \le r$.
  \end{cor}

\begin{notat}\label{not:F-function-2-0}
    For any $1 \le i \le r$ and $n \in \Z$, we let
$E^{1,q}_{i, n}(K) = \{\omega = {\underset{\un{m} \in \Z^r \times \N^{d-r}}\sum}
a_{\un{m}}x^{\un{m}} : a_{\un{m}} = 0 \ \mbox{if} \ m_i < - n\}$ and
$\wt{E}^{1,q}_{i, n}(K) = \{\omega = {\underset{\un{m} \in \Z^r \times \N^{d-r}}\sum}
a_{\un{m}}x^{\un{m}} : a_{\un{m}} = 0 \ \mbox{if} \ m_i \ge - n\}$.
We let
$\Fil^i_{\un{n}}\Omega^q_K = {\underset{1 \le j \le i}\bigcap} E^{1,q}_{j, n_j}(K)$.
We then get a filtration
\begin{equation}\label{eqn:F-function-2-0}
  \Omega^q_K = \Fil^0_{\un{n}}\Omega^q_K(K) \supset \Fil^1_{\un{n}}\Omega^q_K
  \supset \cdots \supset \Fil^r_{\un{n}}\Omega^q_K = \Fil_{\un{n}}\Omega^q_K.
  \end{equation}
\end{notat}

\section{From complete to the non-complete case}\label{sec:com-to-noncom}
We let $A$ be a regular local $F$-finite $\F_p$-algebra with 
maximal ideal $\fm = (x_1, \ldots , x_d)$. We let $\pi = x_1\cdots x_r$ and
$K = A_\pi$, where $1 \le r \le d$ and $L= Q(A)$. We let $\wh{A}$ denote the $\fm$-adic completion of
$A$ with the maximal ideal $\wh{\fm}$ and let $\wh{K} = \wh{A}_\pi$.
For $1 \le i \le r$, we let $A_i = A_{(x_i)}, \ \wh{A}_i = \wh{A_{(x_i)}}$ and
$\wh{K}_i = Q(\wh{A}_i)$. 
For $\un{n} = (n_1, \ldots , n_r) \in \Z^r$ and the invertible ideal
$I = (x^{n_1}_1\cdots x^{n_r}_r) \subset K$, we let $\Fil_{\un{n}}W_m\Omega^q_K =
\Fil_I W_m\Omega^q_K$. $\un{n}/p$ will denote the same as before. We also let $t\un{n}=(tn_1,\ldots,tn_r)$ for any $t \in \Z$. We let $X = \Spec(A)$ and $X_f = \Spec(A[f^{-1}])$ for any
$f \in A \setminus \{0\}$.

\begin{lem}\label{lem:Complete-1}
    The map $\psi_i \colon W_m\Omega^q_K \to W_m\Omega^q_{\wh{K}_i}$, induced by the
    inclusion $K \inj \wh{K}_i$, is injective.
  \end{lem}
  \begin{proof}
    We look at the commutative diagram
    \begin{equation}\label{eqn:Complete-1-0}
      \xymatrix@C.5pc{
        W_m\Omega^q_A \ar[r]^-{\alpha_1} \ar[d] & W_m\Omega^q_{A_i} \ar[r]^-{\alpha_2}
        \ar[d] & W_m\Omega^q_{\wh{A}_i} \ar[d] \\
        W_m\Omega^q_A {\underset{W_m(A)}\otimes} W_m(A)_{[\pi]_m} \ar[r]^-{\beta_1}
        \ar[d]_-{\cong} &
        W_m\Omega^q_{A_i} {\underset{W_m(A)}\otimes} W_m(A)_{[\pi]_m}
        \ar[r]^-{\beta_2} \ar[d]^-{\cong} & 
        W_m\Omega^q_{\wh{A}_i} {\underset{W_m(A)}\otimes} W_m(A)_{[\pi]_m}
        \ar[d]^-{\cong} \\
        W_m\Omega^q_A {\underset{W_m(A)}\otimes} W_m(A)_{[\pi]_m} \ar[r]
        \ar[d]_-{\delta_1} &
       W_m\Omega^q_{A_i} {\underset{W_m(A_i)}\otimes} W_m(A_i)_{[x_i]_m} \ar[r]
       \ar[d]^-{\delta_2} & W_m\Omega^q_{\wh{A}_i} {\underset{W_m(\wh{A}_i)}\otimes}
          W_m(\wh{A}_i)_{[x_i]_m} 
       \ar[d]^-{\delta_3} \\
       W_m\Omega^q_{K} \ar[r]^-{\gamma_1} & W_m\Omega^q_{L} \ar[r]^-{\gamma_2} &
       W_m\Omega^q_{\wh{K}_i}.}
    \end{equation}
    
    Since $A$ is a regular local $\F_p$-algebra, the composition
    $W_m\Omega^q_A \to W_m\Omega^q_L$ of all left vertical
    arrows with $\gamma_1$ is injective (cf. \cite[Prop.~2.8]{KP-Comp}). It follows
    that $\alpha_1$ is injective. The arrow $\alpha_2$ is injective by
    \cite[Lem. 2.9, 2.11]{Morrow-ENS}. As the vertical arrows on the top floor are the
    localizations, it follows that $\beta_1$ and $\beta_2$ are injective.
    Since each $\delta_i$ is an isomorphism by \cite[Thm.~C]{Hesselholt-Acta},
    it follows that $\gamma_1$ and $\gamma_2$ are injective. This achieves the proof.
\end{proof}

Since $A$ is $F$-finite, one knows that $W_m(A)$ is a Noetherian
local ring whose maximal ideal is $\fM_m = \{\un{a} = (a_{m-1}, \ldots, a_0): a_{m-1} \in \fm\}$.
For any $W_m(A)$-module $M$, we let ${M}^{\wedge}$ denote the $\fM_m$-adic (equivalently,
$W_m(\fm)$-adic) completion of $M$. By \cite[Lem. 2.9, 2.11]{Morrow-ENS},
one knows that $W_m\Omega^q_A$ is a finite type $W_m(A)$-module and
\begin{equation}\label{eqn:Non-complete-0}
  W_m\Omega^q_A \otimes_{W_m(A)} W_m(\wh{A}) \xrightarrow{\cong} W_m\Omega^q_{\wh{A}}
    \xleftarrow{\cong} {(W_m\Omega^q_A)}^{\wedge}.
\end{equation}

\begin{lem}\label{lem:Non-complete-1}
  There are canonical isomorphisms of $W_m(\wh{A})$-modules
  \begin{equation}\label{eqn:Non-complete-1-0}
  W_m\Omega^q_A(\log\pi) \otimes_{W_m(A)} W_m(\wh{A}) \xrightarrow{\cong}
  W_m\Omega^q_{\wh{A}}(\log\pi)
    \xleftarrow{\cong} {W_m\Omega^q_A(\log\pi)}^{\wedge}.
    \end{equation}
\end{lem}
\begin{proof}
  Since $W_m\Omega^q_A(\log\pi)$ is a finite type $W_m(A)$-module, the first and the last terms in ~\eqref{eqn:Non-complete-1-0}
  are canonically isomorphic. By the definition of $W_m\Omega^q_A(\log\pi)$ and ~\eqref{eqn:Non-complete-0}, we
  see that the first arrow in
  ~\eqref{eqn:Non-complete-1-0} is surjective. We now look at the commutative diagram
  \begin{equation}\label{eqn:Non-complete-1-1}
    \xymatrix@C1pc{
      W_m\Omega^q_A(\log\pi) \otimes_{W_m(A)} W_m(\wh{A}) \ar[r] \ar[d] &
      W_m\Omega^q_{A_\pi} \otimes_{W_m(A)} W_m(\wh{A}) \ar[d] \\
      W_m\Omega^q_{\wh{A}}(\log\pi) \ar[r] & W_m\Omega^q_{\wh{K}}.}
    \end{equation}
  The top horizontal arrow is injective because $W_m(\wh{A})$ is a flat $W_m(A)$-module.
  Furthermore,
  \begin{equation}\label{eqn:Non-complete-1-2}
  \begin{array}{lll}
  W_m\Omega^q_{K} \otimes_{W_m(A)} W_m(\wh{A}) & \cong & W_m\Omega^q_A \otimes_{W_m(A)}
  W_m(A_\pi) \otimes_{W_m(A)}  W_m(\wh{A}) \\
  & \cong & W_m\Omega^q_{\wh{A}} \otimes_{W_m(\wh{A})}
  (W_m(\wh{A}) \otimes_{W_m(A)} W_m(K)) \\
  & \cong &  W_m\Omega^q_{\wh{A}} \otimes_{W_m(\wh{A})} W_m(\wh{A})_{[\pi]_m} \\
  & \cong & W_m\Omega^q_{\wh{A}} \otimes_{W_m(\wh{A})} W_m(\wh{K}) \cong
  W_m\Omega^q_{\wh{K}}.
  \end{array}
  \end{equation}
  This shows that the right vertical arrow in ~\eqref{eqn:Non-complete-1-1} is
  bijective (as the composition of the isomorphisms in \eqref{eqn:Non-complete-1-2} is same as the right vertical arrow in \eqref{eqn:Non-complete-1-1}). It follows that the left vertical arrow is injective. This finishes
  the proof.
\end{proof}

We shall need the following well-known commutative algebra fact to prove the
next result. 

\begin{lem}\label{lem:Non-complete-2}
  Let $R$ be a commutative Noetherian local ring with maximal ideal $\fm$ and let
  $M$ be a finitely generated $R$-module. Assume that $N' \subset M^{\wedge}$ is
  an $\wh{R}$-submodule. Assume further that $N \subset M \cap N'$ is an $R$-submodule
  which is dense in $N'$ in the $\fm$-adic topology. Then $N' = N^{\wedge}$ and
  $N = M \cap N'$.
  \end{lem}

 \begin{proof}
     First note that $N' \subset M^\wedge$ is closed in the $\fm$-adic topology (cf. \cite[Thms~8.13, 8.10(i)]{Matsumura} ). We see from \cite[Thm.~8.1(i)]{Matsumura} that $N^\wedge$ is the closure of $N$ in $M^\wedge$. This implies $N^\wedge = N'$. For the second claim, note that $M/N \inj (M/N)^\wedge \cong M^\wedge/N^\wedge$, where the injectivity of the first arrow follows from the separateness of the $\fm$-adic topology (cf. Theorem~8.10(i) of loc. cit.) and the second isomorphism follows from Theorem~8.1(ii) of loc. cit. 
 \end{proof}

\begin{lem}\label{lem:Non-complete-4}
  The maps
  \begin{enumerate}
    \item
      $d \colon \Fil_{\un{n}}W_m\Omega^q_K \to \Fil_{\un{n}}W_m\Omega^{q+1}_K$;
    \item
      $V \colon \Fil_{\un{n}}W_m\Omega^q_K \to \Fil_{\un{n}}W_{m+1}\Omega^{q}_K$;
    \item
      $F \colon \Fil_{\un{n}}W_{m+1}\Omega^q_K \to \Fil_{\un{n}}W_{m}\Omega^{q}_K$;
    \item
      $R \colon \Fil_{\un{n}}W_{m+1}\Omega^q_K \to\Fil_{\un{n}/p} W_{m}\Omega^{q}_K$
\end{enumerate}
are continuous in the $W_m(\fm)$-adic topology.
\end{lem}
\begin{proof}
  The claims (3) and (4) are easy to verify. To prove (1) and (2), we note that
  $ \Fil_{\un{n}}W_m\Omega^q_K$ is a finite type $W_m(A)$-module by
  \lemref{lem:Log-fil-4}. We let
  $J_s \subset W_m(A)$ be the ideal $([x_1]^s_m, \ldots , [x_d]^s_m)$.
  Then the $W_m(\fm)$-adic topology of $\Fil_{\un{n}}W_m\Omega^q_K$
  coincides with the topology given by the filtration $\{J_s\Fil_{\un{n}}W_m\Omega^q_K\}$ by
  \cite[Cor.~2.3]{Geisser-Hesselholt-Top} (see its proof).
  Hence, it suffices to note that
  given $s \ge 1$, one has $d([x_i]^{p^mt}_m\Fil_{\un{n}}W_m\Omega^q_K)\subset
  J_s\Fil_{\un{n}}W_m\Omega^{q+1}_K$ and $V([x_i]^{pt}_m\Fil_{\un{n}}W_m\Omega^q_K) \subset
  J_s\Fil_{\un{n}}W_{m+1}\Omega^q_K$ for all $t \ge s$ and all $i$.
\end{proof}

\begin{lem}\label{lem:Non-complete-7}
  For $m \ge 1$, we have the following.
  \begin{enumerate}
    \item
      The map $F \colon W_{m+1}(A) \to W_m(A)$ is a finite ring homomorphism.
    \item
      The $W_m(\fm)$-topology of $W_m(A)$ coincides with its $W_{m+1}(\fm)$-topology
      via $F$.
    \item
      For any $A$-submodule $M$ of $\Omega^{q-1}_K$ (resp. $\Omega^q_K$), 
  $dV^{m-1}(M)$ (resp. $V^{m-1}(M)$) is a $W_{m+1}(A)$-submodule of
  $W_m\Omega^q_K$, where $W_{m+1}(A)$ acts on the latter via $F$.
  \end{enumerate}
\end{lem}
\begin{proof}
The item (1) is well-known result of Langer-Zink (cf. \cite[Thm.~2.6]{Morrow-ENS}).
For (2), we can replace (for any $m$) the
$W_m(\fm)$-adic topology by the one given by the
filtration $\{J_s\}$ defined in the proof of \lemref{lem:Non-complete-4}.
Our claim now follows because $F([x_i]_{m+1}) = [x_i]^p_m$ for every $i$.
To prove (3), we let $a \in W_{m+1}(A)$ and $x \in M$.
We get 
$$F(a)dV^{m-1}(x) =  F(a) FdV^m(x) = F(adV^m(x)) = FdV^m(F^m(a)x) -F(daV^m(x))
 $$ 
$$=FdV^m(F^m(a)x) - FV^m(F^md(a)x)=FdV^m(F^m(a)x) = dV^{m-1}(F^m(a)x),$$
because $FV^m = 0$. This finishes the proof as $F^m(a)x \in M$. The other case is
easier : $F(a)V^{m-1}(x) = V^{m-1}(F^m(a)x)$.
\end{proof}

\begin{lem}\label{lem:Non-complete-3}
 The canonical map $\Fil_{\un{n}} W_m\Omega^q_K \to \Fil_{\un{n}} W_m\Omega^q_{\wh{K}}$
  induces an isomorphism of $W_m(\wh{A})$-modules
  \[
  \phi \colon (\Fil_{\un{n}} W_m\Omega^q_K)^{\wedge} \xrightarrow{\cong}
  \Fil_{\un{n}} W_m\Omega^q_{\wh{K}}.
  \]
\end{lem}
\begin{proof}
  Since $\Fil_{\un{n}} W_m\Omega^q_K$ is a finite type $W_m(A)$-module, we know that
  the canonical
  map $\Fil_{\un{n}} W_m\Omega^q_K \otimes_{W_m(A)} W_m(\wh{A}) \to
  (\Fil_{\un{n}} W_m\Omega^q_K)^{\wedge}$
  is an isomorphism. On the other hand, we have an injection
  $\Fil_{\un{n}} W_m\Omega^q_K \otimes_{W_m(A)} W_m(\wh{A}) \inj
  W_m\Omega^q_K \otimes_{W_m(A)} W_m(\wh{A}) \cong W_m\Omega^q_{\wh{K}}$,
where the last isomorphism is by ~\eqref{eqn:Non-complete-1-2}.
  It follows that $\phi$ is injective.

 Now, first consider the case $q=0$. If $m=1$ then the isomorphism is clear by \lemref{lem:Log-fil-4}(2). For $m >1$, consider the commutative diagram $W_m(\wh A)$ modules.
 \begin{equation}
     \xymatrix@C2pc{
     0 \ar[r] & \Fil_{\un{n}}K\otimes_{W}\wh W  \ar[r]^-{V^{m-1}} \ar[d] & \Fil_{\un{n}}W_m(K)\otimes_{W}\wh W  \ar[r]^-{R} \ar[d] & \Fil_{\un{n}/p}W_{m-1}(K)\otimes_{W}\wh W  \ar[r] \ar[d] & 0 \\
     0 \ar[r] & \Fil_{\un{n}}\wh{K} \ar[r]^-{V^{m-1}} &  \Fil_{\un{n}}W_m(\wh{K}) \ar[r]^-{R} & \Fil_{\un{n}/p}W_{m-1}(\wh{K}) \ar[r]& 0,
     }
 \end{equation}
 where $W=W_m(A)$ and $\wh W= W_m(\wh A)$. The top row is obtained by tensoring \eqref{eqn:Log-fil-1.1} (taking $U=\Spec K, D=(\prod\limits_{1 \le i \le r} x_i^{n_i})$) with $W_m(\wh A)$ (over $W_m(A)$) and noting that $W_m(\wh A)$ is the $W_m(\fm)$-adic completion of $W_m(A)$ (cf. \eqref{eqn:Non-complete-0} for $q=0$) and hence flat over $W_m(A)$. Here the module structure on $\Fil_{\un{n}}K$ is given via $F^{m-1}$. By \lemref{lem:Non-complete-7}(2), we see that the $W_m(\fm)$-adic topology on $\Fil_{\n}K$ (via $F^{m-1}$) is same as the $\fm$-adic topology. Hence $\Fil_{\un{n}}K \otimes_{W_m(A)}W_m(\wh A)=(\Fil_{\un{n}}K)^\wedge$, the $\fm$-adic completion of $\Fil_{\un{n}}K$ as $A$ module. 
 We also note that the $W_{m-1}(\fm)$-topology of $W_{m-1}(A)$ coincides with
  its $W_m(\fm)$-topology via the surjection $R \colon W_m(A) \surj W_{m-1}(A)$.
  It follows that for any finite type $W_{m-1}(A)$-module $M$, one has
  isomorphisms
  \begin{equation}\label{eqn:Non-complete-6-0}
  M \otimes_{W_m(A)} W_{m}(\wh{A})
  \cong {M}^{\wedge}_m \cong {M}^{\wedge}_{m-1} \cong M \otimes_{W_{m-1}(A)} W_{m-1}(\wh{A}),  
\end{equation}
  where $M^{\wedge}_m$ (resp. ${M}^{\wedge}_{m-1}$) is the $W_{m}(\fm)$
  (resp. $W_{m-1}(\fm)$)-adic completion of $M$. Thus, the left and right vertical arrows are isomorphism by $m=1$ case and induction, respectively. Hence, the middle arrow is also an isomorphism.  In particular, we get that $\Fil_{\un{n}}W_m(K)$ is dense in $\Fil_{\un{n}}W_m(\wh K)$.

For $q >0$ case, recall that
  $\Fil_{\un{n}} W_m\Omega^q_{\wh{K}}$ is the sum of
  $\Fil_{\un{n}}W_m(\wh{K}) W_m\Omega^q_{\wh{A}}(\log\pi)$ and
  $d(\Fil_{\un{n}}W_m(\wh{K}))W_m\Omega^{q-1}_{\wh{A}}(\log\pi)$. Hence by Lemmas~\ref{lem:Non-complete-1} and \ref{lem:Non-complete-4} and the $q=0$ case  we
get that $\Fil_{\un{n}} W_m\Omega^q_{K}$ is dense in $\Fil_{\un{n}} W_m\Omega^q_{\wh{K}}$. On the other hand, the lemma is clear if $\un{n}=p^{m-1}\un{n}''$ for some $\un{n}''$; by \lemref{lem:Log-fil-4} (9) and \lemref{lem:Non-complete-1}. So we can choose some $\un{n}'$ such that $p^{m-1}\un{n}' > \un{n}$ and so $\Fil_{\un{n}} W_m\Omega^q_{R} \subset \Fil_{p^{m-1}\un{n'}} W_m\Omega^q_{R}$ for $R \in \{K,\wh K \}$. Now, we apply \lemref{lem:Non-complete-2} with $M=\Fil_{p^{m-1}\un{n'}} W_m\Omega^q_{K}$ and $N=\Fil_{\un{n}} W_m\Omega^q_{K}$ to conclude the proof.
\end{proof}
  
\begin{cor}\label{cor:Non-complete-3}
    Consider $F(\Fil_{\n}W_{m+1}\Omega^q_{K})$ as a $W_{m+1}(A)$-module (via $F$). Then the canonical map $F(\Fil_{\n}W_{m+1}\Omega^q_{K}) \to F(\Fil_{\n}W_{m+1}\Omega^q_{\wh K})$ induces an isomorphism of $W_{m+1}(\wh A)$-modules (via $F$) :  $\phi: F(\Fil_{\n}W_{m+1}\Omega^q_{K})^\wedge \xrightarrow{\cong} F(\Fil_{\n}W_{m+1}\Omega^q_{\wh K})$.
\end{cor}
\begin{proof}
    Consider $\Fil_{\n}W_m\Omega^q_K$ as a finite $W_{m+1}(A)$-module via $F$ (see \lemref{lem:Non-complete-7}(1)), so that  $F(\Fil_{\n}W_{m+1}\Omega^q_{K}) \subset \Fil_{\n}W_m\Omega^q_K$ is a $W_{m+1}(A)$-submodule. By Lemmas \ref{lem:Non-complete-4}(3) and \ref{lem:Non-complete-3} we know that $F(\Fil_{\n}W_{m+1}\Omega^q_{K})$ is dense in $F(\Fil_{\n}W_{m+1}\Omega^q_{\wh K})$ with respect to $W_{m}(\wh \fm)$-adic topology. Also by Lemmas \ref{lem:Non-complete-7}(2) and \ref{lem:Non-complete-3}, we have $$\Fil_{\n}W_{m}\Omega^q_{\wh K} \cong \Fil_{\n}W_{m}\Omega^q_{K} \otimes_{W_{m+1}(A)}W_{m+1}(\wh A),$$ 
    the $W_{m+1}(\fm)$-adic completion (via $F$) of $\Fil_{\n}W_{m}\Omega^q_{K}$. Now, we can apply \lemref{lem:Non-complete-2} with $M=\Fil_{\n}W_{m}\Omega^q_{K}$ and $N=F(\Fil_{\n}W_{m+1}\Omega^q_{K})$ (considering them as modules over $W_{m+1}(A)$, via $F$) to conclude the proof.
\end{proof}

\begin{defn}
    We let $\psi_i \colon W_m\Omega^q_K \to W_m\Omega^q_{\wh{K}_i}$ be the map induced
by the inclusion $K \inj \wh{K}_i$.
We let $\Fil''_{\un{n}}W_m\Omega^q_K$ denote the kernel of the canonical map
 $$\psi = (\psi_i) \colon W_m\Omega^q_K \to {\underset{1 \le i \le r}\bigoplus}
\frac{W_m\Omega^q_{\wh{K}_i}}{\Fil_{n_i}W_m\Omega^q_{\wh{K}_i}}.$$
It is clear that $\Fil_{\un{n}}W_m\Omega^q_K
 \subset  \Fil''_{\un{n}}W_m\Omega^q_K$. 
\end{defn}

\begin{lem}\label{lem:Complete-3}
 If $A$ is complete, then one has an exact sequence of $W_m(A)$-modules
  \[
 0 \to V^{m-1}(\Fil_{\un{n}} \Omega^q_K) + dV^{m-1}(\Fil_{\un{n}}\Omega^{q-1}_K) \to
  \Fil''_{\un{n}}W_m\Omega^q_K \xrightarrow{R} \Fil''_{{\un{n}}/p}W_{m-1}\Omega^{q}_K.
  \]
\end{lem}
 \begin{proof}
   We let $T = V^{m-1}(\Fil_{\un{n}} \Omega^q_K) + dV^{m-1}(\Fil_{\un{n}}\Omega^{q-1}_K)$.
   Using the exact sequence ~\eqref{eqn:V-R-Fil-0}, we only need to show that
   every element 
   $$\omega \in
   \left( V^{m-1}(\Omega^q_K) + dV^{m-1}(\Omega^{q-1}_K) \right) \bigcap  \Fil''_{\un{n}}W_m\Omega^q_K $$
   lies in $T$.
   We shall show more generally that for every $0 \le i \le r$, we
   can write $\omega = V^{m-1}(x) + dV^{m-1}(y)$
with $x \in \Fil^i_{\un{n}}\Omega^q_K$ and $y \in \Fil^i_{\un{n}}\Omega^{q-1}_K$ (cf. Notation~\ref{not:F-function-2-0}).
This will finish the proof of the lemma.

 Now, we can write $\omega = V^{m-1}(x_0) + dV^{m-1}(y_0)$ for some
   $x_0 \in \Omega^q_K = \Fil^0_{\un{n}}\Omega^q_K$ and
   $y_0 \in \Omega^{q-1}_K = \Fil^0_{\un{n}}\Omega^{q-1}_K$.
   Suppose next that $i \ge 1$ and that we have found
   $x \in \Fil^{i-1}_{\un{n}}\Omega^q_K$ and
   $y \in \Fil^{i-1}_{\un{n}}\Omega^{q-1}_K$ such that
   $\omega = V^{m-1}(x) + dV^{m-1}(y)$.
In the notations of
   \corref{cor:F-function-2}, we can uniquely write
   $x = \sum_{\un{m}} a_{\un{m}}x^{\un{m}}$ and $y = \sum_{\un{m}} b_{\un{m}}x^{\un{m}}$
   such that $a_{\un{m}} = 0 = b_{\un{m}}$ for any $\un{m}$ such that $m_j <- n_j$ for some
   $j < i$. 
   We can uniquely write $x = x_1 + x_2$ and $y = y_1 + y_2$ such that
   \begin{enumerate}
     \item
       $x_1, x_2 \in \Fil^{i-1}_{\un{n}}\Omega^q_K$ and
       $y_1, y_2 \in \Fil^{i-1}_{\un{n}}\Omega^{q-1}_K$. 
     \item
       $x_1 \in  \Fil^{i}_{\un{n}}\Omega^q_K, \ y_1 \in  \Fil^{i}_{\un{n}}\Omega^{q-1}_K$.
       In particular,
       $\psi_i(V^{m-1}(x_1) + dV^{m-1}(y_1)) \in \Fil_{n_i}W_m\Omega^q_{\wh K_i}$.
     \item
       No term of the infinite series $x_2$ (resp. $y_2$) lies in
       $\Fil^{i}_{\un{n}}\Omega^q_K$ (resp. $\Fil^{i}_{\un{n}}\Omega^{q-1}_K$).
\end{enumerate}

   The above
   conditions imply that 
   $$\psi_i(V^{m-1}(x_2) + dV^{m-1}(y_2)) =\psi_i(\omega - (V^{m-1}(x_1) + dV^{m-1}(y_1)))
   \in \Fil_{n_i}W_m\Omega^q_{\wh K_i}.$$
   It follows from \lemref{lem:F-group-1} that
   $$V^{m-1}(\psi_i(x_2)) + dV^{m-1}(\psi_i(y_2)) \in
   \left[{\underset{j < -n_i}\prod}F^{m,q}_{j}(\kappa(\wh A_i)) \right] \bigcap
   \Fil_{n_i}W_m\Omega^q_{\wh{K}_i},$$
   where $\kappa(\wh A_i)$ is the residue field of $\wh A_i$.

But the above intersection is trivial by \corref{cor:GH-3} and
   \propref{prop:Fil-decom}. It follows that
   $\psi_i(V^{m-1}(x_2) + dV^{m-1}(y_2)) = V^{m-1}(\psi_i(x_2)) + dV^{m-1}(\psi_i(y_2)) =0$.
   In particular, $V^{m-1}(x_2) + dV^{m-1}(y_2) = 0$ by \lemref{lem:Complete-1}.
   We conclude that $\omega = V^{m-1}(x_1) + dV^{m-1}(y_1)$. We are now done by the
   condition (2) above and proceeding inductively.   
\end{proof}

\begin{prop}\label{prop:Complete-0}
There is a short exact sequence
of $W_m(A)$-modules
  \[
  0 \to V^{m-1}(\Fil_{\un{n}} \Omega^q_K) + dV^{m-1}(\Fil_{\un{n}}\Omega^{q-1}_K) \to
  \Fil_{\un{n}}W_m\Omega^q_K \xrightarrow{R} \Fil_{{\un{n}}/p}W_{m-1}\Omega^{q}_K \to 0.
  \]
\end{prop}
\begin{proof}
  Surjectivity of $R$ follows from \lemref{lem:Log-fil-4}(7).  It remains therefore 
to show that $\Ker(R) \subset V^{m-1}(\Fil_{\un{n}} \Omega^q_K) +
  dV^{m-1}(\Fil_{\un{n}}\Omega^{q-1}_K)$.
  First we assume that $A$ is complete. 
  In this case, it follows from \lemref{lem:Complete-3}
  because $\Fil_{\un{n}}W_m\Omega^q_K \subset  \Fil''_{\un{n}}W_m\Omega^q_K$.

 Suppose now that $A$ is not complete.

We now let $E = V^{m-1}(\Fil_{\un{n}} \Omega^q_K) + dV^{m-1}(\Fil_{\un{n}}\Omega^{q-1}_K)$ and
  $E' = \Ker(R)$. One easily verifies that $E$ is a $W_m(A)$-submodule of
  $\Fil_{\un{n}}W_m\Omega^q_K$. Indeed, \lemref{lem:Log-fil-4} implies that
  $E \subset \Fil_{\un{n}}W_m\Omega^q_K$. Moreover, one has
  $a(V^{m-1}(b) + dV^{m-1}(c)) = V^{m-1}(F(a)b-Fd(a)c) + dV^{m-1}(F(a)c)$
  for any $a \in W_m(A), \ b \in \Fil_{\un{n}} \Omega^q_K$ and
  $c \in \Fil_{\un{n}}\Omega^{q-1}_K$.
  \lemref{lem:Log-fil-4} implies once again that $a(V^{m-1}(b) + dV^{m-1}(c)) \in E$.
Using the complete case, it remains therefore to show that 
$E = \left[V^{m-1}(\Fil_{\un{n}} \Omega^q_{\wh{K}}) +
  dV^{m-1}(\Fil_{\un{n}}\Omega^{q-1}_{\wh{K}})\right] \cap \Fil_{\un{n}}W_m\Omega^q_K$.
But this follows from \lemref{lem:Non-complete-2}
because $E$ is dense in
$V^{m-1}(\Fil_{\un{n}}\Omega^q_{\wh{K}}) + dV^{m-1}(\Fil_{\un{n}}\Omega^{q-1}_{\wh{K}})$ by
Lemmas~\ref{lem:Non-complete-4} and ~\ref{lem:Non-complete-3}.
\end{proof}

\begin{cor}\label{cor:Complete-2}
 The inclusions $K \inj \wh{K}_i$
  induce an exact sequence
  \[
  0 \to  \Fil_{\un{n}}W_m\Omega^q_K \to W_m\Omega^q_K \xrightarrow{\psi}
    {\underset{1 \le i \le r}\bigoplus}
  \frac{W_m\Omega^q_{\wh{K}_i}}{\Fil_{n_i}W_m\Omega^q_{\wh{K}_i}}.
  \]
\end{cor}
\begin{proof}
 The corollary is equivalent to the claim that the inclusion
 $\Fil_{\un{n}}W_m\Omega^q_K  \subset \Fil''_{\un{n}}W_m\Omega^q_K$ is an equality.
 To prove this claim, we first assume that $A$ is complete and prove 
by induction on $m$.
  The case $m =1$ follows directly from \corref{cor:F-function-2}.
  For $m \ge 2$, we use the commutative diagram
  \begin{equation}\label{eqn:Complete-2-0}
    \xymatrix@C.8pc{
      0 \ar[r] & V^{m-1}(\Fil_{\un{n}}\Omega^q_K) + dV^{m-1}(\Fil_{\un{n}}\Omega^{q-1}_K)
      \ar[r] \ar[d] & \Fil_{\un{n}}W_m\Omega^q_K \ar[r]^-{R} \ar[d] &
      \Fil_{\un{n}}W_{m-1}\Omega^q_K \ar[r] \ar[d] &
      0 \\
    0 \ar[r] & V^{m-1}(\Fil''_{\un{n}}\Omega^q_K) + dV^{m-1}(\Fil''_{\un{n}}\Omega^{q-1}_K)
    \ar[r] & \Fil''_{\un{n}}W_m\Omega^q_K \ar[r]^-{R}  &
    \Fil''_{\un{n}}W_{m-1}\Omega^q_K. & }
    \end{equation}

  \propref{prop:Complete-0} says that the top row  in
  ~\eqref{eqn:Complete-2-0} is exact and \lemref{lem:Complete-3}
  implies that the bottom row is exact at $\Fil''_{\un{n}}W_m\Omega^q_K$.
  The left and the right vertical arrows are bijective by induction.
  A diagram chase shows that the middle vertical arrow is also bijective.

  For $A$ not necessarily complete, it suffices to show that
  $\Fil_{\un{n}}W_m\Omega^q_K = \Fil_{\un{n}}W_m\Omega^q_{\wh{K}} \cap W_m\Omega^q_K$.
  But this follows because $W_m(\wh{A})$ is a faithfully flat $W_m(A)$-algebra
  and there are isomorphisms
  \begin{equation}\label{eqn:Complete-2-1}
    \frac{W_m\Omega^q_K}{\Fil_{\un{n}}W_m\Omega^q_K} \otimes_{W_m(A)} W_m(\wh{A})
  \cong \frac{W_m\Omega^q_K \otimes_{W_m(A)} W_m(\wh{A})}
        {\Fil_{\un{n}}W_m\Omega^q_K \otimes_{W_m(A)} W_m(\wh{A})}
        \cong \frac{W_m\Omega^q_{\wh{K}}}{\Fil_{\un{n}}W_m\Omega^q_{\wh{K}}},
  \end{equation}
  where the second isomorphism follows from ~\eqref{eqn:Non-complete-1-2} and
  \lemref{lem:Non-complete-3}.
 \end{proof}

\begin{lem}\label{lem:Complete-5}
  The canonical maps
  \[
  \frac{\Fil_{\un{n}}W_m\Omega^q_K}{dV^{m-1}(\Fil_{\un{n}}\Omega^{q-1}_K)}
  \to \frac{W_m\Omega^q_K}{dV^{m-1}(\Omega^{q-1}_K)};
  \]
  \[
   \frac{\Fil_{\un{n}}W_{m}\Omega^q_K}{V^{m-1}(\Fil_{\un{n}}\Omega^{q}_K)}
  \to \frac{W_{m}\Omega^q_K}{V^{m-1}(\Omega^{q}_K)}
  \]
are injective.
\end{lem}
\begin{proof}
  We shall prove the injectivity of the first map as the other one is completely
  analogous.
We only need to show that every element
  $\omega \in \Fil_{\un{n}}W_m\Omega^q_K \bigcap dV^{m-1}(\Omega^{q-1}_K)$ lies in
$dV^{m-1}(\Fil_{\un{n}}\Omega^{q-1}_K)$. If $A$ is complete, this is proven
exactly as we proved \lemref{lem:Complete-3} (with $x = 0$). If $A$ is not 
complete, it suffices to show that $$\Fil_{\un{n}}W_m\Omega^q_K \bigcap
dV^{m-1}(\Fil_{\un{n}}\Omega^{q-1}_{\wh{K}}) \subseteq dV^{m-1}(\Fil_{\un{n}}\Omega^{q-1}_K).$$

To that end, we note that $\Fil_{\un{n}}W_m\Omega^q_{\wh{K}} =
(\Fil_{\un{n}}W_m\Omega^q_K)^\wedge$ by \lemref{lem:Non-complete-3}. In particular,
$\Fil_{\un{n}}\Omega^{q-1}_K$ is dense in $\Fil_{\un{n}}\Omega^{q-1}_{\wh{K}}$ with respect to
the $\wh{\fm}$-adic topology.  It follows from \lemref{lem:Non-complete-4} that
$dV^{m-1}(\Fil_{\un{n}}\Omega^{q-1}_K)$ is a dense subgroup of 
$dV^{m-1}(\Fil_{\un{n}}\Omega^{q-1}_{\wh{K}})$ with respect to the $W_m(\fm)$-adic
topology.
We conclude from \lemref{lem:Non-complete-7} that $\Fil_{\un{n}}W_m\Omega^q_{\wh{K}}$ is
the $W_{m+1}(\fm)$-adic completion of the $W_{m+1}(A)$-module $\Fil_{\un{n}}W_m\Omega^q_K$
via $F \colon W_{m+1}(A) \to W_m(A)$ and
$dV^{m-1}(\Fil_{\un{n}}\Omega^{q-1}_K)$ is a dense subgroup of 
$dV^{m-1}(\Fil_{\un{n}}\Omega^{q-1}_{\wh{K}})$ with respect to the $W_{m+1}(\fm)$-adic
topology. By \lemref{lem:Non-complete-2}, it remains
to show that $dV^{m-1}(\Fil_{\un{n}}\Omega^{q-1}_K)$ is a $W_{m+1}(A)$-submodule of
$\Fil_{\un{n}}W_m\Omega^q_K$. But this follows again from \lemref{lem:Non-complete-7}.
\end{proof}

\section{Cartier homomorphism II}

\begin{defn}\label{defn:Cartier}
    We define $Z_1\Fil_{\un{n}}W_m\Omega^q_K:=Z_1W_m\Omega^q_K \bigcap \Fil_{\un{n}}W_m\Omega^q_K$, for all $\n \in \Z^r$ and $m \ge 1$. Equivalently, we have $Z_1\Fil_{\un{n}}W_m\Omega^q_K=\Ker (F^{m-1}d: \Fil_{\un{n}}W_m\Omega^q_K \to \Fil_{\un{n}}\Omega^{q+1}_K)$.
\end{defn}

\begin{lem}\label{lem:Complete-6}
  The Cartier map $C \colon Z_1W_m\Omega^q_K \to W_m\Omega^q_K$ restricts to a
  $W_{m+1}(A)$-linear map
  \[
  C \colon Z_1\Fil_{\un{n}}W_m\Omega^q_K \to \Fil_{{\un{n}}/p}W_m\Omega^q_K
  \]
  such that $C(dV^{m-1}(\Fil_{\un{n}}\Omega^{q-1}_K)) = 0$.
\end{lem}
\begin{proof}
  By \corref{cor:Complete-2}, we only need to show that
  $C \circ \psi_i(\omega) = \psi_i \circ C(\omega)$ lies in
  $\Fil_{\lfloor{{n_i}/p}\rfloor}W_m\Omega^q_{\wh{K}_i}$ for every $\omega \in 
  Z_1W_m\Omega^q_K \bigcap \Fil_{\un{n}}W_m\Omega^q_K$ and $1 \le i \le r$. 
But this is shown in the proof of \lemref{lem:Cartier-fil-0}.
  \end{proof}

\begin{lem}\label{lem:Complete-4}
We have the following.
  \begin{enumerate}
  \item
    $Z_1\Fil_{\un{n}}W_m\Omega^q_K = F(\Fil_{\un{n}}W_{m+1}\Omega^q_K)$.
  \item
    $F \colon \Fil_{\un{n}}W_{m+1}\Omega^q_K \to  \Fil_{\un{n}}W_{m}\Omega^q_K$
    induces isomorphisms
    \[
    \ov{F} \colon \Fil_{{\un{n}}/p}W_{m}\Omega^q_K \xrightarrow{\cong}
    \frac{Z_1\Fil_{\un{n}}W_m\Omega^q_K}{dV^{m-1}(\Fil_{\un{n}}\Omega^{q-1}_K)}; \ \
    \ov{C} \colon  \frac{Z_1\Fil_{\un{n}}W_m\Omega^q_K}{dV^{m-1}(\Fil_{\un{n}}\Omega^{q-1}_K)}
     \xrightarrow{\cong} \Fil_{{\un{n}}/p}W_{m}\Omega^q_K.
     \]
\end{enumerate}
\end{lem}
\begin{proof}
 We first note that the lemma is classical for $W_\textbf{.}\Omega^\bullet_K$
  (cf. \cite{Illusie} and \cite{Kato-Duality}). 
To prove (1) and (2), we use \propref{prop:Complete-0} and the identity $FV^m = 0$
to get a commutative diagram
  \begin{equation}\label{eqn:Complete-4-0}
    \xymatrix@C1pc{
     \Fil_{\un{n}}W_{m+1}\Omega^q_K \ar@{->>}[d]_-{R} \ar[r]^-F &  
     \Fil_{\un{n}}W_{m}\Omega^q_K \ar@{->>}[d] \\
     \Fil_{{\un{n}}/p}W_{m}\Omega^q_K \ar[r]^-{\ov{F}} &
     \frac{\Fil_{\un{n}}W_{m}\Omega^q_K}{dV^{m-1}(\Fil_{\un{n}}\Omega^{q-1}_K)}.}
    \end{equation}
  Since $dV^{m-1}(\Fil_{\un{n}}\Omega^{q-1}_K) = FdV^m(\Fil_{\un{n}}\Omega^{q-1}_K)
  \subset F(\Fil_{\un{n}}W_{m+1}\Omega^q_K)$, we thus get maps
  \begin{equation}\label{eqn:Complete-4-1}
\Fil_{{\un{n}}/p}W_{m}\Omega^q_K \stackrel{\ov{F}}{\surj} 
\frac{F(\Fil_{\un{n}}W_{m+1}\Omega^q_K)}{dV^{m-1}(\Fil_{\un{n}}\Omega^{q-1}_K)}
\inj \frac{Z_1\Fil_{\un{n}}W_m\Omega^q_K}{dV^{m-1}(\Fil_{\un{n}}\Omega^{q-1}_K)}
\xrightarrow{\ov{C}} \Fil_{{\un{n}}/p}W_{m}\Omega^q_K.
\end{equation}

We now look at the commutative diagram
\begin{equation}\label{eqn:Complete-4-2}
    \xymatrix@C1pc{
  \Fil_{{\un{n}}/p}W_{m}\Omega^q_K  \ar[r]^-{\ov{F}} \ar[d] &
      \frac{Z_1\Fil_{\un{n}}W_m\Omega^q_K}{dV^{m-1}(\Fil_{\un{n}}\Omega^{q-1}_K)}
      \ar[d] \ar[r]^-{\ov{C}}  & \Fil_{{\un{n}}/p}W_{m}\Omega^q_K \ar[d] \ar[r]^-{\ov{F}} &
      \frac{Z_1\Fil_{\un{n}}W_m\Omega^q_K}{dV^{m-1}(\Fil_{\un{n}}\Omega^{q-1}_K)} \ar[d] \\
W_{m}\Omega^q_K \ar[r]^-{\ov{F}} &
      \frac{Z_1W_m\Omega^q_K}{dV^{m-1}(\Omega^{q-1}_K)}
      \ar[r]^-{\ov{C}} & W_{m}\Omega^q_K \ar[r]^-{\ov{F}} &
      \frac{Z_1\Fil_{\un{n}}W_m\Omega^q_K}{dV^{m-1}(\Fil_{\un{n}}\Omega^{q-1}_K)}.}
\end{equation}
One knows classically that the composition of any two adjacent maps in the bottom row
is identity. Since the vertical arrows are injective by \lemref{lem:Complete-5},
we conclude that the same holds in the top row too. This proves the lemma.
\end{proof}

We let $B_0\Fil_{\un{n}}\Omega^q_K =0$ and $B_1\Fil_{\un{n}}\Omega^q_K =
d\Fil_{\un{n}}\Omega^{q-1}_K$. We let $Z_0\Fil_{\un{n}}\Omega^q_K = \Fil_{\un{n}}\Omega^q_K$.
For $i \ge 2$, we let $Z_i\Fil_{\un{n}}\Omega^q_K$ (resp. $B_i\Fil_{\un{n}}\Omega^q_K$)
be the inverse image of $Z_{i-1}\Fil_{\un{n}}\Omega^q_K$
(resp. $B_{i-1}\Fil_{\un{n}}\Omega^q_K$) under the composite map
$Z_{i-1}\Fil_{\un{n}}\Omega^q_K \surj
\frac{Z_{i-1}\Fil_{\un{n}}\Omega^q_K}{B_1\Fil_{\un{n}}\Omega^q_K}
\xrightarrow{\ov{C}} \Fil_{{\un{n}}/p}\Omega^q_K$. 
The following is easily deduced from ~\eqref{eqn:Complete-4-0} and
~\eqref{eqn:Complete-4-2}.
\begin{lem}\label{lem:Complete-9}
\begin{enumerate}
  \item
$Z_i\Fil_{\un{n}}\Omega^q_K = F^i(\Fil_{\un{n}}W_{i+1}\Omega^q_K), \ \
  B_i\Fil_{\un{n}}\Omega^q_K = F^{i-1}d(\Fil_{\un{n}}W_{i}\Omega^{q-1}_K)$.
\item
  $Z_i\Fil_{\un{n}}\Omega^q_K \supset Z_{i+1}\Fil_{\un{n}}\Omega^q_K \supset
  B_j\Fil_{\un{n}}\Omega^q_K \supset B_{j-1}\Fil_{\un{n}}\Omega^q_K \ \
  \forall \ \ i \ge 0, \ \ j \ge 1$.
\item
  The maps
$\ov{C} \colon \frac{Z_i\Fil_{\un{n}}\Omega^q_K}{B_i\Fil_{\un{n}}\Omega^q_K}
  \xrightarrow{\cong}
  \frac{Z_{i-1}\Fil_{{\un{n}}/p}\Omega^q_K}{B_{i-1}\Fil_{{\un{n}}/p}\Omega^q_K}$
  and $\ov{F} \colon
  \frac{Z_{i-1}\Fil_{{\un{n}}/p}\Omega^q_K}{B_{i-1}\Fil_{{\un{n}}/p}\Omega^q_K}
  \xrightarrow{\cong} \frac{Z_i\Fil_{\un{n}}\Omega^q_K}{B_i\Fil_{\un{n}}\Omega^q_K}$
are inverses of each other.
\end{enumerate}
\end{lem}

\begin{lem}\label{lem:Complete-8}
  We have $Z_i\Fil_{\un{n}}\Omega^q_K = Z_i\Omega^q_K \bigcap \Fil_{\un{n}}\Omega^q_K$
  and $B_i\Fil_{\un{n}}\Omega^q_K = B_i\Omega^q_K \bigcap \Fil_{\un{n}}\Omega^q_K$.
\end{lem}
\begin{proof}
  If $i = 1$, the first claim follows from \defref{defn:Cartier} while the second
  claim follows from \lemref{lem:Complete-5}.
  To prove the first claim for $i \ge 2$, we look at the commutative
  diagram of exact sequences
  \[
  \xymatrix@C.8pc{
  0 \ar[r] &  B_1\Fil_{\un{n}}\Omega^q_K  \ar[r] \ar[d] & Z_{i}\Fil_{\un{n}}\Omega^q_K 
  \ar[r]^-{C} \ar[d] &  Z_{i-1}\Fil_{{\un{n}}/p}\Omega^q_K  \ar[r] \ar[d] & 0 \\
  0 \ar[r] &  B_1\Omega^q_K \bigcap \Fil_{\un{n}}\Omega^q_K  \ar[r] &
  Z_{i}\Omega^q_K  \bigcap \Fil_{\un{n}}\Omega^q_K 
  \ar[r]^-{C} &  Z_{i-1}\Omega^q_K \bigcap \Fil_{{\un{n}}/p}\Omega^q_K &}
  \]
  The right vertical arrows is bijective by induction and we showed above that the
  left vertical arrow is bijective. This implies the same for the middle vertical
  arrow. The second claim for $i \ge 2$ is proven similarly by replacing
  the $Z_i$-groups (resp. $Z_{i-1}$-groups) by the $B_i$-groups (resp. $B_{i-1}$-groups)
  in the above diagram.
\end{proof}

\begin{lem}\label{lem:VR-0}
  There exists a short exact sequence
  \begin{equation}\label{eqn:VR-0-0}
  0 \to V^{m-1}(Z_1\Fil_{\un{n}}\Omega^q_K) + dV^{m-1}(\Fil_{\un{n}}\Omega^{q-1}_K) \to
  Z_1\Fil_{\un{n}}W_m\Omega^q_K \xrightarrow{R} Z_1\Fil_{{\un{n}}/p}W_{m-1}\Omega^q_K \to 0.
  \end{equation}
\end{lem}
\begin{proof}
 By comparing with the classical case and using \lemref{lem:Log-fil-4},
  one checks that ~\eqref{eqn:VR-0-0} is a well defined complex.
  Moreover, Proposition~\ref{prop:Complete-0} and \lemref{lem:Complete-4}(1)
  together imply that the map $R$ is surjective. We now let $\omega \in \Ker(R)$.
  We then know from Proposition~\ref{prop:Complete-0} that $\omega =
  V^{m-1}(a) + dV^{m-1}(b)$ for some $a \in \Fil_{\un{n}}\Omega^q_K$ and
  $b \in  \Fil_{\un{n}}\Omega^{q-1}_K$. By \defref{defn:Cartier},
  we then get $da = F^{m-1}d(V^{m-1}(a) + dV^{m-1}(b)) = 0$. In other words,
  $\omega \in V^{m-1}(Z_1\Fil_{\un{n}}\Omega^q_K) + dV^{m-1}(\Fil_{\un{n}}\Omega^{q-1}_K)$.
\end{proof}

\begin{defn}
    For $m \ge 1, n \in \Z^r$, we let $$ Z_1R^q_{m,\un{n}}(K) =
\Ker(Z_1\Fil_{\un{n}}\Omega^q_K \oplus \Fil_{\un{n}}\Omega^{q-1}_K
\xrightarrow{V^{m},\ dV^{m}} Z_1\Fil_{\un{n}}W_{m+1}\Omega^q_K) \text{ and}$$
$$R^q_{m,\un{n}}(K) = \Ker(\Fil_{\un{n}}\Omega^q_K \oplus \Fil_{\un{n}}\Omega^{q-1}_K
\xrightarrow{V^{m}, \ dV^{m}} \Fil_{\un{n}}W_{m+1}\Omega^q_K).$$
\end{defn}
We observe that $Z_1R^q_{m,\un{n}}(K) = R^q_{m,\un{n}}(K)$. Indeed, any element
  $(a,b) \in  R^q_{m,\un{n}}(K)$ has the property that $a \in Z_1\Omega^q_K
  \cap \Fil_{\un{n}}\Omega^q_K = Z_1\Fil_{\un{n}}\Omega^q_K$ by
  \corref{cor:R-m-q}. In particular, $(a,b) \in  Z_1R^q_{m-1,\un{n}}(K)$.
\begin{defn}\label{def:two-term}
   For $\un{n} \ge (-1, \ldots , -1)$, we
let $C^{m,q}_{\un{n}, \bullet}(A)$ denote the 2-term complex
\begin{equation}\label{eqn:two-term}
    \left(Z_1\Fil_{\un{n}}W_m\Omega^q_K \xrightarrow{1-C} \Fil_{\un{n}}W_m\Omega^q_K\right).
\end{equation}
\end{defn}

We look at the diagram (for $\un{n} \ge (-1, \ldots , -1)$)

  \begin{equation}\label{eqn:VR-1}
    \xymatrix@C2.5pc{
      0 \ar[r] & Z_1R^q_{m-1,\un{n}}(K) \ar[d]_-{\phi} \ar[r] &  
  Z_1\Fil_{\un{n}}\Omega^q_K \oplus \Fil_{\un{n}}\Omega^{q-1}_K \ar[rr]^-{V^{m-1},\ dV^{m-1}} 
  \ar[d]^-{\phi}& & Z_1\Fil_{\un{n}}W_m\Omega^q_K \ar[d]^-{(1-C)} \\
   0 \ar[r] & R^q_{m-1,\un{n}}(K) \ar[r] &  
 \Fil_{\un{n}}\Omega^q_K \oplus \Fil_{\un{n}}\Omega^{q-1}_K \ar[rr]^-{V^{m-1},\ dV^{m-1}} 
 && \Fil_{\un{n}}W_m\Omega^q_K.}
    \end{equation}

  In this diagram, we let $\phi = 1-\wt{C}$, where $\wt{C}(a, b) = C(a)$. Since
  $CdV^{m-1} = 0$ by \lemref{lem:Complete-6}, this diagram is commutative.
  We have observed that $Z_1R^q_{m-1,\un{n}}(K) = R^q_{m-1,\un{n}}(K)$. Finally, we observe that all vertical arrows are defined because
  $\Fil_{{\un{n}}/p}W_m\Omega^q_K \subset \Fil_{\un{n}}W_m\Omega^q_K$ under our
  assumption on $\un{n}$.
  We now claim that $\phi \colon R^q_{m-1,\un{n}}(K) \to R^q_{m-1,\un{n}}(K)$ is bijective.
  To show this, it is enough to check that $\wt{C} = (C,0)$ is nilpotent on
  $R^q_{m-1,\un{n}}(K)$.
  But this follows from the fact that $R^q_{m-1,\un{n}}(K) \subset B_m\Omega^q_K \oplus
  \Omega^{q-1}_K$ (cf. \corref{cor:R-m-q}) and $C^{m+1} = 0$ on $B_m\Omega^q_K$.
  Combining these claims with \lemref{lem:VR-0}, we deduce the following.

  \begin{cor}\label{cor:VR-2}
    In the derived category of abelian groups, there exists a distinguished triangle (cf. \defref{def:two-term})
    \[
     C^{1,q}_{\un{n}, \bullet}(A) \xrightarrow{V^{m-1}} C^{m,q}_{\un{n}, \bullet}(A)
    \xrightarrow{R}  C^{m-1,q}_{{\un{n}}/p, \bullet}(A) \xrightarrow{+} C^{1,q}_{\un{n}, \bullet}(A) [1] .
    \]
  \end{cor}
  \begin{proof}
    We let $F^{1,q}_{\un{n},_\bullet}(A)$ (resp. $\wt{F}^{1,q}_{\un{n},_\bullet}(A)$)
    denote the 2-term complex consisting of the middle (resp. left)
    column of ~\eqref{eqn:VR-1}. Using \lemref{lem:VR-0},
   we only need to show that the composite map
    \[
    C^{1,q}_{\un{n}, \bullet}(A) \inj F^{1,q}_{\un{n},_\bullet}(A) \surj
    \frac{F^{1,q}_{\un{n},_\bullet}(A)}{\wt{F}^{1,q}_{\un{n},_\bullet}(A)}
    \]
    is a quasi-isomorphism.
    But this follows from ~\eqref{eqn:VR-1} and the claims following it.
    \end{proof}

\begin{lem}\label{lem:Complete-10}
  We have the following.
  \begin{enumerate}
\item
    $\Ker(F^{m-1}d \colon \Fil_{\un{n}}W_m\Omega^q_K \to \Fil_{\un{n}}\Omega^{q+1}_K)
    = F(\Fil_{\un{n}}W_{m+1}\Omega^q_K)$.
  \item
    $\Ker(F^{m-1}\colon \Fil_{\un{n}}W_m\Omega^q_K \to \Fil_{\un{n}}\Omega^q_K) =
    V(\Fil_{\un{n}}W_{m-1}\Omega^q_K)$.
  \item
    $\Ker(dV^{m-1}\colon \Fil_{\un{n}}\Omega^q_K \to \Fil_{\un{n}}W_m\Omega^{q+1}_K)
    = F^m(\Fil_{\un{n}}W_{m+1}\Omega^q_K)$.
  \item
    $\Ker(V \colon \Fil_{\un{n}}W_{m}\Omega^q_K \to \Fil_{\un{n}}W_{m+1}\Omega^{q}_K)
    = dV^{m-1}(\Fil_{\un{n}}\Omega^{q-1}_K)$.
  \item
    $\Ker(V^{m-1} \colon \Fil_{\un{n}}\Omega^q_K \to \Fil_{\un{n}}W_{m}\Omega^{q}_K)
    = F^{m-1}dV(\Fil_{\un{n}}W_{m-1}\Omega^{q-1}_K)$.
  \item 
    $\Ker(dV^{m-1}\colon \Fil_{\un{n}}\Omega^q_K \to \Fil_{\un{n}}W_m\Omega^{q+1}_K/V^{m-1}\Fil_\n \Omega^{q+1}_K)
    = F^{m-1}(\Fil_{\un{n}}W_{m}\Omega^q_K)$.
   \item
    $\Ker(V^{m-1} \colon \Fil_{\un{n}}\Omega^q_K \to \Fil_{\un{n}}W_{m}\Omega^{q}_K/dV^{m-1}\Fil_{\un{n}}\Omega^{q-1}_K )
    = F^{m}dV(\Fil_{\un{n}}W_{m}\Omega^{q-1}_K)$. 
    
  \end{enumerate}
\end{lem}
\begin{proof}
We first note that all statements are classical for $W_\textbf{.}\Omega^\bullet_K$.
Now, the item (1) follows directly from \lemref{lem:Complete-4}(1).

The item (2) is a special case of the more general claim that
\begin{equation}\label{eqn:ker-F-l}
    \Ker(F^l \colon \Fil_{\un{n}}W_{m+l}\Omega^q_K \to  \Fil_{\un{n}}W_{m}\Omega^q_K)
= V^m(\Fil_{\un{n}}W_{l}\Omega^q_K)
\end{equation}
for all $m \ge 0$ and $l \ge 1$.
Since this general claim is known for $W_\textbf{.}\Omega^\bullet_K$ by \propref{prop:basics}, one gets that
$V^m(\Fil_{\un{n}}W_{l}\Omega^q_K) \subseteq
\Ker(F^l \colon \Fil_{\un{n}}W_{m+l}\Omega^q_K \to  \Fil_{\un{n}}W_{m}\Omega^q_K)$.
We shall prove the reverse inclusion by induction on $l$. We first consider the
case $l = 1$. It suffices to show
that $V^m(\Omega^q_{K})
\bigcap \Fil_{\un{n}}W_{m+1}\Omega^q_K = V^m(\Fil_{\un{n}}\Omega^q_K)$. But this follows
from \lemref{lem:Complete-5}.

We now assume $l \ge 2$ and $m \ge 1$ (the $m = 0$ case is trivial). 
We let $\omega \in \Fil_{\un{n}}W_{m+l}\Omega^q_K$ be such that
$F^l(\omega) = 0$. We let $\omega' = F(\omega)$ so that $\omega' \in
\Fil_{\un{n}}W_{m+l-1}\Omega^q_K$ and $F^{l-1}(\omega') = 0$.
The induction hypothesis
implies that $F(\omega) = \omega' = V^{m}(x)$ for some $x \in
\Fil_{\un{n}}W_{l-1}\Omega^q_K$. This implies that
$F^{l-2}d(x) = F^{m+l-2}dV^m(x) = F^{m+l-2}dF(\omega) = 0$.
We conclude from item (1) that $x = F(y)$ from some $y \in \Fil_{\un{n}}W_{l}\Omega^q_K$.
This yields
\[
F(\omega - V^m(y)) = F(\omega) - FV^m(y) = F(\omega) - V^mF(y) = F(\omega) - V^m(x) =0.
\]
Using the $l =1$ case, we get
$\omega - V^m(y) \in V^m(\Fil_{\un{n}}W_{l}\Omega^q_K)$.
In particular, $\omega \in V^m(\Fil_{\un{n}}W_{l}\Omega^q_K)$.

To prove (3), it suffices to show that $$F^m(\Fil_{\un{n}}W_{m+1}\Omega^q_K) =
F^m(W_{m+1}\Omega^q_K) \bigcap \Fil_{\un{n}}W_{m+1}\Omega^q_K.$$ But this follows from
Lemmas~\ref{lem:Complete-9} and ~\ref{lem:Complete-8} because of the
classical identity $Z_m\Omega^q_K = F^m(W_{m+1}\Omega^q_K)$ (cf. \propref{prop:basics}).  

The proof of
(4) is identical to that of the $l =1$ case of \eqref{eqn:ker-F-l} using \lemref{lem:Complete-5}
and \propref{prop:basics}.

To prove (5), we can assume $m \ge 2$. In the latter case,
we only need to show that $\Ker(V^{m-1}) \subset
F^{m-2}d(\Fil_{\un{n}}W_{m-1}\Omega^{q-1}_K)$. By \lemref{lem:Complete-9}, we can replace
$F^{m-2}d(\Fil_{\un{n}}W_{m-1}\Omega^{q-1}_K)$ by $B_{m-1}\Fil_{\un{n}}\Omega^q_K$.
By \propref{prop:basics}, we are thus reduced to showing that
$B_{m-1}\Omega^q_K \bigcap \Fil_{\un{n}}\Omega^q_K \\ = B_{m-1}\Fil_{\un{n}}\Omega^q_K$.
But this follows from \lemref{lem:Complete-8}.

The proof of (6) and (7) are analogous to the proof of (3) and (5), respectively, using \lemref{lem:Complete-5}.
\end{proof}

\begin{lem}\label{lem:VR-3}
  Assuming $\un{n} \ge (-1, \ldots , -1)$, we have
  \[
  d(\Fil_{\un{n}}W_m\Omega^{q}_K) \subset {\rm Image}(Z_1\Fil_{\un{n}}W_m\Omega^{q+1}_K
  \xrightarrow{1-C} \Fil_{\un{n}}W_m\Omega^{q+1}_K).
  \]
\end{lem}
\begin{proof}
  For $\omega \in \Fil_{\un{n}}W_m\Omega^{q}_K$, we have
  $d(\omega) = FdV(\omega) - RdV(\omega) + FdV^2R(\omega) - RdV^2R(\omega) +
  FdV^3R^2(\omega) - RdV^3R^2(\omega) + \cdots + FdV^mR^{m-1}(\omega) -
  RdV^{m}R^{m-1}(\omega)$, as $RdV^mR^{m-1}(\omega) = 0$. Hence, we have
  $$d(\omega) =
  (F-R)((dV + dV^2R + \cdots + dV^mR^{m-1})(\omega)) = (1-C)Fd(x),$$ where
  we let $x = (V + V^2R + \cdots + V^mR^{m-1})(\omega)$. It suffices therefore to show
  that $Fd(x) \in Z_1\Fil_{\un{n}}W_m\Omega^{q+1}_K$.
  To this end, we first note that $F^{m-1}d(Fd(x)) = pF^{m}d^2(x) = 0$,
  and this implies that $Fd(x) \in Z_1W_m\Omega^{q+1}_K$. Secondly,
  \lemref{lem:Log-fil-4} implies that $Fd(x) \in \Fil_{\un{n}}W_m\Omega^{q+1}_K$
  because $\omega \in \Fil_{\un{n}}W_m\Omega^{q}_K$. This achieves the proof.
\end{proof}

\begin{rem}\label{rem:VR-3}
    The same proof also shows that $$d(\Fil_{-\un{1}}W_m\Omega^{q}_K) \subset {\rm Image}(\Fil_{-\un{1}}W_{m+1}\Omega^{q+1}_K
  \xrightarrow{R-F} \Fil_{-\un{1}}W_m\Omega^{q+1}_K).$$
\end{rem}

\begin{lem}\label{lem:VR-4}
  The abelian group $\frac{\Fil_\n W_m\Omega^q_K}{d(\Fil_\n W_m\Omega^{q-1}_K)}$ is
  generated by the images of elements of the form
  $V^i([x]_{m-i})\dlog([x_1]_m)\wedge \cdots \wedge \dlog([x_q]_m)$, where
  $[x]_{m-i} \in \Fil_\n W_{m-i}(K), \ x_i \in K^\times$ and $0 \le i \le m-1$.
\end{lem}
\begin{proof}
  The $q =0$ case of the lemma is clear and so we assume $q \ge 1$.
 We know from \lemref{lem:Log-fil-4}(2) that
  ${\Fil_\n\Omega^q_K} = \pi^{-\n}\Omega^q_A(\log\pi)$ ($\pi^{-\n}=x_1^{-n_1}\cdots x_r^{-n_r}$),
  which is clearly generated as an abelian group by
  elements of the desired kind. This proves our assertion when $m =1$.
  Using this case and induction on $m$, the general case follows at once by
  \propref{prop:Complete-0} and the
  observation that $[x]_{m-i-1} \in \Fil_{{\un{n}}/p}W_{m-i-1}(K)$ implies
  $[x]_{m-i} \in \Fil_{{\un{n}}}W_{m-i}(K)$ and
  $R(V^i([x]_{m-i})w_1) =  V^i([x]_{m-i-1})w_2$ if we let
  $w_1 = \dlog([x_1]_m)\wedge \cdots \wedge \dlog([x_q]_m))$ and
  $w_2 = \dlog([x_1]_{m-1})\wedge \cdots \wedge \dlog([x_q]_{m-1})$.
  \end{proof}

Recall from \propref{prop:basics} that the multiplication map
$p \colon W_m\Omega^q_K \to W_{m}\Omega^q_K$ has a factorization
$W_m\Omega^q_K \stackrel{R}{\surj}  W_{m-1}\Omega^q_K \stackrel{\un{p}}{\inj}
W_{m}\Omega^q_K$.

\begin{lem}\label{lem:Complete-7}
Let $\omega \in W_{m-1}\Omega^q_K$.
Then $\omega \in \Fil_{{\un{n}}/p}W_{m-1}\Omega^q_K$ if and only if
$\un{p}(\omega) \in \Fil_{\un{n}}W_m\Omega^q_K$.
\end{lem}
\begin{proof}
  The `only if' part is clear because $R 
 \colon \Fil_{\un{n}}W_{m}\Omega^q_K \to
  \Fil_{{\un{n}}/p}W_{m-1}\Omega^q_K$ is surjective. For the reverse implication,
we let $\wt{\omega} \in W_m\Omega^q_K$ be such that $R(\wt{\omega}) = \omega$.
This yields
$VF(\wt{\omega}) = p\wt{\omega} = \un{p}(w) \in \Fil_{\un{n}}W_{m}\Omega^q_K$.
It follows from \lemref{lem:Complete-10}(2) that $VF(\wt{\omega}) = V(y')$ for some
$y' \in \Fil_{\un{n}}W_{m-1}\Omega^q_K$. \propref{prop:basics}(8) now implies
that $F(\wt{\omega}) - y' = FdV^{m-1}(z')$ for some $z' \in \Omega^{q-1}_K$.
That is, $F(\wt{\omega} - dV^{m-1}(z')) = y' \in \Fil_{\un{n}}W_{m-1}\Omega^q_K$.
\lemref{lem:Complete-10}(1) in turn implies that
$F(\wt{\omega} - dV^{m-1}(z')) = F(y'')$ for some $y'' \in \Fil_{\un{n}}W_{m}\Omega^q_K$.
By \propref{prop:basics}(7), we can find $z'' \in \Omega^q_K$ such that
$\wt{\omega} - dV^{m-1}(z') - y'' = V^{m-1}(z'')$.
That is, $\wt{\omega} - y'' = V^{m-1}(z'') + dV^{m-1}(z')$.
Equivalently, $R(\wt{\omega} - y'') = 0$. That is,
$\omega = R(\wt{\omega}) = R(y'') \in \Fil_{{\un{n}}/p}W_{m}\Omega^q_K$.
This concludes the proof.
\end{proof}

The following result will be used later on in this thesis.

\begin{lem}\label{lem:Gersten-0}
  Let $X=\Spec A$ and $X_\pi=D((\pi))$. Then $H^i_\zar(X_\pi, \sK^M_{q, X_\pi}) = 0$ for $i > 0$.
\end{lem}
\begin{proof}
  We shall prove the lemma by induction on $r =$ no. of components of the divisor given by $\pi$.
  Suppose first that $r =1$ and let $Y = \Spec({A}/{(\pi)})$.
  For any regular scheme $W$ and $n \ge 1$, we let $G_n(W)$ denote the Gersten
  complex
  \[
    {\underset{w \in W^{(0)}}\oplus} K^M_{n}(k(w)) \xrightarrow{\partial}
    {\underset{w \in W^{(1)}}\oplus} K^M_{n-1}(k(w)) \xrightarrow{\partial} \cdots .
    \]
    We then have an exact sequence
    \[
    0 \to G_{q-1}(Y)[-1] \to G_q(X) \to G_q(X_\pi) \to 0.
    \]
    Since $H^i(G_q(X)) = H^i(G_{q-1}(Y)) = 0$ for $i > 0$ by
    \cite[Prop.~10(8)]{Kerz-JAG}), it follows that $H^i_\zar(X_\pi, \sK^M_{q, X_\pi})
    = H^i(G_q(X_\pi)) = 0$ for $i > 0$.

    Suppose now that $r \ge 2$ and set $\pi_1=x_2 \cdots x_r$ and $A_1 = A_{\pi_1}$. We let, $X_1 = \Spec(A_1)$, $Y_1 = \Spec(A_1/(x_1))$, 
    and $U = X_1 \setminus Y = X_\pi$. Then we have the exact sequence
    \[
    0 \to G_{q-1}(Y_1)[-1] \to G_q(X_1) \to G_q(U) \to 0.
    \]
    We see by induction on $r$ that $H^i(G_{q-1}(Y_1)) = 0= H^i(G_q(X_1))$ for $i > 0$. Hence we conclude that $H^i_\zar(X_\pi, \sK^M_{q, X_\pi})= H^i(G_{q}(U)) = 0$ for $i > 0$. 
 \end{proof}

\begin{defn}\label{defn:1-c complex}
    Let $X$ be a Noetherian regular $F$-finite $\F_p$-scheme and let $E = \sum_i E_{i}$ be a simple normal crossing divisor on $X$
with irreducible components $E_1, \ldots , E_r$.  Let $j \colon U \inj X$ be the inclusion of the complement of
$E$ in
$X$. Let $\eta_i$ denote the generic
point of $E_i$ and let $\wh{K}_i$ be the quotient field of $\wh{\sO_{X, \eta_i}}$. Let $j_i \colon
\Spec(\sO_{X, \eta_i}) \to X$ denote the canonical map.
If $D = \sum_i n_iE_i \in \Div_E(X)$ with $n_i \ge -1$ for each $i$, we let
$W_m\sF^{q,\bullet}_{D}$ denote the 2-term complex of Zariski sheaves
\begin{equation}\label{eqn:Fil-D-complex}
    \left(Z_1\Fil_{D}W_m\Omega^q_U \xrightarrow{1-C} \Fil_{D}W_m\Omega^q_U\right),
\end{equation}
where $Z_1\Fil_{D}W_m\Omega^q_U :=\Fil_{D}W_m\Omega^q_U \bigcap j_*Z_1W_m\Omega^q_U $.
\end{defn}

Combining the local results --- \corref{cor:Complete-2}, \propref{prop:Complete-0}, Lemmas \ref{lem:Complete-4} and \ref{lem:Complete-10}, \corref{cor:VR-2} --- together with \lemref{lem:Log-fil-4}(3), we obtain the following.

\begin{thm}\label{thm:Global-version}
  Let $D \in \Div_E(X)$. We have the following.
  \begin{enumerate}
    \item
    There exists an  exact sequence of Zariski sheaves of $W_m\sO_X$-modules
    \[
    0 \to \Fil_DW_m\Omega^q_U \to j_*(W_m\Omega^q_U) \to
    \stackrel{r}{\underset{i =1}\bigoplus}
    \frac{(j_i)_*(W_m\Omega^q_{\wh{K}_i})}{(j_i)_*(\Fil_{n_i}W_m\Omega^q_{\wh{K}_i})}.
    \]
  \item
    There is a short exact sequence of Zariski sheaves of $W_m\sO_X$-modules
    \[
    0 \to V^{m-1}(\Fil_D\Omega^q_U) + dV^{m-1}(\Fil_D\Omega^{q-1}_U) \to
    \Fil_DW_m\Omega^q_U \xrightarrow{R}  \Fil_{D/p}W_{m-1}\Omega^q_U \to 0.
    \]
    \item
    $\Ker(F^{m-1}d \colon \Fil_{D}W_m\Omega^q_U \to \Fil_{D}\Omega^{q+1}_U)=Z_1\Fil_{D}W_m\Omega^q_U = F(\Fil_{D}W_{m+1}\Omega^q_U) $.
  \item
    There exist isomorphisms of Zariski sheaves of $W_m\sO_X$-modules
   \[
    \ov{F} \colon \Fil_{{D}/p}W_{m}\Omega^q_U \xrightarrow{\cong}
    \frac{Z_1\Fil_{D}W_m\Omega^q_U}{dV^{m-1}(\Fil_{D}\Omega^{q-1}_U)}; \ \
    \ov{C} \colon  \frac{Z_1\Fil_{D}W_m\Omega^q_U}{dV^{m-1}(\Fil_{D}\Omega^{q-1}_U)}
     \xrightarrow{\cong} \Fil_{{D}/p}W_{m}\Omega^q_U.
    \]
  \item
    $\Ker(F^{m-1}\colon \Fil_{D}W_m\Omega^q_U \to \Fil_{D}\Omega^q_U) =
    V(\Fil_{D}W_{m-1}\Omega^q_U)$.
  \item
    $\Ker(dV^{m-1}\colon \Fil_{D}\Omega^q_U \to \Fil_{D}W_m\Omega^{q+1}_U)
    = F^m(\Fil_{D}W_{m+1}\Omega^q_U)= Z_m\Fil_D\Omega^q_U$.
  \item
    $\Ker(V \colon \Fil_{D}W_{m}\Omega^q_U \to \Fil_{D}W_{m+1}\Omega^{q}_U)
    = dV^{m-1}(\Fil_{D}\Omega^{q-1}_U)$.
  \item
$\Ker(V^m \colon \Fil_{D}\Omega^q_U \to \Fil_{D}W_{m+1}\Omega^{q}_U)
    = F^mdV(\Fil_{D}W_m\Omega^{q-1}_U)= B_m\Fil_D\Omega^q_U$.
  \item 
    $\Ker(dV^{m-1}\colon \Fil_{D}\Omega^q_U \to \Fil_{D}W_m\Omega^{q+1}_U/V^{m-1}\Fil_D \Omega^{q+1}_U)
    = F^{m-1}(\Fil_{\un{n}}W_{m}\Omega^q_K)$.
   \item
    $\Ker(V^{m-1} \colon \Fil_{D}\Omega^q_U \to \Fil_{D}W_{m}\Omega^{q}_U/dV^{m-1}\Fil_{D}\Omega^{q-1}_U )
    = F^{m}dV(\Fil_{D}W_{m}\Omega^{q-1}_U)$.
  \item
    Let $D\ge -E$. In the derived category of Zariski sheaves of abelian groups on $X$,
    there exists a distinguished triangle
    \[
     W_1\sF^{q\bullet}_{D} \xrightarrow{V^{m-1}} W_m\sF^{q,\bullet}_{D}
    \xrightarrow{R} W_{m-1}\sF^{q,\bullet}_{D/p} \xrightarrow{+}  W_1\sF^{q\bullet}_{D} [1].
    \]
    \item $\varinjlim\limits_{D\ge -E} W_m\sF^{q,\bullet}_{D}= \left(j_*Z_1W_m\Omega^q_U \xrightarrow{1-C} j_*W_m\Omega^q_U\right)=Rj_*W_m\Omega^q_U.$
     
  \end{enumerate}
\end{thm}

\chapter{Comparison with Kato's filtration}\label{chap:Kato-coh}
Let $k=\F_p$. For $q \in \Z$ and $n = p^mr \in \N$ with $(p,r) =1$, we let
${\Z}/n(q) = {\Z}/r(q) \oplus W_m\Omega^q_{(-,\log)}[-q]$
and consider it as a complex of {\'e}tale sheaves on the big {\'e}tale site
${\Et}_k$ of $\Spec(k)$, where the first summand on the right is the standard
{\'e}tale twist (cf. \cite[p.~163]{Milne-EC}). If $r' \in \N$ such that $r \mid r'$ and $(p,r')=1$, we have a natural map $\Z/r(q) \to \Z/r'(q)$, induced by the maps $Z/r \cong \frac{1}{r}\Z/\Z \inj \frac{1}{r'}\Z/\Z \cong \Z/r'$ and $\mu_r \inj \mu_{r'}$ ($\mu_r$ is the sheaf of $r$-th roots of unity on ${\Et}_k$). On the other hand, the canonical map $\ov p : W_m\Omega^q_{(-,\log)} \to W_{m+1}\Omega^q_{(-,\log)}$ induces a canonical map $\Z/p^m(q) \to \Z/p^{m+1}(q)$. As a result, we have canonical maps $\Z/n(q) \to \Z/n'(q)$, when $n \mid n'$. For ring $R$, we let
$H^q_m(R) = H^q_\et(\Spec(R), {\Z}/{p^m}(q-1))$.
We let $H^q(R) = H^q_\et(\Spec(R), {\Q}/{\Z}(q-1)) =
{\varinjlim}_n H^q_\et(\Spec(R), {\Z}/n(q-1))$, where the transition maps are induced by the canonical map $\Z/n(q) \to \Z/n'(q)$ (if $n \mid n'$) discussed above. It follows that $H^q(R)\{p\} =
{\varinjlim}_m H^q_m(R)$, with the transition maps induced by the map $\Z/p^m(q) \to \Z/p^{m+1}(q)$. For a scheme $X$, one defines the groups $H^q(X)$ and $H^q_m(X)$ in a similar way. The purpose of this chapter is to provide a cohomological description of Kato’s ramification filtration.

\section{Comparison with Local Kato filtration}\label{sec:loc-Kato}
We let $K^M_*(R)$ denote the (improved) Milnor $K$-theory of $R$ as defined in
\cite{Kerz-JAG}.
This coincides with the Milnor $K$-theory of $R$ as defined in \cite{Kato-86} if $R$
is a local ring with infinite residue field (cf. \cite[Lem.~2.2]{GK-Duality}).
We let $\dlog \colon {K^M_q(R)}/{p^m} \to W_m\Omega^q_R$ denote the dlog map
induced by $\dlog(\{x_1, \ldots , x_q\}) = \dlog[x_1]_m \wedge \cdots \wedge
\dlog[x_q]_m$ (cf. \cite{Illusie}, \cite[Lem.~3.2.8]{Zhao}) and let $W_m\Omega^q_{R,\log}$ denote its image.
This map is multiplicative with respect to the product structures of
$K^M_*(R)$ and $W_m\Omega^\bullet_{R}$.
The map $\dlog \colon {K^M_q(R)}/{p^m} \to W_m\Omega^q_{R,\log}$ is bijective if
$R$ is a regular local ring (cf. \cite[Thm.~5.1]{Morrow-ENS}).

\subsection{Kato filtration}\label{sec:katofil}

If $R$ is a regular local ring, the exact sequence
\begin{equation}\label{eqn:Milnor-0}
  0 \to W_m\Omega^q_{{(-)}, \log} \to Z_1W_m\Omega^q_{(-)} \xrightarrow{1-C}
    W_m\Omega^q_{(-)} \to 0
\end{equation}
of sheaves on $(\Spec R)_\et$ gives an exact sequence
\begin{equation}\label{eqn:Milnor-0.1}
  0 \to W_m\Omega^q_{R, \log} \to Z_1W_m\Omega^q_R \xrightarrow{1-C} W_m\Omega^q_R
  \xrightarrow{\delta^q_m} H^{q+1}_m(R) \to 0.
  \end{equation}
We also have a commutative diagram (cf. \cite[\S~1.3]{Kato-89})
\begin{equation}\label{eqn:Milnor-1}
  \xymatrix@C1pc{
    {K^M_q(R)}/{p^m} \otimes W_m\Omega^{q'}_{R} \ar[r]^-{\cup}
    \ar[d]_-{id \otimes \delta^{q}_m} \ar[dr]^-{\lambda^{q+q'}_m} &
    W_m\Omega^{q+q'}_{R} \ar[d]^-{\delta^{q+q'}_{m}} \\
     {K^M_q(R)}/{p^m} \otimes H^{q'+1}_m(R) \ar[r]^-{\cup} &
     H^{q+q'+1}_m(R),}
\end{equation}
where the bottom row is induced by the Bloch-Kato map
$\beta^q_{R} \colon K^M_q(R) \to H^q_\et(R, {\Z}/{p^m}(q))$,
followed by the cup product on the {\'e}tale cohomology of
$W_m\Omega^\bullet_{(-), \log}$ 
on $\Spec(R)$. The top row is induced by the dlog map, followed by the wedge
product. We shall denote the composite map
$W_m\Omega^q_R \xrightarrow{\delta^q_m} H^{q+1}_m(R) \to H^{q+1}(R)$
also by $\delta^q_m$.

We now let $A$ be an $F$-finite Henselian discrete valuation ring containing $k$.
Let $\fm = (\pi)$
denote the maximal ideal of $A$. Let $\ff = {A}/{\fm}$ and $K = Q(A)$. 
\begin{lem}\label{lem:Kato-fil-9}
  The map $ H^1_\et(K, W_m\Omega^q_{K, \log}) \to \ _{p^m}H^{q+1}(K)$ is a bijection.
\end{lem}
\begin{proof}
  We let $X = \Spec(K)$.
Since the {\'e}tale sheaf $\sK^M_{q,X}$ has no $p$-torsion (see \cite[Thm.~8.1]{Geisser-Levine}),
  there is an exact sequence of sheaves
  \[
  0 \to {\sK^M_{q,X}}/{p^m} \to {\varinjlim}_r  {\sK^M_{q,X}}/{p^r}
  \xrightarrow{p^m} {\varinjlim}_r  {\sK^M_{q,X}}/{p^r} \to 0.
  \]
  Considering the cohomology and observing that
  $H^0_\et(X,  {\sK^M_{q,X}}/{p^r}) \cong H^0_\et(X,  W_r\Omega^q_{X, \log}) \cong
  W_r\Omega^q_{K, \log} \cong {K^M_q(K)}/{p^r}$,
  we get an exact sequence
  \[
  0 \to  {K^M_q(K)}/{p^m} \to K^M_q(K)\otimes {\Q_p}/{\Z_p}   \xrightarrow{p^m}
  {K^M_q(K)} \otimes {\Q_p}/{\Z_p} \xrightarrow{\partial} \hspace*{3cm}
  \]
  \[
  \hspace*{5cm}
  H^1_\et(X, W_r\Omega^q_{X, \log}) \to H^{q+1}(K)\{p\}  \xrightarrow{p^m}  H^{q+1}(K)\{p\}.
  \]
  Since $K^M_q(K)\otimes {\Q_p}/{\Z_p} \xrightarrow{p^m} {K^M_q(K)}\otimes {\Q_p}/{\Z_p}$
is surjective, the lemma follows.
\end{proof}

We fix integers $n, q \ge 0$ and $m \ge 1$. Let $T^{m,q}_n(K)$ denote the cokernel of the map
$(1-C) \colon Z_1\Fil_nW_m\Omega^q_K \to \Fil_nW_m\Omega^q_K$. 
\begin{lem}\label{lem:hyp}
        Let $X=\Spec A, m\ge 1, q \ge 0$. Let $E=V((\pi)),\ U=X\setminus E$ and $D_n=nE$, for $n \ge 0$. Then, we have 
  \begin{enumerate}
      \item $H^i_\et(X, Z_1 \Fil_{D_{n}} W_m \Omega^q_U) =
  H^i_\et(X, \Fil_{D_{n}} W_m \Omega^q_U) = 0, \text{ for $i \ge 1$.}$ 
  \item \[
\H^i_\et(X, W_m\sF^q_{D_n}) = \left\{\begin{array}{ll}
K^M_q(K)/p^m & \mbox{if $i = 0$} \\
T^{m,q}_n(K) & \mbox{if $i = 1$} \\
 0 & \mbox{if $i > 1$.}
\end{array}\right.
\]
  \end{enumerate}
   \end{lem}
   \begin{proof}
       (1) Recall that higher cohomology of coherent sheaves on affine space is zero. Hence $H^i_\et(X, Z_1 \Fil_{D_{n}} W_m \Omega^q_U)=H^i_\et(W_m(X), Z_1 \Fil_{D_{n}} W_m \Omega^q_U)=0$. Similarly, we have $H^i_\et(X, \Fil_{D_{n}} W_m \Omega^q_U) = 0$. Here, $W_m(X)$ denotes the scheme $(X, W_m\sO_{X})$.  

    (2) We look at the following hypercohomology long exact sequence
    \begin{equation}\label{eqn:hyp}
         0 \to \H^0_\et(W_m\sF^q_{D_n}) \to H^0_\et(Z_1 \Fil_{D_{n}} W_m \Omega^q_U) \xrightarrow{1-C} H^0(\Fil_{D_{n}} W_m \Omega^q_U) \to \H^1_\et(W_m\sF^q_{D_n}) \xrightarrow{} 0,
    \end{equation}
    where $\H^i_\et(\sE)$ (resp. $H^i_\et(\sE)$) is the short hand notation for $\H^i_\et(X,\sE)$ (resp. $H^i_\et(X,\sE)$) and the last term of \eqref{eqn:hyp} is $0$ because of item (1). Since $W_m\Omega^q_{A} \subset \Fil_{n}W_m\Omega^q_{K} \subset W_m\Omega^q_{K}$, we have 
    \begin{equation*}
        \begin{array}{cl}
           K^M_q(A)/p^m   &\xrightarrow[\cong]{\dlog}\Ker(1-C: Z_1W_m\Omega^q_{A} \to W_m\Omega^q_{A} ) \\
             &\subset \Ker(1-C: Z_1\Fil_{n}W_m\Omega^q_{K} \to \Fil_{n}W_m\Omega^q_{K} )= \H^0_\et(A,W_m\sF^q_{D_n})\\
             &\subset \Ker(1-C: Z_1W_m\Omega^q_{K} \to W_m\Omega^q_{K} )= K^M_q(K)/p^m.
        \end{array}
    \end{equation*}
    Since $A$ is DVR, we note that $\dlog(K^M_q(K)/p^m) \subset Z_1W_m\Omega^q_{K}$ is generated by (as a group) elements of the form 
    \begin{enumerate}
        \item $\dlog([a_1]_m)\wedge \cdots \wedge \dlog([a_q]_m)$ and
        \item $\dlog([a_1]_m)\wedge \cdots \wedge \dlog([a_{q-1}]_m)\wedge\dlog([\pi_x]_m)$,
    \end{enumerate}
    where where $a_i \in A^\star$.  However, the elements of the form (1) and (2) also lie within $Z_1\Fil_{0}W_m\Omega^q_{K} \subset Z_1\Fil_{n}W_m\Omega^q_{K}$, for $n \ge 0$. Hence, we conclude that $$K^M_q(K)/p^m \cong \H^0_\et(A,W_m\sF^q_{D_n}), \text{ and}$$
    $$\H^1_\et(X,W_m\sF^q_{D_n}) =\coker(1-C: Z_1\Fil_{n}W_m\Omega^q_{K} \to \Fil_{n}W_m\Omega^q_{K} )=T^{m,q}_n(K).$$
 Also, we get that $\H^i_\et(X,W_m\sF^q_{D_n})=0$ by looking at the hypercohomology long exact sequence and using item (1). This proves the lemma.
   \end{proof}
We have an exact
sequence
\begin{equation}\label{eqn:Milnor-2}
  0 \to W_m\Omega^q_{K, \log} \to Z_1\Fil_nW_m\Omega^q_K \xrightarrow{1-C}
  \Fil_nW_m\Omega^q_K \xrightarrow{\partial^{m,q}_n} T^{m,q}_n(K) \to 0
  \end{equation}
which canonically maps to the exact sequence
~\eqref{eqn:Milnor-0.1} with $R = K$. We let $\ov{\delta}^{m,q}_n \colon
T^{m,q}_n(K) \to H^{q+1}(K)$ be the induced map.

By \lemref{lem:Log-fil-4}(8) (for $D=0$, $D'=nV(\pi)$), we see that the product map
on the top row of ~\eqref{eqn:Milnor-1} (with $R =K$) restricts to
a map
\begin{equation}\label{eqn:Milnor-3}
K^M_q(K)/p^m \otimes \Fil_nW_m\Omega^{q'}_K \xrightarrow{\cup} \Fil_nW_m\Omega^{q+q'}_K.
\end{equation}
It follows from Lemmas~\ref{lem:VR-3} and ~\ref{lem:VR-4}
that the composite map $ K^M_q(K) \otimes \Fil_nW_m(K) \xrightarrow{\cup}
\Fil_nW_m\Omega^q_K \xrightarrow{\partial^{m,q}_n} T^{m,q}_n(K)$ is surjective.
Combining this with \propref{prop:Cartier-fil-1} and \corref{cor:Complete-2} (for $r=1$), we get the following.

\begin{cor}\label{cor:VR-5}
  The map
  \[
 \ov{\delta}^{m,q}_n \colon T^{m,q}_n(K) \to
    {\rm Image} (\Fil_nW_m(K) \otimes K^M_q(K)  \xrightarrow{\lambda^q_m} H^{q+1}(K))
    \]
    is a bijection.
\end{cor}
\begin{proof}
    By \lemref{lem:Kato-fil-9}, we only need to show the map $T^{m,q}_n(K)\to H^{q+1}_m(K)$ is injective. If $K$ is complete then this follows from \propref{prop:Cartier-fil-1}. If $K$ is not complete, then we reduce to the complete case by \corref{cor:Complete-2} (for $r=1$).
\end{proof}

We now recall Kato's filtration on $H^{q}(K)$ from \cite{Kato-89}. For any
$A$-algebra $R$, let $R^h$ denote the Henselization of $R$ with respect to the
ideal $\fm R$. We let
$\Spec({R}/{\fm R}) \xrightarrow{\iota} \Spec(R) \xleftarrow{j} \Spec(R[\pi^{-1}])$
be the inclusions and let
\begin{equation}\label{eqn:Kato-0}
V^q(R) = {\varinjlim}_{l} H^q_\et(\Spec({R}/{\fm R}), \iota^*Rj_*({\Z}/{l}(q-1)))
\cong H^q(R^h[\pi^{-1}]).
\end{equation}
The following definition is due to Kato \cite[Defn.~2.1]{Kato-89}, inspired by
a filtration of $W_m(K)$ by Brylinski \cite{Brylinski}. 

\begin{defn}\label{defn:Kato-1}
We let $\Fil^{\bk}_nH^{q}(K)$ be the set of all elements
  $\chi \in H^q(K)$ such that $\{\chi, 1 + \pi^{n+1}T\} = 0$ in $V^{q+1}(A[T])$
under the product map $$ H^q(K)\otimes K^M_1((A[T])^h[\pi^{-1}])  \to V^{q+1}(A[T]).$$ 
We let $\gr^{\bk}_0H^q(K) = \Fil^{\bk}_0H^q(K)$ and
$\gr^{\bk}_nH^q(K) = {\Fil^{\bk}_nH^{q}(K)}/{\Fil^{\bk}_{n-1}H^{q}(K)}$ if $n \ge 1$.
  \end{defn}

The following result of Kato (cf. \cite[Thm.~3.2]{Kato-89})
provides a description of $\Fil^{\bk}_nH^{q}(K)$ in terms of Milnor $K$-theory and
Witt-vectors.

\begin{thm}\label{thm:Kato-2}
  One has the following.
  \begin{enumerate}
  \item
    $\Fil^{\bk}_nH^1(K)\{p\} = {\underset{m \ge 1}\bigcup} \delta^0_m(\Fil_nW_m(K))$.
  \item
 $\Fil^{\bk}_nH^{q+1}(K)\{p\}$ coincides with the image of
    \[
    K^M_q(K) \otimes \Fil^{\bk}_nH^1(K)\{p\} \xrightarrow{\cup} H^{q+1}(K).
    \]
    \end{enumerate}
\end{thm}

In \cite[Page~2 of the Erratum]{Kerz-Saito-ANT}, Kerz and Saito proved the following refined version
of \thmref{thm:Kato-2}(1).

\begin{thm}\label{thm:Kato-3}
  One has $\ _{p^m}\Fil^{\bk}_nH^1(K) = \delta^0_m(\Fil_nW_m(K))$.
\end{thm}

\subsection{Proof of the main theorem}

The main result of this section is the following refinement of
 \thmref{thm:Kato-2}.

 \begin{thm}\label{thm:Kato-4}
   For $n, q \ge 0$ and $m \ge 1$, the map $\delta^q_m \colon
   \Fil_nW_m\Omega^q_K \to H^{q+1}(K)$ induces an isomophism
   \[
  \ov{\delta}^{m,q}_n \colon T^{m,q}_n(K) \xrightarrow{\cong} \ _{p^m}\Fil^{\bk}_nH^{q+1}(K).
 \]
\end{thm}

 \vskip .3cm
 
 We shall prove this theorem in several steps. We fix $n, q \ge 0$ and $m \ge 1$ as
 before. When $n > 0$, we shall write $n = p^rl$, where $r \ge 0$ and $(p,l) =1$.
 We let $S^{m,q}_n(K) = \ _{p^m}\Fil^{\bk}_nH^{q+1}(K)$. Since $T^{m,q}_n(K)$
 is a $p^m$-torsion group, we already deduce from
 \corref{cor:VR-5}, \thmref{thm:Kato-2} and ~\eqref{eqn:Milnor-1} that the map
 $\ov{\delta}^{m,q}_n$ has a factorization
 $\ov{\delta}^{m,q}_n\colon T^{m,q}_n(K) \inj S^{m,q}_n(K)$. 
In what follows, our goal is to show that this map is surjective.

\vskip .3cm

We begin with the following.

\begin{lem}\label{lem:Kato-fil-0}
  There is a commutative diagram
   \begin{equation}\label{eqn:Kato-fil-0-0}
   \xymatrix@C1pc{
     Z_1\Fil_nW_m\Omega^q_K \ar[r]^-{1-C} \ar[d]_-{V} &  \Fil_nW_m\Omega^q_K
     \ar[d]^-{V} \\
     Z_1\Fil_nW_{m+1}\Omega^q_K \ar[r]^-{1-C} &  \Fil_nW_{m+1}\Omega^q_K.}
   \end{equation}
   \end{lem}
\begin{proof}
  To prove the lemma, it
  suffices to show that $VC = CV$ and $V(Z_1W_m\Omega^q_K) \subset
  Z_1W_{m+1}\Omega^q_K$ (cf. \lemref{lem:Log-fil-4}(6)).
  For showing the first claim, it is enough to
   show that $VCF = CVF$. But this is clear because $VCF = VR = RV = CFV = CVF$.
   For the second claim, note that $\omega \in Z_1W_m\Omega^q_K$ implies
   $F^{m-1}d(\omega) = 0$. This implies in turn that $F^mdV(\omega) =
   F^{m-1}d(\omega) = 0$. That is, $V(\omega) \in Z_1W_{m+1}\Omega^q_K$.
\end{proof}

Using \lemref{lem:Kato-fil-0}, we get a commutative diagram
 (cf. \cite[\S~1.3]{Kato-89})
 \begin{equation}\label{eqn:Kato-5}
   \xymatrix@C2.5pc{
     \Fil_nW_m\Omega^q_K \ar[d]_-{V} \ar@{->>}[r]^-{\delta^{m,q}_n} & T^{m,q}_n(K)
     \ar[d]^-{V}
     \ar@{^{(}->}[r]^-{\ov{\delta}^{m,q}_n} & S^{m,q}_n(K) \ar@{^{(}->}[d] \\
     \Fil_nW_{m+1}\Omega^q_K \ar@{->>}[r]^-{\delta^{m+1,q}_n} & T^{m+1,q}_n(K)
     \ar@{^{(}->}[r]^-{\ov{\delta}^{m+1,q}_n} &  S^{m+1,q}_n(K),}
   \end{equation}
 where the middle $V$ is the unique map induced by the
 commutative diagran ~\eqref{eqn:Kato-fil-0-0} and the right vertical arrow
 is the canonical inclusion $\ _{p^m}\Fil^{\bk}_nH^{q+1}(K) \inj
 \ _{p^{m+1}}\Fil^{\bk}_nH^{q+1}(K)$. Since the canonical inclusion $H^{q+1}_m(K) \inj H^{q+1}_{m+1}(K)$ is induced by $\ov p = V : W_m\Omega^q_{(-,\log)} \to W_{m+1}\Omega^q_{(-,\log)}$, we see that the right square of \eqref{eqn:Kato-5} is commutative.

 The commutative diagram
  \begin{equation}\label{eqn:Kato-6}
   \xymatrix@C1pc{
     T^{m,q}_{n-1}(K) \ar[r] \ar[d]_-{\ov{\delta}^{m,q}_{n-1}} & T^{m,q}_{n}(K)
     \ar[d]^-{\ov{\delta}^{m,q}_{n}} \\
     S^{m,q}_{n-1}(K) \ar@{^{(}->}[r] & S^{m,q}_n(K)}
  \end{equation}
  implies that $T^{m,q}_{n-1}(K) \inj T^{m,q}_{n}(K)$ for every $n \ge 1$.
  We let $\wt{T}^{m,q}_n(K) = \frac{T^{m,q}_{n}(K)}{T^{m,q}_{n-1}(K)}$ and
  $\wt{S}^{m,q}_n(K) = \frac{S^{m,q}_{n}(K)}{S^{m,q}_{n-1}(K)}$ and let 
  $\psi^{m,q}_n \colon \wt{T}^{m,q}_n(K) \to \wt{S}^{m,q}_n(K)$ be the map induced by
 the vertical arrows in ~\eqref{eqn:Kato-6}.
 Since the commutative diagram ~\eqref{eqn:Kato-fil-0-0} is compatible with
 the change in values of $n \ge 0$, it induces maps
 $V \colon \wt{T}^{m,q}_n(K) \to \wt{T}^{m+1,q}_n(K)$ as $n$ varies over all
 positive integers. We let $\wt{M}^{m,q}_n(K) =
 \frac{\wt{T}^{m,q}_n(K)}{V(\wt{T}^{m-1,q}_n(K))}$ and $\wt{N}^{m,q}_n(K) =
 \frac{\wt{S}^{m,q}_n(K)}{\wt{S}^{m-1,q}_n(K)}$.
It is clear that $\ov{\delta}^{m,q}_n$
 descends to a homomorphism $\ov{\delta}^{m,q}_n \colon \wt{M}^{m,q}_n(K) \to
 \wt{N}^{m,q}_n(K)$.

Form now on, we shall assume $n>1$, unless otherwise stated.
 
 \begin{lem}\label{lem:Kato-fil-1}
   The map $V \colon \wt{T}^{m,q}_n(K) \to \wt{T}^{m+1,q}_n(K)$ is surjective for every
   $m \ge r+1$.
 \end{lem}
 \begin{proof}
   Since $V(a\dlog([x_1]_m)\wedge \cdots \wedge \dlog([x_q]_m) =
   V(a)\dlog([x_1]_m)\wedge \cdots \wedge \dlog([x_q]_m)$,
   ~\eqref{eqn:Kato-fil-0-0} implies that the diagram
   \begin{equation}\label{eqn:Kato-fil-1-0}
   \xymatrix@C1pc{
K^M_q(K) \otimes \Fil_nW_m(K) \ar@{->>}[r]^-{\partial^{m,q}_n} \ar[d]_-{id \otimes V}
& T^{m,q}_{n}(K) \ar[d]^-{V} \\
K^M_q(K) \otimes \Fil_nW_{m+1}(K) \ar@{->>}[r]^-{\partial^{m+1,q}_n} & T^{m+1,q}_{n}(K)}
   \end{equation}
   is commutative. Furthermore, the horizontal arrows are surjective by
   \corref{cor:VR-5}. It suffices therefore to show that
   $\frac{\Fil_nW_{m+1}(K)}{\Fil_{n-1}W_{m+1}(K) + V(\Fil_nW_{m}(K))} = 0$ for
   $m \ge r+1$.
   
   To that end, we look at the commutative diagram of exact sequences
  \begin{equation}\label{eqn:Kato-fil-1-1}
   \xymatrix@C.8pc{ 
     0 \ar[r] & \Fil_{n-1}W_{m}(K) \ar[r]^-{V} \ar[d] & \Fil_{n-1}W_{m+1}(K)
     \ar[d] \ar[r]^-{R^m} & \Fil_{\lfloor{{(n-1)}/{p^m}}\rfloor}W_1(K) \ar[d] \ar[r] & 0 \\
   0 \ar[r] & \Fil_{n}W_{m}(K) \ar[r]^-{V} & \Fil_{n}W_{m+1}(K)
   \ar[r]^-{R^{m}} & \Fil_{\lfloor{{n}/{p^m}}\rfloor}W_1(K) \ar[r] & 0.}
  \end{equation}
  Since $m \ge r+1$, we see that $\lfloor{{(n-1)}/{p^m}}\rfloor =
  \lfloor{{n}/{p^m} -{1}/{p^m}}\rfloor = \lfloor{{n}/{p^m}}\rfloor$.
  This implies that the right vertical arrow in ~\eqref{eqn:Kato-fil-1-1} is
  bijective, which proves the asserted claim.
\end{proof}

 For the next three lemmas, we assume that $1 \le m \le r$ and consider the map
 $$\wt{\theta}_1 \colon A \otimes \overbrace{A^\times \otimes \cdots \otimes A^\times}_q \to 
 \wt{T}^{m,q}_{n}(K)$$
 $$ \wt{\theta}_1(x \otimes y_1 \otimes \cdots \otimes y_q)
 = \partial^{m,q}_n([x\pi^{-np^{1-m}}]_m\dlog([y_1]_m)\wedge \cdots \wedge \dlog([y_q]_m))
 \text{ mod } T^{m,q}_{n-1}(K).$$ Note that this is defined because $np^{1-m} = p^{r+1-m}l\in \N$ and $[x\pi^{-np^{1-m}}]_m \in \Fil_nW(K)$.

 \begin{lem}\label{lem:Kato-fil-2}
   $\wt{\theta}_1$ descends to a group homomorphism
   $\theta^q_1 \colon \Omega^q_{\ff} \to \wt{M}^{m,q}_n(K)$.
 \end{lem}
 \begin{proof}
   For $\{y_1, \ldots , y_q\} \subset A^\times$, we write
   $\dlog([y_1]_m)\wedge \cdots \wedge \dlog([y_q]_m) = \dlog(\un{y})$. 
   If $x = a\pi$ for some $a \in A$, then
   $\partial^{m,q}_n([x]_m\dlog(\un{y})) =
   \partial^{m,q}_n([a\pi^{-(n-p^{m-1})p^{1-m}}]_m\dlog(\un{y})) \in T^{m,q}_{n-1}(K)$, as $[a\pi^{-(n-p^{m-1})p^{1-m}}]_m \in \Fil_{n-1}W_m(K)$. 
   Suppose next that $y_i = 1 + a\pi$ for some $a \in A$ and $1 \le i \le q$.
   By \cite[Lem.~3.5]{Kato-89}, we see that
   as an element of $\Fil_nW_m\Omega^1_K$, $[x]_m\dlog([1+a\pi]_m)$ actually lies
   in $\Fil_{n-1}W_m\Omega^1_K$. It follows that
   $[x]_m\dlog(\un{y}) \in \Fil_{n-1}W_m\Omega^q_K$ and hence
   $\partial^{m,q}_n([x]_m\dlog(\un{y}))$ dies in $\wt{T}^{m,q}_n(K)$.
   We have thus shown that $\wt{\theta}_1$ descends a group homomorphism
\begin{equation}\label{eqn:Kato-fil-2-0}
  \wt{\theta}_1 \colon  F \otimes F^\times \otimes \cdots \otimes F^\times \to
  \wt{T}^{m,q}_n(K) \surj \wt{M}^{m,q}_n(K).
\end{equation}

We now let $\{x_1, \ldots , x_l\}$ and $\{x'_1, \ldots , x'_{l'}\}$ be elements of
$A^\times$ such that $\sum_i x_i - \sum_j x'_j = a\pi$ for some $a \in A$.
Let $y_1, \ldots , y_{q-1} \in A^\times$.
We claim $\wt{\theta}_1$ kills elements of the form
$$(\sum_i x_i \otimes x_i- \sum_j x'_i \otimes x'_i) \otimes y_1 \otimes \cdots \otimes y_{q-1} .$$
Writing $w_0 = [\pi^{-np^{1-m}}]_m\dlog[y_1]_m \cdots \dlog[y_{q-1}]_m$, we get the following in ${T}^{m,q}_n(K)$.
$$w_0\left(\sum_i [x_i]_m\dlog[x_i]_m -
\sum_j [x'_i]_m\dlog[x'_j]_m\right) = w_0d\left(\sum_i[x_i]_m - \sum_j [x'_j]_m\right)$$
$$ =w_0d\left([\sum_i x_i - \sum_jx'_i)]+V(\un{b})\right) = w_0d[a\pi]_m + w_0dV(\un{b})$$
for some $\un{b}=(a_{m-2}, \ldots , a_1, a_0) \in W_{m-1}(A)$.

Since $w_0 \in \Fil_nW_m\Omega^{q-1}_K$, we get that $[a\pi]_mw_0 \in \Fil_{n-1}W_m\Omega^{q-1}_K$, and hence $w_0d[a\pi]_m=[a\pi]_mw_0 \dlog[a\pi]_m \in \Fil_{n-1}W_m\Omega^q_K$. So, it dies in $\wt{M}^{m,q}_n(K)$. On the other hand,
$$w_0dV(\un{b})=dV(F(w_0)\un{b})
  - V(Fd(w_0)\un{b}) \in (1-C)(\Fil_nW_m\Omega^{q}_K) +V(\Fil_nW_{m-1}\Omega^{q}_K)$$
 by Lemmas~\ref{lem:Log-fil-4}(4),(5),(6) and \ref{lem:VR-3} and hence also dies in $\wt{M}^{m,q}_n(K)$. 
We conclude from
\cite[Lem.~4.2]{Bloch-Kato} that $\wt{\theta}_1$ descends to a homomorphism
$\theta^q_1 \colon \Omega^q_{\ff} \to \wt{M}^{m,q}_n(K)$.

\end{proof}

 An identical proof shows that $\wt{\phi} \colon A \otimes A^\times \otimes \cdots
 \otimes A^\times \to \wt{T}^{r+1,q}_{n}(K)$ given by
 $\wt{\theta}_1(x \otimes y_1 \otimes \cdots \otimes y_q)
 = \partial^{m,q}_n([x\pi^{-l}]_m\dlog([y_1]_m)\wedge \cdots \wedge \dlog([y_q]_m))$
 mod $T^{r+1,q}_{n-1}(K)$, induces a homomorphism
\begin{equation}\label{eqn:Kato-fil-2-2}
   \phi^q\colon \Omega^q_{\ff} \to \wt{M}^{r+1,q}_n(K).
   \end{equation}

 \begin{lem}\label{lem:Kato-fil-3}
   $\theta^q_1(Z_1\Omega^q_{\ff}) = 0$.
 \end{lem}
 \begin{proof}
   By \cite[\S~3, p.~111]{Kato-89}, $Z_1\Omega^q_{\ff}$ is generated by the elements of the
   types
   \begin{enumerate}
   \item
     $x\dlog(x)\wedge \dlog(y_1) \wedge \cdots \wedge \dlog(y_{q-1})$ and
   \item
     $x^p\dlog(y_1) \wedge \cdots \wedge \dlog(y_{q})$,
   \end{enumerate}
    where $x \in \ff, y_j \in \ff^\star, 1 \le j \le q$.
  
   We now let $x, y_1, \ldots y_q \in A^\times, \ w = x^p\pi^{-np^{1-m}}$ and $\dlog(\un{y})_l=\dlog([y_1]_l) \wedge \cdots \wedge
   \dlog([y_q]_l)$. Since $m-1 < r$ and $p^r|n$, we can write
   $w = (x\pi^{-(n/p)p^{1-m}})^p$. Letting $w' = x\pi^{-(n/p)p^{1-m}}$, 
   we get by \lemref{lem:Complete-4}(1) that
   $$[w']_m \in \Fil_{n/p}W_m(K) \text{ and }
   [w]_m\dlog(\un{y})_m= F([w']_{m+1}\dlog(\un{y})_{m+1}) \in Z_1\Fil_nW_m\Omega^q_K.$$
   In particular,
   $\partial^{m,0}_n \circ (1-C)([w]_m\dlog(\un{y})_m) = 0$ in $T^{m,q}_n(K)$.
   From this, we get $$\partial^{m,q}_n([w]_m\dlog(\un{y})_m) = \partial^{m,q}_n \circ C([w]_m\dlog(\un{y})_m)
   = \partial^{m,q}_n([w']_m\dlog(\un{y})_m).$$
   Since the last term is in $\Fil_{n/p}W_m\Omega^{q}_K \subset \Fil_{n-1}W_m\Omega^q_K$,
   it follows that $\partial^{m,q}_n([w]_m\dlog(\un{y})_m)$ dies in $\wt{M}^{m,q}_n(K)$. This shows that the type (2) elements goes to zero under the map $\theta_1^q$.
   
   Next, we let $a = \pi^{-np^{1-m}}$ and $w_l'' = \dlog([y_1]_l) \wedge \cdots \wedge
   \dlog([y_{q-1}]_l)$ so that
   $$[xa]_m\dlog([x]_m)w_m'' = [a]_md([x]_m)w_m'' = d([xa]_mw_m'') - [x]_md([a]_m)w_m''.$$
   \lemref{lem:VR-3} implies $\partial^{m,q}_n([xa]_m\dlog([x]_m)w''_m) =
   \partial^{m,q}_n([x]_md([a]_m)w''_m)$. On the other hand,
   \[
   \begin{array}{lll}
     [x]_md([a]_m)w''_m & = & [x]_mw''_m(-{n}/{p^{m-1}})[a]_m\dlog[\pi]_m \\
     & = & p(-{n}/{p^{m}}) [xa]_mw''_m\dlog[\pi]_m \\
     & = & (-{n}/{p^{m}}) VF([xa]_mw''_m\dlog[\pi]_m).
   \end{array}
   \]
   Since
   $F([xa]_mw''\dlog[\pi]_m )\in \Fil_nW_{m-1}\Omega^q_K$,
   we conclude that $\partial^{m,q}_n([x]_md([a]_m)w'')$ dies in
   $\wt{M}^{m,q}_n(K)$. Hence, so does $\partial^{m,q}_n([xa]_m\dlog([x]_m)w'')$ (image of the elements of type (1)).
   This finishes the proof.
   \end{proof}

We next consider the map
 $$\wt{\theta}_2 \colon A \otimes A^\times \otimes \cdots \otimes A^\times \to 
 \wt{T}^{m,q}_{n}(K)$$ 
$$\wt{\theta}_2(a \otimes y_1 \otimes \cdots \otimes y_{q-1})
 = \partial^{m,q}_n([x\pi^{-np^{1-m}}]_m
 \dlog([y_1]_m)\wedge \cdots \wedge \dlog([y_{q-1}]_m) \wedge \dlog([\pi]_m)),$$ 
 $$ \hspace{10cm}\text{ mod }
 T^{m,q}_{n-1}(K).$$

\begin{lem}\label{lem:Kato-fil-4}
  $\wt{\theta}_2$ descends to a group homomorphism
  $\theta^q_2 \colon \frac{\Omega^{q-1}_{\ff}}{Z_1\Omega^{q-1}_{\ff}} \to \wt{M}^{m,q}_n(K)$.
\end{lem}
\begin{proof}
  By \lemref{lem:Log-fil-4}(8), we see that the multiplication by $\dlog([\pi]_m)$
  preserves $\Fil_nW_m\Omega^\bullet_K$. Since $F^{m-1}d(a\dlog([\pi]_m) =
  F^{m-1}d(a)\dlog([\pi]_m)$ and $C(a\dlog([\pi]_m)) = C(a)\dlog([\pi]_m)$ for $a \in Z_1W_\textbf{.}\Omega^\bullet_K$,
  it follows that the map $W_m\Omega^{q-1}_K \xrightarrow{\dlog([\pi]_m)}
  W_m\Omega^q_K$ induces maps
  \[
  T^{m,q-1}_n(K) \xrightarrow{\dlog([\pi]_m)} T^{m,q}_n(K), \ \wt{T}^{m,q-1}_n(K)
  \xrightarrow{\dlog([\pi]_m)} \wt{T}^{m,q}_n(K).
  \]
  Since $V(a\dlog([\pi]_m)) = V(a)\dlog([\pi]_m)$, we get that these maps
  also induce
  \[
  \wt{M}^{m,q-1}_n(K) \xrightarrow{\dlog([\pi]_m)} \wt{M}^{m,q}_n(K).
\]
Since ${\theta}^q_2 = \dlog([\pi]_m) \circ {\theta}^{q-1}_1 = {\theta}^{q}_1 \circ
\dlog([\pi]_m)$, we are done by Lemmas~\ref{lem:Kato-fil-3}. 
\end{proof}

Combining the previous three lemmas, we get a homomorphism
\begin{equation}\label{eqn:Kato-fil-5}
  (\theta^q_1, \theta^q_2) \colon \frac{\Omega^{q}_{\ff}}{Z_1\Omega^{q}_{\ff}}  \bigoplus
  \frac{\Omega^{q-1}_{\ff}}{Z_1\Omega^{q-1}_{\ff}} \to
  \wt{M}^{m,q}_n(K) \ \ {\rm for} \ \ 1 \le m \le r.
  \end{equation}

\begin{lem}\label{lem:Kato-fil-6}
  We have the following.
  \begin{enumerate}
    \item
      The map $\phi^q \colon \Omega^q_{\ff} \to \wt{M}^{r+1,q}_n(K)$ is surjective.
    \item
       For $1 \le m \le r$, the map $(\theta^q_1, \theta^q_2)$
       in ~\eqref{eqn:Kato-fil-5} is surjective.
       \end{enumerate}
\end{lem}
\begin{proof}
  The map $\frac{\Fil_nW_{m}(K)}{\Fil_{n-1}W_{m}(K)} \otimes K^M_q(K) \to
  \wt{T}^{m,q}_n(K)$ is surjective by \corref{cor:VR-5} and
  $\Fil_nW_{m}(K) = [\pi^{-\alpha}]_{m}W_{m}(A)$ by \lemref{lem:Log-fil-4}(9),
  where $\alpha = np^{1-m}$ if $1 \le m \le r$ and $\alpha = l$ if $m = r+1$.
  Using Lemmas~\ref{lem:VR-3} and ~\ref{lem:VR-4}, it only remains to show that
  we can replace $K^M_q(K)$ by $K^M_q(A)$ on the left hand side of this surjection
  in order to prove (1).
  To show this, it is enough to consider the case $q =1$ and to show that
  $[x\pi^{-l}]_{r+1}\dlog([\pi]_{r+1})$ lies in the image of $A \otimes A^\times$
   under $\phi^1$ if $x \in A^\times$. But $[x\pi^{-l}]_{r+1}\dlog([\pi]_{r+1}) =
  -l^{-1}[x]_{r+1}d([\pi^{-l}]_{r+1})$. On the other hand,
  $[x]_{r+1}d([\pi^{-l}]_{r+1}) = d([x\pi^{-l}]_{r+1}) - [x\pi^{-l}]_{r+1}\dlog([x]_{r+1})$.
  It follows that $\partial^{r+1,1}_n([x\pi^{-l}]_{r+1}\dlog([\pi]_{r+1})) =
  \phi^1(l^{-1}x \otimes x)$.

  To prove (2), we can repeat the above steps which reduces finally to showing that
  $\partial^{m,1}_n([x\pi^{-\alpha}]_m\dlog([y]_m))$ lies in the image of
  $(\theta^1_1, \theta^1_2)$, where $y \in K^\times$. To show this, we write $y = u\pi^{\beta}$ for
  some $u \in A^\times$ and $\beta \in \Z$. This yields
  $[x\pi^{-\alpha}]_m\dlog([y]_m) = [x\pi^{-\alpha}]_m\dlog([u]_m) +
  \beta[x\pi^{-\alpha}]_m\dlog([\pi]_m) = \theta^1_1(x \otimes u) + \theta^1_2(\beta x)$.
  This finishes the proof.
\end{proof}

\begin{lem}\label{lem:Kato-fil-7}
  For $m \ge 1$, the map
  $\ov{\delta}^{m,q}_n \colon \wt{M}^{m,q}_n(K) \to \wt{N}^{m,q}_n(K)$ is bijective.
  \end{lem}
\begin{proof}
When $m \ge r+2$, we are done by \lemref{lem:Kato-fil-1} and \cite[Lem~3.6.1]{Kato-89} as both groups are zero. Otherwise, we
look at the maps
  \[
  \Omega^q_{\ff} \xrightarrow{\phi^q} \wt{M}^{r+1,q}_n(K) \xrightarrow{\ov{\delta}^{r+1,q}_n}
  \wt{N}^{r+1,q}_n(K)
  \]
  and (when $1 \le m \le r$)
  \[
  \frac{\Omega^{q}_{\ff}}{Z_1\Omega^{q}_{\ff}}  \bigoplus
  \frac{\Omega^{q-1}_{\ff}}{Z_1\Omega^{q-1}_{\ff}} \xrightarrow{\theta^q_1, \theta^q_2}
  \wt{M}^{m,q}_n(K) \xrightarrow{\ov{\delta}^{m,q}_n} \wt{N}^{m,q}_n(K).
  \]
  Kato shows in \cite[\S~3]{Kato-89} (see the last step in proof of Theorem~3.2
  on p.~114) that the above two composite maps are bijective. 
  We now apply \lemref{lem:Kato-fil-6} to finish the proof.
\end{proof}

\begin{lem}\label{lem:Kato-fil-8}
  For $m \ge 1$, the square ~\eqref{eqn:Kato-6} is Cartesian.
\end{lem}
\begin{proof}
  Since $T^{m,q}_{n-1}(K) \to T^{m,q}_n(K)$ is injective, as we observed following
  the diagram ~\eqref{eqn:Kato-6}, it is sufficient to show that the nap
  $\ov{\delta}^{m,q}_n \colon \wt{T}^{m,q}_n(K) \to \wt{S}^{m,q}_n(K)$ is injective.
  We shall prove this by induction on $m$.
  When $m =1$, the map $\wt{T}^{m,q}_n(K) \to \wt{M}^{m,q}_n(K)$ is bijective and
  we are done by \lemref{lem:Kato-fil-7}. We now assume $m \ge 2$ and
  look at the commutative diagram
  \begin{equation}\label{eqn:Kato-fil-8-0}
    \xymatrix@C1pc{
      & \wt{T}^{m-1,q}_n(K) \ar[r]^-{V} \ar[d]_-{\ov{\delta}^{m-1,q}_n} &
      \wt{T}^{m,q}_n(K) \ar[r] \ar[d]^-{\ov{\delta}^{m,q}_n} &  \wt{M}^{m,q}_n(K)
      \ar[d]^-{\ov{\delta}^{m,q}_n} \ar[r] & 0 \\
      0 \ar[r] & \wt{S}^{m-1,q}_n(K) \ar[r] &
      \wt{S}^{m,q}_n(K) \ar[r] &  \wt{N}^{m,q}_n(K) \ar[r] & 0.}
    \end{equation}

  It is clear from the definition of the groups ${S}^{m,q}_n(K)$ that
  $S^{m-1,q}_n(K) \bigcap S^{m,q}_{n-1}(K) = S^{m-1,q}_{n-1}(K)$, and this implies that
  the bottom row of ~\eqref{eqn:Kato-fil-8-0} is exact.
  The left vertical arrow in this diagram is injective by induction and the
  right vertical arrow is bijective by \lemref{lem:Kato-fil-7}.
  It follows that the middle vertical arrow is injective.
 \end{proof}

We can now finally prove the following.

\begin{thm}\label{thm:Kato-fil-10}
  For $m \ge 1$ and $n,q \ge 0$, the map
  $$\ov{\delta}^{m,q}_n \colon T^{m,q}_n(K) \to \ _{p^m}\Fil^{\bk}_n H^{q+1}(K)$$ is an
  isomorphism, where $T^{m,q}_n(K)=\H^1_\et(\Spec A, W_m\sF^{q,\bullet}_n)$.
\end{thm}
\begin{proof}
  By \corref{cor:VR-5}, we only need to show that $\ov{\delta}^{m,q}_n$ is surjective.
  For this, we let $w \in \ _{p^m}\Fil^{\bk}_n H^{q+1}(K) = S^{m,q}_n(K)$.
  By \lemref{lem:Kato-fil-9}, we see that $w \in H^1_\et(K, W_m\Omega^q_{K, \log})$.
  Since $\delta^{m,q} \colon W_m\Omega^q_K \surj H^1_\et(K, W_m\Omega^q_{K, \log})$
  and by \lemref{lem:Log-fil-4}(3)
  it follows that $  {\varinjlim}_n T^{m,q}_n(K)$ $= H^1_\et(K, W_m\Omega^q_{K, \log})$.
  In particular, $w \in T^{m,q}_{n'}(K) \cap S^{m,q}_n(K)$ for all $n' \gg 0$.
  We conclude from \lemref{lem:Kato-fil-8} that $w \in  T^{m,q}_{n}(K)$.
\end{proof}

\section{Comparison with Global Kato filtration}\label{sec:the group}
Let $X$ and $W_m\sF^{q,\bullet}_{D}$ be as in \defref{defn:1-c complex}. We also recall all the notations from there. Furthermore, assume $D \ge 0$ (i.e. $n_i\ge 0$). In this section we shall describe the group $\H^1_\et(X, W_m\sF^{q,\bullet}_{D})$ in terms of local Kato filtration studied in \S~\ref{sec:loc-Kato}. 
 Recall the complex $$W_m\sF^{q,\bullet}_{D} =\left[Z_1\Fil_{D} W_m\Omega^q_U \xrightarrow{1 -C}
\Fil_{D}W_m\Omega^q_U\right].$$
We note that $j^*W_m\sF^{q,\bullet}_{D} = W_m\Omega^q_{U,\log}$.

\subsection{Global Kato filtration}
For $x_i \in E^{(0)}$, let $K_i=Q(\sO_{X,x_i}^h)$ and $j_i: H^1_\et(U,W_m\Omega^q_{U,\log}) \to H^1_\et(K_i,W_m\Omega^q_{K_i,\log})$ be the canonical restriction map. \lemref{lem:Kato-fil-9} implies $H^1_\et(K_i,W_m\Omega^q_{K_i,\log})= \ _{p^m}(H^{q+1}(K_i))$. 

We define the global Kato filtration of $H^{q+1}_m(U)=H^1_\et(U,W_m\Omega^q_{U,\log})$ as follows:
\begin{defn}(cf. \cite[Defn~2.7]{Kerz-Saito-ANT})\label{defn:Log-fil-D} 
For any $q \ge 0, m \ge 1$, we let    $$\Fil^\log_DH^1_\et(U,W_m\Omega^q_{U,\log}):= ker \left(H^1_\et(U,W_m\Omega^q_{U,\log}) \xrightarrow{\bigoplus j_i} \bigoplus \limits_{x_i \in E^{(0)}} \frac{H^1_\et(K_i,W_m\Omega^q_{K_i,\log})}{_{p^m}(\Fil^{\bk}_{n_i}H^{q+1}(K_i))}\right).$$
\end{defn}

 For $q=0$, we have 
\begin{equation}\label{eqn:fundamental}
    \pi_1^{ab}(X,D)/p^m=\Hom_\Ab(\Fil^\log_DH^1_\et(U,\Z/p^m), \Q/\Z),
\end{equation}
where $\pi_1^{ab}(X,D)$ is the (logarithmic) abelianized {\et}ale fundamental group with modulus $D$, defined as $\Hom_\Ab(\Fil^\log_DH^1_\et(U,\Q/\Z), \Q/\Z)$ (cf. \cite[Def~2.7]{Kerz-Saito-ANT}). 
Recall from \cite[Lem~3.6]{Kerz-Saito-ANT}, that we have an isomorphism $$\H^{1}_\et(X, W_m\sF^{q,\bullet}_{D}) \cong \Fil^\log_DH^1_\et(U, W_m\Omega^q_{U,\log}), \text{ if }q=0.$$

Our goal in this section is to extend the above isomorphism to arbitrary $q \ge 0$. This will be one of the key results for proving the duality theorem for Hodge-Witt cohomology with modulus over local fields in \S~\ref{chap:duality-localfields} and the Lefschetz theorem for Brauer groups in \S~\ref{chap:Lef}. 

We begin with the following lemma.

\begin{lem}\label{lem:Z-1-fil}
    Let $k$ be an $F$-finite field of characteristic $p>0$. Let $R=k[[Y_1,Y_2]]$, $\pi=Y_1^nY_2^m, n, m \in \Z$ and $K=R_{Y_1Y_2}$ be localisation of $R$ by the element $Y_1Y_2$. Let's write (cf. \eqref{eqn:Multi-0})
    $$F^{1,q}_0=F^{1,q}_0(R)=\Omega^q_k \oplus  \Omega^{q-1}_k \dlog Y_1 \oplus  \Omega^{q-1}_k \dlog Y_2 \oplus  \Omega^{q-2}_k \dlog Y_1 \dlog Y_2$$ and  $$Z_1F^{1,q}_0=Z_1\Omega^q_k \oplus  (Z_1\Omega^{q-1}_k) \dlog Y_1 \oplus  (Z_1\Omega^{q-1}_k) \dlog Y_2 \oplus ( Z_1\Omega^{q-2}_k) \dlog Y_1 \dlog Y_2$$
    Then $Z_1\Fil_{(n,m)}\Omega^q_K$ has the following unique presentation.
    $$Z_1\Fil_{(n,m)}\Omega^q_K= \sum\limits_{\substack{ip \ge -n \\ jp \ge -m}}Y_1^{ip}Y_2^{jp} \left( Z_1F^{1,q}_0 \right)  + \sum\limits_{\substack{i \ge -n \\  j  \ge -m \\ p \nmid i \text{ or } p \nmid j}} d(Y_1^iY_2^j F^{1,q-1}_0). $$ 
    Moreover, one has
    \begin{enumerate}
        \item $(\Fil_{(n,m)}\Omega^q_K)_{Y_1Y_2}= \Omega^q_K$ and $(Z_1\Fil_{(n,m)}\Omega^q_K)_{Y_1Y_2}= Z_1\Omega^q_K$, where we consider $Z_1\Omega^q_K$ and $Z_1\Fil_{(n,m)}\Omega^q_K$ as $R$-modules via Frobenius action.
        
        \item Every element $\omega \in (\Fil_{(n,m)}\Omega^q_K)_{Y_1}$ can be written uniquely as
        $$\omega = \sum\limits_{\substack{i\ge -N_0 \\ j\ge -m}} Y_1^iY_2^j a_{i,j} , \text{  for some $N_0 >0$ and } a_{i,j} \in F^{1,q}_0.$$

        \item Every element $\omega \in (\Fil_{(n,m)}\Omega^q_K)_{Y_2}$ can be written uniquely as
        $$\omega = \sum\limits_{\substack{i\ge -n \\ j\ge -M_0}} Y_1^iY_2^j b_{i,j} , \text{  for some $M_0 >0$ and } b_{i,j} \in F^{1,q}_0.$$

        \item Every element $\omega \in (Z_1\Fil_{(n,m)}\Omega^q_K)_{Y_1}$ can be written uniquely as
        $$\omega = \sum\limits_{\substack{i\ge -N'_0 \\ jp\ge -m}} Y_1^{ip}Y_2^{jp} e_{i,j} + \sum\limits_{\substack{i\ge -N'_0 \\ j\ge -m \\ p\nmid i \text{ or } p \nmid j}}d(Y_1^iY_2^jf_{i,j}),$$
        for some $N'_0 >0$ and  $e_{i,j} \in Z_1F^{1,q}_0, \  f_{i,j} \in F^{1,q-1}_0.$

        \item Every element $\omega \in (Z_1\Fil_{(n,m)}\Omega^q_K)_{Y_2}$ can be written uniquely as
        $$\omega = \sum\limits_{\substack{ip\ge -n \\ j\ge -M'_0}} Y_1^{ip}Y_2^{jp} e'_{i,j} + \sum\limits_{\substack{i\ge -n \\ j\ge -M'_0 \\ p\nmid i \text{ or } p \nmid j}}d(Y_1^iY_2^jf'_{i,j}),$$
        for some $M'_0 >0$ and  $e'_{i,j} \in Z_1F^{1,q}_0, \  f'_{i,j} \in F^{1,q-1}_0.$
    \end{enumerate}
    
\end{lem}
\begin{proof}
    The first part of the lemma follows from \corref{cor:F-function-2} and the Cartier isomorphism in \lemref{lem:Complete-4}(2) for $m=1$. The rest of the part is an easy consequence of the first part and \corref{cor:F-function-2},
\end{proof}

\begin{lem}\label{lem:1-C-inj}
    Let $R$ be as above. $\fm =(Y_1,Y_2)$. Let $n \ge 0, m\ge 0.$ Consider the map
    $$1-C : Z_1\Fil_{(n,m)}\Omega^q_K \to \Fil_{(n,m)}\Omega^q_K.$$
    Then the induced map 
    $$(1-C)^* : H^2_{\fm}(R, Z_1\Fil_{(n,m)}\Omega^q_K) \to H^2_{\fm}(R, \Fil_{(n,m)}\Omega^q_K)$$
    is injective, where $H^*_\fm(R,-)$ is the Zariski (eqv. {\'e}tale) cohomology with support at the closed point $V(\fm)$ of $\Spec R$.
\end{lem}

\begin{proof}
            For any $R$-module $M$, note that 
        $$H^2_{\fm}(R, M)= H^1(\Spec R-\{\fm\},M). $$
        and thus, by using \v{c}ech cohomology, we have
        $$H^1(\Spec R-\{\fm\},M)= coker (M_{Y_1} \oplus M_{Y_2} \xrightarrow{(+ -)}M_{Y_1Y_2}).$$
        Now consider the following commutative diagram.
        \begin{equation}\label{eqn:1-C-diag}
            \xymatrix{
            (Z_1\Fil_{(n,m)}\Omega^q_K)_{Y_1} \oplus (Z_1\Fil_{(n,m)}\Omega^q_K)_{Y_2} \ar[r]^-{(+ -)}_-{\theta_1} \ar[d]_-{f\oplus g} & (Z_1\Fil_{(n,m)}\Omega^q_K)_{Y_1Y_2} \ar[d]_-{h}  \\
            (\Fil_{(n,m)}\Omega^q_K)_{Y_1} \oplus (\Fil_{(n,m)}\Omega^q_K)_{Y_2} \ar[r]^-{(+ -)}_-{\theta_2}& (\Fil_{(n,m)}\Omega^q_K)_{Y_1Y_2}, 
            }
        \end{equation}
where $f,g,h$ are the maps induced by $1-C$. 

To show $(1-C)^*$ is injective, we need to show that $h(\alpha) \in {\rm Image~} \theta_2$ implies $\alpha \in {\rm Image~} \theta_1$, for all $\alpha \in (Z_1\Fil_{(n,m)}\Omega^q_K)_{Y_1Y_2}$.

Let $\alpha \in (Z_1\Fil_{(n,m)}\Omega^q_K)_{Y_1Y_2} = Z_1\Omega^q_K$, such that $h(\alpha) \in {\rm Image~} \theta_2$. So by \lemref{lem:Z-1-fil}, we can write $h(\alpha)= a - b$,  where 
$$a= \sum\limits_{\substack{i \ge -N \\ j \ge -m}} Y_1^iY_2^j a_{i,j} \text{ and } b= \sum\limits_{\substack{i \ge -n \\ j \ge -M}} Y_1^iY_2^j b_{i,j}; \text{  \  $a_{i,j}, b_{i,j} \in F^{1,q}_0$}.$$
and $\alpha$ as 
\begin{equation}\label{eqn:sum-alpha-0}
    \alpha = \sum\limits_{\substack{ip \ge -N \\ jp \ge -M}}Y_1^{ip}Y_2^{jp}c_{i,j} +\sum\limits_{\substack{i \ge -N \\ j \ge -M \\ p \nmid i \ or \ p \nmid j}} d(Y_1^iY_2^jd_{i,j});  \text{   $c_{i,j} \in Z_1F^{1,q}_0$ and $d_{i,j} \in F^{1,q-1}_0$}
\end{equation}
for some $N >>0,M>>0$.

Note that $d(Y_1^iY_2^jd_{i,j})= Y_1^iY_2^j(i  d_{i,j}\cdot\dlog Y_1 +jd_{i,j}\cdot \dlog Y_2+d(d_{i,j}))$. Hence, $d(Y_1^iY_2^jd_{i,j})$ can be written as $Y_1^iY_2^jd''_{i,j}$, where $d''_{i,j}=(i  d_{i,j}\cdot\dlog Y_1 +jd_{i,j}\cdot \dlog Y_2+d(d_{i,j})) \in F^{1,q}_0$. 

So, we can rearrange the sum in \eqref{eqn:sum-alpha-0} and write it as 
\begin{equation}
    \alpha = \sum\limits_{\substack{i \ge -n \\ j \ge -M}} Y_1^iY_2^jc'_{i,j} + \sum\limits_{\substack{j \ge -m \\ i \ge -N}}Y_1^iY_2^jd'_{i,j} + \sum\limits_{\substack{-N < i < -n \\ -M < j < -m}}Y_1^iY_2^je'_{i,j},
\end{equation}
where $c'_{i,j}, d'_{i,j}, e'_{i,j} \in  F^{1,q}_0$ and each term of the above sum lies in $Z_1\Omega^q_K$. 

Hence, we conclude that $\sum\limits_{\substack{i \ge -n \\ j > -M}}Y_1^iY_2^jc'_{i,j}$ (resp. $\sum\limits_{\substack{j \ge -m \\ i > -N}}Y_1^iY_2^jd'_{i,j}$) lies in  $(Z_1\Fil_{(n,m)}\Omega^q_K)_{Y_2}$ (resp. $(Z_1\Fil_{(n,m)}\Omega^q_K)_{Y_1}$) (cf. \lemref{lem:Z-1-fil}(4),(5)).
It follows that 
\begin{equation}\label{eqn:sum-alpha}
    \tau = \sum\limits_{\substack{i \ge -n \\ j > -M}}Y_1^iY_2^jc'_{i,j} + \sum\limits_{\substack{j \ge -m \\ i > -N}}Y_1^iY_2^jd'_{i,j} \in {\rm Image~} \theta_1.
\end{equation}
Let's write $c= \sum\limits_{\substack{-N < i < -n \\ -M < j < -m}}Y_1^iY_2^je'_{i,j}$.

\textbf{Claim}: $c=0.$

Since $h(\alpha) \in {\rm Image~} \theta_2$, it follows from (\ref{eqn:sum-alpha}) and from commutativity of (\ref{eqn:1-C-diag}) that $h(c)= h(\alpha) - h(\tau) \in {\rm Image~} \theta_2 $. We shall show that this implies $c=0$.

Indeed, if $-N_0$ be the least exponent of $Y_1$ appearing in $c$ , then write $c$ as $$c= \sum\limits_{j<-m}Y_1^{-N_0}Y_2^je'_{-N_0,j} +  \sum\limits_{\substack{-N_0 <i < -n \\ j < -m}}Y_1^iY_2^je'_{i,j}.$$
Now from the expression of $\alpha$ in \eqref{eqn:sum-alpha-0}, we note that, $Y_1^iY_2^je'_{i,j}$ is either of the form $d(Y_1^iY_2^j\ov e_{i,j})$ (if $p\nmid i$ or $p \nmid j$), or it is of the form $Y_1^{pi_1}Y_2^{pi_2} e'_{i,j}$ (in case $p$ divides both $i, j$), where $\ov e_{i,j} \in F^{1,q-1}_0$. So
\begin{equation}
    C(Y_1^iY_2^je'_{i,j})  = \left\{\begin{array}{cl}
   0 ,& \text{if } p\nmid i \text{ or } p\nmid j  \\
       Y_1^{i/p}Y_2^{j/p}C(e'_{i,j}) , & \text{otherwise.}
\end{array}\right.
\end{equation}

That is, $C(Y_1^iY_2^je'_{i,j})$ is either zero or the exponent of $Y_1$ in $C(Y_1^iY_2^je'_{i,j})$ is strictly greater than $i$ (because $i<0$). 
This implies $(1-C)(c)$ can be written as 
$$(1-C)(c)=\sum\limits_{j<-m}Y_1^{-N_0}Y_2^je'_{-N_0,j} + \sum\limits_{\substack{i>-N_0 \\ j}}Y_1^iY_2^je^{''}_{i,j},$$
where $e'_{i,j}, e''_{i,j} \in F^{1,q}_0$.
Since $(1-C)(c) \in {\rm Image~} \theta_2$, it implies that every term of the above sum has the property that either the exponent of $Y_1$ is at least $-n$ or the exponent of $Y_2$ is at least $-m$. Hence by the uniqueness of \lemref{lem:Z-1-fil}, we conclude
$$Y_1^{-N_0}Y_2^je'_{-N_0,j}=0, \ \text{for all } j<-m.$$
Thus, $-N_0$ is not the least exponent of $Y_1$ appearing in $c$, which is a contradiction. Hence we get the claim. 

Since the claim is proved, we get $\alpha \in {\rm Image~} \theta_1$.
\end{proof}

Recall, for any bounded complex of abelian sheaves $\sF$ on $X_\et$, we have the coniveau spectral sequence 
        \begin{equation}\label{eqn:spectral}
        E_1^{p,q}=\bigoplus\limits_{x \in X^{(p)} \cap E}\H^{p+q}_x(X,\sF) \implies \H^{p+q}_E(X,\sF),
        \end{equation}
where $E^{(p)}=$ set of codimension $p$ points of $E$.

This spectral sequence will be used to prove the next two lemmas. 
\begin{lem}\label{lem:H^1-fil-1}
        $\H^1_E(X, W_m\sF^{q,\bullet}_D) =0$.
\end{lem}
\begin{proof}
Let $\sF_m= W_m\sF^{q,\bullet}_D$, for $D \ge 0$.

\textbf{Claim 1}: For $x \in X^{(1)}\cap E$, we have $\H^1_x(X, \sF_m)=0$.
     
    {\sl Proof :}  We let $A_x=\sO_{X,x}^h, K_x=Q(A_x)$ and $\pi_x$ be a uniformizer of $A_x$. Since $\H^1_x(X, \sF_m)= \H^1_x(A_x,\sF_m)$, we have the following localisation exact sequence
    \begin{equation}\label{eqn:Coh-V0-2}
         \cdots \to \H^0_\et(A_x,\sF_m) \xrightarrow{f} \H^0_\et(K_x,\sF_m) \to \H^1_x(A_x,\sF_m) \to \H^1_\et(A_x,\sF_m) \xrightarrow{g} 
    \H^1_\et(K_x,\sF_m).
    \end{equation}
    By \lemref{lem:hyp}(2), we see that $f$ is an isomorphism in \eqref{eqn:Coh-V0-2}.
    Also $g$ is injective by \propref{prop:Cartier-fil-1}, because $\H^1_\et(A_x,\sF_m) =\coker(1-C: Z_1\Fil_{D_x}W_m\Omega^q_{K_x} \to \Fil_{D_x}W_m\Omega^q_{K_x} )$ (by \lemref{lem:hyp}(2)). Hence we get $\H^1_x(A_x,\sF_m)=0$. This proves Claim 1.

    Using Claim 1, we get $E^{1,0}_\infty=0$. Also note that, $E^{0,1}_\infty=0$. Hence the lemma follows.
   \end{proof}

\begin{lem}\label{lem:H^1-fil-2}
        Let $\sF_m$ be as in \lemref{lem:H^1-fil-1}. Then the canonical map $\H^2_E(X,\sF_m) \to \bigoplus\limits_{x \in X^{(1)}\cap E} \H^2_{x}(X, \sF_m)$ is injective.
\end{lem}
\begin{proof} We make the following claim.

       \textbf{Claim 2} : Let $x \in X^{(2)}\cap E$. Then $\H^2_x(X,\sF_m)=0$.

        {\sl Proof :} By \thmref{thm:Global-version}(11), it is enough to prove the claim for $m=1$ and for all $D \ge 0$. Since $\H^2_x(X,\sF_1)=\H^2_x(R_x,\sF_1)$, where $R_x=\sO_{X,x}^h$; by hypercohomology long exact sequence we have 
        $$0=H^1_x(R_x, \Fil_D\Omega^q_{U_x}) \to \H^2_x(R_x,\sF_1) \to H^2_x(R_x, Z_1\Fil_D\Omega^q_{U_x}) \xrightarrow{(1-C)^*} H^2_x(R_x, \Fil_D\Omega^q_{U_x}), $$
        where $U_x= \Spec R_x -E_x$, where $E_x$ is pull back of $E$ in $\Spec R_x$.

        Here $H^1_x(R_x, \Fil_D\Omega^q_{U_x})=0$ (\cite[Prop~3.5.4.b]{Bruns-Herzog}) because codimension of $\{x\}$ in $X$ is 2 and $\Fil_D\Omega^q_{U_x}$ is free $R_x$-module. So, to prove the claim, it is enough to show that $(1-C)^*$ is injective. 
        
        Let $\fm$ be the maximal ideal of $R_x$, $\wh R_x$ be its $\fm$ adic completion and $\wh{U}_x=$ be the inverse image of $U_x$ to $\Spec \wh{R}_x$.
        Since $ Z_1\Fil_D\Omega^q_{U_x}$ and $\Fil_D\Omega^q_{U_x}$ are coherent sheaves on $\Spec R_x$, and by \corref{cor:Non-complete-3} and Lemmas \ref{lem:Non-complete-3}, \ref{lem:Complete-4} their $\fm$-adic completions are $ Z_1\Fil_D\Omega^q_{\wh{U}_x}$ and $\Fil_D\Omega^q_{\wh{U}_x}$, respectively; \cite[Prop~3.5.4.(d)]{Bruns-Herzog} implies $H^2_x(R_x, M)\cong H^2_x(\wh{R}_x, \wh M)$ for $M \in \{\Fil_D\Omega^q_{U_x},  Z_1\Fil_D\Omega^q_{U_x}\}$.
        Moreover, we have a commutative diagram
        \begin{equation}
            \xymatrix{
             H^2_x(R_x, Z_1\Fil_D\Omega^q_{U_x}) \ar[r]^{(1-C)^*}\ar[d]^{\cong}    &  H^2_x(R_x, Z_1\Fil_D\Omega^q_{U_x}) \ar[d]^\cong\\
             H^2_x(\wh R_x, Z_1\Fil_D\Omega^q_{\wh U_x}) \ar[r]^{(1-C)^*}    & H^2_x(\wh R_x, Z_1\Fil_D\Omega^q_{\wh U_x}).
            }
        \end{equation}
        
        Hence, it is enough to show the bottom arrow is injective. By Cohen's structure theorem, we can write $\wh R_x=k[[X,Y]]$, where $k$ is the residue field of $R_x$; $D$ will be of the form $V(\pi)$, where $\pi=X^nY^m$, where $n \ge 0$ and $m\ge 0$.
        Now the claim follows from \lemref{lem:1-C-inj}.

        Using claim 3, we get $E^{2,0}_\infty=0$ in \eqref{eqn:spectral}. Also, $E^{0,2}_\infty=0$. Hence 
        $$\H^2_E(X, \sF_m)= E^{1,1}_\infty \inj E^{1,1}_1 = \bigoplus\limits_{x \in X^{(1)}\cap E} \H^2_{x}(X, \sF_m).$$
        This finishes the proof.
    \end{proof}

Now we prove the main theorem of this chapter.

\begin{thm}\label{thm:H^1-fil}
    We have a canonical isomporphism
    $$\H^1_\et(X, W_m\sF^{q,\bullet}_D) \cong \Fil^\log_DH^1_\et(U,W_m\Omega^q_{U,\log}).$$
\end{thm}

\begin{proof}
    We have the following localisation exact sequence 
    $$\H^1_E(X, W_m\sF^{q,\bullet}_D) \to \H^1_\et( X,W_m\sF^{q,\bullet}_D) \to H^1_\et(U, W_m\Omega^q_{U,\log}) \to \H^2_E(X, W_m\sF^{q,\bullet}_D).$$
By \lemref{lem:H^1-fil-1}, we have $\H^1_E(X, W_m\sF^{q,\bullet}_D)=0$. Now we consider the following commutative diagram of localisation exact sequences 
\begin{equation}\label{eqn:loc-2}
    \xymatrix@C.8pc{
    0\ar[r]& \H^1_\et( X,W_m\sF^{q,\bullet}_D) \ar[r] \ar[d]& H^1_\et(U, W_m\Omega^q_{U,\log}) \ar[r] \ar[d] & \H^2_E(X, W_m\sF^{q,\bullet}_D) \ar[d]\ar[r]& \\
    \ar[r]&\H^1_\et(A_x,W_m\sF^{q,\bullet}_{D_x}) \ar[r]^-{1} & H^1_\et(K_x,W_m\Omega^q_{K_x,\log}) \ar[r]^-{2}& \H^2_x(A_x,W_m\sF^{q,\bullet}_{D_x}) \ar[r]&, 
    }
\end{equation}
where $A_x=\sO_{X,x}^h, K_x=Q(A_x), x \in E^{(0)}$, $D_x$ is the pull back of $D$ in $\Spec A_x$. The vertical arrows are restriction maps induced by the morphism $\Spec A_x \to X$.  Moreover, the map (2) in the bottom row is surjective. This follows because the next term $\H^2(A_x,W_m\sF^{q,\bullet}_{D_x})=0$, by looking at the hypercohomology long exact sequence for $W_m\sF^{q,\bullet}_{D_x}$ and noting that $H^1_\et(A_x, \Fil_{n_x}W_m\Omega^q_{K_x}) = 0 =H^2_\et(A_x, Z_1\Fil_{n_x}W_m\Omega^q_{K_x})$ (cf. \lemref{lem:hyp}). 

Also, the map (1) in the bottom row of \eqref{eqn:loc-2} is injective by Lemmas \ref{lem:Kato-fil-9}, \ref{lem:hyp}(2) and \corref{cor:VR-5}. Furthermore, we have $\H^1_\et(A_x,W_m\sF^{q,\bullet}_{D_x}) \cong \ _{p^m}\Fil^{\bk}_{n_x}H^{q+1}(K_x)$ by \thmref{thm:Kato-4}.

As a result, composite of the morphisms $H^1_\et(U, W_m\Omega^q_{U,\log}) \to \H^2_E(X, W_m\sF^{q,\bullet}_D) \to \H^2_x(A_x,W_m\sF^{q,\bullet}_{D_x})$ coincides with the map $H^1_\et(U, W_m\Omega^q_{U,\log}) \to \frac{H^1(K_x,W_m\Omega^q_{K_x,\log})}{_{p^m}\Fil^{\bk}_{n_x}H^{q+1}(K_x)}$, induced by canonical restriction. Hence, we have 
\begin{equation*}
    \begin{array}{cl}
  \H^1_\et( X,W_m\sF^{q,\bullet}_D)&\cong \Ker \left(H^1_\et(U, W_m\Omega^q_{U,\log}) \to \H^2_E(X, W_m\sF^{q,\bullet}_D)\right)    \\
     &\cong  \Ker \left(H^1_\et(U, W_m\Omega^q_{U,\log}) \to \bigoplus\limits_{x \in X^{(1)}\cap E} \H^2_{x}(X, W_m\sF^{q,\bullet}_D)\right) \\
     &= \Ker \left(H^1_\et(U, W_m\Omega^q_{U,\log}) \to \bigoplus\limits_{x \in X^{(1)}\cap E} \frac{H^1(K_x,W_m\Omega^q_{K_x,\log})}{_{p^m}\Fil^{\bk}_{n_x}H^{q+1}(K_x)}\right),
\end{array}
\end{equation*}
where, the second isomorphism follows from \lemref{lem:H^1-fil-2}. This finishes the proof of the theorem.

\end{proof}

\subsection{Relation with Brauer group with modulus} For $q=1$, we shall establish a relationship between the group $\H^1_\et(X, W_m\sF^{q,\bullet}_{D})$ and the group $\Br^\divv(X|D)$ defined in ~\cite[Def~8.7]{KRS}.

\begin{defn}\label{defn:BGM}
     We let $\Br^\divv(X|D)$ denote the subgroup of $\Br(U)$ consisting of elements $\chi$
  such that for every $x \in E^{(0)}$, the image $\chi_{x}$ of
$\chi$ under the canonical map $\Br(U) \to\Br((K_{x}))$ 
lies in $\Fil_{n_{x}} \Br(K_{x})$. We call $\Br^\divv(X|\un{0})$ to be the tame Brauer group of $U$, where $\un{0}$ is the zero divisor. 
\end{defn}
\begin{rem}
    Our definition is slightly different from the one defined in \cite{KRS}. Here we have shifted the filtration index by $1$.
\end{rem}
We note the following observations.
\begin{lem}\label{lem:ell-torsion}
    We have
    \begin{enumerate}
        \item $\Br^\divv(X|\un{0})$ contains the prime to $p$-part of $\Br(U)$.
        \item $\Br^\divv(X|D) \subset \Br^\divv(X|D')$ if $D \le D'$; $\varinjlim\limits_{|D|\subset E}\Br^\divv(X|D)= \Br(U)$.
        \item $_{\ell^m}\Br^\divv(X|D)= \ _{\ell^m}\Br(U)$ for all $D \ge 0$, $m \ge 1$ and any prime $\ell \neq p$.
    \end{enumerate}
\end{lem}
\begin{proof}
    (1) follows from the fact that any prime to $p$-torsion element of $\Br(U)$ maps to $\Fil_{0}\Br(K_x)$ for all $x \in E^{(0)}$ (see \cite[Prop~6.1]{Kato-89}). (2) easily follows from the definition. (3) follows from (1) and (2). 
\end{proof}

We conclude the chapter with the following corollary.
\begin{cor}\label{cor:F^q}
    We have an exact sequence
    \begin{equation}
        0 \to \Pic(U)/p^m \to \H^1_\et(X, W_m\sF^{1,\bullet}_{D}) \to \ _{p^m}\Br^\divv(X|D) \to 0.
    \end{equation}
\end{cor}

\begin{proof}
    Recall that we have an isomorphism of {\'e}tale sheaves $\dlog : \G_{m,U}/p^m \to W_m\Omega^1_{U,\log}$ on $U_\et$.
    We consider the following exact sequence of {\'e}tale sheaves on $U_\et$.
    $$0 \to \G_{m,U} \xrightarrow{p^m}\G_{m,U}\xrightarrow{\dlog} W_m\Omega^1_{U,\log}\to 0.$$
    Taking cohomology and identifying $H^1_\et(U, \G_{m,U})$ (resp. $H^2_\et(U,\G_{m,U}$)) with $\Pic(U)$ (resp. $\Br(U)$), we get the following commutative diagram of exact sequences.
    \begin{equation}
        \xymatrix@C.5pc{
        0 \ar[r] & \Pic(U)/p^m \ar[r] \ar[d] & H^1_\et(U,W_m\Omega^1_{U,\log}) \ar[r] \ar[d] & \ _{p^m}\Br(U) \ar[r] \ar[d]& 0 \\
        0 \ar[r] & 0\ar[r] & \bigoplus\limits_{x \in E^{(0)}}\frac{H^1_\et(K_x, W_m\Omega^1_{K_x,\log})}{\ _{p^m}\Fil^{\bk}_{n_x}H^{q+1}(K_x)} \ar[r]^-{\cong} & \bigoplus\limits_{x \in E^{(0)}}\frac{\ _{p^m}\Br(K_x)}{\ _{p^m}\Fil_{n_x}\Br(K_x)} \ar[r] & 0,
        }
    \end{equation}
where the vertical maps are induced by canonical restrictions. Now the result follows using snake lemma, \thmref{thm:H^1-fil} and noting that $\frac{\ _{p^m}\Br(K_x)}{\ _{p^m}\Fil_{n_x}\Br(K_x)} \subset \ _{p^m}\left(\frac{\Br(K_x)}{\Fil_{n_x}\Br(K_x)}\right)$.
    
\end{proof}

\chapter{Duality theorem for Hodge-Witt cohomology over finite fields}\label{chap:Duality}
In this chapter, we shall prove a duality theorem for the cohomology of the logarithmic
Hodge-Witt sheaves with modulus over finite fields. We begin by defining the
relevant complexes of sheaves, proving some of their properties and
constructing their pairings over a more
general field before we restrict to finite fields.

\section{Definition of complexes and their properties}\label{sec:Complexes}
Let $X$ be a Noetherian regular F-finite $\F_p$-scheme. Let $d$ denote the rank of the locally free sheaf
$\Omega^1_X$.
Let $E = E_1 + \cdots + E_r$ be a simple normal crossing divisor on $X$ and let
$j \colon U \inj X$ be the inclusion of the complement of $E$. For
$\un{n} = (n_1, \ldots , n_r) \in \Z^r$, we let
$D_{\un{n}} = \stackrel{r}{\underset{i =1}\sum} n_i[E_i] \in \Div_E(X)$. 
We shall write $\Fil_{D_{\un{n}}}W_m\Omega^q_U$ as $\Fil_{\un{n}}W_m\Omega^q_U$.
We shall say that $\un{n} \ge \n'$ if $n_i \ge n'_i$ for each $i$.
We shall have analogous interpretations for the notations
$\un{n} \le \n'$ and $\un{n} =\un{n'}$.
We let $t\un{n}:=(tn_1,\ldots,tn_r)$, $\un{t}:=(t,\ldots,t)\in \Z^r$ for $t \in \Z$, $\un{n} \pm \n' := (n_1\pm n'_1,
\ldots , n_r \pm n'_r)$ and $\n \pm t :=\n \pm \un{t}$. We also let $\n/p =(\lfloor n_1/p \rfloor,\ldots, \lfloor n_r/p \rfloor)$. We shall have analogous notations for $\n/p^i$, for $i \ge 1$.

We let $W_m\Omega^q_{X, \log}$ denote the image of the
map $\dlog \colon \sK^M_{q,X} \to W_m\Omega^q_X$ (cf. \S~\ref{chap:Kato-coh}).
Letting $\sK^{m}_{q,X} = {\sK^M_{q,X}}/{p^m}$, this induces an isomorphism of Zariski
sheaves $\dlog \colon \sK^m_{q,X} \xrightarrow{\cong} W_m\Omega^q_{X,\log}$.
For $\un{n} \ge \un{0}$, we let $\sI_{\un{n}} = \Ker(\sO_X \surj \sO_{D_{\un{n}}})$ and
let $\sK^M_{q, X|D_{\un{n}}}$ be the {\'etale} (or Nisnevich) sheaf on $X$ given by the image of the
canonical map $(1 + \sI_{\un{n}})^\times \otimes j_*\sO^\times_U \otimes \cdots \otimes
j_*\sO^\times_U \to j_*\sK^M_{q,U}$.
One knows that $\sK^M_{q, X|D_{\un{n}}} \subset \sK^M_{q,X}$ (cf. \cite[Prop.~1.1.3]{JSZ}).
We let
$\sK^m_{q, X|D_{\un{n}}} = {\sK^M_{q, X|D_{\un{n}}}}/{(p^m\sK^M_{q,X} \cap \sK^M_{q, X|D_{\un{n}}})}$.
Letting $W_m\Omega^q_{X|D_{\un{n}},\log}$ be the image of the map
$\dlog \colon \sK^M_{q, X|D_{\un{n}}} \to W_m\Omega^q_X$ for $\n \ge \un{1}$, one gets an isomorphism
$\dlog \colon \sK^m_{q, X|D_{\un{n}}} \xrightarrow{\cong} W_m\Omega^q_{X|D_{\un{n}},\log}$
(cf. \cite[Thm.~1.1.5]{JSZ}). One also has the following.

\begin{lem}\label{lem:Complex-5}
  $W_m\Omega^q_{X|D_{\un{n}},\log} \subset \Fil_{-\un{n}}W_m\Omega^q_U$, for $\n \ge \un{1}$.
\end{lem}
\begin{proof}
  Using the definition of $W_m\Omega^q_{X|D_{\un{n}},\log}$ and
  \lemref{lem:Log-fil-4}(8, 10), it suffices to show that 
  $\dlog([1+ at]_m) \in \Fil_{-\un{n}}W_m\Omega^1_K$
  in the notations of \S~\ref{sec:com-to-noncom}
  provided $A$ is regular local ring,
  $t = \stackrel{r}{\underset{i =1}\prod} x^{n_i}_i$ with each $n_i \ge 1$
  and $a \in A$. To this end, we use \cite[Lem.~1.2.3]{Geisser-Hesselholt-JAMS} to write
  $[1+ at]_m = (ta_{m-1}, \ldots , ta_0) + [1]_m$, where $a_i \in A$.
  This implies that $\dlog([1+ at]_m) = [1+at]^{-1}_m (d(x) + d([1]_m))$,
  where $x = (ta_{m-1}, \ldots , ta_0)$. Since $d([1]_m) \in W_m\Omega^1_{\F_p} = 0$
  and $x \in \Fil_{-\un{n}}W_m(K)$, we conclude from \lemref{lem:Log-fil-4}(4) that
  $\dlog([1+ at]_m) \in \Fil_{-\un{n}}W_m\Omega^1_K$.
\end{proof}

We now let  $\un{n} \ge -1, \ q \ge 0$ and $m \ge 1$, and 
consider the following complexes in $\sD(X_\et)$.
\begin{equation}\label{eqn:Complex-0}
  W_m\sF^{q, \bullet}_{\un{n}} = \left[Z_1\Fil_{\un{n}}W_m\Omega^q_U
    \xrightarrow{1 - C} \Fil_{\un{n}}W_m\Omega^q_U\right];
\end{equation}
\begin{equation}\label{eqn:Complex-1}
  W_m\sG^{q,\bullet}_{\un{n}} = \left[\Fil_{{-\un{n}-1}}W_m\Omega^q_U \xrightarrow{1-\ov{F}}
    \frac{\Fil_{-\un{n}-1}W_m\Omega^q_U}{dV^{m-1}(\Fil_{-\un{n}-1}\Omega^{q-1}_U)}\right];
\end{equation}
\begin{equation}\label{eqn:Complex-2}
  W_m\sF^{q,\bullet} = \left[Z_1W_m\Omega^q_X  \xrightarrow{1 - C} W_m\Omega^q_X\right];
  \ \ W_m\sG^{q,\bullet} = \left[W_m\Omega^q_X  \xrightarrow{1 - \ov{F}}
    \frac{W_m\Omega^q_X}{dV^{m-1}(\Omega^{q-1}_X)}\right];
  \end{equation}
  \begin{equation}\label{eqn:Complex-3}
  W_m\sH^{\bullet} = \left[W_m\Omega^d_X \xrightarrow{1-C} W_m\Omega^d_X\right]
  \cong W_m\Omega^d_{X,\log}, \ \ (\text{Note }Z_1W_m\Omega^d_X=W_m\Omega^d_X).
  \end{equation}
These complexes are defined in view of \thmref{thm:Global-version}.
The following lemma explains these complexes as objects of $\sD(X_\et)$
for certain values of $\un{n}$.

\begin{lem}\label{lem:Complex-6}
  The sequences (cf. \lemref{lem:Complex-5})
  \begin{enumerate}
    \item
     \ \ $0 \to W_m\Omega^q_{X|D_{\un{n}},\log} \to \Fil_{-\un{n}}W_m\Omega^q_U
    \xrightarrow{1 -\ov{F}}
    \frac{\Fil_{-\un{n}}W_m\Omega^q_U}{dV^{m-1}(\Fil_{-\un{n}}\Omega^{q-1}_U)}$;
\item
   \ \ $0 \to W_m\Omega^q_{X|D_{\un{1}},\log} \to \Fil_{-\un{1}}W_m\Omega^q_U
    \xrightarrow{1 -\ov{F}}
    \frac{\Fil_{-\un{1}}W_m\Omega^q_U}{dV^{m-1}(\Fil_{-\un{1}}\Omega^{q-1}_U)} \to 0$;
     \item
  \ \ $0 \to W_m\Omega^q_{X|D_{\un{1}},\log} \to Z_1\Fil_{-\un{1}}W_m\Omega^q_U
    \xrightarrow{1 - C}
    \Fil_{-\un{1}}W_m\Omega^q_U \to 0$;
 \item
    \ \ $0 \to j_*(W_m\Omega^q_{U,\log}) \to Z_1\Fil_{\un{0}}W_m\Omega^q_U
    \xrightarrow{1 - C}
    \Fil_{\un{0}}W_m\Omega^q_U \to 0$;
  \item
\ \ $0 \to j_*(W_m\Omega^q_{U,\log}) \to \Fil_{\un{0}}W_m\Omega^q_U
    \xrightarrow{1 -\ov{F}}
    \frac{\Fil_{\un{0}}W_m\Omega^q_U}{dV^{m-1}(\Fil_{\un{0}}\Omega^{q-1}_U)} \to 0$;
  \item 
  \ \ $0 \to \Omega^q_{X|D_\n, \log} \to Z_1\Fil_{-\n}\Omega^q_U \xrightarrow{1-C}\Fil_{(-\n)/p}\Omega^q_U \to 0$;
    \end{enumerate}
are exact, where $\un{n} \ge 1$ in (1) and (6).
\end{lem}
\begin{proof}
  It is clear that $(1- \ov{F})(W_m\Omega^q_{X|D_{\un{n}},\log}) = 0$.
  Using \cite[Lem.~2]{CTSS}, (1) will therefore follow if we show that
\begin{equation}\label{eqn:Complex-4-2}
  \Fil_{-\un{n}}W_m\Omega^q_U \bigcap W_m\Omega^q_{X,\log} \subset W_m\Omega^q_{X|D_{\un{n}},\log}
  \ \ {\rm for} \ \ \un{n} \ge \un{0}.
  \end{equation} 
  We shall prove this inclusion by induction on $m$. The case $m = 1$ follows
  directly from \lemref{lem:Log-fil-4}(2) and \cite[Thm.~1.2.1]{JSZ}. We now let
  $m \ge 2$ and $x \in \Fil_{-\un{n}}W_m\Omega^q_U \bigcap W_m\Omega^q_{X,\log}$.
  \lemref{lem:Log-fil-4}(5,7) then implies that
  $$R(x) \in \Fil_{({-\un{n})}/p}W_{m-1}\Omega^q_U \bigcap W_{m-1}\Omega^q_{X,\log} \text{ and }
  F(x) \in \Fil_{-\un{n}}W_{m-1}\Omega^q_U \bigcap W_{m-1}\Omega^q_{X,\log}.$$
  Since we also have $R(x) = F(x)$, it follows that
  $R(x) \in \Fil_{-\un{n}}W_{m-1}\Omega^q_U \bigcap W_{m-1}\Omega^q_{X,\log}$.
  In particular, the induction hypothesis implies that
  $R(x) = F(x) \in  W_{m-1}\Omega^q_{X|D_{\un{n}},\log}$.

  We now look at the commutative diagram (cf. \cite[Proof of Thm.~2.3.1]{JSZ})
  \begin{equation}\label{eqn:Complex-4-3}
    \xymatrix@C.8pc{
      & & W_m\Omega^q_{X|D_{\un{n}},\log} \ar[r]^-{R} \ar@{^{(}->}[d] & 
      W_{m-1}\Omega^q_{X|D_{\un{n}},\log}  \ar[r] \ar@{^{(}->}[d] & 0 \\
      0 \ar[r] & \Omega^q_{X, \log} \ar[r]^-{\un{p}^{m-1}} &
      W_m\Omega^q_{X,\log} \ar[r]^-{R} &
      W_{m-1}\Omega^q_{X,\log} \ar[r] & 0,}
    \end{equation}
  whose bottom row is exact (cf. \cite[Lem.~3]{CTSS}). Note that the map $R$ on the
  top is surjective by definition of the sheaves $W_m\Omega^q_{X|D_{\un{n}},\log}$.
  Using this diagram, we can
  find $y \in W_m\Omega^q_{X|D_{\un{n}},\log}$ and $z \in \Omega^q_{X, \log}$ such that
  $\un{p}^{m-1}(z) = x-y \in \Fil_{-\un{n}}W_m\Omega^q_U$. We conclude from
  \lemref{lem:Complete-7} that $z \in \Fil_{-\un{n}'}\Omega^q_U \bigcap
  \Omega^q_{X,\log}$, where $\un{n}' = -(-\un{n}/{p^{m-1}})$.
  In particular, $z \in \Omega^q_{X|D_{n'}, \log}$ by induction.
  It follows that $x-y = \un{p}^{m-1}(z) \in W_m\Omega^q_{X|D_{p^{m-1}\un{n}',\log}}
  \subset W_m\Omega^q_{X|D_{\un{n}},\log}$, and hence, $x \in W_m\Omega^q_{X|D_{\un{n}},\log}$.
  This proves (1).

  To prove (2), we only need to show that $1-\ov{F}$ is surjective.
By \remref{rem:VR-3} and \lemref{lem:VR-4}, and the fact that
  $$\ov{F}(w \dlog([x_1]_m) \wedge \cdots \wedge \dlog([x_q]_m)) =
  \ov{F}(w)\dlog([x_1]_m) \wedge \cdots \wedge \dlog([x_q]_m)),$$
  it is sufficient to find pre-images of elements of the form $V^i([x]_{m-i})$
  under the map
  $\Fil_{-\un{1}}W_m(K)  \xrightarrow{1 - \ov{F}} \Fil_{-\un{1}}W_m(K)$.
  We shall show that this map is in fact surjective.
  Using the commutative diagram of exact sequences
  \begin{equation}\label{eqn:Complex-4-1}
    \xymatrix@C.9pc{
      0 \ar[r] & \Fil_{- \un{1}}W_1(K) \ar[r]^-{V^{m-1}} \ar[d]_-{1 -\ov{F}} &
 \Fil_{- \un{1}}W_m(K) \ar[r]^-{R} \ar[d]_-{1 -\ov{F}} &      
 \Fil_{- \un{1}}W_{m-1}(K) \ar[r] \ar[d]_-{1 -\ov{F}} & 0 \\
 0 \ar[r] & \Fil_{- \un{1}}W_1(K) \ar[r]^-{V^{m-1}} &
 \Fil_{- \un{1}}W_m(K) \ar[r]^-{R} & \Fil_{- \un{1}}W_{m-1}(K) \ar[r] & 0}
  \end{equation}
and induction on $m$, it suffices to show that
  $\Fil_{- \un{1}}W_1(K) \xrightarrow{1 - \ov{F}} \Fil_{-\un{1}}W_1(K)$ is surjective
  when the ring $A$ in Lemma~\ref{lem:VR-3} is strict Henselian. But this is
  elementary (cf. \cite[Claim~1.2.2]{JSZ}). 

  To prove the surjectivity of $1-C$ in (3), it suffices to show, using \lemref{lem:Complete-4}(2) and the commutative
  diagram  
  \begin{equation}\label{eqn:Complex-4-4}
  \xymatrix@C.8pc{
  0 \ar[r] & Z_1\Fil_{-\un{1}}W_m\Omega^q_U \ar[r] \ar[d]^-{1-C} & \Fil_{{-\un{1}}}W_{m}\Omega^q_U \ar[r] \ar[d]^-{1-\ov F} & \frac{\Fil_{{-\un{1}}}W_{m}\Omega^q_U}{Z_1\Fil_{-\un{1}}W_m\Omega^q_U \ar[r]} \ar@{=}[d] & 0 \\
  0 \ar[r] & \Fil_{{-\un{1}}}W_{m}\Omega^q_U \ar[r]^-{-\ov F}& \frac{\Fil_{-\un{1}}W_{m}\Omega^q_U}{dV^{m-1}(\Fil_{-\un{1}}\Omega^{q-1}_U)} \ar[r] & \frac{\Fil_{-\un{1}}W_{m}\Omega^q_U}{\ov F(\Fil_{{-\un{1}}}W_{m}\Omega^q_U )} \ar[r] & 0,
  }
   \end{equation}
  that $1 - \ov{F}$ is surjective. The latter claim follows from (2). The
  remaining part of (3) is easily deduced from (1) and the above diagram. 

  The proof of (6) is identical using the commutative diagram and \cite[Thm~1.2.1]{JSZ}
  \begin{equation}
      \xymatrix@C.8pc{ 0 \ar[r] & Z_1\Fil_{-\un{n}}\Omega^q_U \ar[r] \ar[d]^-{1-C} & \Fil_{{-\un{n}}}\Omega^q_U \ar[r] \ar[d]^-{1-\ov F} & \frac{\Fil_{{-\un{n}}}\Omega^q_U}{Z_1\Fil_{-\un{n}}\Omega^q_U \ar[r]} \ar@{=}[d] & 0 \\
  0 \ar[r] & \Fil_{({-\un{n}})/p}\Omega^q_U \ar[r]^-{-\ov F}& \frac{\Fil_{-\un{n}}W_{m}\Omega^q_U}{d(\Fil_{-\un{n}}\Omega^{q-1}_U)} \ar[r] & \frac{\Fil_{-\un{n}}W_{m}\Omega^q_U}{\ov F(\Fil_{({-\un{n}})/p}\Omega^q_U )} \ar[r] & 0.}
  \end{equation}

The proof of the surjectivity of $1-C$ (resp. $1-\ov{F}$) in (4) (resp. (5))
is identical to that in (3) (resp. (2)). Also, one checks using an analogue of
~\eqref{eqn:Complex-4-4} for $\Fil_0W_m\Omega^q_U$
that $\Ker(1-C) = \Ker(1-\ov{F})$. We thus have to only show
that $\Ker(1-C) = j_*(W_m\Omega^q_{U, \log})$ to finish the proofs of (4) and (5).

To that end, we look at the commutative diagram of exact sequences
\begin{equation}\label{eqn:Complex-4-5}
  \xymatrix@C.8pc{
    0 \ar[r] & \Ker(1-C) \ar[r] \ar[d] & Z_1\Fil_0W_m\Omega^q_U \ar[r]^-{1-C}
    \ar@{^{(}->}[d] & \Fil_0W_m\Omega^q_U \ar[r] \ar@{^{(}->}[d] & 0 \\
    0 \ar[r] & j_*(W_m\Omega^q_{U, \log}) \ar[r] & j_*(Z_1W_m\Omega^q_U) \ar[r]^-{1-C} &
    j_*(W_m\Omega^q_U). &}
  \end{equation}
The middle and the right vertical arrows are the canonical inclusion maps.
It follows that $\Ker(1-C) \inj  j_*(W_m\Omega^q_{U, \log})$.
To prove this inclusion is a bijection, we note that the map
$\dlog \colon j_*(\sK^M_{q,U}) \to j_*(W_m\Omega^q_{U, \log})$ 
has a factorization
\begin{equation}\label{eqn:Complex-4-6}
  \xymatrix@C.8pc{
    j_*(\sK^M_{q,U}) \ar[r]^-{\dlog} \ar[d] & Z_1W_m\Omega^q_X(\log E) \ar@{^{(}->}[d] \\
    j_*(W_m\Omega^q_{U, \log}) \ar@{^{(}->}[r] & j_*(Z_1W_m\Omega^q_U).}
  \end{equation}
Hence, it suffices to show that $j_*(\sK^M_{q,U}) \to j_*(W_m\Omega^q_{U, \log})$ is
surjective. Since $j_*({\sK^m_{q,U}}) \to
j_*(W_m\Omega^q_{U, \log})$ is an isomorphism, we are reduced to
showing that the canonical map $j_*(\sK^M_{q,U}) \to j_*({\sK^m_{q,U}})$ is
surjective in $X_\et$. To prove the latter, it is enough to show that the map
$j_*(\sK^M_{q,U}) \to j_*({\sK^m_{q,U}})$ is surjective in $X_\zar$.

To prove the surjectivity of $j_*(\sK^M_{q,U}) \to j_*({\sK^m_{q,U}})$ in
$X_\zar$, we can assume that $X = \Spec(A)$ and
$U = X_\pi$ as in \S~\ref{sec:com-to-noncom}. In this case, the exact sequence
\begin{equation}\label{eqn:Complex-4-7}
0 \to \sK^M_{q, X_\pi} \xrightarrow{p^m}  \sK^M_{q, X_\pi} \to {\sK^m_{q,X_\pi}} \to 0
\end{equation}
reduces the problem to showing that $H^1_\zar(X_\pi,  \sK^M_{q, X_\pi}) = 0$,
which is \lemref{lem:Gersten-0}.
\end{proof}

\begin{cor}$($cf. \cite[Thm.~1.1.6]{JSZ}$)$\label{cor:Complex-7}
  For $\un{n} \ge 1$ and $1 \le r \le m-1$, one has an exact sequence
  \[
  0 \to W_{r}\Omega^q_{X|D_{\lceil{{\un{n}}/p^{m-r}}\rceil}, \log} \xrightarrow{{\un{p}}^{m-r}}
   W_{m}\Omega^q_{X|D_{{\un{n}}}, \log} \xrightarrow{R^{r}}  W_{m-r}\Omega^q_{X|D_{{\un{n}}}, \log}
   \to 0.
   \]
\end{cor}
\begin{proof}
  The surjectivity of $R^r$ follows directly from the definition of
  $W_{m}\Omega^q_{X|D_{{\un{n}}}, \log}$. The remaining part of the corollary follows
  from its classical non-modulus version (cf. \cite[Lem.~3]{CTSS}) and the
  identity $\Fil_{-\un{n}}W_m\Omega^q_U \bigcap W_m\Omega^q_{X,\log} =
  W_m\Omega^q_{X|D_{\un{n}},\log}$ that we showed in the proof of \lemref{lem:Complex-6}.
\end{proof}

\begin{lem}\label{lem:Complex-4}
  One has $W_m\sF^{0, \bullet}_0 \cong {\Z}/{p^m}$ and
    $W_m\sG^{d, \bullet}_0 \cong W_m\Omega^d_{X,\log}$. 
  \end{lem}
\begin{proof}
The first isomorphism is clear from the definition. For the second
  isomorphism, it is enough to show (by \lemref{lem:Complex-6}(2)) that the canonical map
  $W_{m}\Omega^d_{X|D_{\un{1}}, \log} \to W_{m}\Omega^d_{X, \log}$ is an isomorphism. But this follows from \cite[Prop~2.2.5]{Kerz-Zhao}, because the proof works in our situation.
  
\end{proof}

\section{Pairings of complexes}\label{sec:Pairing}
We continue with the set up of \S~\ref{sec:Complexes} and define the pairings
between the complexes. We consider the first two pairings
\begin{equation}\label{eqn:Pair-0}
  \< \ , \ \>^{0,0}_0 \colon Z_1\Fil_{\un{n}}W_m\Omega^q_U \times
  \Fil_{-\un{n}-1}W_m\Omega^{d-q}_U
  \to \Fil_{-1}W_m\Omega^d_U \subset W_m\Omega^d_X
\end{equation}
and
\begin{equation}\label{eqn:Pair-1}
  \< \ , \ \>^{1,0}_1 \colon \Fil_{\un{n}}W_m\Omega^q_U \times \Fil_{-\un{n}-1}W_m\Omega^{d-q}_U
  \to \Fil_{-1}W_m\Omega^d_U \subset W_m\Omega^d_X
\end{equation}
by letting $\<w_1,w_2\>^{0,0}_0 := w_1 \wedge w_2$ and $\<w_1,w_2\>^{1,0}_1 :=
w_1 \wedge w_2$.

We define the third pairing
\begin{equation}\label{eqn:Pair-2}
  \< \ , \ \>^{0,1}_1 \colon Z_1\Fil_{\un{n}}W_m\Omega^q_U \times
  \frac{\Fil_{-\un{n}-1}W_m\Omega^{d-q}_U}{dV^{m-1}(\Fil_{-\un{n}-1}\Omega^{d-q-1}_U)}
  \to \Fil_{-1}W_m\Omega^d_U \subset W_m\Omega^d_X
\end{equation}
by letting $\<w_1,w_2\>^{0,1}_1 := -C(w_1 \wedge w_2)$.

Note that the first two pairings are defined by \lemref{lem:Log-fil-4}, and to
show that the third pairing is also defined, we only need to prove that
$C(w_1 \wedge w_2) = 0$ if $w_1 \in Z_1\Fil_{\un{n}}W_m\Omega^q_U$ and
$w_2 \in dV^{m-1}(\Fil_{-\un{n}-1}\Omega^{d-q-1}_U)$.
To prove it, we can write $w_1 = F(w'_1)$ (cf. \thmref{thm:Global-version}(3))
and $w_2 = dV^{m-1}(w'_2)$, where
$w'_1 \in \Fil_{\un{n}}W_{m+1}\Omega^q_U$ and $w'_2 \in \Fil_{-\un{n}-1}\Omega^{d-q-1}_U$.
Then
\[
C(w_1 \wedge w_2) = C(F(w'_1) \wedge dV^{m-1}(w'_2)) = CF(w'_1 \wedge dV^m(w'_2)) =
R(w'_1 \wedge dV^m(w'_2))
\]
\[
\hspace*{4cm} = R(w'_1)\wedge dRV^m(w'_2) = 0,
\]
because $RV^m =0$.

\begin{lem}\label{lem:Pair-3}
  For $\un{n} \ge \un{0}$, the above pairings of {\'e}tale sheaves induce a pairing of the
  complexes
  \begin{equation}\label{eqn:Pair-3-0}
    W_m\sF^{q, \bullet}_{\un{n}} \times W_m\sG^{d-q, \bullet}_{\un{n}} \xrightarrow{{\< \ , \ \>}}
    W_m\sH^{\bullet}
  \end{equation}
  such that one has a commutative diagram of pairings
  \begin{equation}\label{eqn:Pair-3-1}
  \xymatrix@C1.5pc{  
    W_{m}\sF^{q, \bullet}_{\un{n}} \times W_{m}\sG^{d-q, \bullet}_{\un{n}}  \ar[r]^-{\< \ , \ \>}
    \ar@<5ex>[d] &
  W_m\sH^{\bullet} \ar@{=}[d] \\
  W_{m} \sF^{q, \bullet} \times W_{m}\sG^{d-q, \bullet}  \ar[r]^-{\< \ , \ \>}
  \ar@<7ex>[u] & W_{m}\sH^{\bullet},}
  \end{equation}
  in which the vertical arrows are the canonical maps.
\end{lem}
\begin{proof}
  To prove that ~\eqref{eqn:Pair-3-0} is defined, it is enough to show
  (cf. \cite[\S~1, p.~175]{Milne-Duality}) that
  \[
  (1-C)(w_1 \wedge w_2) = - C(w_1 \wedge (1-\ov{F})(w_2)) + (1-C)(w_1) \wedge w_2
  \]
  for all $w_1 \in Z_1\Fil_{\un{n}}W_m\Omega^q_U$ and
  $w_2 \in \Fil_{-\un{n}-1}W_m\Omega^{d-q}_U$. But this can be easily checked
  using \cite[Lem.~1.1(c)]{Milne-Duality} and the identity $C\ov{F} = \id$
  (cf. \cite[Lem.~7.1]{GK-Duality}). The commutativity of ~\eqref{eqn:Pair-3-1} is
  clear.
\end{proof}

Using \lemref{lem:Complex-6}(1), we get the following.

\begin{cor}\label{cor:Pair-4}
  \lemref{lem:Pair-3} remains valid if one replaces $W_m\sG^{q, \bullet}_{\un{n}}$
  by $W_m\Omega^q_{X|D_{\un{n}+1}, \log}$.
  \end{cor}

Using this corollary and considering the hypercohomology, we get a commutative
diagram of the cup product pairings of abelian groups (for $\un{n} \ge \un{0}$)
 \begin{equation}\label{eqn:Pair-4-1}
  \xymatrix@C1.5pc{  
    \H^i_\et(X, W_{m}\sF^{q, \bullet}_{\un{n}}) \times
    H^j_\et(X, W_{m}\Omega^{d-q}_{X|D_{\un{n}+1}, \log})  \ar[r]^-{\cup}
    \ar@<5ex>[d] &
  H^{i+j}_\et(X, W_m\Omega^d_{X, \log}) \ar@{=}[d] \\
  H^i_\et(X, W_{m}\Omega^q_{X, \log}) \times H^j_\et(X, W_{m}\Omega^{d-q}_{X,\log})
  \ar[r]^-{\cup} \ar@<10ex>[u] & H^{i+j}_\et(X, W_{m}\Omega^d_{X,\log}).}
  \end{equation}

\section{The duality theorem over finite fields}\label{sec:DT}
We now specialize to the situation when $X$ is projective as well as geometrically connected over $k$, where $k$ is a finite field. Other assumptions
remain the same as in \S~\ref{sec:Complexes}. Then dim $X= \text{ rank }\Omega^1_X=d$. We fix $\un{n} \ge \un{0}$.
By \cite[Cor.~1.12]{Milne-Zeta}, there is a bijective trace map
$\Tr \colon H^{d+1}_\et(X, W_{m}\Omega^d_{X,\log}) \xrightarrow{\cong} {\Z}/{p^m}$.
Composing the pairings in ~\eqref{eqn:Pair-4-1} with this trace map, we get
a pairing
\begin{equation}\label{eqn:Pair-4-2}
 \H^i_\et(X, W_{m}\sF^{q, \bullet}_{\un{n}}) \times
    H^{d+1-i}_\et(X, W_{m}\Omega^{d-q}_{X|D_{\un{n}+1}, \log})  \xrightarrow{\cup}
    {\Z}/{p^m}
    \end{equation}
which is compatible with the pairing for the cohomology of the sheaves
$W_m\Omega^{\bullet}_{X,\log}$ obtained in \cite[Thm.~1.14]{Milne-Zeta}.
Our goal is to show that this is a perfect pairing of finite abelian groups.
We begin with the following.

\begin{lem}\label{lem:Fin-coh}
  The groups $\H^i_\et(X, W_{m}\sF^{q, \bullet}_{\un{n}})$ and
  $H^{i}_\et(X, W_{m}\Omega^{q}_{X|D_{\un{n}}, \log})$ are all finite.
\end{lem}
\begin{proof}
  To prove the finiteness of $\H^i_\et(X, W_{m}\sF^{q, \bullet}_{\un{n}})$, it suffices to
  show that the cohomology of $Z_1\Fil_{\un{n}}W_m\Omega^q_{U}$ and
  $\Fil_{\un{n}}W_m\Omega^q_{U}$ are finite. But \lemref{lem:Log-fil-4}(1) says that these
  are sheaves of coherent $W_m\sO_X$-modules. In particular, their Zariski
  (equivalently, {\'e}tale) cohomology are finite type $W_m(k)$-modules. The desired
  claim now follows because $W_m(k)$ is finite.
  The finiteness of $H^{i}_\et(X, W_{1}\Omega^{q}_{X|D_{\un{n}}, \log})$ follows by
  considering the long exact cohomology sequence associated to
  the sheaf exact sequence of \cite[Thm.~1.2.1]{JSZ}. The desired finiteness for $m > 1$
  now follows by induction using \corref{cor:Complex-7}.
\end{proof}

For a ${\Z}/{p^m}$-module $M$, we let $M^\star = \Hom_{{\Z}/{p^m}}(M, {\Z}/{p^m})$.
Recall that a pairing of finite ${\Z}/{p^m}$-modules $M \times N \to {\Z}/{p^m}$ is
called perfect if the induced map $M \to N^\star$ is an isomorphism.
Equivalently, the map $N \to M^\star$ is an isomorphism. 

In the following lemma, we don't assume that $k$ is finite and $X$ is projective, because it will be used in other chapters.
\begin{lem}\label{lem:Perfect-0.5}
    Let $X$ be as in \secref{sec:Complexes}. Then the $\sO_X$ modules $Z_1\Fil_{\n}\Omega^q_U$ and $\frac{\Fil_{\un{n}}\Omega^{q+1}_U}{d(\Fil_{\un{n}}\Omega^{q}_U)}$ (via Frobenius) are locally free of finite rank, for all $\n \in \Z^r$.
\end{lem}
\begin{proof}
    The sheaf of $\sO_X$-modules $\Fil_{\un{n}}\Omega^q_U$ is locally free of finite rank by Lemmas
\ref{lem:Log-fil-4}(2) and \ref{lem:LWC-2}. To see that the same is true for
$Z_1\Fil_{\un{n}}\Omega^q_U$ and
    $\frac{\Fil_{\un{n}}\Omega^{q+1}_U}{d(\Fil_{\un{n}}\Omega^{q}_U)}$, 
  recall first that if $A$ is the local ring of a closed point on $E \subset X$ with
  maximal ideal $(x_1, \ldots , x_r, x_{r+1}, \ldots , x_N)$, then $\Omega^1_A(\log \pi)$ is a free
  $A$-module with basis $\{\dlog(x_1), \ldots, \dlog(x_r), dx_{r+1}, \ldots , dx_N,$ $dy_1, \ldots, dy_s\}$, for some $y_1,\ldots,y_s \in A$ (cf, \lemref{lem:LWC-2}). Also, $\{x_1,\ldots,x_N,y_1,\ldots,y_s\}$ forms a differential basis (see \lemref{lem:LWC-3}) and hence a $p$-basis (see \cite[\S~II, \S~III, Thm.~1]{Tyc} for the definition of $p$-basis and its relation with differential basis) of $A$ over $\F_p$. 

  Using the above facts about the local rings of $X$, it is easy to check that locally
  (depending on $\un{n}$) we can find two finite dimensional $\F_p$-vector spaces
  $V_1, V_2$ and an $\F_p$-linear map $d_0 \colon V_1 \to V_2$ such that
  $d \colon \Omega^q_X(\log E)(D_{\un{n}}) \to \Omega^{q+1}_X(\log E)(D_{\un{n}})$ is locally of the
    form $(V_1 \xrightarrow{d_0} V_2)\otimes_{\F_p} \sO_{X}^p$
    (cf. \cite[Lem.~1.7]{Milne-Duality}).
    In particular, $\Ker(d)$ and $\coker(d)$ are locally free $\sO^p_{X}$-modules
    of finite rank. The ring isomorphism $\sO_X \xrightarrow{p} \sO^p_X$ now implies 
    that they are locally free $\sO_X$-modules of finite rank.
    We can now apply \lemref{lem:Log-fil-4}(2) to conclude that
    $Z_1\Fil_{\un{n}}\Omega^q_U$ and
    $\frac{\Fil_{\un{n}}\Omega^{q+1}_U}{d(\Fil_{\un{n}}\Omega^{q}_U)}$ are finite rank
    locally free sheaves of $\sO_X$-modules.
\end{proof}

\begin{lem}\label{lem:Perfect-0}
  ~\eqref{eqn:Pair-4-2} is a perfect pairing of finite groups when $m =1$.
\end{lem}
\begin{proof}

    Now, we let $\sE^q_{\un{n}} =
    \frac{\Fil_{\un{n}}\Omega^{q}_U}{d(\Fil_{\un{n}}\Omega^{q-1}_U)}$ and $q' = d-q$.
To prove the lemma, we note that $\Omega^q_{X|D_{\un{n}},\log} \cong
W_1\sG^{q,\bullet}_{\un{n}}$ in $\sD(X_\et)$ for every $q, \un{n} \ge \un{0}$ by
\cite[Thm.~1.2.1]{JSZ}.
    Using the definition of the above pairings, we thus get a commutative diagram
    of long exact sequences
    \begin{equation}\label{eqn:Perfect-0-0}
      \xymatrix@C.8pc{
        \cdots \ar[r] & H^{d-i}_\et(X, \sE^{q'}_{-\un{n}-1}) \ar[r] \ar[d]_-{\alpha_i} &
        H^{d+1-i}_\et(X, \Omega^{q'}_{X|D_{\un{n}+1}, \log}) \ar[r] \ar[d]^-{\beta_i} &
        H^{d+1-i}_\et(X, \Fil_{-\un{n}-1}\Omega^{q'}_U) \ar[r] \ar[d]^-{\gamma_i} & \cdots \\
        \cdots \ar[r] & H^{i}_\et(X, Z_1\Fil_{\un{n}}\Omega^{q}_U)^\star \ar[r] &
        \H^i_\et(X, W_1\sF^{q, \bullet}_{\un{n}})^\star \ar[r] &
        H^{i-1}_\et(X, \Fil_{\un{n}}\Omega^q_U)^\star \ar[r] & \cdots .}
      \end{equation}
    
    Since ~\eqref{eqn:Pair-1} and ~\eqref{eqn:Pair-2} are perfect pairings of
    locally free sheaves (cf. Step~2 in the proof of \cite[Thm.~4.1.4]{JSZ}),
    it follows from the Grothnedick-Serre duality for the structure map
    $X \to \Spec(\F_p)$ that $\alpha_i$ and $\gamma_i$ in ~\eqref{eqn:Perfect-0-0}
    are bijective maps between finite groups for all $i$. 
    It follows that $\beta_i$ is also an isomorphism of finite groups for all $i$.
    This concludes the proof. 
\end{proof}

The main result of this section is the following.

\begin{thm}\label{thm:Perfect-finite}
  Assume that $X$ is a smooth projective and geometrically connected scheme of
  dimension $d$ over a finite field $k$ of characteristic $p$. Then \eqref{eqn:Pair-4-2}
  is a perfect pairing of finite abelian groups for all $q, \un{n} \ge \un{0}$ and $m \ge 1$.
\end{thm}
\begin{proof}
  The finiteness part is already shown in \lemref{lem:Fin-coh}. To prove
  the perfectness of the pairing, we can assume $m > 1$ by \lemref{lem:Perfect-0}.
  We let $i' = d+1-i$ and $q' = d-q$.
  By  \thmref{thm:Global-version}(11) and \corref{cor:Complex-7}, we get a
  commutative diagram (cf. \cite[Lem.~4.4]{GK-Duality-1} and its proof for the commutativity) 
  \begin{equation}\label{eqn:Perfect-finite-0}
    \xymatrix@C.5pc{
      \cdots \ar[r] &  H^{i'}_\et(W_{m-1}\Omega^{q'}_{X|D_{\lceil{(\un{n}+1)}/p}\rceil, \log})
      \ar[r]^-{\un{p}} \ar[d]_-{\alpha_i} &
      H^{i'}_\et(W_m\Omega^{q'}_{X|D_{\un{n}+1}, \log}) \ar[r]^-{R^{m-1}} \ar[d]^-{\beta_i} & 
      H^{i'}_\et(W_1\Omega^{q'}_{X|D_{\un{n}+1}, \log}) \ar[r]  \ar[d]^-{\gamma_i} & \cdots \\
      \cdots \ar[r] & \H^{i}_\et(W_{m-1}\sF^{q, \bullet}_{{\un{n}}/p})^\star \ar[r]^-{R^\star} &
        \H^i_\et(W_m\sF^{q, \bullet}_{\un{n}})^\star \ar[r]^-{(V^{m-1})^\star} &
        \H^{i}_\et(W_1\sF^{q, \bullet}_{\un{n}})^\star \ar[r] & \cdots .}
      \end{equation}
  In this diagram, we have used the shorthand $\H^*_\et(\sE^\bullet)$ for the (hyper)cohomology
  $\H^*_\et(X, \sE^\bullet)$ and the vertical arrows are induced by the
  pairing ~\eqref{eqn:Pair-4-2}. Since $\lceil{{(\un{n}+1)}/p}\rceil =
  \lfloor{{\un{n}}/p}\rfloor +1$,
it follows by induction on $m$ that $\alpha_i$ and $\gamma_i$ are bijective for all $i$.
    We conclude that $\beta_i$ is bijective for all $i$. 
\end{proof}

\begin{rem}
    By the above theorem, we get the following isomorphism of finite groups
    $$ H^{d+1-i}_\et(X, W_{m}\Omega^{d-q}_{X|D_{\un{n}+1}, \log})  \xrightarrow{\cong}
    \Hom_{\Z/p^m}(\H^i_\et(X, W_{m}\sF^{q, \bullet}_{\un{n}}),{\Z}/{p^m}).$$
    Taking limit over $\n$ (note that we have canonical maps $ W_{m}\Omega^{d-q}_{X|D_{\un{n'}+1}, \log} \to  W_{m}\Omega^{d-q}_{X|D_{\un{n}+1}, \log}$ and $W_{m}\sF^{q, \bullet}_{\un{n}} \to W_{m}\sF^{q, \bullet}_{\un{n'}}$ if $\n' \ge \n$.), we have the following isomorphism of abelian groups.
    \begin{equation}\label{eqn:JSZ-1}
        \begin{array}{cl}
            \varprojlim\limits_{\n}H^{d+1-i}_\et(X, W_{m}\Omega^{d-q}_{X|D_{\un{n}+1}, \log})   & \xrightarrow{\cong}
    \Hom_{\Z/p^m}(\varinjlim\limits_{\n}\H^i_\et(X, W_{m}\sF^{q, \bullet}_{\un{n}}),{\Z}/{p^m}) \\
             & \cong
    \Hom_{\Z/p^m}(H^i_\et(U, W_{m}\Omega^q_{U,\log}),{\Z}/{p^m}),
        \end{array}
    \end{equation}
    where the last isomorphism follows from the fact that $\varinjlim\limits_\n W_{m}\sF^{q, \bullet}_{\un{n}}\cong Rj_*W_m\Omega^q_{U,\log}$. We can endow the left hand side of \eqref{eqn:JSZ-1} with profinite topology and $H^i_\et(U, W_{m}\Omega^q_{U,\log})$ with discrete topology (cf. \cite[P. 1325]{JSZ}). Hence, \eqref{eqn:JSZ-1} gives an isomorphism of profinite groups. On the other hand, \thmref{thm:Perfect-finite} also shows that the map $$\<\alpha, -\>: \varprojlim\limits_{\n}H^{d+1-i}_\et(X, W_{m}\Omega^{d-q}_{X|D_{\un{n}+1}, \log}) \to \Z/p^m$$ is continuous, for $\alpha \in H^i_\et(U, W_{m}\Omega^q_{U,\log})$ (cf. \cite[Lem~4.1.3]{JSZ}). Combining all these claims, we can conclude the following perfect pairing of topological abelian groups (cf. \cite[Thm.~4.1.4]{JSZ}).
    $$H^i_\et(U, W_{m}\Omega^q_{U,\log}) \times
    \varprojlim\limits_{\n}H^{d+1-i}_\et(X, W_{m}\Omega^{d-q}_{X|D_{\un{n}+1}, \log})  \xrightarrow{\cup}
    {\Z}/{p^m}.$$
\end{rem}

\begin{cor}\label{cor:Duality-log}
  For $X$ as in \thmref{thm:Perfect-finite}, there is a perfect pairing of
  finite abelian groups
  \[
  H^i_\et(X, j_*(W_m\Omega^q_{U, \log})) \times H^{d+1-i}_\et(X, W_m\Omega^{d-q}_{X|E,\log})
  \xrightarrow{\cup} {\Z}/{p^m}.\]
  \end{cor}
\begin{proof}
Combine \lemref{lem:Complex-6} and \thmref{thm:Perfect-finite}.
\end{proof}

\chapter{A duality theorem for 2-local fields}
In this chapter, we shall prove a duality theorem for 2-local fields using the class field theory developed in \cite{Kato80-2} and \cite{Kato-Inv}. This will be used in \S~\ref{chap:duality-localfields} later.
\section{Kato topology}\label{sec:Kato-top}
Since we shall use Kato's topology extensively in the next chapter, we briefly
recall it here.
More details can be found in \cite[\S~7]{Kato80-1} (see also
\cite[\S~7.1]{KRS}).

Let $K$ be a 2-local field of characteristic $p>0$, i.e., a
complete discrete valuation field of characteristic $p > 0$
whose residue field $\ff$ is a local field. Recall that, a local field is a complete discrete valued field with finite residue field. Let $A$ denote the ring of integers of
$K$ with maximal ideal $(\pi)$. One knows that there
is a canonical isomorphism of rings $\phi \colon \ff[[\pi]] \xrightarrow{\cong}
A$ such that the composition $\ff \to \ff[[\pi]] \xrightarrow{\phi}A \to \ff$ is the identity map (cf. \cite[Thm~28.3, also the proof of Lem.~1]{Matsumura}). The Kato topology on $A$ is the unique topology such that
$\phi$ is an isomorphism of topological rings when $\ff[[\pi]]$ is endowed with the
product of the valuation (also called the adic) topology of $\ff$. We remark that Kato topology on $A$ does not depends on the choice of the uniformizer $\pi$ and $\phi$.
The Kato topology of $A^\times$
is the subspace topology induced from that of $A$.
The Kato topology of $K^\times$ is the unique topology which is compatible
with the group structure of $K^\times$ and for which
$A^\times$ is open in $K^\times$. The Kato topology of $K^M_2(K)$ is the finest
topology which is compatible with the group structures and for which
the product map $K^\times \times K^\times \to K^M_2(K)$ is continuous.
This makes $K^M_2(K)$ into a topological abelian group.

\begin{rem}\label{rem:A^*-dual}
    The canonical topological isomorphism $\phi \colon \ff[[\pi]] \xrightarrow{\cong} A$ implies $A^\times \xrightarrow{\cong} \ff^\times \oplus (1+\pi.\ff[[\pi]])$ is also a topological isomorphism, where $\ff^\times \oplus (1+\pi.\ff[[\pi]])$ is given direct sum topology. 
\end{rem}

Now we fix some notations.
\begin{notat}\label{not:fil-r}
    If $A$ is a Henselian discrete valuation ring
with maximal ideal $\fm$ and field of fraction $K$,
 we shall write  $K^M_n(A|\fm^r):=H^0_{nis}(\Spec A, \sK^M_{n,\Spec A|V(\fm^r)})$ (see \S~ \ref{sec:Complexes}) as
$\Fil_r K^M_n(K)$) for $n, r \ge 1$ and $\Fil_0 K^M_n(K) = K^M_n(A)$, where $K^M_n$ is the $n$th Milnor K-theory $K^M_n$. Let $m,q \ge 1$ and $n \ge 0$ be three integers. 
If char $A=p$, we let $\kappa^m_q(A|n) =
\frac{\Fil_n K^M_q(K)}{p^m K^M_q(K) \cap \Fil_n K^M_q(K)}$ and for $n \ge 1$, define
$\varpi^m_q(A|n)$ so that there is an exact sequence
\begin{equation}\label{eqn:Kato-0.5}
  0 \to \kappa^m_q(A|n) \to \kappa^m_q(A|1) \to \varpi^m_q(A|n) \to 0.
  \end{equation}
It also follows that 
\begin{equation}\label{eqn:Kato-0.5.1}
    \varpi^m_q(A|n)= \Ker\left(\frac{K^M_q(K)}{\Fil_n K^M_q(K)+p^mK^M_q(K)} \xrightarrow{\theta}\frac{K^M_q(K)}{\Fil_1 K^M_q(K)+p^mK^M_q(K)}\right),
\end{equation}
where $\theta$ is the natural quotient map.
\end{notat}

For any integers $m, r \ge 0$,
the Kato topology of $K^M_2(K)$ induces the quotient topology on ${K^M_2(K)}/m$
and the subspace topology on $\Fil_r K^M_2(K)$. We shall use the term Kato topology
for these topologies as well.
By the Kato topologies of $\kappa^m_r(A|n)$ (resp. $\varpi^m_r(A|n)$), we shall
mean the subspace topology of $\kappa^m_r(A|n)$ induced from that of
${K^M_2(K)}/{p^m}$ (resp. the
  quotient topology of $\kappa^m_r(A|1)$ via ~\eqref{eqn:Kato-0.5}).
  Let $\partial_K \colon K^M_2(K) \surj \ff^\times$ be the Tame symbol map.
Via this surjective homomorphism, the Kato topology of $K^M_2(K)$ induces
the quotient topology (called the Kato topology) on $\ff^\times$.
On the other hand, the latter is also equipped with its valuation topology.

\begin{lem}\label{lem:Kato-adic}
The two topologies of $\ff^\times$ coincide.
\end{lem}
\begin{proof}
  To show that the adic topology of $\ff^\times$ is finer than the Kato topology,
  it suffices to show that if $U \subset \ff^\times$ is a subgroup which is
  open in the Kato topology, then it contains an open subgroup in the adic topology.
  To that end, we look at the composite map $T \colon K^\times \times K^\times
  \to K^M_2(K) \xrightarrow{\partial_K} \ff^\times$. By definition, $T$ is 
  continuous with respect to the Kato topology of $\ff^\times$.
  It follows that the map $T_\pi \colon K^\times \to \ff^\times$,
  given by $T_\pi(a) = \partial_K(\{a, \pi\})$ is a continuous homomorphism.
  In particular, $T_\pi^{-1}(U)$ contains an open subgroup of $K^\times$.
  By definition of the Kato topology of $K^\times$ and
  \cite[\S~7, Rem.~1]{Kato80-1}, $T^{-1}_\pi(U)$ contains a subgroup of the
  form $V_0 \times V_1$, where $V_0 \subset \ff^\times$ is an open subgroup in the
  adic topology and
  $V_1 \subset {\underset{i \ge 1}\prod} \ff \pi^i$ is an open subgroup
  in the Kato topology of $\pi A \subset A$. It is clear from the
  definition of $T$ that $T_\pi(V_0 \times V_1) = V_0$ which implies that
  $V_0 \subset U$.

To show that the Kato topology of $\ff^\times$ is finer than the adic topology,
  it is equivalent to show that $\partial_K$ is continuous with respect  to the
  Kato topology on $K^M_2(K)$  and adic topology on $\ff^\times$. But this is equivalent to show that the map $T: K^* \times K^*\to \ff^*$ is continuous with respect  to the
  Kato topology on $K^* \times K^*$ and adic topology on $\ff^\times$. Let $x \in \ff^*$ and $xU_n$ be a basic open set containing $x$. Here, $U_n=1+\pi_1^n\sO_\ff$ and $\pi_1$ is a uniformizer of $\sO_\ff$. Also, let $a, b \in K^*$ such that $T(a,b)=x$. We need to find two Kato open subsets $A$ and $B$ in $K^*$ containing $a$ and $b$, respectively, such that $T(A\times B)\subset xU_n$. We write $a=\pi^{n_a}a_1$ and $b= \pi^{n_b}b_1$, such that $a_1, b_1 \in A^*$ and $n_a, n_b \in \Z$. Note, $T(a,b)=\ov a_1^{n_b}\cdot \ov b_1^{-n_a}=x$, where $\ov a_1, \ov b_1$ are the images of $a_1$ and $b_1$, respectively, in $\ff^*$. Since the map $\tau:\ff^* \times \ff^* \to \ff^*$, given by $\tau(c,d)=c^{n_b}\cdot d^{-n_a}$ is continuous homomorphism of topological groups (with respect to adic topology), there exists $n'$ such that $\tau(\ov a_1U_{n'}\times \ov b_1 U_{n'})\subset xU_n$. Then, for $A=a.U_{n'}(1+\pi.\ff[[\pi]])$ and $B=b.U_{n'}(1+\pi.\ff[[\pi]])$, we have $T(A\times B)= \tau(\ov a_1U_{n'}\times \ov b_1 U_{n'})\subset xU_n$. Since $A$ and $B$ are Kato open in $K^*$ (cf. \remref{rem:A^*-dual}), we conclude that $\partial_K$ is continuous.
  \end{proof}

In the sequel, the default topology of ${K^M_2(K)}/n$ ($n \ge 0$) and its subspaces
      (e.g., ${K^M_2(A)}/n$) will be the Kato topology, and the same for
      ${\ff^\times}/n$ will be the adic topology.
If $K'$ is a Henselian discrete valuation field with completion $K$, then the
Kato topology of $K'^\times$ is the subspace topology induced from that
of $K^\times$. The Kato topology on $K^M_2(K')$ is the finest topology
such that the product map $K'^\times \times K'^\times \to K^M_2(K')$ is
continuous. It is clear that the canonical map $K^M_2(K') \to K^M_2({K})$ is
continuous.

\begin{notat}\label{not:Top}
    Before we proceed further, we recall some topological notions.
For $G \in \Ab$, we let $G^\vee =  \Hom_\Ab(G, {\Q}/{\Z})$.
For $G \in \Tab$ (see page VIII for the notations $\Ab, \Tab$), we let $G^\star =  \Hom_\Tab(G, {\Q}/{\Z})$.
If $G$ is either a profinite or a torsion locally compact Hausdorff topological
abelian group, then $G^\star$ is a topological abelian group with its compact open
topology and satisfies the Pontryagin duality, i.e., the evaluation map
$G \to (G^*)^*$ is an isomorphism in $\Tab$ (cf. \cite[\S~2.9]{Pro-fin}).
For $A, B \in \Tab$, we let $\Hom_\cf(A,B) = (\Hom_{\Tab}(A,B))_\tor$.
We shall write $\Hom_\cf(A,{\Q}/{\Z})$ as $A^\star_{\fr}$.
In this thesis, we shall endow $\Q$ as well as the direct sums of its subgroups and
subquotients (e.g., finite abelian groups) with the discrete topology.
\end{notat}

\begin{lem}\label{lem:K-dual}
  The map $\partial_K \colon K^M_2(K) \surj \ff^\times$ induces an exact sequence
  \begin{equation}\label{eqn:K-dual-0}
    0 \to (\ff^\times)^\star \to {K^M_2(K)}^\star \to {K^M_2(A)}^\star \to 0.
  \end{equation}
 \end{lem}
 \begin{proof}
    Since $K^M_2(A) = \Ker(\partial_K)$, the exactness of ~\eqref{eqn:K-dual-0}
    except at $K^M_2(A)^\star$ follows immediately from \lemref{lem:Kato-adic}.
    To prove the exactness at ${K^M_2(A)}^\star$, we let
    $\phi_K \colon K^M_2(K) \to K^M_2(A)$ be
    given by $\phi_K(\{a, b\}) = \{a,b\} - \{\partial_K(\{a,b\}), \pi\}$. Since $\partial_K$, the map $\psi_K:\ff^*\to K_2^M(K)$, where $\psi_K(a)= \{a,\pi\}$ and the identity map on $K_2^M(K)$ are group homomorphism, we see that $\phi_K$ is a group homomorphism whose restriction to
    $K^M_2(A)$ is identity. It remains to show that $\phi_K$ is continuous with
    respect to the Kato topology. Equivalently, we need to show that 
    $\phi_K \colon K^M_2(K) \to K^M_2(K)$ is continuous. As $\partial_K$ is
    continuous, the latter claim follows if we show that the map
    $\psi_K \colon \ff^\times \to K^M_2(K)$ is continuous. But this map is the composition of
    $\ff^\times \inj K^\times \to K^M_2(K)$, where the first map is the canonical
inclusion and the second map takes $a$ to $\{a, \pi\}$.
The second map is clearly continuous and the first map is continuous by
\cite[\S~7, Rem.~1]{Kato80-1}. 
\end{proof}

\begin{remk}\label{remk:K-dual-1}
The same proof as that of \lemref{lem:K-dual} shows that the sequence
\begin{equation}\label{eqn:K-dual-2}
    0 \to ({\ff^\times}/n)^\star \to ({K^M_2(K)}/n)^\star \to ({K^M_2(A)}/n)^\star \to 0
  \end{equation}
is exact for every $n \in \Z$.  Also, by ~\remref{rem:A^*-dual}, we get $(K^\times /n)^\star \surj (A^\times/n)^\star \surj (\Fil_1 K^\times /n)^\star$ for every $n \in \Z$.
\end{remk}

\section{The reciprocity homomorphism}\label{sec:RHom}
Let $K$ be a Henselian discrete valuation field of characteristic $p > 0$
whose residue field $\ff$ is a local field.
Let $A$ denote the ring of integers of $K$ and $\fm = (\pi)$ the maximal ideal of
$A$. For any site $\tau$ on $\Spec A$ (resp. $\Spec K$), the cohomology groups $H^i_\tau(\Spec A, \sF)$ (resp. $H^i_\tau(\Spec K, \sF)$) will be written as $H^i_\tau (A, \sF)$ (resp. $H^i_\tau(K, \sF)$).

For any $r \ge 1$, we consider $H^r(K)$ (see \S~\ref{chap:Kato-coh} for the definition) as a discrete topological group.
For $q \le 2$, recall from \cite{Kato80-1} that the
reciprocity homomorphism $$\rho^q_K \colon K^M_q(K) \to \Hom_\Ab(H^{3-q}(K), \Q/\Z)$$ is induced by the cup
product pairing
\begin{equation}\label{eqn:KR**}
  K^M_q(K) \times H^{3-q}(K) \xrightarrow{{\rm NR}_K \times \id}
  H^q_\et(K, {\Q}/{\Z}(q)) \times H^{3-q}(K) \xrightarrow{\cup} H^3(K) \cong {\Q}/{\Z}.
\end{equation}
One knows that $\rho^q_K$ is a continuous homomorphism with respect to the Kato
topology of $K^M_q(K)$ and the profinite topology of $\left(H^{3-q}(K)\right)^\star$. Indeed, the continuity in the case of $q=2$ follows from \cite[Thm.~1(iii) and Thm.~2(1)]{Kato80-1}, while the continuity in the case of $q=1$ follows from Thm.~2(2) of loc. cit.

\begin{notat}\label{not:G-q-n}
    Let's write
$\Fil_n H^q(K)$ for $\Fil^{bk}_n H^q(K)$ defined in \defref{defn:Kato-1}.
 Also, we write $\Fil_{n} H^q(K, {\Z}/{p^m}(q-1))$ or $\Fil_{n} H^1_\et(K, W_m\Omega^{q-1}_{K,\log}))$ for $T^{m,q}_n(K)$ defined in \eqref{eqn:Milnor-2}.
 Letting $G^{q,n}_K = \left(\frac{H^q(K)}{\Fil_{n-1} H^q(K)}\right)^\star$
  and $G^{q,n}_{K,m} = \left(\frac{H^q(K, {\Z}/{p^m}(q-1))}{\Fil_{n-1} H^q(K, {\Z}/{p^m}(q-1))}
  \right)^\star$, we have the following.
\end{notat}

\begin{lem}\label{lem:LR-mod}
    For $n \ge 1, 1\le q\le 2$ and $m \ge 1$, ~\eqref{eqn:KR**} induces continuous homomorphisms (cf. Notations~\ref{not:fil-r},\ref{not:G-q-n})
    \[
      (1) \ \rho^{q,n}_{K} \colon \Fil_{n} K^M_q(K) \to G^{3-q,n}_K; \
      (2) \ \rho^{q,n}_{K,m} \colon \kappa^m_2(A|n) \to  G^{3-q,n}_{K,m}.
    \]
  \end{lem}
\begin{proof}
    The existence of $\rho^{q,n}_{K,m}$ follows from that of $\rho^{q,n}_K$ using
   \thmref{thm:Kato-fil-10}. Item (1) follows from
    ~\eqref{eqn:KR**} and \cite[Rem.~6.6]{Kato-89}.
\end{proof}

\section{Duality for \texorpdfstring{$\kappa^m_q(A|n)$ and
$\varpi^m_q(A|n)$}{k(A|n) and w(A|n)}}

\begin{lem}\label{lem:Kato-equiv}
 Let $A$ be as in \S~\ref{sec:RHom} and $n \ge 1, m \ge 1, \ q \ge 1$ are integers, 
 the following hold.
  \begin{enumerate}
 \item
    The canonical maps ${K^M_q(K)}/{p^m} \to {K^M_q(\wh{K})}/{p^m}$ and
    $\kappa^m_q(A|n) \to \kappa^m_q(\wh{A}|n)$ are injective, where $\wh A$ (resp. $\wh K$) is the $\fm$-adic completion of $A$ (resp. $K$).

 \item
   The canonical maps
$\frac{K^M_q(K)}{\Fil_n K^M_q(K)} \to \frac{K^M_q(\wh{K})}{\Fil_n K^M_q(\wh{K})}$ and
$\varpi^m_q(A|n) \to \varpi^m_q(\wh{A}|n)$
  are bijective.
  \item
  $H^q_\et(R, {\Z}/{p^m} (q-1)) \to
  H^q_\et(\wh{R}, {\Z}/{p^m} (q-1))$ is bijective if $R \in \{A, K\}$.
  \item
    $\Fil_{n} H^q(K, {\Z}/{p^m}(q-1)) \to \Fil_{n} H^q(\wh{K}, {\Z}/{p^m}(q-1)) $ is bijective.
  \end{enumerate}
\end{lem}
\begin{proof}
  The first part of item (1) follows because $W_m\Omega^2_K \inj W_m\Omega^2_{\wh{K}}$. For the second part, we note that $\kappa^m_q(A|n)  \subset {K^M_q(K)}/{p^m}$ (and similarly in the case of $\wh A$).
  The first part of item (2) and the ($R = K$) case of item (3)  follow from 
  \cite[Lem.~1]{Kato83} and \cite[Lem.~21]{Kato-Inv}. For the second part of item (2), we note that the result is true for $\frac{K^M_q(K)}{\Fil_n K^M_q(K)+p^mK^M_q(K)}$ (follows from the first part of (2)). Then the claim for $\varpi^m_q(A|n)$ follows from the previous observation and \eqref{eqn:Kato-0.5.1}. To prove ($R =A$) case of item $(3)$ first consider the following commutative diagram obtained by the localisation sequence for the sheaf $W_m\Omega^{q-1}_{(-),\log}$ in $\sD(\Spec B_\et,\Z/p^m)$, for $B=A \text{ and } \wh{A}$.
  \begin{equation}\label{eqn:A-A^}
      \xymatrix@C.8pc{
      0 \ar[r]&H^1_\et(A, W_m\Omega^{q-1}_{(-),\log}) \ar[r] \ar[d] &H^1_\et(K, W_m\Omega^{q-1}_{(-),\log}) \ar[r] \ar[d] & H^2_x(A, W_m\Omega^{q-1}_{(-),\log}) \ar[r] \ar[d]^-{\gamma} & 0 \\
      0 \ar[r] &H^1_\et(\wh A, W_m\Omega^{q-1}_{(-),\log}) \ar[r] & H^1_\et(\wh K, W_m\Omega^{q-1}_{(-),\log}) \ar[r] & H^2_{x'}(\wh A, W_m\Omega^{q-1}_{(-),\log}) \ar[r] & 0,
      }
  \end{equation}
where $x,x'$ are the closed points of $\Spec A$ and $\Spec \wh A$, respectively. Using $(R=K)$ case, it is enough to show the natural map $\gamma: H^2_x(A, W_m\Omega^{q-1}_{(-),\log}) \to H^2_{x'}(\wh{A}, W_m\Omega^{q-1}_{(-),\log})$ is isomorphism. Since $W_m\Omega^{q-1}_{(-),\log} \cong \left [ Z_1W_m\Omega^{q-1}_{(-)} \xrightarrow{1-C} W_m\Omega^{q-1}_{(-)} \right ]$, we have the following commutative diagram (letting $q'=q-1$). 
\begin{equation}
    \xymatrix@C.8pc{
    H^1_x(Z_1W_m\Omega^{q'}_A) \ar[r] \ar[d]^-{\alpha_1} &H^1_x(W_m\Omega^{q'}_A) \ar[r] \ar[d]^-{\beta_1} & H^2_x(W_m\Omega^{q'}_{A,\log}) \ar[r] \ar[d]^-{\gamma} & H^2_x(Z_1W_m\Omega^{q'}_{A}) \ar[r] \ar[d]^-{\alpha_2} & H^2_x(W_m\Omega^{q'}_{A}) \ar[d]^-{\beta_2} \\
    H^1_{x'}(Z_1W_m\Omega^{q'}_{\wh A}) \ar[r] &H^1_{x'}(W_m\Omega^{q'}_{\wh A}) \ar[r] & H^2_{x'}(W_m\Omega^{q'}_{\wh A,\log}) \ar[r] & H^2_{x'}(Z_1W_m\Omega^{q'}_{\wh A}) \ar[r] & H^2_{x'}(W_m\Omega^{q'}_{\wh A}),
    }
\end{equation}
where $H^1_x(W_m\Omega^{q'}_A)$ (and similarly the other groups) is the short hand notation for $H^1_x(A, W_m\Omega^{q'}_{A})$. First, we note that the local cohomologies of $W_m\Omega^{q'}_A$ and $Z_1W_m\Omega^{q'}_A$ on Spec $A$ is same as those on $\Spec W_m(A)$. We recall that the $\fM_m$-adic (equivalently $W_m(\fm)$-adic) completion of $W_m(A)$ is $W_m(\wh A)$ (cf. \lemref{lem:Non-complete-1}, for $q=0$). Since $Z_1W_m\Omega^{q'}_{A}$, $W_m\Omega^{q'}_{A}$ can be regarded as finitely generated $W_m(A)$-modules, and their  $\fM_m$-adic completions are $Z_1W_m\Omega^{q'}_{\wh A}$, $W_m\Omega^{q'}_{\wh A}$ (by \corref{cor:Non-complete-3}, Lemmas~\ref{lem:Non-complete-3}, \ref{lem:Complete-4}), respectively, we get that $\alpha_i$ and $\beta_i$ ($i=1,2$) are isomorphisms by \cite[Prop~3.5.4.(d)]{Bruns-Herzog}. Thus we get $\gamma$ is isomorphism.

Proof of (4) is identical to that in (3), ($R=A$) case, except that one need to use $W_m\sF^{q-1,\bullet}_n$ complex instead of $W_m\Omega^{q-1}_{(-),\log}$ and also \corref{cor:Coh-V0-00} to obtain a diagram like \eqref{eqn:A-A^}.
\end{proof}

\begin{lem}\label{lem:Local-inj-6}
   Let $A$ and $K$ be as in \S~ \ref{sec:RHom}. For $m \ge 1, 1 \le q\le 2$ and $n \ge 1$, the map $\rho^{q,n}_{K,m}$ of \lemref{lem:LR-mod}
   induces a monomorphism (cf. Notation~\ref{not:fil-r})
   \[
     \wt{\rho}^{q,n}_{K,m} \colon  \varpi^m_q(A|n) \inj     
\left(\frac{\Fil_{n-1} H^1_\et(K, W_m\Omega^{2-q}_{K,\log})}{\Fil_{0} H^1_\et(K, W_m\Omega^{2-q}_{K,\log})}\right)^\star,
 \]
 where the dual on the right is taken with respect to the discrete topology.
\end{lem}

\begin{proof}
First note that the map $\wt{\rho}^{q,n}_{K,m}$ is defined by the following commutative diagram.
\begin{equation*}
    \xymatrix@C.8pc{
     0 \ar[r]& \kappa^m_q(A|n)\ar[r] \ar[d]^-{\rho^{q,n}_{K,m}} & \kappa^m_q(A|1) \ar[r] \ar[d]^-{\rho^{q,1}_{K,m}} & \varpi^m_q(A|n)\ar[r] \ar@{.>}[d]^-{\wt{\rho}^{q,n}_{K,m}} & 0\\
     0 \ar[r]& \left(\frac{H^1_\et(K, W_m\Omega^{2-q}_{K,\log})}{\Fil_{n-1} H^1_\et(K, W_m\Omega^{2-q}_{K,\log})}
  \right)^\star \ar[r]& \left(\frac{H^1_\et(K, W_m\Omega^{2-q}_{K,\log})}{\Fil_{0}H^1_\et(K, W_m\Omega^{2-q}_{K,\log})}
  \right)^\star\ar[r]& \left(\frac{\Fil_{n-1} H^1_\et(K, W_m\Omega^{2-q}_{K,\log})}{\Fil_{0} H^1_\et(K, W_m\Omega^{2-q}_{K,\log})}\right)^\star \ar[r]& 0. 
    }
\end{equation*}
    Now we need to show that the map $\wt{\rho}^{q,n}_{K,m}$ is injective. We can assume that $K$ is complete in view of
\lemref{lem:Kato-equiv}.
  We now let $n' = \lceil n/p \rceil$ and claim that the following is a commutative diagram of exact sequences.
 \begin{equation}\label{eqn:Local-inj-7}
    \xymatrix@C.8pc{ 
      0 \ar[r] & \varpi^{m-1}_q(A|n') \ar[r]^-{\ov p} \ar[d]_-{\wt{\rho}^{q,n}_{K,m-1}} &
    \varpi^{m}_q(A|n) \ar[rr]^-{R^{m-1}} \ar[d]^-{\wt{\rho}^{q,n}_{K,m}} &&  
    \varpi^{1}_q(A|n) \ar[r] \ar[d]^-{\wt{\rho}^{q,n}_{K,1}} & 0 \\
    0 \ar[r] &
 \left(\frac{\Fil_{n'-1} H^1_\et(K, W_{m-1}\Omega^{2-q}_{K,\log})}{\Fil_0H^1_\et(K, W_{m-1}\Omega^{2-q}_{K,\log})}\right)^\star \ar[r]^-{R^*} & 
 \left(\frac{\Fil_{n-1} H^1_\et(K, W_{m}\Omega^{2-q}_{K,\log})}{\Fil_0 H^1_\et(K, W_m\Omega^{2-q}_{K,\log})}\right)^\star \ar[rr]^-{(\ov p^{m-1})^*} &&
 \left(\frac{\Fil_{n-1} H^1_\et(K, \Omega^{2-q}_{K,\log})}{\Fil_0 H^1_\et(K, \Omega^{2-q}_{K,\log})}\right)^\star
 \ar[r] & 0.}
\end{equation}

The exactness of the top row follows by comparing the exact sequence of \corref{cor:Complex-7} (taking $r=m-1$) for $\kappa^m_q(A|n)$ and $\kappa^m_q(A|1)$.

Now by \thmref{thm:Global-version}(11), we have an exact sequence
$$ \cdots \to \H^0_\et(W_{1}\sF^{2-q,\bullet}_{\lfloor n/p \rfloor}) \to \H^1_\et(W_{m-1}\sF^{2-q,\bullet}_n) \xrightarrow{V} \H^1_\et(W_{m}\sF^{2-q,\bullet}_n) \xrightarrow{R^{m-1}} \H^1_\et(W_{1}\sF^{2-q,\bullet}_{\lfloor n/p \rfloor}) \to \cdots,$$
where we use the short hand notation $\H^i_\et(\sE)$ for $\H^i(A,\sE)$. By \lemref{lem:hyp}(2), we have the short exact sequence
$$0 \to \H^0_\et(W_{m-1}\sF^{2-q,\bullet}_n) \xrightarrow{V=\ov p} \H^0_\et(W_{m}\sF^{2-q,\bullet}_n) \xrightarrow{R^{m-1}} \H^0_\et(W_{1}\sF^{2-q,\bullet}_{\lfloor n/p \rfloor}) \to 0,$$
and hence the following short exact sequence.
$$0\to\H^1_\et(W_{m-1}\sF^{2-q,\bullet}_n) \xrightarrow{V} \H^1_\et(W_{m}\sF^{2-q,\bullet}_n) \xrightarrow{R^{m-1}} \H^1_\et(W_{1}\sF^{2-q,\bullet}_{\lfloor n/p \rfloor}) \to0.$$
Comparing the above exact sequence with the similar one with $n=0$, we get the following exact sequence.
$$0\to\frac{\H^1_\et(W_{m-1}\sF^{2-q,\bullet}_n)}{\H^1_\et(W_{m-1}\sF^{2-q,\bullet}_0)} \xrightarrow{V} \frac{\H^1_\et(W_{m}\sF^{2-q,\bullet}_n)}{\H^1_\et(W_{m}\sF^{2-q,\bullet}_0)} \xrightarrow{R^{m-1}} \frac{\H^1_\et(W_{1}\sF^{2-q,\bullet}_{\lfloor n/p \rfloor})}{\H^1_\et(W_{1}\sF^{2-q,\bullet}_0)} \to0.$$
Here, we have also used the fact that the map $\H^1_\et(W_{m}\sF^{2-q,\bullet}_0) \to \H^1_\et(W_{m}\sF^{2-q,\bullet}_n)$ is injective, which follows from \corref{cor:VR-5} and \lemref{lem:hyp}(2). 
Now, to show the exactness of the bottom row in \eqref{eqn:Local-inj-7}, we take dual in the above exact sequence, use \thmref{thm:Kato-fil-10} and observe that $n'-1=\lfloor (n-1)/p \rfloor$. Thus, the commutativity of \eqref{eqn:Local-inj-7} follows from the commutativity of \eqref{eqn:Perfect-finite-0}. It follows from \eqref{eqn:Local-inj-7} and an induction argument that it suffices to show $\wt{\rho}^{q,n}_{K,1}$ is injective.

For this, note that $\varpi^1_q(A|n)$ and $\left(\frac{\Fil_{n-1} H^1_\et(K, \Omega^{2-q}_{K,\log})}{\Fil_{0} H^1_\et(K, \Omega^{2-q}_{K,\log})}\right)$ have filtrations given by 
$$\left\{\frac{\kappa^1_q(A|i)}{\kappa^1_q(A|n)} \right\}_{i \le n} \text{  and  } \left\{ \left(\frac{\Fil_{i} H^1_\et(K, \Omega^{2-q}_{K,\log})}{\Fil_{0} H^1_\et(K, \Omega^{2-q}_{K,\log})}\right) \right\}_{i \le n-1} \text{,  respectively.}$$

It suffices to show that for all $i \le n-1$, the following map induced by $\rho^{q,n}_{K,1}$
$$ \frac{\kappa^1_q(A|i)}{\kappa^1_q(A|i+1)} \to \left(\frac{\Fil_{i} H^1_\et(K, \Omega^{2-q}_{K,\log})}{\Fil_{i-1} H^1_\et(K, \Omega^{2-q}_{K,\log})} \right)^\star$$
is injective.
But this follows from \cite[Lem~22.2 and 22.3]{Kato-Inv} by identifying the above map with the map induced by the following pairings which don't have any left and right kernels. 
$$  \Omega^1_{\ff} \times \ff \to \Omega^1_{\ff}\xrightarrow{\text{Tr o res}}\F_p, \ \ \ \text{if } p\nmid i;  \hspace{2cm} \ff/\ff^p \times \ff/\ff^p \to \Omega^1_\ff \xrightarrow{\text{Tr o res}}\F_p, \ \ \ \text{ if } p|i $$
$$(v,w)\mapsto \text{Tr(res}(v.w));\ \ \ \ \hspace{2.2cm} (v,w)\mapsto \text{Tr(res}(dv.w)).$$
\end{proof}
We conclude this chapter with a stronger claim about the dual of $\wt{\rho}^{q,n}_{K,m}$, which we can 
easily deduce from what Kato already showed about $(\rho^{q}_K)^*$.

\begin{lem}\label{lem:Dual-iso}
    For $m \ge 1$, $n \ge 0$ and $q=1,2$, the map
  \[
    (\wt{\rho}^{q,n}_{K,m})^\star \colon
\frac{\Fil_{n-1} H^1_\et(K, W_m\Omega^{2-q}_{K,\log})}{\Fil_{0} H^1_\et(K,  W_m\Omega^{2-q}_{K,\log})} 
    \to (\varpi^m_q(A|n))^\star
  \]
  is an isomorphism, where the dual on the right
  is taken with respect to the Kato topology.
\end{lem}

\begin{proof}
     We can assume that $K$ is complete.
We consider the commutative diagram
  \begin{equation}\label{eqn:Dual-iso-0}
    \xymatrix@C.8pc{
      H^1_\et(K, W_m\Omega^{2-q}_{K,\log})
      \ar[r] \ar[d] &
      \frac{H^1_\et(K, W_m\Omega^{2-q}_{K,\log})}{\Fil_{0} H^1_\et(K, W_m\Omega^{2-q}_{K,\log})} 
      \ar[r] \ar[d]^-{(\rho^{q,1}_{K,m})^\star} & 0 \\
      ({K^M_q(K)}/{p^m})^\star \ar[r] &
      (\kappa^m_q(A|1))^\star \ar[r] &  0,}
  \end{equation}
  where the vertical arrows are the duals of the reciprocity maps.
  The bottom row is exact by Remark~\ref{remk:K-dual-1}.
  The left vertical arrow is bijective
  by \cite[\S~3, Thm.~3]{Kato80-2} and \lemref{lem:Kato-fil-9}. It follows that the right vertical arrow is surjective. It is also injective by \cite[Rem.~6.6]{Kato-89}, noting that $\Fil_{0} H^1_\et(K, W_m\Omega^{2-q}_{K,\log})=\Fil^{bk}_{0} H^{3-q}(K) \cap H^1_\et(K, W_m\Omega^{2-q}_{K,\log})$, which is \thmref{thm:Kato-fil-10}.
  In particular,  $(\wt{\rho}^{n}_{K,m})^\star$ is injective.

To prove that $(\wt{\rho}^{n}_{K,m})^\star$ is surjective,
let $\chi \in  (\varpi^m_2(A|n))^\star$. The surjectivity of
  $(\rho^{q,1}_{K,m})^\star$ implies that there exists an element $\chi' \in
  \frac{H^1_\et(K, W_m\Omega^{2-q}_{K,\log})}{\Fil_{0} H^1_\et(K, W_m\Omega^{2-q}_{K,\log})}$ such that
  $(\rho^{q,1}_{K,m})^\star(\chi') = \chi$. We only have to show that
  $\chi' \in \frac{\Fil_{n-1} H^1_\et(K, W_m\Omega^{2-q}_{K,\log})}
  {\Fil_{0} H^1_\et(K, W_m\Omega^{2-q}_{K,\log})}$. But this again follows from
  \cite[Rem.~6.6]{Kato-89} and \thmref{thm:Kato-fil-10}, where the latter one asserts $\Fil_{n} H^1_\et(K, W_m\Omega^{2-q}_{K,\log})=\Fil^{bk}_{n} H^{3-q}(K) \cap H^1_\et(K, W_m\Omega^{2-q}_{K,\log})$.
\end{proof}

\chapter{Duality theorem for Hodge-Witt cohomology over local fields}\label{chap:duality-localfields}
In this chapter, we shall prove a duality theorem for the cohomology of logarithmic Hodge-Witt sheaves with modulus for curves over local fields. Due to the lack of duality results for logarithmic Hodge-Witt cohomology (without modulus) for varieties over local fields, except in dimension one (cf. \cite[Prop~4]{Kato-Saito-Ann}, \cite[Thm~4.7]{KRS}), we focus solely on the case of curves. Let $X$ be a smooth projective geometrically connected curve over a local field $k$ of characteristic $p>0$. We let $E = \{x_1, \ldots , x_r\}$, $j \colon U \inj X$, $D_\n$ be as in \S~\ref{sec:Complexes}. Also, let $A_x :=\sO^h_{X,x}$, $K_x :=Q(A_x)$ and $X^h_x=\Spec A_x$ for any $x \in X_{(0)}$.

Since the cohomology group $H^i_\et(X,W_m\Omega^q_{X|D_\n,\log})$ is not finite in general (unlike the case in \S~\ref{sec:DT}), we choose a different approach to prove the duality theorem in this case. We begin by recalling some pairings of complexes.

\section{Pairing of complexes}
By \lemref{lem:Pair-3} and \corref{cor:Pair-4}, we have a commutative diagram (for $\n \ge \un{0}$)
\begin{equation}\label{eqn:comp-4}
    \xymatrix@C.8pc{
         W_m\Omega^{2-q}_{X|D_{\un{1}},\log} \times W_m\sF^{q,\bullet}_{\un{0}} \ar@<7ex>[d] \ar[r]^-{\cup} & W_m\sH^{2,\bullet} \ar@{=}[d] \\
        W_m\Omega^{2-q}_{X|D_{\un{n}+1}, \log} \times W_m\sF^{q,\bullet}_{\un{n}}  \ar@<7ex>[u]
       \ar[r]^-{\cup} &  W_m\sH^{2,\bullet}.}
\end{equation}
For any $x \in X_{(0)}$, let $\tau_x^h \colon \Spec(k(x)) \inj X^h_x$ and
$\tau_x \colon \Spec(k(x)) \inj X$ be the inclusion maps, and let
$j_x \colon X^h_x \to X$ be the canonical {\'e}tale map. 
 Recall (cf. \cite[\S~3.1]{Zhao}) that any pairing of complexes of sheaves
$\sA^\cdot \otimes \sB^\cdot \to \sC^\cdot$ on $X_\et$ gives rise to a pairing
$\tau_x^* \sA^\cdot \otimes^L \tau_x^{!} \sB^\cdot \to
\tau_x^{!} \sC^\cdot$ in $\sD(k(x)_\et)$. Applying this to ~\eqref{eqn:comp-4},
we get a commutative diagram of pairings 
\begin{equation}\label{eqn:G-comp-5}
    \xymatrix@C.8pc{
        \tau_x^*W_m\Omega^{2-q}_{X|D_{\un{1}},\log}\ \ \times \ \ \tau_x^{!}W_m\sF^{q,\bullet}_{\un{0}} \ar@<7ex>[d] \ar[r]^-{\cup} & \tau^{!}_x W_m\sH^{2,\bullet} \ar@{=}[d] \\
        \tau_x^*W_m\Omega^{2-q}_{X|D_{\un{n}+1}, \log} \ \ \times \ \   \tau_x^{!}W_m\sF^{q,\bullet}_{\un{n}}  \ar@<7ex>[u]
       \ar[r]^-{\cup} & \tau^{!}_x W_m\sH^{2,\bullet} }
\end{equation}
in  $\sD(k(x)_\et)$ which is compatible with ~\eqref{eqn:comp-4}. We shall write $W_m\sF^{q,\bullet}_\n$ as $W_m\sF^{q}_\n$ in the sequel.

\section{Relation with local duality}\label{sec:LG-dual}
For any
$\un{n} \ge 1, q \ge 0$ and any smooth scheme $X$ (need not be a curve) over a field, we let $W_m\sW^q_{\un{n}}$ be defined by the following short exact sequence 
\begin{equation}\label{eqn:Coker-0}
  0 \to  W_m\Omega^q_{X|D_{\un{n}}, \log} \to W_m\Omega^q_{X|D_{\un{1}},\log} \to
  W_m\sW^q_{\un{n}} \to 0,
\end{equation}
of Nisnevich (resp. {\'e}tale) sheaves on $X$. We recall the following result without proof, as it will be used in various places throughout the thesis.
\begin{prop}\label{prop:JSZ}
Let $X, D_\n$ be as above. For $\nu \in \{1,\ldots,r\}$, we let $\delta_\nu=(0,\dots,1,\dots,0)$ $\in \N^r$, where $1$ is in the $\nu$-th place. For $\n \ge \un 1, q  \ge 0$ and $m\ge 1$, we define $$\gr^{\un n,\nu}W_m\Omega^q_{X\log}:=W_m\Omega^q_{X|D_{\n},\log}/W_m\Omega^q_{X|D_{\n+\nu},\log},$$  
 as Nisnevich (resp. {\'e}tale) sheaf on $X$. Then we have
 \begin{enumerate}
     \item $\gr^{\un n,\nu}W_m\Omega^q_{X\log}$ is a coherent $\sO_{D_\nu}^{p^e}$-module, for some $e>>0$.
     \item $W_m\sW^q_{\un{n}}$ has a filtration whose successive quotients are of the form $\gr^{\un n,\nu}W_m\Omega^q_{X\log}$.
 \end{enumerate}
\end{prop}
\begin{proof}
    We may assume $X$ is local. In this case, Item (1) is same as \cite[Prop.~2.2.5]{Kerz-Zhao} (see also \cite[Prop.~1.1.9]{JSZ}). Item (2) follows from \eqref{eqn:Coker-0}. 
\end{proof}
From now on, we restrict ourselves to the case of curves over local fields, assuming that $X$ is as defined at the beginning of this chapter.
In this case, we can write
$$W_m\sW^q_{\un{n}} = \ {\underset{x \in E}\bigoplus} (\tau_x)_* \circ
(\tau_x)^*(W_m\sW^q_{\un{n}}).$$
In particular, the proper base change theorem for the
torsion {\'e}tale sheaves implies that
\begin{equation}\label{eqn:W-n}
    H^i_\et(X, W_m\sW^q_{\un{n}}) = {\underset{x \in E}\bigoplus}
H^i_\et (X^h_x, W_m\sW^q_{\un{n}}).
\end{equation} 
For $n \ge 0$, we let $$W_m\sV^q_{\un{n}} = \left[\frac{Z_1 \Fil_{D_{\un{n}}}W_m\Omega^q_U}{Z_1 W_m\Omega^q_X(\log E)}
  \xrightarrow{1-C} \frac{\Fil_{D_{\un{n}}}W_m\Omega^q_U}{W_m\Omega^q_X(\log E)}\right]$$
  so that we have a
term-wise short exact sequence of complexes on $X_\nis$ (resp. $X_\et$):
\begin{equation}\label{eqn:G-comp-3}
  0 \to W_m\sF_{\un{0}}^q \to W_m\sF^q_{\un{n}} \to W_m\sV^q_{\un{n}} \to 0.
  \end{equation}

By \cite[Prop.~4.4]{KRS} (see also \cite[\S~3, Prop.~4]{Kato-Saito-Ann}),
there is a bijective trace homomorphism
$\Tr_X \colon H^2_\et(X, W_m\Omega^2_{X, \log}) \xrightarrow{\cong} {\Z}/{p^m}$. By definition of the pairing between complexes of sheaves (cf.
\cite[\S~1, p.~175]{Milne-Duality}) the pairing on the top row of ~\eqref{eqn:comp-4}
is equivalent to a morphism of complexes $W_m\Omega^{2-q}_{X|D_{\un{1}},\log}  \to
\sHom^\cdot(W_m\sF^{q}_{\un{0}}, W_m \sH)$. 
Composing this with the canonical map $\sHom^\cdot(W_m\sF^{q}_{\un{0}}, W_m \sH) \to 
{\bf R}\sHom^\cdot(W_m\sF^{q}_{\un{0}}, W_m \sH)$, we get a morphism $\theta \colon
W_m\Omega^{2-q}_{X|D_{\un{1}}, \log} \to  
{\bf R}\sHom^\cdot(W_m\sF^{q}_{\un{0}}, W_m \sH)$ in $\sD(X_\et)$.
 Similarly, the bottom row of
~\eqref{eqn:comp-4} gives rise to
a morphism in the derived category
$\theta_\n \colon W_m\Omega^{2-q}_{X|D_{\n+1}, \log} \to
{\bf R}\sHom^\cdot(W_m\sF^{q}_{\n}, W_m \sH)$. Moreover, we have a
commutative diagram
\begin{equation}\label{eqn:G-comp-4}
  \xymatrix@C.8pc{
    W_m\Omega^{2-q}_{X|D_{\n+1}, \log} \ar[r]^-{\theta_{\n}} \ar[d] &
    {\bf R}\sHom^\cdot(W_m\sF^{q}_{\n}, W_m \sH) \ar[d]  \\
W_m\Omega^{2-q}_{X|D_{\un{1}}, \log} \ar[r]^-{\theta} &  {\bf R}\sHom^\cdot(W_m\sF^{q}_{\un{0}}, W_m \sH).}
\end{equation}

Letting $\psi_\n^{2-q}$ denote the induced map between the homotopy fibers, we
get a commutative diagram of distinguished triangles in $\sD(X_\et, {\Z}/{p^m})$:
\begin{equation}\label{eqn:G-comp-5.5}
  \xymatrix@C.8pc{
    W_m \sW^{2-q}_{\n+1}[-1] \ar[d]_-{\psi_{\un n}^{2-q}} \ar[r] &
    W_m\Omega^{2-q}_{X|D_{\n+1}, \log} \ar[d]_-{\theta_{\un n}} \ar[r] &
     W_m\Omega^{2-q}_{X|D_{\un 1}, \log} \ar[d]^-{\theta} \ar[r]^-{+1} & \\ 
     {\bf R}\sHom^\cdot(W_m \sV^{q}_{\n}, W_m \sH) \ar[r] &
     {\bf R}\sHom^\cdot(W_m\sF^{q}_{\n}, W_m \sH) \ar[r] &
     {\bf R}\sHom^\cdot(W_m\sF^{q}_{\un 0}, W_m \sH) \ar[r]^-{+1} & .}
 \end{equation}
  Now recall that, for $\sA,\sB,\sC \in \sD(X_\et,\Z/p^m)$, we have
 canonical isomorphism $${\bf R}\Hom^\cdot({\Z}/{p^m},
 {\bf R}\sHom^\cdot(\sB^\cdot, \sC^\cdot)) \xrightarrow{\cong}
 {\bf R}\Hom^\cdot(\sB^\cdot, \sC^\cdot),$$ (cf. \cite[Prop.~4.3]{Gamst})
 and also the canonical map
 $${\bf R}^i\Hom^\cdot(\sB^\cdot, \sC^\cdot) \to
 \Hom({\bf R}^j\Hom({\Z}/{p^m}, \sB^\cdot),
{\bf R}^{i+j}\Hom({\Z}/{p^m}, \sC^\cdot)),$$ for $i, j\in \Z$. So, any map $\sA \to  {\bf R}\sHom^\cdot(\sB^\cdot, \sC^\cdot)$ in $\sD(X_\et,\Z/p^m)$ induces a canonical map 
$$H^i_\et(X,\sA)= {\bf R}^i\Hom^\cdot({\Z}/{p^m},
 \sA) \to {\bf R}^i\Hom^\cdot(\sB^\cdot, \sC^\cdot) \to$$
$$\to \Hom({\bf R}^{2-i}\Hom({\Z}/{p^m}, \sB^\cdot),
{\bf R}^{2}\Hom({\Z}/{p^m}, \sC^\cdot))= \Hom(H^{2-i}_\et(X,\sB), H^{2}_\et(X,\sC)).$$
Thus, applying ${\bf R}^i\Hom^\cdot({\Z}/{p^m},
 (-))$ to \eqref{eqn:G-comp-5.5} and using the above canonical maps, we get a commutative diagram of long exact hypercohomology sequences 
\begin{equation}\label{eqn:G-comp-6}
  \xymatrix@C.5pc{
    0 \ar[r] & H^0(W_m\Omega^{{2-q}}_{X|D_{{\n+1}}, \log}) \ar[r]
    \ar[d]_-{\theta^*_{{\n}}} &
    H^0(W_m\Omega^{2-q}_{X|D_{\un{1}}, \log}) \ar[r] \ar[d]^-{\theta^*} &
    H^0(W_m \sW^{2-q}_{{\n+1}}) \ar[r]^-{\partial}
    \ar[d]^-{(\psi^{2-q}_{{\n}})^*} &
    H^1(W_m\Omega^{2-q}_{X|D_{\n+1}, \log}) \ar[r]
    \ar[d]^-{\theta^*_{{\n}}} & \cdots \\
    0 \ar[r] & \H^2(W_m\sF^{q}_{\n})^\vee \ar[r] &
    \H^2(W_m\sF^{q}_{\un{0}})^\vee \ar[r] & \H^1(W_m\sV^{q}_{\n})^\vee \ar[r] &
    \H^1(W_m\sF^{q}_{\n})^\vee \ar[r] & \cdots ,}
\end{equation}
where $n \ge 0$.

If $x \in E$, then ~\eqref{eqn:G-comp-5} gives rise to a commutative diagram
\begin{equation}\label{eqn:G-comp-8}
  \xymatrix@C.8pc{
\tau^*_x W_m\Omega^{2-q}_{X|D_{\un n+1}, \log} \ar[r]^-{\theta_{\n,x}} \ar[d] &
{\bf R}\sHom^\cdot(\tau^{!}_x W_m\sF^{q}_{\n}, \tau^{!}_x
W_m \sH) \ar[d] \\
\tau^*_x W_m\Omega^{2-q}_{X|D_{\un 1}, \log} \ar[r]^-{\theta_x} &
{\bf R}\sHom^\cdot(\tau^{!}_xW_m\sF^{q}_{\un 0}, \tau^{!}_xW_m \sH),}
\end{equation}
which is compatible with ~\eqref{eqn:G-comp-4}.

Since $\tau^*_x$ is exact, it is easy to check that the homotopy fibers of
the left and the right vertical arrows are $\tau^*_x  W_m \sW^{2-q}_{\n}[-1]$
and ${\bf R}\sHom^\cdot(\tau^{!}_x W_m \sV^q_{\n}, \tau^{!}_x
W_m \sH)$, respectively. We also note that the maps
$j^*_x W_m \sW^{2-q}_{\n} \to (\tau^h_x)_* \tau^*_x  W_m \sW^{2-q}_{\n}$
and $(\tau^h_x)_* \tau^{!}_x W_m \sV^q_{\n} \to j^*_x W_m \sV^q_{\n}$
are isomorphisms on the {\'e}tale (and Nisnevich) site of $X^h_x$.
Since the canonical (i.e., forget support) map
$H^2_x(X, W_m\Omega^2_{X, \log}) \to H^2_\et(X, W_m\Omega^2_{X, \log})$ is bijective (cf. \cite[Prop.~4.4]{KRS}),
letting $\psi_{\n,x}^{2-q}$ denote the induced map between the mapping cones
of the vertical arrows in ~\eqref{eqn:G-comp-8}, considering the
corresponding maps between the hypercohomology and using the proper base change,
we obtain a commutative diagram 
\begin{equation}\label{eqn:G-comp-9}
  \xymatrix@C.5pc{
 0 \ar[r] & G^{{2-q},0}_{\n}(x) \ar[r]
    \ar[d]_-{\theta^*_{{\n}}} &
    H^0(W_m\Omega^{2-q}_{X^h_x|D_{\un{1}}, \log}) \ar[r] \ar[d]^-{\theta^*_x} &
    H^0(\tau^*_x W_m \sW^{2-q}_{{\n+1}}) \ar[r]^-{\partial}
    \ar[d]^-{(\psi^{2-q}_{{\n},x})^*} & G^{{2-q},1}_{\n}(x) \ar[r]
    \ar[d]^-{\theta^*_{{\n},x}} & \cdots \\
    0 \ar[r] & \H^2_x(W_m\sF^{q}_{\n})^\vee \ar[r] &
    \H^2_x(W_m\sF^{q}_{\un{0}})^\vee \ar[r] & \H^1_x(W_m\sV^{q}_{\n})^\vee \ar[r] &
    \H^1_x(W_m\sF^{q}_{\n})^\vee \ar[r] & \cdots }
\end{equation}
of {\'e}tale cohomology on $X^h_x$ whose rows are exact and which is compatible with
~\eqref{eqn:G-comp-6}. Here, we let
$G^{q,i}_{\n}(x) = H^i_\et(X^h_x, W_m\Omega^{q}_{X^h_x|D_{{\n+1}}, \log})$.

\subsection{The maps \texorpdfstring{$(\psi_{\n,x}^q)^*$}{Psi} vs. \texorpdfstring{$\wt{\rho}^{q,\n+1}_{K,m}$}{Rho}}
\label{sec:Loc-der-rec}
Our goal now is to compute the cohomology of $W_m\sW^q_{\n}$ and $W_m\sV^q_\n$,
and show that $(\psi^q_{\n, x})^*$ coincides with the reciprocity map
$\wt{\rho}^{q,n_x+1}_{K,m}$ of \lemref{lem:Local-inj-6}.
We write $D_\n = \sum\limits_{x \in E} n_x [x]$, where $n_x$ is the multiplicity of $D_\n$ at $x\in E$ (i.e. $n_x=n_i$ if $x=x_i\in E$).

  \begin{lem}\label{lem:Coker-1}
Let $q\ge 0, \n \ge \un 1$. For $x \in D_{\un{n}}$ , we have the following.
\begin{enumerate}
\item
$H^i_x(X, W_m\sW^q_{\un{n}}) \xrightarrow{\cong} H^i_\et(X^h_x, W_m\sW^q_{\un{n}})$, for all $i \ge 0$. 
\item
\[
H^0_\et(X^h_x, W_m\sW^q_{\un{n}}) = \left\{\begin{array}{ll}
0 & \mbox{if $q = 0$} \\
\varpi^m_q(A_x|n_x)  & \mbox{if $q \ge 1$}.
\end{array}\right.
\]
\item
$H^i_\et(X^h_x, W_m\sW^q_{\un{n}}) = 0$ for $i \ge 1$ and $q \ge 0$.
\end{enumerate}
\end{lem}
\begin{proof}
  The items (1)  and (2) are clear from the definition of $W_m\sW^q_{\un{n}}$ while
 (3) follows from \propref{prop:JSZ} which says
  that $ W_m\sW^q_{\un{n}}$ has a finite decreasing filtration whose graded quotients
  are coherent sheaves on $D_{\un{n}}$. 
\end{proof}

\begin{lem}\label{lem:support}
     For $x\in X^{(1)}$, we have 
    $$H^1_x(X^h_x, W_m\Omega^q_{X|D_\n,\log}) \cong \frac{ K^M_q(K_x)}{\Fil_{n_x} K^M_q(K_x)+p^mK^M_q(K_x)}. $$
\end{lem}
\begin{proof}
    If $x \not \in D_\n$, then $H^1_x(X^h_x, W_m\Omega^q_{X|D_\n,\log})=H^1_x(X^h_x, W_m\Omega^q_{X,\log})$ $= H^0(k(x),W_m\Omega^{q-1}_{X,\log} )$ $= K^M_{q-1}(k(x))/p^m$, where the second equality is a consequence of the Purity theorem \cite[Thm~3.2]{Shiho} for $W_m\Omega^q_{X_x^h,\log}$ on $X^h_{x,\et}$. But this is same as the right hand side group in the lemma, because $n_x=0$ and so $\Fil_{n_x} K^M_q(K_x)=K^M_q(A_x)$. If $x \in D_\n$, we claim 
   $H^1_\et(X^h_x, W_m\Omega^q_{X_x^h|D_{n_x},\log})=0$ for all $n_x \ge 1$.  By induction on $m$ and \corref{cor:Complex-7}, it is enough to prove the claim for $m=1$. In this case, it is enough to show the map (cf. \cite[Thm~1.2.1]{JSZ})
   $$1-\ov F: \Fil_{-n_x}\Omega^q_{A_x} \to \Fil_{-n_x}\Omega^q_{A_x}/d(\Fil_{-n_x}\Fil_{n_x}\Omega^{q-1}_{A_x}), \ \ X_x^h=\Spec A_x;$$ is surjective. 
   From the proof of \lemref{lem:VR-4}, we note that the abelian group $\Fil_{-n_x}\Omega^q_{A_x}$ is generated by elements of the form $a\dlog(x_1)\wedge\cdots\wedge\dlog(x_q)$, where $a \in \Fil_{-n_x}K_x=\pi_x^{n_x}A_x$ and $x_i\in K_x$, for all $i$. Here $\pi_x$ is a uniformizer of $A_x$. Now, it is enough to show that $1-\ov F: \Fil_{-n_x}K_x \to \Fil_{-n_x}K_x$ is surjective. Let $a \in \pi_x^{n_x}A_x$. Since $\ov a$, the image of $a$ is zero in $k(x)$, by Hensel lemma, there exists $y \in A_x$ such that $y^p-y-a=0$ in $A_x$. Let $\ov y$ be the image of $y$ in $k(x)$. Then, $\ov y^p-\ov y=0 \in k(x)$, which implies $\ov y \in \F_p \subset A_x$. Then letting $y_1=y-\ov y$, we get $y_1(y_1^{p-1}-1)=a \in \pi_x^{n_x}A_x$. But, note that $y_1 \in \pi_xA_x$, which implies $(y_1^{p-1}-1)$ is a unit in $A_x$. Hence, we conclude that $y_1 \in \pi_x^{n_x}A_x$ and the claim follows.
  
  Now the corollary follows from the localisation sequence for $W_m\Omega^q_{X_x^h|D_{n_x},\log}$.
  \end{proof}

\begin{lem}\label{lem:Coh-V0}
  For $x \in E, \n \ge \un{0}$, we have 
  \begin{enumerate}
      \item $H^i_\et(X^h_x, Z_1 \Fil_{D_{\un{n}}} W_m \Omega^q_U) =
  H^i_\et(X^h_x, \Fil_{D_{\un{n}}} W_m \Omega^q_U) = 0, \text{ for $i \ge 1$.}$ 
  \item \[
\H^i_\et(X^h_x, W_m\sF^q_{\un{n}}) = \left\{\begin{array}{ll}
K^M_q(K_x)/p^m & \mbox{if $i = 0$} \\
\Fil_{n_x} H^1_\et(K_x, W_m\Omega^q_{K_x, \log}) & \mbox{if $i = 1$} \\
 0 & \mbox{if $i > 1$.}
\end{array}\right.
\]
  \end{enumerate}
\end{lem}

\begin{proof}
    This follows from \lemref{lem:hyp} and \thmref{thm:Kato-fil-10}.
\end{proof}
\begin{cor}\label{cor:Coh-V0-00}
    We have 
    $$\H^i_x(X^h_x, W_m\sF^q_{\un{n}}) = \left\{\begin{array}{ll}
\frac{ H^1_\et(K_x, W_m\Omega^q_{K_x, \log})}{\Fil_{n_x} H^1_\et(K_x, W_m\Omega^q_{K_x, \log})} & \mbox{if $i = 2$} \\
 0 & \mbox{if $i \neq 2$.}
\end{array}\right. $$
\end{cor}

\begin{proof}
    First note that, the map $\H^i_\et(X^h_x, W_m\sF^q_\n) \to
    \H^i_\et(K_x, W_m\sF^q_\n)$ is injective by \thmref{thm:Kato-fil-10}. Now the corollary follows from \lemref{lem:Coh-V0} and localisation exact sequence
     \[
    \cdots \to \H^i_x(X, W_m\sF^q_\n) \to  \H^i_\et(X^h_x, W_m\sF^q_\n) \to
    \H^i_\et(K_x, W_m\sF^q_\n) \to \H^{i+1}_x(X, W_m\sF^q_\n) \to \cdots .
  \]
  We also observe that $\H^i_\et(K_x, W_m\sF^q_\n)=H^i_\et(K_x, W_m\Omega^q_{K_x,\log})=0$ for $i \ge 2$, because $cd_p(K_x)\le 1$.
\end{proof}
\begin{lem}\label{lem:Coh-V1}
    We have
    \[
      \H^i_\et(X, W_m\sV^q_{\un{n}}) = \left\{\begin{array}{ll}
    {\underset{x \in E}\bigoplus}
\frac{\Fil_{n_x} H^1_\et(K_x, W_m\Omega^q_{K_x, \log})}{\Fil_0 H^1_\et({K}_x, W_m\Omega^q_{K_x, \log})}
& \mbox{if $i = 1$} \\
 0 & \mbox{otherwise.}
\end{array}\right.
\]
\end{lem}
\begin{proof}
Since 
$\H^i_\et(X,  W_m\sV^q_{\un{n}}) \cong {\underset{x \in E}\bigoplus}
\H^i_\et(X^h_x, W_m\sV^q_{\un{n}})$, we can replace $X$ by $X^h_x$
in order to prove the lemma. We now consider the long exact sequence
  \begin{equation}\label{eqn:Coh-V1-0}
\cdots \to \H^0_\et(W_m\sF_{\un 0}^q) \xrightarrow{\alpha_0} \H^0_\et(W_m\sF^q_{\un{n}}) \to \H^0_\et(W_m\sV^q_{\un{n}}) \to \H^{1}_\et(W_m\sF_{\un 0}^q) \xrightarrow{\alpha_1} \cdots ,
\end{equation}

where $\H^i_\et(W_m\sF_{\un 0}^q)$ (and similarly the other groups) is the short hand notation for $\H^i_\et(X^h_x,W_m\sF_{\un 0}^q)$ for all i. It follows from ~\eqref{eqn:Coh-V0-2}
that $\alpha_0$ is a bijection between two groups, each isomorphic to
  $K^M_q(K_x)/p^m$.
  
  ~\lemref{lem:Coh-V0} and \thmref{thm:Kato-fil-10} imply that $\alpha_1$ is injective. It follows that
  $\H^0_\et(X^h_x, W_m\sV^q_{\un{n}})$ $= 0$. We now apply \lemref{lem:Coh-V0} and
  ~\eqref{eqn:Coh-V1-0} to conclude the proof.

\end{proof}

\begin{lem}\label{lem:Coh-V2}
  The diagram
  \begin{equation}\label{eqn:Coh-V2-0}
    \xymatrix@C.8pc{
      H^0_\et(X, W_m \sW^q_{{\n+1}}) \ar[r] 
      \ar[d]_-{(\psi^q_{{\n}})^*} & {\underset{x \in E}\bigoplus}
      \varpi^m_q(A_x|n_x+1) \ar[d]^-{\underset{x \in E}{\oplus}\wt{\rho}^{q,n_x +1}_{K_x, m} } \\
      \H^1_\et(X, W_m\sV^{2-q}_{\n})^\vee \ar[r] &
      {\underset{x \in E}\bigoplus} \left(\frac{\Fil_{n_x}
       H^1_\et(K_x, W_m\Omega^{2-q}_{K_x, \log})}{\Fil_{0} H^1_\et(K_x, W_m\Omega^{2-q}_{K_x,\log})}\right)^\vee}
  \end{equation}
  commutes, where the horizontal arrows are canonical isomorphisms induced by Lemmas \ref{lem:Coker-1}, \ref{lem:Coh-V1} and \eqref{eqn:W-n}.
\end{lem}

\begin{proof}
The diagram ~\eqref{eqn:Coh-V2-0} is the composition of
  two squares
  \begin{equation}\label{eqn:Coh-V2-1}
    \xymatrix@C.8pc{
      H^0_\et(X, W_m \sW^q_{\n+1}) \ar[r] \ar[d]_-{(\psi^q_{\n})^*} &
      {\underset{x \in E}\bigoplus}
      H^0_\et(X^h_x, W_m \sW^q_{\n+1}) \ar[r] \ar[d]^-{\oplus \ (\psi^q_{\n,x})^*} &
      {\underset{x \in E}\bigoplus} \varpi^m_q(K_x|n_x+1)
      \ar[d]^-{\oplus \ \wt{\rho}^{q,n_x +1}_{K_x, m}} \\
      \H^1_\et(X, W_m\sV^{2-q}_{\n})^\vee \ar[r] &   
{\underset{x \in E}\bigoplus}
      \H^1_x(X^h_x, W_m\sV^{2-q}_{\n})^\vee \ar[r] &
      {\underset{x \in E}\bigoplus} \left(\frac{\Fil_{n_x}
       H^1_\et(K_x, W_m\Omega^{2-q}_{K_x,\log})}{\Fil_{0} H^1_\et(K_x, W_m\Omega^{2-q}_{K_x,\log})}\right)^\vee,}
\end{equation}
Since ~\eqref{eqn:G-comp-6} and ~\eqref{eqn:G-comp-9} are compatible, it follows
that the left square in ~\eqref{eqn:Coh-V2-1} is commutative.
It remains to show that the right square commutes.
We can therefore assume $E = \{x\}$. We write $\un{n} = n$.

Now, it follows from \lemref{lem:Coker-1} that the map
  $\kappa^m_q(A_x|1) \cong H^0_\et(X^h_x, W_m\Omega^q_{X|D_1,\log})$ $ \to
  H^0_\et(X^h_x, W_m\sW^q_{n+1})$ is surjective.
  \corref{cor:Coh-V0-00} (with $\n = 0$) and
  \lemref{lem:Coh-V1} imply that
  $\H^2_x(X^h_x, W_m\sF^{2-q}_0)^\vee \to \H^1_x(X^h_x, \sV^{2-q}_{n})^\vee$
  is also surjective. We conclude that ~\eqref{eqn:G-comp-9} breaks into
  a commutative diagram of short exact sequences
  \begin{equation}\label{eqn:Coh-V2-2}
\xymatrix@C.8pc{
 0 \ar[r] & H^0_\et(X^h_x, W_m\Omega^q_{X^h_x|D_{n+1}, \log}) \ar[r]
    \ar[d]_-{\theta^*_{n}} &
    H^0_\et(X^h_x, W_m\Omega^q_{X^h_x|D_1, \log}) \ar[r] \ar[d]^-{\theta^*_x} &
    H^0_\et(X^h_x, W_m \sW^q_{n+1}) \ar[r]
    \ar[d]^-{(\psi^q_{n,x})^*} \ar[r] & 0  \\
    0 \ar[r] & \H^2_x(X^h_x, W_m\sF^{2-q}_{n})^\vee \ar[r] &
    \H^2_x(X^h_x, W_m\sF^{2-q}_0)^\vee \ar[r] &
    \H^1_x(X^h_x, W_m\sV^{2-q}_{n})^\vee \ar[r] & 0.}
\end{equation}

By definition of Kato's reciprocity map (cf. \S~\ref{sec:RHom}), it follows
from \corref{cor:Coh-V0-00} that
  the left square in ~\eqref{eqn:Coh-V2-2} is the same as the left square in
  the commutative diagram
  \begin{equation}\label{eqn:Coh-V2-4}
\xymatrix@C.8pc{
  0 \ar[r] & \kappa^m_q(A_x|n+1) \ar[r] \ar[d]_-{\rho^{q,n+1}_{K_x, m}} &
  \kappa^m_q(A_x|1) \ar[d]^-{\rho^{q,1}_{K_x, m}} \ar[r] &
\varpi^m_q(A_x|n+1) \ar[d]^-{\wt{\rho}^{q,n+1}_{K_x,m}} \ar[r] & 0 \\
0 \ar[r] & G^{3-q,n+1}_{K_x,m} \ar[r] & G^{3-q,1}_{K_x,m} \ar[r] &
\left(\frac{\Fil_{n} H^1_\et(K_x, W_m\Omega^{2-q}_{K_x,\log})}
{\Fil_{0} H^1_\et(K_x, W_m\Omega^{2-q}_{K_x,\log})}\right)^\vee\ar[r] & 0.}
\end{equation}
From this, it immediately follows that
$(\psi^q_{n,x})^* = \wt{\rho}^{q,n+1}_{K_x,m}$.
\end{proof}

\section{Topologies on the cohomology groups}\label{sec:Top-1}
Our next goal is to endow the cohomology groups $H^i_\et(X, W_m\Omega^q_{X|D_{\n+1}, \log})$ and $\H^j_\et(X, W_m\sF^q_\n)$ with the structure of topological abelian groups so that we get a perfect pairing of topological abelian groups consisting them. We fix some notations. 

\begin{notat}
    For $\n, i, j \ge 0, 1 \le q \le 2$, we let  $$G^{q,i}_\n(m) := H^i_\et(X, W_m\Omega^q_{X|D_{\n+1}, \log})
\text{     ,     } F^{q,j}_\n(m) := \H^j_\et(X, W_m\sF^q_\n).$$ Note that by \lemref{lem:Complex-4}, 
$G^{2,i}_{\un{0}}(m) = H^i_\et(X, W_m\Omega^2_{X,\log})$ and
$F^{0,j}_{\un{0}}(m) = H^j_\et(X, {\Z}/{p^m})$.
It is clear that $G^{q,i}_\n(m) = 0$ for $i \ge 3$
(cf. \cite[Chap.~VI, Rem.~1.5]{Milne-EC}). As for $i =2$, it follows from Lemmas~\ref{lem:Coker-1}, ~\ref{lem:Coh-V1} and ~\eqref{eqn:G-comp-6} that 
\begin{equation}\label{eqn:Top--1}
    G^{q,2}_\n(m)=G^{q,2}_{\un{0}}(m)  , \ F^{q,0}_\n(m)=F^{q,0}_{\un{0}}(m), \text{ for $n \ge 0$}.
\end{equation}
\end{notat}

\subsection{Topology of \texorpdfstring{$H^i_\et(X, W_m\Omega^q_{X|D_{\n+1}, \log})$}{the cohomology groups} for \texorpdfstring{$i \neq 1$}{i not 1}}

We consider four cases.

\textbf{Case 1 $\left( i=2,q=2 \right)$:} The pairing ~\eqref{eqn:comp-4} induces an isomorphism
$G^{2,2}_{\un{0}}(m) \xrightarrow{\cong} F^{0,0}_{\un{0}}(m)^\vee \cong {\Z}/{p^m}$ by
\cite[\S~3, Prop.~4]{Kato-Saito-Ann} (see also \cite[Thm.~4.7]{KRS}).
We therefore endow $G^{2,2}_\n(m)$ with the discrete
topology via the isomorphism \eqref{eqn:Top--1} and hence 
\begin{equation}\label{eqn:Top-0.5}
G^{2,2}_\n(m) \times  F^{0,0}_{\n}(m) \to {\Z}/{p^m}.
\end{equation}
is a perfect pairing.

\textbf{Case 2 $(i=0,q=2)$:} First consider the case $\un n=0$. In this case, by \cite[\S~3, Prop.~4]{Kato-Saito-Ann} (see also \cite[Thm.~4.7]{KRS}), $F^{0,2}_{\un 0}(m)$ has discrete topology, $G^{2,0}_{\un 0}(m)$ has profinite topology and the map $G^{2,0}_{\un 0}(m) \to F^{0,2}_{\un 0}(m)^\vee$ is an isomorphism of profinite groups. It follows then from
\lemref{lem:Local-inj-6}, ~\lemref{lem:Coh-V2}, ~\eqref{eqn:G-comp-6} and
 that
the map $G^{2,0}_\n(m) \to F^{0,2}_{\n}(m)^\vee$ is an isomorphism.
Since $F^{0,2}_{\un{0}}(m)$ has discrete topology and
$F^{0,2}_{\n}(m)$ is its quotient,
the latter is a discrete topological abelian group. We therefore endow
$G^{2,0}_\n(m)$ with the profinite topology so that
\begin{equation}\label{eqn:Top-0}
G^{2,0}_\n(m) \times  F^{0,2}_{\n}(m) \to {\Z}/{p^m}.
\end{equation}
is a perfect pairing of topological abelian groups.
We also note that $\varinjlim\limits_{\n}F^{0,2}_\n(m)= H^2_\et(U, \Z/p^m)=0$ (by \thmref{thm:Global-version}(12)), because $U$ is affine and $cd_pU \le 1$. Hence 
\begin{equation}\label{eqn:Duality-open--1}
    \varprojlim\limits_{\n}G^{2,0}_{\n}(m) = (\varinjlim\limits_{\n}F^{0,2}_\n(m))^\vee=0.
\end{equation}
\textbf{Case 3 ($i=0,q=1$):} First consider the sub case $\n=\un{0}$. Since $W_m\sF^1_{\un{0}} \cong j_*W_m\Omega^1_{U,\log}$ by \lemref{lem:Complex-6}(4), we consider the following exact sequences.
\begin{equation}\label{eqn:res-6}
    0 \to W_m\Omega^1_{X|D_{\un{1}},\log} \to W_m\Omega^1_{X,\log} \to \bigoplus\limits_{x \in E} (\tau_{x})_ \star W_m\Omega^1_{\Spec\kappa(x), \log} \to 0,
\end{equation}
\begin{equation*}
  \hspace{1.5cm}  0 \to W_m\Omega^1_{X,\log} \to j_*W_m\Omega^1_{U,\log}\xrightarrow{\res} \bigoplus\limits_{x \in E} (\tau_{x})_ \star W_m\Omega^0_{\Spec\kappa(x), \log} \to 0,\hspace{1.2cm} 
\end{equation*}
where $\kappa(x)$ is the residue field of $x$ and $\res$ is induced by the residue map of Milnor K-theory.
This induces the following commutative diagram.
\begin{equation}
    \xymatrix@C.8pc{
    0 \ar[r] & H^0_\et(X,W_m\Omega^1_{X|D_1,\log}) \ar[r] \ar[d] & H^0_\et(X, W_m\Omega^1_{X,\log}) \ar[r] \ar[d] & \bigoplus\limits_{x \in E} \kappa(x)^\times /p^m  \ar[d]\\
    0 \ar[r] & \H^2_\et(W_m\sF^1_{\un{0}})^\vee \ar[r] & H^2_\et(X, W_m\Omega^1_{X,\log})^\vee \ar[r] & \bigoplus\limits_{x \in E} H^1_\et(\kappa(x),\Z/p^m)^\vee.
    }
\end{equation}

By a similar argument to that in \lemref{lem:Coh-V2}, we conclude that the right vertical arrow is the reciprocity map of $1$-local field $\kappa(x)$, which is known to be injective. The middle vertical arrow is isomorphism by \cite[Thm~4.7]{KRS}. This implies the left vertical arrow is also isomorphism. Since $H^2_\et(X, W_m\Omega^1_{X,\log})$ is discrete (cf. \cite[Prop.~3.6]{KRS}), $F^{1,2}_{\un{0}}(m)$ being its quotient, is a discrete group. Therefore $G^{1,0}_{\un{0}}(m)$ is endowed with profinite topology and
\begin{equation}
G^{1,0}_{\un{0}}(m) \times  F^{1,2}_{\un{0}}(m) \to {\Z}/{p^m}
\end{equation}
is a perfect pairing of topological abelian groups. Now by \lemref{lem:Local-inj-6}, ~\lemref{lem:Coh-V2}, and  ~\eqref{eqn:G-comp-6} we get $G^{1,0}_{\n}(m) \xrightarrow{\cong} F^{1,2}_{\n}(m)^\vee$ is isomorphism for $n \ge 0$. $F^{1,2}_{\n}(m)$ is discrete (being quotient of $F^{1,2}_{0}(m)$). So we endow $G^{1,0}_{\n}(m)$ with profinite topology and hence get a perfect pairing
\begin{equation}\label{eqn:Top-1}
G^{1,0}_{\n}(m) \times  F^{1,2}_{\n}(m) \to {\Z}/{p^m}
\end{equation}
for all $n \ge 0$.

Since $\varinjlim\limits_{\n}F^{1,2}_\n(m)= H^2_\et(U, W_m\Omega^1_{U, \log})=0$ (by \thmref{thm:Global-version}(12) and $cd_pU \le 1$), we have 
\begin{equation}\label{eqn:Duality-open-0}
    \varprojlim\limits_{\n}G^{1,0}_{\n}(m) = (\varinjlim\limits_{\n}F^{1,2}_\n(m))^\vee=0.
\end{equation}

\textbf{Case 4 ($i=2,q=1$):} Again first we consider the sub case $\n=0$. \eqref{eqn:res-6} induces a commutative diagram
\begin{equation}
    \xymatrix@C.5pc{
    0 \ar[r] & H^0_\et(X,W_m\Omega^1_{X,\log}) \ar[r] \ar[d]_-{1} & F^{1,0}_{\un{0}}(m) \ar[r] \ar[d]_-{2} & \bigoplus\limits_{x \in E} \Z/p^m \ar[r] \ar[d]_-{3} & H^1_\et(X,W_m\Omega^1_{X,\log}) \ar[d]_-{1'} \\
    0 \ar[r] & H^2_\et(X,W_m\Omega^1_{X,\log})^\vee \ar[r] & G^{1,2}_{\un{0}}(m)^\vee \ar[r] & \bigoplus\limits_{x \in E} H^1_\et(\kappa(x),W_m\Omega^1_{\kappa(x),\log})^\vee \ar[r] &  H^1_\et(X,W_m\Omega^1_{X,\log})^\vee, 
    }
\end{equation}
where $(1)$ is isomorphism and $(1')$ is injection by \cite[Prop.~3.6, Thm~4.7]{KRS}. $(3)$ is the dual of the isomorphism $H^1_\et(\kappa(x),W_m\Omega^1_{\kappa(x),\log}) \xrightarrow{inv_x} \Z/p^m$ (invariant map of class field theory of $\kappa(x)$) and so is isomorphism. This implies $(2)$ is isomorphism. By loc.cit, $H^0_\et(X,W_m\Omega^1_{X,\log})$ is profinite. Thus we endow  discrete topology on $G^{1,2}_{\un{0}}(m)$ so that $F^{1,0}_{\un{0}}(m) \xrightarrow{\cong} G^{1,2}_{\un{0}}(m)^\vee$ is an isomorphism of profinite group. Now by ~\eqref{eqn:Top--1}, we know $G^{1,2}_{\n}(m)=G^{1,2}_{\un{0}}(m)$ and $F^{1,0}_{\n}(m)=F^{1,0}_{\un{0}}(m)=H^0_\et(U,W_m\Omega^1_{U,\log})$. Hence we have perfect pairings
\begin{equation}\label{eqn:Top-2}
G^{1,2}_{\n}(m) \times  F^{1,0}_{\n}(m) \to {\Z}/{p^m}, \text{ for all } n \ge 0,
\end{equation}
$$\varprojlim_\n G^{1,2}_{\n}(m) \times H^0_\et(U,W_m\Omega^1_{U,\log}) \to \Z/p^m.$$

\subsection{Topology of \texorpdfstring{$H^1_\et(X, W_m\Omega^2_{X|D_{\n+1}, \log})$}{the cohomology groups for (i,j)=(1,2)}}
Now consider the case $(i=1,q=2)$. To define topology on $G^{2,1}_{\n}(m)$ we shall use a recipe.

Let $G$ and $G'$ be two abelian groups such that $G'$ is a torsion group. Let
$G' \times G \to {\Q}/{\Z}$ be a pairing which has no left kernel. Let
$\beta \colon G' \to G^\vee$ be the induced injective homomorphism.
We let $G$ be endowed with the weakest topology
with respect to which the map $\beta(f) \colon G \to  {\Q}/{\Z}$ is continuous
for every $f \in G'$ (called the weak topology induced by $G'$).
This makes $G$ a topological abelian group for which 
the identity element has a fundamental system of open neighborhoods of the form
$\Ker(\beta(f_1)) \cap \cdots \cap \Ker(\beta(f_r))$ with
$f_1, \ldots , f_r \in G'$.
It is clear that $\beta$ has a factorization $\beta \colon G' \inj G^\star$, where
the dual on the right is taken with respect to the weak topology.

\begin{lem}\label{lem:Top-1}
The map $\beta \colon G' \to G^\star$ is a bijection.
\end{lem}
\begin{proof}
  Let $f \in G^\star$. We can find $f_1, \ldots , f_r \in G'$ such that
  the continuous homomorphism $f \colon G \to {\Q}/{\Z}$ has the property that
  $H:= \stackrel{r}{\underset{i=1}\cap} \Ker(\beta(f_i)) \subset \Ker(f)$.
  We now look at 
  \begin{equation}\label{eqn:Top-1-0}
    \xymatrix@C.8pc{
      G \ar@{->>}[r] \ar[dr]_-{\Delta} \ar@/^2pc/[rrr]^-{f}
      & {G}/H \ar@{->>}[r] \ar@{^{(}->}[d] &
      \frac{G}{\Ker(f)} \ar[r]_-{\ov{f}} & {\Q}/{\Z} \\
      & \stackrel{r}{\underset{i =1}\bigoplus} \frac{G}{\Ker(\beta(f_i))}.
      \ar@/_2pc/[urr]_-{\wt{f}} & &}
  \end{equation}

 Since ${\Q}/{\Z}$ is injective, there exists $\wt{f}$ such that the above diagram
  is commutative. Since $G'$ is a torsion group, we see that each summand of the
  the group on the bottom is a finite cyclic group.
  It follows that we can replace ${\Q}/{\Z}$ by a finite cyclic group 
  in ~\eqref{eqn:Top-1-0}. Using the standard description of homomorphisms between
  finite cyclic groups, we can find integers $n_1, \ldots , n_r$ such
  that $f = \wt{f} \circ \Delta = \stackrel{r}{\underset{i =1}\sum} n_i \beta(f_i)
  = \beta(\stackrel{r}{\underset{i =1}\sum} n_i f_i) \in G'$.
\end{proof}

For a topological abelian group $G$ whose topology is $\tau$, we let $G^{\pf}$
denote the profinite completion of $G$ with respect to $\tau$. That is, $G^{\pf}$ is
the inverse limit ${\varprojlim}_U \ {G}/U$ with the inverse limit
topology, where $U$ runs through the $\tau$-open subgroups of finite index in $G$.
It is clear that $G^{\pf}$ is a profinite abelian group and the
profinite completion map $G \to G^{\pf}$ is continuous with dense image.
We note the following elementary lemma. 
\begin{lem}\label{lem:PF-compln}
  If $G$ is a topological abelian group, then the map $(G^{\pf})^\star \to G^\star$ induces  a
  topological isomorphism $(G^{\pf})^\star \to (G^\star)_\tor$ for the discrete topology of $(G^\star)_\tor$. In particular if $G$ is torsion topological abelian group, then $(G^{\pf})^\star \to G^\star$ is isomorphism.
\end{lem}
\begin{proof}
    Since the profinite completion map $G \to G^{\pf}$ is continuous with dense image, it follows that the canonical map $(G^{\pf})^\star \to G^\star$ is injective and its image lies in $(G^\star)_\tor$ (as any continuous map $f:G^\pf\to \Q/\Z$ has finite order). On the other hand, any continuous map $f;G \to \Q/\Z$ of finite order factors through $G/\Ker f$ and hence through $G^\pf$, because $\Ker f$ is an open subgroup of finite index in $G$. This proves the lemma.
\end{proof}

Let $G$ be a torsion topological abelian group and let  $G^\star$ be endowed with
the discrete topology. Then the evaluation map ${\rm ev}_G \colon G \to
(G^\star)^\vee$ is continuous with respect to the profinite topology of the target
group. This implies that ${\rm ev}_G$ has a canonical factorization
$G \to G^{\pf} \xrightarrow{{\rm ev}^{\pf}_G} (G^\star)^\vee$ with ${\rm ev}^{\pf}_G$
continuous.
\begin{cor}\label{cor:PF-compln-0}
  The map ${\rm ev}^{\pf}_G \colon G^{\pf} \to  (G^\star)^\vee$ is a topological
  isomorphism.
\end{cor}
\begin{proof}
  Follows from \lemref{lem:PF-compln} using the Pontryagin duality.
\end{proof}

To apply the above recipe, we use the pairing $F^{0,1}_{\n}(m) \times G^{2,1}_{\n}(m) \to \Z/p^m$ (cf.~\eqref{eqn:comp-4})
to endow $G^{2,1}_{\n}(m)$ with the structure of a
topological abelian group via the weak topology induced by $F^{0,1}_{\n}(m)$.
 To check that the hypothesis of \lemref{lem:Top-1} is satisfied, 
we look at the commutative diagram
\begin{equation}\label{eqn:Top-2.5}
  \xymatrix@C.8pc{
    0 \ar[r] & F^{0,1}_{\un{0}}(m) \ar[r] \ar[d]_-{(\theta^*)^\vee} &
    F^{0,1}_{\n}(m) \ar[r] \ar[d]^-{(\theta^*_{\n})^\vee} &
    \H^1_\et(X, W_m\sV^0_{\n}) \ar[d]^-{((\psi^2_{\n})^*)^\vee}  \\
    0 \ar[r] & G^{2,1}_{\un{0}}(m) ^\vee \ar[r] &
    G^{2,1}_{\n}(m)^\vee \ar[r]^-{\partial^\vee} &
    H^0_\et(X,W_m \sW^2_{{\n+1}})^\vee,}
  \end{equation}
  obtained by dualizing ~\eqref{eqn:G-comp-6}.
Lemmas~\ref{lem:Coker-1} and ~\ref{lem:Coh-V1} imply that
  the rows of this diagram are exact. In combination with
  \lemref{lem:Dual-iso}, \lemref{lem:Coh-V2} and \cite[Prop.~3.4]{Kato-Saito-Ann}
  (see also \cite[Thm.~4.7]{KRS}), these results also imply that
  $(\theta^*_{\n})^\vee$ is a monomorphism. This proves that the desired hypothesis
  is satisfied.

We conclude from \lemref{lem:Top-1}
  that with respect to the weak topology of $G^{2,1}_{\n}(m)$,
  the map $(\theta^*_{\n})^\vee \colon F^{0,1}_{\n}(m) \to G^{2,1}_{\n}(m)^\star$ is a
  bijection.

\begin{remk}\label{remk:Top-4}
    Via the pairing $F^{0,1}_{\un{0}}(m) \times G^{2,1}_{\un{0}}(m)  \to {\Z}/{p^m}$, we can endow
    $G^{2,1}_{\un{0}}(m) $ with the weak topology induced by
    $F^1_{\un{0}}(m)$. Using the Pontryagin duality between the
    profinite topology of $G^{2,1}_{\un{0}}(m) $ and the discrete
    topology of $F^1_{\un{0}}(m)$ (cf. \cite[Thm~4.7]{KRS}),
    it is easy to check that
    $\{\Ker(\beta(f))|f \in F^{0,1}_{\un{0}}(m) \}$ forms a sub-base for
    the profinite topology of $G^{2,1}_{\un{0}}(m) $. It follows that
    the weak topology of $G^{2,1}_{\un{0}}(m) $ coincides with the
    profinite topology. Furthermore, $G^{2,1}_{\n}(m) \to  G^{2,1}_{\un{0}}(m) $
    is continuous.
    \end{remk}
By combining the isomorphism  $(\theta^*_{\n})^\vee \colon F^{0,1}_{\n}(m)  \to G^{2,1}_{\n}(m)^\star$, \corref{cor:PF-compln-0} and the Pontryagin
duality, we get the following case of the duality.
\begin{cor}\label{cor:Duality-Spl}
     For $n \ge 0$, there is a perfect pairing of
     topological abelian groups
     \[
       (G^{2,1}_{\n}(m))^{\pf} \times F^{0,1}_{\n}(m) \to {\Z}/{p^m}.
     \]
     Hence, we have an isomprphism $(G^{2,1}_{\n}(m))^\pf \cong \pi_1^{ab}(X,D_n)/p^m$.
   \end{cor}
   \begin{proof}
       First part of the corollary is clear from the previous discussion. The second part follows from the first part, \eqref{eqn:fundamental} and \thmref{thm:H^1-fil}.
   \end{proof}
 \textbf{Case 5 $(i=1, q=2)$:} We have a perfect pairing of topological abelian groups
 \begin{equation}
      (G^{2,1}_{\n}(m))^{\pf} \times F^{0,1}_{\n}(m) \to {\Z}/{p^m},
 \end{equation}
 where $(G^{2,1}_{\n}(m))^{\pf}$ has profinite topology and $ F^{0,1}_{\n}(m)$ has discrete topology.

Even though we don't know much about the topology of $G^{2,1}_{\n}(m)$, the following lemma ensures that the topology is separated. This property is crucial while proving Class field theory in modulus setup.

\begin{lem}\label{lem:Top-5}
      The profinite completion map $G^{2,1}_{\n}(m) \to (G^{2,1}_{\n}(m))^{\pf}$
        is injective, and the induced subspace topology of $G^{2,1}_{\n}(m)$
      coincides with its weak topology. In particular, the latter topology is
      Hausdorff.
    \end{lem}
\begin{proof}
  By \corref{cor:Duality-Spl}, the injectivity claim is equivalent to that the map
      $\theta^*_{\n}$ in the following (cf.~\eqref{eqn:G-comp-6}) is injective.
      \begin{equation}
          \xymatrix@C.5pc{
    H^0(W_m\Omega^2_{X, \log}) \ar[r] \ar[d]^-{\theta^*} &
    H^0(W_m \sW^2_{{\n+1}}) \ar[r]^-{\partial}
    \ar[d]^-{(\psi^2_{{\n}})^*} &
    H^1(W_m\Omega^2_{X|D_{\n+1}, \log}) \ar[r]
    \ar[d]^-{\theta^*_{{\n}}} &  H^1(W_m\Omega^2_{X, \log}) \ar[r] \ar[d]^-{\theta^*}& \cdots \\
    \H^2(W_m\sF^{0}_{\un{0}})^\vee \ar[r] & \H^1(W_m\sV^{0}_{\n})^\vee \ar[r] &
    \H^1(W_m\sF^{0}_{\n})^\vee \ar[r] &  \H^2(W_m\sF^{0}_{\un{0}})^\vee \ar[r] & \cdots ,}
    \end{equation}
      
      But this follows because $(\psi^2_{\n})^*$ is injective by
      \lemref{lem:Local-inj-6} and \lemref{lem:Coh-V2} while the maps $\theta^*$
      are bijective by \cite[\S~3, Prop.~4]{Kato-Saito-Ann} (see also
      \cite[Thm.~4.7]{KRS}). The claim, that the weak topology of $G^{2,1}_{\n}(m)$
      coincides with the subspace
      topology induced from the profinite topology of $(G^{2,1}_{\n}(m))^\pf$,
      follows from \lemref{lem:Top-1} whose proof shows that every open subgroup of
      $G^{2,1}_{\n}(m)$ has finite index.
      \end{proof}

      \begin{lem}\label{lem:Top-6}
     For $\n \ge \n'$, the canonical map $G^{2,1}_{\n}(m) \to G^{2,1}_{\n'}(m)$ is
     surjective continuous and the induced map $(G^{2,1}_{\n}(m))^\pf \to (G^{2,1}_{\n'}(m))^\pf$
     is a topological quotient.
   \end{lem}
   \begin{proof}
   We consider the exact sequence in $X_\et$
   $$0 \to W_m\Omega^2_{X|D_\n,\log} \to W_m\Omega^2_{X|D_{\n'},\log}\to W_m\Omega^2_{X|D_{\n'},\log}/W_m\Omega^2_{X|D_\n,\log}\to 0.$$
   Since $\sA=W_m\Omega^2_{X|D_{\n'},\log}/W_m\Omega^2_{X|D_\n,\log}$ is supported on $E$, by \propref{prop:JSZ}, we have a filtration of $\sA$ whose successive quotients are coherent sheaves on $E$. Hence $H^1_\et(X,\sA)=0$. This implies the map $G^{2,1}_{\n}(m) \to G^{2,1}_{\n'}(m)$ is surjective. Now we look at the commutative diagram
     \begin{equation}\label{eqn:Top-6-0}
       \xymatrix@C2pc{
         G^{2,1}_{\n}(m)\ar@{->>}[r] \ar@{^{(}->}[d] & G^{2,1}_{\n'}(m)
           \ar@{^{(}->}[d] \\
           (G^{2,1}_{\n}(m))^\pf \ar[r] &  (G^{2,1}_{\n'}(m))^\pf.}
         \end{equation}
By \thmref{thm:H^1-fil}, we see that the transition map
$F^{0,1}_{\n'}(m) \to F^{0,1}_{\n}(m)$ is injective. 
We conclude from \corref{cor:Duality-Spl} that
the bottom horizontal arrow in ~\eqref{eqn:Top-6-0} is a surjective continuous
homomorphism between two profinite groups. It must therefore be a topological
quotient map. Finally, the claim, that the horizontal arrow in the top row of
~\eqref{eqn:Top-6-0} is continuous, follows from \lemref{lem:Top-5} which
asserts that each term of this row is a subspace of the corresponding term in the
bottom row.
\end{proof}
\begin{cor}\label{cor:Duality-open}
We have a perfect pairing of topological groups $$\varprojlim_\n (G^{2,1}_{\n}(m))^\pf \times H^1_\et(U,\Z/p^m) \to \Z/p^m.$$    
\end{cor}
\begin{proof}
    By \thmref{thm:Global-version}(12), we have $\varinjlim_\n  F^{0,1}_{\n}(m)=H^1_\et(U,\Z/p^m)$. We now combine \corref{cor:Duality-Spl}, \cite[Lem.~7.2]{GK-Duality}
  and \cite[Lem.~2.6]{KRS} to conclude the proof.
\end{proof}

\subsection{The Brauer-Manin pairing with modulus}\label{sec:Br-Ma}
 In this section we recall the Brauer-Manin pairing for modulus studied in \cite{KRS}. This will be necessary to endow  $G^{1,1}_{\n}$ with a topology.
 
 Let $\Pic(X)$ be the group of isomorphism classes of line bundles on $X$. For a fixed effective divisor $D$, the relative Picard group $\Pic(X|D)$ (usually written as
$\Pic(X,D)$ in the literature) is the set of isomorphism classes
of pairs $(\sL, u \colon \sO_D \xrightarrow{\cong} \sL|_D)$, where
$\sL$ is an invertible sheaf on $X$ and $\sL|_D := \sL \otimes_{\sO_X} \sO_D$.
One says that $(\sL, u) \cong (\sL', u')$ if there is an isomorphism
$w \colon \sL \xrightarrow{\cong} \sL'$ such that $u' = w|_D \circ u$.
$\Pic(X|D)$ is an abelian group under the tensor product of invertible sheaves and
their trivializations along $D$ whose identity element is $(\sO_X, \id)$.
There is a canonical homomorphism $\Pic(X|D) \to \Pic(X)$ which is surjective. By \cite[Lem~2.1]{SV-Invent} we have $\Pic(X|D)=H^1_\et(X,\sK^M_{1,X|D})=H^1_\Zar(X,\sK^M_{1,X|D})$.

Let $\Pic^0(X)$ be the subgroup of degree zero divisor classes in $\Pic(X)$ and $\Picc(X)$ be the Picard scheme of $X$. $\Picc(X)(k)$ has a structure of totally disconnected locally compact Hausdorff topological group, whose topology is induced by the adic (a.k.a the valuation) topology of $k$ (cf. \cite[\S~6.3]{KRS}). We call the topology to be adic topology. Recall that the connected component $\Picc^0(X)$ of $\Picc(X)$ is a connected abelian variety over $k$. It follows from \cite[Thm~10.5.1]{CTS} that the group of $k$ rational points $\Picc^0(X)(k)$ is compact as adic topological group. Hence, $\Picc^0(X)(k)$ is a profinite (because profinite = compact + totally disconnected + Hausdorff) topological abelian group. This topology also induces an adic topology on $\Pic(X)$ (note, $\Pic(X)=\Picc(X)(k)$ if $X(k)\neq \phi$) such that the connected component $\Pic^0(X)$ is a profinite topological group.

 Since we have an exact sequence : $\Z^r \to \Pic(X) \to \Pic(U)\to 0$, we endow  $\Pic(U)$ with the quotient topology. Also we have an exact sequence of topological group : $0 \to \Pic^0(X) \to \Pic(X) \to \Z \to 0$. As a result we also conclude that the quotient topology on $\Pic(X)/p^m$ and $\Pic(U)/p^m$ are profinite for all $m \ge 1$.  

By \cite[Thm~5.6]{KRS}, there exists a group scheme $\Picc(X|D_{\n})$ ($\n\ge \un{1}$) over $k$ such that $\Pic(X|D_{\n})=\Picc(X|D_{\n})(k)$. Hence $\Pic(X|D_{\n})$ has a structure of adic topological group whose topology is induced from $k$, which is locally compact and Hausdorff.

Let's write $\Br(X|D_{\n})$ for the Brauer group with modulus $D_{\n}$ defined in \defref{defn:BGM} (cf. \cite[Defn~8.9]{KRS}). We endow $\Br(X)$ and  $\Br(X|D_{\n})$ with discrete topology. By \cite[\S~9]{Saito}, \cite[Prop~9.4, 11.2]{KRS} there exist continuous pairings (called the Brauer-Manin pairings) 
\begin{equation}\label{eqn:Br-Ma-1}
    \Br(X) \times \Pic(X) \to \Q/\Z, \ \ \Br(X|D_{\n}) \times \Pic(X|D_{\n+1}) \to \Q/\Z; \ \ \n \ge \un{0},
\end{equation}
compatible with the maps $\Br(X) \inj \Br(X|D_{\n})$ and $\Pic(X|D_{\n+1}) \surj \Pic(X)$. Moreover, by \cite[Thm~9.2]{Saito}, \cite[Lem~12.2(1)]{KRS}, we have isomorphisms
\begin{equation}\label{eqn:Br-Ma-2}
    \Br(X) \cong (\Pic(X))^\star, \ \  \ \Br(X|D_{\n}) \cong (\Pic(X|D_{\n+1}))^\star.
\end{equation}
Since $\Br(X|D_{\n})$ is torsion group, it follows that $(\Pic(X|D_{\n+1}))^\star=(\Pic(X|D_{\n+1}))^\star_\tor$. As a result, we get  $(\Pic(X|D_{\n+1})^\pf)^*= (\Pic(X|D_{\n+1}))^*$ by  \lemref{lem:PF-compln}. Taking $p^m$ torsion, we get $(\Pic(X|D_{\n+1})^\pf/p^m)^*= (\Pic(X|D_{\n+1})/p^m)^*$. Note that the quotient topology on  $\Pic(X|D_{\n+1})^\pf/p^m$ is profinite. Hence, by Pontryagin duality and \corref{cor:PF-compln-0}, we have 
\begin{equation}\label{eqn:Pic-1}
    \Pic(X|D_{\n+1})^\pf/p^m \cong ((\Pic(X|D_{\n+1})/p^m)^*)^\vee \cong (\Pic(X|D_{\n+1})/p^m)^\pf.
\end{equation}
Now \eqref{eqn:Br-Ma-2} induces the following isomorphisms.
\begin{equation}\label{eqn:Br-Ma-3}
  \xymatrix@C.8pc{
  (1) \ \ _{p^m}\Br(X) \cong (\Pic(X)/p^m)^\star, &(2)\ \ _{p^m}\Br(X|D_{\n}) \cong (\Pic(X|D_{\n+1})/p^m)^\star, \\
  (3) \ \Pic(X)/p^m \cong (\ _{p^m}\Br(X))^\star,&(4)  \ \Pic(X|D_{\n+1})^\pf/p^m \cong (\ _{p^m}\Br(X|D_{\n}))^\star,
  }
\end{equation}
where the last two isomorphisms of \eqref{eqn:Br-Ma-3} are obtained by taking Pontryagin dual of the first two isomorphisms, \eqref{eqn:Pic-1} and also noting $\Pic(X)/p^m$ is profinite. Moreover, the isomorphisms in bottom row are topological isomorphisms.

We end this section by recalling the following commutative diagram of short exact sequences of topological groups (see the proof of \cite[Cor~9.6]{KRS}).
\begin{equation}\label{eqn:diag-1}
    \xymatrix@C1pc{
      0 \ar[r] & {\Pic(X)}/{p^m} \ar[r] \ar[d]_{\alpha_{p^m}} & H^1(X, W_m\Omega^1_{X, \log})
      \ar[d]^-{\beta_{p^m}} \ar[r] & _{p^m}\Br(X) \ar[r] \ar[d]^-{\gamma_{p^m}} & 0 \\
      0 \ar[r] & (_{p^m}\Br(X))^\star \ar[r]  & H^1(X, W_m\Omega^1_{X, \log})^\star \ar[r] &
      ({\Pic(X)}/{p^m})^\vee \ar[r]&0, }
\end{equation}
where $\alpha_{p^m}$ and $\gamma_{p^m}$ are induced by \eqref{eqn:Br-Ma-1} and $\beta_{p^m}$ is induced from bottom row of \eqref{eqn:Pair-4-1}.

\subsection{Topology of \texorpdfstring{$H^1_\et(X, W_m\Omega^1_{X|D_{\n+1}, \log})$}{the cohomology groups for (i,j)=(1,1)}}\label{sec:Top-3} 
Finally, we consider the case $(i=1,q=1)$. First we prove a general proposition.
\begin{prop}\label{prop:mod-p-iso}
    Let $Y$ be an $N$-dimensional smooth scheme finite type over an $F$-finite field $k$. Let $D_\n$ be as in \S~\ref{sec:Complexes}. For all $\n \ge \un{1}, d \ge N, q \ge 0, m\ge 1$, we have
    $$H^{d}_\et(Y,\sK^M_{q,Y|D_{\n}}/p^m) \xrightarrow{\cong} H_\et^{d}(Y,W_m\Omega^q_{Y|D_{\n},\log}).$$
\end{prop}
\begin{proof}
    The map is induced by the canonical map $\dlog: \sK^M_{q,Y|D_{\n}}/p^m \to W_m\Omega^q_{Y|D_{\n},\log}$ which is an isomorphism of Nisnevich (hence also {\'e}tale) sheaves  if $\n=\un{1}$ by \cite[Thm~2.2.4]{Kerz-Zhao}. Since $\sK^M_{q,Y|D_{\n}}$ is $p$-torsion free, the exact sequence
    $$0 \to \sK^M_{q,Y|D_{\n}} \to \sK^M_{q,Y|D_{\un{1}}} \to \frac{\sK^M_{q,Y|D_{\un{1}}}}{\sK^M_{q,Y|D_{\n}}} \to 0$$
    induces another exact sequence of the form
    \begin{equation}\label{eqn:Im-phi}
        0 \to \ _{p^m}\left(\frac{\sK^M_{q,Y|D_{\un{1}}}}{\sK^M_{q,Y|D_{\n}}}\right) \to \sK^M_{q,Y|D_{\n}}/p^m \xrightarrow{\phi} \sK^M_{q,Y|D_{\un{1}}}/p^m
    \end{equation}   
    for every $\un n\ge \un 1$.
   It is enough to show that $H^d_\et(Y, \ _{p^m}\left(\frac{\sK^M_{q,Y|D_{\un{1}}}}{\sK^M_{q,Y|D_{\n}}}\right))=0$ for $d \ge N$ and $\text{Image } \phi \cong W_m\Omega^q_{Y|D_{\n},\log}$, for all $m\ge 1$.
    First we claim that, for $m>>0$, we have an isomorphism of Nisnevich (resp. {\'e}tale) sheaves \begin{equation}\label{eqn:dlog-1}
        \dlog:\frac{\sK^M_{q,Y|D_{\un{1}}}}{\sK^M_{q,Y|D_{\n}}} \xrightarrow{\cong} \frac{W_m\Omega^q_{Y|D_{\un{1}},\log}}{W_m\Omega^q_{Y|D_{\un{n}},\log}}.
    \end{equation} 
     To prove the claim, note that the bijective map $\dlog:\sK^M_{q,Y|D_{\un{1}}}/p^m \to W_m\Omega^q_{Y|D_{\un 1},\log}$ and the surjective map $\dlog:\sK^M_{q,Y|D_{\un{n}}} \to W_m\Omega^q_{Y|D_{\un n},\log}$ together imply that 
    \begin{equation}\label{eqn:Im-phi-1}
        \dlog:\frac{\sK^M_{q,Y|D_{\un{n}}}+\ p^m\sK^M_{q,Y|D_{\un{1}}}}{p^m\sK^M_{q,Y|D_{\un{1}}}}\xrightarrow{\cong} W_m\Omega^q_{Y|D_{\un n},\log}
    \end{equation}
    is an isomorphism of Nisnevich (resp. {\'e}tale) sheaves for all $m \ge 1$. Furthermore, $\text{Image } \phi \cong W_m\Omega^q_{Y|D_{\n},\log}$. Now, \eqref{eqn:dlog-1} easily follows once we observe that $p^m\sK^M_{q,Y|D_{\un{1}}}\subset \sK^M_{q,Y|D_{\un{n}}}$ for $m>>0$.
    Next, we note that the sheaf $\frac{W_m\Omega^q_{Y|D_{\un{1}},\log}}{W_m\Omega^q_{Y|D_{\un{n}},\log}}$ which is supported on $E$, has a filtration whose successive quotients are coherent sheaves on $Y$ (\cf. \propref{prop:JSZ}(1)). Hence $H^d_\et(Y,\frac{\sK^M_{q,Y|D_{\un{1}}}}{\sK^M_{q,Y|D_{\n}}}) \cong H^d_\et(Y, \frac{W_m\Omega^q_{Y|D_{\un{1}},\log}}{W_m\Omega^q_{Y|D_{\un{n}},\log}})=0$ for $d \ge N$. As the right hand side of \eqref{eqn:dlog-1} is $p^m$-torsion, we equivalently get $H^d_\et(Y, \ _{p^m}\left(\frac{\sK^M_{q,Y|D_{\un{1}}}}{\sK^M_{q,Y|D_{\n}}}\right))=0$ for $d \ge N, m >>0$. Hence, $H^{d}_\et(Y,\sK^M_{q,Y|D_{\n}}/p^m) \xrightarrow{\cong} H_\et^{d}(Y,W_m\Omega^q_{Y|D_{\n},\log})$ for $m>>0$.
    
    Now we prove by decreasing induction on $m$ that $H^d_\et(Y, \ _{p^m}\left(\frac{\sK^M_{q,Y|D_{\un{1}}}}{\sK^M_{q,Y|D_{\n}}}\right))=0$, for all $m \ge 1$. Fix some $m$ for which we have $_{p^m}\left(\frac{\sK^M_{q,Y|D_{\un{1}}}}{\sK^M_{q,Y|D_{\n}}}\right)=\frac{\sK^M_{q,Y|D_{\un{1}}}}{\sK^M_{q,Y|D_{\n}}}=\frac{W_m\Omega^q_{Y|D_{\un{1}},\log}}{W_m\Omega^q_{Y|D_{\un{n}},\log}}$. Now consider the following exact sequence (for $1<i < m$).
    $$0 \to \ _{p^i}W_m\Omega^q_{Y|D_{\n},\log} \xrightarrow{\phi_1} \ _{p^i}W_m\Omega^q_{Y|D_{\un{1}},\log} \to \ _{p^i}\left(\frac{W_m\Omega^q_{Y|D_{\un{1}},\log}}{W_m\Omega^q_{Y|D_{\un{n}},\log}}\right) \to \hspace{2 cm}$$
    $$\hspace{5cm}  \to  W_m\Omega^q_{Y|D_{\n},\log}/p^i \xrightarrow{\phi_2} W_m\Omega^q_{Y|D_{\un{1}},\log} /p^i.$$
    Enough to show $H^d_\et(Y,\coker \phi_1)=0=H^d_\et(Y,\Ker\ \phi_2)$ for all $d \ge N$.
    \\
    To prove this, we make the following claims.
    
    \textbf{Claim 1:} $\ _{p^i}W_m\Omega^q_{Y|D_{\n},\log} \cong W_{i}\Omega^q_{Y|D_{\lceil\n /p^{m-i}\rceil},\log}, \text{ for } i,m \ge 1, q \ge 0 \text{ and } \un n\ge \un 1.$

    Proof: Since $W_m\Omega^q_{Y|D_{\n},\log} \subset W_m\Omega^q_{Y}$, and $\Ker (p^i: W_m\Omega^q_{Y} \to W_m\Omega^q_{Y})  = \Ker (R^i: W_m\Omega^q_{Y} \to W_{m-i}\Omega^q_{Y})$, we conclude that $\Ker (p^i: W_m\Omega^q_{Y|D_{\n},\log} \to W_m\Omega^q_{Y|D_{\n},\log})  = \Ker (R^i: W_m\Omega^q_{Y|D_{\n},\log} \to W_{m-i}\Omega^q_{Y|D_{\n},\log})$. But, by \corref{cor:Complex-7}, we have an isomorphism $$\ov p^{m-i}:W_{i}\Omega^q_{Y|D_{\lceil\n /p^{m-i}\rceil},\log} \xrightarrow{\cong} \Ker (R^i: W_m\Omega^q_{Y|D_{\n},\log} \to W_{m-i}\Omega^q_{Y|D_{\n},\log}).$$ This proves claim 1.

    From this claim, we obtain $\coker \ \phi_1 \cong \frac{W_i\Omega^q_{Y|D_{\un{1}},\log}}{W_i\Omega^q_{Y|D_{\un{n}},\log}}$ is supported on $E$. Therefore, by \propref{prop:JSZ}, we have $H^d_\et(Y,\coker \ \phi_1)$ $=0$ for $d \ge N$.

    \textbf{Claim 2 :} $\Ker\ \phi_2 \cong W_{m-i}\Omega^q_{Y|D_{\lceil \n/p^i \rceil},\log} / W_{m-i}\Omega^q_{Y|D_{\n},\log}$.

    Proof: To see this, first we note that (as $p=\ov p R$)
    $$p^i(W_m\Omega^q_{Y|D_{\n},\log})= \ov{p}^i R^i(W_m\Omega^q_{Y|D_{\n},\log})= \ov{p}^i(W_{m-i}\Omega^q_{Y|D_{\n},\log}) .$$
Now let $\alpha \in W_m\Omega^q_{Y|D_{\n},\log}$ be any local section such that $\alpha = p^i \beta$, where $\beta \in W_m\Omega^q_{Y|D_{\un{1}},\log}$ be a local section. Since $R^i(\beta)\in W_{m-i}\Omega^q_{Y,\log}$ and $\ov p^iR^i(\beta)=\alpha \in \Fil_{D_{-\un n}}W_m\Omega^q_U$, by repeatedly applying \lemref{lem:Complete-7}, we get $R^i\beta \in W_{m-i}\Omega^q_{Y,\log} \cap \Fil_{D_{-\lceil \un n/p^i \rceil}}W_{m-i}\Omega^q_U=W_{m-i}\Omega^q_{Y|D_{\lceil \n/p^i \rceil},\log}$, where the last equality follows from \lemref{lem:Complex-6}(1).
So we get  $$\alpha = \ov{p}^iR^i\beta \in \ov{p}^i(W_{m-i}\Omega^q_{Y|D_{\lceil \n/p^i \rceil},\log}).$$
This implies $\Ker\ \phi_2 \subset \frac{\ov{p}^i(W_{m-i}\Omega^q_{Y|D_{\lceil \n/p^i \rceil},\log})}{\ov{p}^i(W_{m-i}\Omega^q_{Y|D_{\n},\log})}$. Also we have $\frac{\ov{p}^i(W_{m-i}\Omega^q_{Y|D_{\lceil \n/p^i \rceil},\log})}{\ov{p}^i(W_{m-i}\Omega^q_{Y|D_{\n},\log})} \subset \Ker\ \phi_2$, because $\lceil \n/p^i \rceil \ge \un{1}$ and so $\ov{p}^i(W_{m-i}\Omega^q_{Y|D_{\lceil \n/p^i \rceil},\log}) \subset \ov{p}^i(W_{m-i}\Omega^q_{Y|D_{\un 1},\log})=p^i(W_{m}\Omega^q_{Y|D_{\un 1},\log})$. Now the claim 2 follows from the isomorphism
$$\frac{W_{m-i}\Omega^q_{Y|D_{\lceil \n/p^i \rceil},\log}}{W_{m-i}\Omega^q_{Y|D_{\n},\log}} \xrightarrow{\ov{p}^i} \frac{\ov{p}^i(W_{m-i}\Omega^q_{Y|D_{\lceil \n/p^i \rceil},\log})}{\ov{p}^i(W_{m-i}\Omega^q_{Y|D_{\n},\log})}.$$
By \propref{prop:JSZ}, $\Ker\ \phi_2$ has a filtration whose successive quotients are coherent sheaves supported on $E$. Hence $H^d_\et(Y, \Ker\ \phi_2)=0$ for all $d\ge N$. This proves the proposition.
\end{proof}

\begin{cor}\label{cor:G^1,1}
    For $m \ge 1,\n \ge \un{0}$, we have an exact sequence 
    \begin{equation}\label{eqn:G^1,1}
        0 \to \Pic(X|D_{\n+1})/p^m \to H^1_\et(X,W_m\Omega^1_{X|D_{\n+1},\log}) \to \ _{p^m}H^2_\et(X,\sK^M_{1,X|D_{\n+1}}) \to 0.
    \end{equation}
\end{cor}
\begin{proof}
    Note that $\Pic(X|D_{\n})\cong H^1_\et(X,\sK^M_{1,X|D_{\n}})$. Now the corollary immediately follows from ~\propref{prop:mod-p-iso}, using the exact sequence $0 \to \sK^M_{1,X|D_{\n+1}}\xrightarrow{p^m}\sK^M_{1,X|D_{\n+1}} \to \sK^M_{1,X|D_{\n+1}}/p^m \to 0$ in $X_\et$. 
\end{proof}

\textbf{Case 6 ($i=1,q=1$):} 
We endow $G^{1,1}_{\n}(m)$ with the structure of weakest topological abelian group such that $\Pic(X|D_{\n+1})/p^m$ becomes an open subgroup of $G^{1,1}_{\n}(m)$ in \eqref{eqn:G^1,1}. As a result we see the induced quotient topology on $ \ _{p^m}H^2_\et(X,\sK^M_{1,X|D_{\n+1}})$ in \eqref{eqn:G^1,1}) is discrete (cf. \cite[Lem~12.5]{KRS}) and \eqref{eqn:G^1,1} becomes an exact sequence of topological abelian groups. We let $$(G^{1,1}_{\n}(m))^\wedge:=H^1_\et(X,W_m\Omega^1_{X|D_{\n+1},\log})^\wedge= \varprojlim\limits_{V} H^1_\et(X,W_m\Omega^1_{X|D_{\n+1},\log})/V,$$
where $V$ varies over all open subgroups of finite index in $\Pic(X|D_{\n+1})/p^m$. Here we also claim that 
\begin{equation}\label{eqn:H^2-grp}
    H^2_\et(X,\sK^M_{1,X|D_{\n}}) \cong H^2_\et(X,\sK^M_{1,X|D_{\n'}})
\end{equation}
for $\n \ge \n' \ge \un{1}$. To see this, we consider the exact sequence $0 \to \sK^M_{1,X|D_{\n}} \to \sK^M_{1,X|D_{\n'}} \to \frac{\sK^M_{1,X|D_{\n'}}}{\sK^M_{1,X|D_{\n}}} \to 0 $ in $X_\et$. By a similar argument to that in \propref{prop:mod-p-iso}, we see that $\frac{\sK^M_{1,X|D_{\n'}}}{\sK^M_{1,X|D_{\n}}} \cong \frac{W_m\Omega^1_{X|D_{\un n'},\log}}{W_m\Omega^1_{X|D_{\un n},\log}}$, for some $m>>0$. Now, by \propref{prop:JSZ}, we get $H^i_\et(X,\frac{W_m\Omega^1_{X|D_{\un n'},\log}}{W_m\Omega^1_{X|D_{\un n},\log}})=0$ for $i=1,2$ and hence the claim. 

A locally compact Hausdorff topological group $G$ is called ``torsion-by-profinite" if there is an exact sequence
\begin{equation*}
  0 \to G_\pf \to G \to G_{\text{dt}} \to 0,
\end{equation*}
where $G_\pf$  is an open and profinite subgroup of $G$ and $G_{\text{dt}}$
is a torsion group with the quotient topology (necessarily discrete). Note that, the Pontryagin dual $G^*=\Hom_\Tab(G,\T)$ is again a ``torsion-by-profinite" group. Moreover, if $G$ is torsion, then $G^*=\Hom_\Tab(G,\Q/\Z)$. We also recall that ``torsion-by-profinite" groups satisfy Pontryagin duality, namely, the evaluation map $G\to (G^*)^*$ is topological isomorphism (cf. \cite[\S~2.9]{Pro-fin}). The category of ``torsion-by-profinite" will be denoted as $\pfd$. Note that $\pfd$ contains both profinite and discrete groups. 

We have the following corollary: 
\begin{cor}\label{cor:G^1,1-dual} 
For all $\n \ge 0$, we have the following exact sequences in $\pfd$.    
    \begin{equation}\label{eqn:Pic-2}
        0 \to \Pic(X|D_{\n+1})^\pf/p^m \to (G^{1,1}_{\n}(m))^\wedge \to \ _{p^m}H^2_\et(X,\sK^M_{1,X|D_{\un{1}}}) \to 0,
    \end{equation}
    \begin{equation}\label{eqn:Pic-3}
        0 \to (\ _{p^m}H^2_\et(X,\sK^M_{1,X|D_{\un{1}}}))^\star \to (G^{1,1}_{\n}(m))^\star \to (\Pic(X|D_{\n+1})^\pf/p^m )^\star \to 0.
    \end{equation}
\end{cor}
\begin{proof}
    Note that, \eqref{eqn:G^1,1} and \eqref{eqn:H^2-grp} induce an exact sequence of topological abelian groups
    $$0 \to (\Pic(X|D_{\n+1})/p^m)/V \to H^1_\et(X,W_m\Omega^1_{X|D_{\n+1},\log})/V \to \ _{p^m}H^2_\et(X,\sK^M_{1,X|D_{\un{1}}}) \to 0,$$
    for any open subgroup $V$ of finite index in $\Pic(X|D_{\n+1})/p^m$. Now \eqref{eqn:Pic-2}  is obtained by taking inverse limit over $V$ and noting that $\Pic(X|D_{\n+1})^\pf/p^m\cong (\Pic(X|D_{\n+1})/p^m)^\pf$, which is \eqref{eqn:Pic-1}.

    Now by taking Pontryagin dual in \eqref{eqn:Pic-2}, we get an exact sequence in $\pfd$ 
    $$ 0 \to (\ _{p^m}H^2_\et(X,\sK^M_{1,X|D_{\un{1}}}))^\star \to ((G^{1,1}_{\n}(m))^\wedge)^\star \to (\Pic(X|D_{\n+1})^\pf/p^m )^\star \to 0.$$
    We also note that $(\Pic(X|D_{\n+1})/p^m)^*= ((\Pic(X|D_{\n+1})/p^m)^\pf)^*$ (by \lemref{lem:PF-compln}). Since $\Pic(X|D_{\n+1})/p^m$ is an open subgroup of $G^{1,1}_{\n}(m)$, the exact sequence \eqref{eqn:Pic-3} is also obtained by taking topological dual in \eqref{eqn:G^1,1}. Compairing this exact sequence with the above one, we get $((G^{1,1}_{\n}(m))^\wedge)^\star=(G^{1,1}_{\n}(m))^\star$. Hence, \eqref{eqn:Pic-3} is also topologically exact.
\end{proof}
By \corref{cor:F^q}, we have an exact sequence
\begin{equation}\label{eqn:F^1,1}
    0 \to \Pic(U)/p^m \to F^{1,1}_{\n}(m) \to   \ _{p^m}\Br(X|D_{\n}) \to   0,
\end{equation}
where we endow $F^{1,1}_{\n}(m)$ with the structure of weakest topological group such that $\Pic(U)/p^m$ becomes an open subgroup in \eqref{eqn:F^1,1}. Hence  \eqref{eqn:F^1,1} becomes a short exact sequence in $\pfd$ (because $\Pic(U)/p^m$ is a profinite group). Taking Pontryagin dual, we have the following exact sequence in $\pfd$.
\begin{equation}\label{eqn:F^1,1-dual}
    0 \to (\ _{p^m}\Br(X|D_{\n}))^\star \to (F^{1,1}_{\n}(m))^\star \to (\Pic(U)/p^m)^\star \to 0
\end{equation}

Now recall from \cite[Lem~12.5]{KRS} that there is a pairing
\begin{equation}\label{eqn:G_m}
    \Pic(U) \times H^2_\et(X,\sK^M_{1,X|D_{1}}) \to \Q/\Z
\end{equation}
compatible with the left pairing of \eqref{eqn:Br-Ma-1} and induces an isomorphism $$H^2_\et(X,\sK^M_{1,X|D_{\un{1}}}) \cong (\Pic(U))^\star.$$ As a result, we get the following isomorphisms of topological abelian group.
\begin{equation}\label{eqn:Pic-U-dual}
    \xymatrix@C.5pc{
    (1)\ _{p^m}H^2_\et(X,\sK^M_{1,X|D_{1}}) \xrightarrow{\cong} (\Pic(U)/p^m)^\star &(2)\ \Pic(U)/p^m \xrightarrow{\cong} (\ _{p^m}H^2_\et(X,\sK^M_{1,X|D_{1}}))^\star,
    }
\end{equation}
where the right isomorphism is obtained from the left one by applying Pontryagin duality and noting that $\Pic(U)/p^m$ is profinite.

We now consider the following diagram of exact sequences.
\begin{equation}\label{eqn:diag-2}
    \xymatrix@C.8pc{
    0 \ar[r]& \Pic(U)/p^m \ar[r] \ar[d]_{\delta_{p^m}}^-{\cong}& F^{1,1}_{\n}(m) \ar[r] \ar[d]_{\beta'_{p^m}} & \ _{p^m}\Br(X|D_{\n}) \ar[r] \ar[d]^-{\alpha'_{p^m}} & 0 \\
    0 \ar[r] &(\ _{p^m}H^2_\et(X,\sK^M_{1,X|D_{\un{1}}}))^\vee \ar[r] & (G^{1,1}_{\n}(m))^\vee \ar[r] & (\Pic(X|D_{\n+1})/p^m )^\vee \ar[r]& 0
    ,}
\end{equation}
where $\delta_{p^m}$, $\beta'_{p^m}$,   $\alpha'_{p^m}$ are the maps induced from \eqref{eqn:G_m}, \eqref{eqn:Pair-4-1} and \eqref{eqn:Br-Ma-1}, respectively. Moreover, $\delta_{p^m}$ is isomorphism by \eqref{eqn:Pic-U-dual}(2), because $\ _{p^m}H^2_\et(X,\sK^M_{1,X|D_{\un{1}}})$ is discrete. We claim that the above diagram is commutative. For this consider the diagram. 
\begin{equation}
    \xymatrix@C.8pc{ &
      \Pic(X)/p^m \ar[rr] \ar@{->>}[dr] \ar[dd]^-{\alpha_{p^m}} & &
      H^1_\et(X,W_m\Omega^1_{X,\log}) \ar[dr] \ar[dd]^<<<<<<<{\beta_{p^m}} & \\
      & & \Pic(U)/p^m \ar[dd]^<<<<<<<{\delta_{p^m}} \ar[rr] & & F^{1,1}_{\n} \ar[dd]^-{\beta'_{p^m}} \\
      & (_{p^m}\Br(X))^\vee \ar[rr] \ar[dr] & &
      H^1_\et(X,W_m\Omega^1_{X,\log})^\vee \ar[dr] & \\
      & & _{p^m}H^2_\et(X,\sK^M_{1,X|D_{\un{1}}}))\vee \ar[rr] & &(G^{1,1}_{\n}(m))^\vee},
\end{equation}
where the front face and the back face are the left square of \eqref{eqn:diag-2} and \eqref{eqn:diag-1}, respectively after replacing $(-)^\star$ by $(-)^\vee$ in the latter square. Note that the top floor and bottom floor are commutative by naturality of cohomology, while the right sided face is commutative by compatibility of the pairing \eqref{eqn:Pair-4-1}. The left face is commutative by the construction of $\delta_{p^m}$ (see \cite[Lem~12.5]{KRS}). Since the back face is commutative (by \eqref{eqn:diag-1}), we conclude the front is also commutative.

Now the right square of \eqref{eqn:diag-2} commutes because of the construction of the Brauer-Manin pairing in \cite[Prop~9.4]{KRS} and the following commutative diagram (cf. \cite[Prop~4.4]{KRS}).
\begin{equation}
    \xymatrix@C.8pc@R1pc{
    \bigoplus\limits_{x\in X^{(1)}}F^{1,1}_{\n,x}(m) \  \times  \ \bigoplus\limits_{x\in X^{(1)}}G^{1,1}_{\n,x}(m)  \ar@<8ex>[dd]   \ar[r]& \underset{x \in X^{(1)}}{\oplus}H^1_\et(\kappa(x),W_m\Omega^1_{\kappa(x),\log}) \ar[d]^-{\theta}_-{\cong} \ar@/^2pc/[rdd]^-{inv_{\kappa(x)}} \\
      & \underset{x \in X^{(1)}}{\oplus}H^2_x(X,W_m\Omega^2_{X,\log}) \ar[d]  \\
    F^{1,1}_{\n}(m)  \ \times \ G^{1,1}_{\n}(m) \ar@<7ex>[uu] \ar[r] & H^2_\et(X,W_m\Omega^2_{X,\log}) \ar[r]^-{Tr}& \Z/p^m,
    }
\end{equation}
where $G^{1,1}_{\n,x}=H^1_x(X_x^h,W_m\Omega^1_{X|D_{\n+1},\log})=K_x^\times/(\Fil_{n_x+1}K_x^\times+p^mK_x^\times)$ (resp. $\Z/p^m$) and $F^{1,1}_{\n,x}=\H^1_\et(X_x^h, W_m\sF^{1}_\n)=\ _{p^m}\Fil_{n_x}\Br(K_x) (\text{ resp. }_{p^m}\Br(A_x)=\Z/p^m)$, if $x \in D_\n$ (resp. $x \not\in D_\n$)  by \lemref{lem:support} and \thmref{thm:Kato-fil-10}, respectively. Here the top pairing is induced by the local class field theory for Brauer group (cf. \cite[\S~3.4]{Kato80-2}).
Now we prove the main result of this section.
\begin{prop}\label{prop:G^1,1}
    The pairing $G^{1,1}_{\n}(m) \times F^{1,1}_{\n}(m) \to \Z/p^m$ induces isomorphisms
    $$F^{1,1}_{\n}(m) \xrightarrow{\cong} (G^{1,1}_{\n}(m))^\star, \ \ \  (G^{1,1}_{\n}(m))^\wedge \xrightarrow{\cong}(F^{1,1}_{\n}(m))^\star.$$
\end{prop}
\begin{proof}
    By \eqref{eqn:Pic-U-dual}(2), \eqref{eqn:Br-Ma-3}(2) and the exact sequence \eqref{eqn:Pic-3}, we get that the map $\beta'_{p^m}$ in \eqref{eqn:diag-2} is injective and its image lies in $(G^{1,1}_{\n}(m))^\star$. Moreover, this is surjective onto $(G^{1,1}_{\n}(m))^\star$. This follows by replacing the bottom row of \eqref{eqn:diag-2} by \eqref{eqn:Pic-3} and applying  \eqref{eqn:Br-Ma-3} and \eqref{eqn:Pic-U-dual} once again. 
    
    For the second one, note that $((G^{1,1}_{\n}(m))^\wedge)^* =(G^{1,1}_{\n}(m))^*$ (see the proof of \corref{cor:G^1,1-dual}). Then $F^{1,1}_{\n}(m) \xrightarrow{\cong} ((G^{1,1}_{\n}(m))^\wedge)^\star$ is a topological isomorphism in $\pfd$ (because $\delta_{p^m}$ is topological isomorphism). Now, by taking Pontryagin dual we get the isomorphism $(G^{1,1}_{\n}(m))^\wedge \xrightarrow{\cong}(F^{1,1}_{\n}(m))^\star$  in $\pfd$.
\end{proof}
\begin{rem}\label{rem:G^1,1}
    We call $G^{1,1}_{\n}(m)^\wedge$ to be the ``torsion-by-profinite" completion of $G^{1,1}_{\n}(m)$.

\end{rem}
\begin{cor}\label{cor:Duality-open-2}
    We have a perfect pairing of topological abelian groups
    $$\varprojlim_\n (G^{1,1}_{\n}(m))^\wedge \times H^1_\et(U,W_m\Omega^1_{U,\log}) \to \Z/p^m$$
    \end{cor}
    \begin{proof}
        We note that $\varinjlim_\n  F^{1,1}_{\n}(m)=H^1_\et(U,W_m\Omega^1_{U,\log})$ by \thmref{thm:Global-version}(12). Also, note that the map $\Br(X|D_\n)\to \Br(X|D_{\n'})$ is injective if $\n' \ge \n$. Hence, by \eqref{eqn:Br-Ma-3}, the transition map $\Pic(X|D_{\n'+1})^\pf/p^m \to \Pic(X|D_{\n+1})^\pf/p^m$ is topological quotient map. Also, by \eqref{eqn:Pic-2} (resp. \eqref{eqn:F^1,1}), the transition map $(G^{1,1}_{\n'}(m))^\wedge \to (G^{1,1}_{\n}(m))^\wedge$ (resp. $F^{1,1}_{\n}(m) \to F^{1,1}_{\n'}(m)$) is topological quotient (resp. topological injection). We endow $\varprojlim_\n (G^{1,1}_{\n}(m))^\wedge$ (resp. $H^1_\et(U,W_m\Omega^1_{U,\log})$) with inverse limit (resp. direct limit) topology. As a results we have the following exact sequences in $\pfd$.
        $$0 \to \varprojlim\limits_\n\Pic(X|D_{\n+1})^\pf/p^m \to \varprojlim_\n (G^{1,1}_{\n}(m))^\wedge \to \ _{p^m}H^2_\et(X,\sK^M_{1,X|D_{\un{1}}}) \to 0,$$
        $$0 \to \Pic(U)/p^m \to H^1_\et(U,W_m\Omega^1_{U,\log}) \to \ _{p^m}\Br(U) \to 0.$$

        Moreover, note that the canonical map $G=\varprojlim_\n (G^{1,1}_{\n}(m))^\wedge \to (G^{1,1}_{\n}(m))^\wedge$ is surjective (because the transition map $\Pic(X|D_{\n'+1})^\pf/p^m \to \Pic(X|D_{\n+1})^\pf/p^m$ is surjective if $\n' \ge \n$). By \cite[Lem~2.6]{KRS} we conclude that $\varinjlim\limits_\n (G^{1,1}_\n(m))^* \to G^*$ is isomorphism (note $(G^{1,1}_\n(m))^*=((G^{1,1}_\n(m))^\wedge)^*$). Now, by \propref{prop:G^1,1}, we have topological isomorphism $H^1_\et(U,W_m\Omega^1_{U,\log}) \xrightarrow{\cong} G^*$ in $\pfd$. Hence by Pontryagin duality for ``torsion-by-profinite" groups, we get the result.    
    \end{proof}
\subsection{The duality theorem over local fields}
Finally, we summarise all the duality we proved the in previous subsections.
\begin{thm}\label{thm:Duality-curve}
   Let $X, U, D_\n$ be as in the beginning of \S~\ref{chap:duality-localfields}. Then for all $i \ge 0$ and $\n \ge \un{0}$, we have perfect pairings of topological abelian groups
   \begin{enumerate}
       
       \item  \hspace{1cm} $H^i_\et(X, W_m\Omega^j_{X|D_{\n+1}, \log})^\pf \times
    \H^{2-i}_\et(X, W_m\sF^{2-j,\bullet}_{D_{\n}}) \to {\Z}/{p^m},\ (i,j)\neq (1,1), (2,1)$; \\
    where $H^i_\et(X, W_m\Omega^j_{X|D_{\n+1}, \log})$ is already profinite if $i \neq 1$.
    \item \hspace{1cm}$H^2_\et(X,W_m\Omega^1_{X|D_{\n+1},\log}) \times \H^{0}_\et(X,W_m\sF^{1,\bullet}_{D_{\n}}) \to \Z/p^m$, \\
    where $H^2_\et(X,W_m\Omega^1_{X|D_{\n+1},\log})$ has discrete topology.
    \item \hspace{1cm} $H^1_\et(X,W_m\Omega^1_{X|D_{\n+1},\log})^\wedge \times \H^1_\et(X,W_m\sF^{1,\bullet}_{D_{\n}}) \to \Z/p^m,$ \\ 
    where $H^1_\et(X,W_m\Omega^1_{X|D_{\n+1},\log})$ is an extension of discrete group by an adic topological group and $\H^1_\et(X,W_m\sF^{1,\bullet}_{D_{\n}})$ is in $\pfd$ and $(H^1_\et(X,W_m\Omega^1_{X|D_{\n+1},\log}))^\wedge$ is the ``torsion by profinite" completion of $H^1_\et(X,W_m\Omega^1_{X|D_{\n+1},\log})$ (cf. \remref{rem:G^1,1}).
    \item \hspace{1cm} $\varprojlim\limits_\n H^i_\et(X,W_m\Omega^j_{X|D_{\n+1},\log})^\wedge \times H^{2-i}_\et(U,W_m\Omega^{2-j}_{U,\log}) \to \Z/p^m$, \\  
    where $(-)^\wedge$ is the ``torsion-by-profinite" completion for $i=j=1$ and the profinite completion, otherwise. (Note, these groups are all zero if $i=0$.)
   \end{enumerate}
\end{thm}
\begin{proof}
    Combine \eqref{eqn:Top-0.5} \eqref{eqn:Top-0},  \eqref{eqn:Top-1}, \eqref{eqn:Top-2}, \corref{cor:Duality-Spl}, \propref{prop:G^1,1}. For the item (4), combine \eqref{eqn:Duality-open--1}, \eqref{eqn:Duality-open-0}, Corollaries~\ref{cor:Duality-open}, \ref{cor:Duality-open-2}.
\end{proof}

\chapter{Lefschetz hypersurface theorem }\label{chap:Lef}

Lefschetz hypersurface (a.k.a the Weak Lefschetz) theorem in {\et}ale cohomology (cf. \cite[Thm~VI.7.1]{Milne-EC}, \cite{Katz}) plays an important role for studying $\ell$-adic {\et}ale cohomology of smooth schemes over a field. In this chapter, we shall consider the $p$-adic analogue of these theorems.  

Let $X$, $E$, $U$, $D = D_{\un{n}}$ be as in \S ~\ref{sec:Complexes}. Moreover, we also assume $X$ is smooth projective over an $F$-finite base field $k$ and we fix an embedding $X \inj \P^L_k$. Let dim $X=N$. Let $Y$ be a general smooth hypersurface section such that $E' = Y \times_X E$ is also a simple normal crossing divisor. Let $D' := D \times_X Y$ and $i: Y\inj X$ denote the closed embedding. Also, let $j' : U' \inj Y $ be the inclusion of the complement of $E'$ in $Y$. Also, write $V$ for the complement of $E \cup Y$ in $X$.

Main purpose of this chapter is to prove a Lefschetz theorem for the cohomology groups $H^i_\et(X,W_m\Omega^q_{X,\log})$,  
 $\H^i_\et(X,W_m\sF^{q,\bullet}_D)$ and  $H^i_\et(X, W_m\Omega^q_{X|D,\log})$. 

\section{Some exact sequences}

Let $X_\log$ and $Y_\log$ be the log schemes whose log structures on $X$ and $Y$ are given by the sheaves of monoids $$\sM_X = \sO_X \cap j_*\sO_U^* \  \text{and} \ \sM'_Y = \sO_Y \cap j'_*\sO_{U'}^{*}, \ \text{respectively.}$$
We identify $\Omega^1_{X_\log}$ (resp. $\Omega^1_{Y_\log}$) with $\Omega^1_X(\log E)$ (resp. $\Omega^1_Y(\log E')$) by \lemref{lem:LWC-2}.

\begin{lem}\label{lem:s.e.s-2}
    For all $q \ge 0$, we have the following exact sequence.
    \begin{equation}\label{eqn:s.e.s-2}
       0 \to \sI/\sI^2 \otimes_{\sO_Y}\Omega^{q-1}_Y(\log E') \xrightarrow{\ov{d}} i^*\Omega^q_{X}(\log E) \to \Omega^q_Y(\log E') \to 0,
    \end{equation}
    where $\sI$ is the ideal sheaf $\sO_X(-Y)$.
\end{lem}
\begin{proof}

Note that $i:Y_\log \to X_\log$ is a strict closed immersion of log schemes. This means, $Y \inj X$ is a closed immersion and $i^*\sM_X \xrightarrow{\cong}\sM'_Y$ (cf. \cite[p.275]{Ogus}).  Hence by \cite[Prop~IV.2.3.2]{Ogus}, we have an exact sequence
\begin{equation}\label{eqn:s.e.s-2.5}
    \sI/\sI^2 \xrightarrow{\ov{d}} i^*\Omega^1_{X}(\log E) \to \Omega^1_Y(\log E') \to 0.
\end{equation}
 Since $X$ is $F$-finite and smooth over $k$, we have an exact sequence of finite type locally free sheaves :
$$0 \to f^*\Omega^1_k \to \Omega^1_X \to \Omega^1_{X/k} \to 0,$$
by \cite[Thm~25.1]{Matsumura}, where $f:X \to \Spec k$. So we conclude $\rank \ \Omega^1_X(\log E)=\rank \  \Omega^1_X= N+ \rank \ \Omega^1_k$, where the first equality follows from the description of $\Omega^1_A(\log \pi)$ in \lemref{lem:LWC-2}.

Similarly, we also get $\rank \ \Omega^1_Y(\log E')= N-1 +\rank \ \Omega^1_k$.  Since $\sI/\sI^2$ is locally free of rank one, by counting rank we get the map $\ov d$ in \eqref{eqn:s.e.s-2.5} is injective. Now the lemma follows by taking higher exterior power in \eqref{eqn:s.e.s-2.5} and \cite[Ex~II 5.16(d)]{Hartshorne-AG}. 

\end{proof}
We refer the reader to \defref{defn:Log-fil-3} for the definition of $\Fil_D \Omega^q_U$.
\begin{cor}\label{cor:fil D-Y}
    For any $D \in \Div_E(X)$, we have the following exact sequences.
    \begin{equation}\label{eqn:fil D-Y}
        0 \to \Fil_{D-Y} \Omega^q_V \to \Fil_D \Omega^q_U \to i_* \Fil_{D'} \Omega^q_{U'} \to 0,
    \end{equation}
    \begin{equation}\label{eqn:no-log-1}
        0 \to \Fil_{-Y} \Omega^q_{V_1} \to \Omega^q_X \to i_*  \Omega^q_{Y} \to 0, 
    \end{equation}
    where $V_1:= X-Y$ and $D' := D \times_X Y$.
\end{cor}

\begin{proof}
    Note that, the second exact sequence follows from (\ref{eqn:fil D-Y}), considering trivial log structure on $X$ and $Y$ (i.e. putting $E$ to be zero divisor in (\ref{eqn:fil D-Y})).

    For the first one, note that (\ref{eqn:s.e.s-2}) gives the following exact sequence of $\sO_X$ modules.
     \begin{equation}\label{eqn:s.e.s-3}
        0 \to \sO_X(-Y) \cdot \Omega^q_X(\log E) + d(\sO_X(-Y))\cdot\Omega^{q-1}_X(\log E) \to \Omega^q_X(\log E) \to i_*\Omega^q_Y(\log E') \to 0.
\end{equation}
Now note,
\begin{equation}\label{eqn:fil-Y}
    \begin{array}{ll}
        \sO_X(-Y) \cdot \Omega^q_X(\log E) + d(\sO_X(-Y))\cdot\Omega^{q-1}_X(\log E)  & = \Omega^q_X(\log (E+Y)) (-Y)\\
         & = \Fil_{-Y} \Omega^q_V.
    \end{array}
\end{equation}
 Thus, we obtain the following exact sequence.   
    \begin{equation}
        0 \to \Fil_{-Y} \Omega^q_V \to  \Omega^q_X(\log E) \to i_*  \Omega^q_Y(\log E') \to 0.
    \end{equation}
    Now tensoring by $\sO_X(D)$ over $\sO_X$ and using \lemref{lem:Log-fil-4}(2), we get the result.
\end{proof}

\begin{cor}\label{cor:fil-D-Y-2}
    For any $D\in \Div_E(X)$ and $q \ge 0$, we have the following exact sequences.
    \begin{equation}\label{eqn:fil-D-Y-2}
        0 \to Z_1\Fil_{D-Y}\Omega^q_{V} \to Z_1\Fil_{D}\Omega^q_{U} \xrightarrow{\theta_{D}} i_*Z_1\Fil_{D'}\Omega^q_{U'} \to 0.
    \end{equation}
    \begin{equation}\label{eqn:no-log-3}
        0\to Z_1\Fil_{-Y}\Omega^q_{V_1} \to Z_1\Omega^q_X \xrightarrow{\theta} i_*Z_1\Omega^q_{Y} \to 0.
    \end{equation}
\end{cor}

\begin{proof}
    By \corref{cor:fil D-Y}, we have the following diagram.
    \begin{equation}
        \xymatrix{
        0 \ar[r]& \Fil_{D-Y}\Omega^q_{V} \ar[r] \ar[d]_-{d}& \Fil_{D}\Omega^q_{U} \ar[r]^-{\theta_D} \ar[d]_-{d} & i_* \Fil_{D'}\Omega^q_{U'} \ar[r] \ar[d]_-{i_*d} & 0 \\
        0 \ar[r]& \Fil_{D-Y}\Omega^{q+1}_{V} \ar[r] & \Fil_{D}\Omega^{q+1}_{U} \ar[r]^-{\theta_D} & i_* \Fil_{D'}\Omega^{q+1}_{U'} \ar[r] & 0.
        }
    \end{equation}
Since $Z_1\Fil_D\Omega^q_U= \Ker\ (d:\Fil_D\Omega^q_U \to \Fil_D\Omega^{q+1}_U)$, this gives the following left exact sequence.
$$0 \to Z_1\Fil_{D-Y}\Omega^q_{V} \to Z_1\Fil_{D}\Omega^q_{U} \xrightarrow{\theta_{D}} i_*Z_1\Fil_{D'}\Omega^q_{U'}, $$
where, we note that $i_*Z_1\Fil_{D'}\Omega^q_{U'} = ker \ i_*d$, because $i_*$ is (left) exact. To show the right exactness of (\ref{eqn:fil-D-Y-2}), we consider the following diagram.
$$\xymatrix{
\Fil_{D}W_{2}\Omega^q_{U} \ar[r]^-{\theta_D} \ar[d]_-{F} & i_* \Fil_{D'}W_{2}\Omega^q_{U'} \ar[d]_-{i_*F}  \\
 Z_1\Fil_{D}\Omega^q_{U} \ar[r]^-{\theta_{D}} & i_*Z_1\Fil_{D'}\Omega^q_{U'}. 
}$$
Note that $i_*F$ are surjective by \lemref{lem:Complete-4}(1) and exactness of $i_*$. Also, the top $\theta_D$ is surjective because locally the map is $\Fil_{\n}W_2\Omega^q_{(R_\pi)}\surj \Fil_{\n}W_2\Omega^q_{((R/t)_{\ov \pi})}$, where $R$ is a RLR and $t,\pi,\ov \pi$ and $\un n$ corresponds to $Y,E$, $E'$ and $D$, respectively. This map is surjective by definition of $\Fil_{\n}W_2\Omega^q_{(R_\pi)}$ (cf. \defref{defn:Log-fil-3}) and noting that $\Fil_{\n}W_2(R_\pi)\surj \Fil_{\n}W_2((R/t)_{\ov \pi})$. Hence, we conclude the bottom $\theta_D$ is also surjective.

The proof of (\ref{eqn:no-log-3}) is similar and hence we skip the proof.

\end{proof}

Now recall the complexes 
$$W_m\sF^{q,\bullet}_{D,U} =\left[Z_1\Fil_{D} W_m\Omega^q_U \xrightarrow{1 -C}
\Fil_{D}W_m\Omega^q_U\right], \\ W_m\sF^{q,\bullet}_X = \left[Z_1W_m\Omega^q_{X} \xrightarrow{1 -C}
W_m\Omega^q_{X}\right],$$
where $D \ge 0 $ in $\Div_E(X)$. Here, we use the notation $W_m\sF^{q,\bullet}_{D,U}$ instead of $W_m\sF^{q,\bullet}_{D}$, to keep track of the log structure corresponding to the open subscheme $U$. Let
$$ \sN^{q,\bullet}_{D,U}:= \left[Z_1\Fil_{-D}\Omega^q_U \xrightarrow{1-C}\Fil_{(-D)/p}\Omega^q_U \right], \ \ \text{ for }D\ge E. $$
By \lemref{lem:Complex-6}(6), we know $\sN^{q,\bullet}_{D,U}$ is quasi-isomorphic to $\Omega^q_{X|D,\log}$.
 
Then  Corollaries \ref{cor:fil D-Y} and  \ref{cor:fil-D-Y-2} give :

\begin{cor}\label{cor:ker-lef}
    We have the following short exact sequences of complexes in $\sD(X_\et)$.
    \begin{equation}
        0 \to W_1\sF^{q,\bullet}_{D-Y,V} \to W_1\sF^{q,\bullet}_{D,U} \to i_*W_1\sF^{q,\bullet}_{D',U'} \to 0
    \end{equation}
    \begin{equation}
      0 \to W_1\sF^{q,\bullet}_{-Y,V_1} \to W_1\sF^{q,\bullet}_X  \to i_*W_1\sF^{q,\bullet}_Y \to 0.
    \end{equation}
    \begin{equation}
        0 \to \sN^{q,\bullet}_{D-Y,V} \to \sN^{q,\bullet}_{D,U} \to i_*\sN^{q,\bullet}_{D',U'} \to 0.
    \end{equation}
    In particular, we have exact triangles in $\sD(X_\et)$.
    \begin{equation}
       W_1\sF^{q,\bullet}_{-Y,V_1} \to \Omega^q_{X,\log} \to i_*\Omega^q_{Y,\log} ,
    \end{equation}
    \begin{equation}
        \sN^{q,\bullet}_{D-Y,V} \to \Omega^q_{X|D,\log} \to i_*\Omega^q_{Y|D',\log}.
    \end{equation}
\end{cor}

\section{Weak Lefschetz for \texorpdfstring{$p$}{p}-adic {\'e}tale cohomology}

Recall that we have morphisms of sheaves
$$W_m\sF^{q,\bullet}_{D,U} \to i_*W_m\sF^{q,\bullet}_{D',U'}, \ W_m\Omega^q_{X|D,\log} \to i_*W_m\Omega^q_{Y|D',\log}, \ 
W_m\Omega^q_{X,\log} \to i_*W_m\Omega^q_{Y,\log}.$$ 
These induce the following group homomorphisms.
$$\H^i_\et(X,W_m\sF^{q,\bullet}_{D,U}) \to \H^i_\et(X,i_*W_m\sF^{q,\bullet}_{D',U'})=^*\H^i_\et(Y,W_m\sF^{q,\bullet}_{D',U'})$$
$$H^i_\et(X,W_m\Omega^q_{X|D,\log}) \to H^i_\et(X, i_*W_m\Omega^q_{Y|D',\log})=^*H^i_\et(Y,W_m\Omega^q_{Y|D',\log}),$$
$$H^i_\et(X,W_m\Omega^q_{X,\log}) \to H^i_\et(X, i_*W_m\Omega^q_{Y,\log})=^*H^i_\et(Y,W_m\Omega^q_{Y,\log}),$$
where the equalities "$=^*$" hold because of exactness of $i_*$.

\begin{lem}\label{lem:ample}
    Fix a divisor $D \in \Div_E(X)$. Then there exists a smooth hypersurface section  $Y$ such that $Y$ intersects $E$ transversally. Moreover, for all $n \ge 1$ and
    \begin{enumerate}
        \item  for all $i, q\ge 0$ such that $i+q \le N-1$, we have
        \begin{enumerate}
            \item \ $ H^{i}(X,\Omega^q_X(\log E)(D-nY))=0, $
            \item \ $H^{i}(X,\Omega^q_X(-nY))=0;$
        \end{enumerate} 
        \item  for all $i, q\ge 0$ such that $i+q \le N-2$, we have
        \begin{enumerate}
            \item \ $ H^i(X,i_*\Omega^{q}_Y(\log \ E')(D-nY) )= 0 $
            \item \ $H^i(X,i_*\Omega^{q}_Y (-nY))=0$;
        \end{enumerate}
    \end{enumerate}
   where these cohomologies are taken with respect to Zariski topology (eqv. {\'e}tale topology).
\end{lem}
\begin{proof}
Let $\sO_X(1)$ be the ample line bundle on $X$ corresponding to the embedding $X \inj \P^L_k$. 
For $D \in \Div_E(X)$, there exists $n_0 >0$ such that 
\begin{equation}\label{eqn:ample}
    H^i(X,\Omega^q_X(\log E)(D-n))=H^i(X,\Omega^q_X(-n))=0, \ \ \text{for all } n \ge n_0, i+q \le N-1.
\end{equation}
Using Serre's vanishing theorem, this is a consequence of the Grothendieck-Serre duality and the fact the $\Omega^q_X(\log E)$, $\Omega^q_X$ are locally free sheaves of finite rank (see \cite[Lem~5.1]{Zero cycle}). We also note that we obtain the identities of \eqref{eqn:ample} for all $i,q$ such that $i \le N-1$. But, for the proof, we only need the condition $i+q \le N-1$. 

By Bertini theorem \cite[Thm~3.9, 4.5]{Bertini}, there exists $H \in H^0(\P^L_k,\sO(n'))$ for all $n' >> n_0$ such that $Y=H \cap X$ and $Y_i= H \cap E_i$ are smooth for all $i \le r$. In particular, $Y$ intersects $E$ transversally. This $Y$ together with \eqref{eqn:ample} prove the item (1) of the lemma.

Now we set $A^q_n:=\Omega^q_Y(\log E')(D-nY)$ and $B^q_n:=\Omega^q_Y(-nY)$, where $n \ge 1, q \ge 0$.

\textbf{Claim:} $H^i(Y,A^q_n)=H^i(Y,B_n^q)=0$ for all $i+q \le N-2$ and $n \ge 1$.

Proof: We show this only for $A^q_n$, because the other case is analogous. We proceed by induction on $q \ge 0$.

If $q=0$ then $A^q_n=\sO_Y(D-nY)$. In this case the claim is clear from the exact sequence
$$0 \to \sO_X(D-(n+1)Y) \to \sO_X(D-nY) \to i_*\sO_Y(D-nY) \to 0,$$
together with \eqref{eqn:ample}.

Now for $q \ge 1$, we consider the following two exact sequences.
\begin{equation}\label{eqn:s.e.s-5}
    0 \to \Omega^{q-1}_Y(\log E')(D-(n+1)Y) \xrightarrow{d} i^*\Omega^q_X(\log E)(D-nY) \to \Omega^q_Y(\log E')(D-nY) \to 0,
\end{equation}
$$ 0 \to \Omega^q_X(\log E)(D-(n+1)Y)\to \Omega^q_X(\log E)(D-nY) \to i_*i^*\Omega^q_X(\log E)(D-nY) \to 0,$$
where, the first one is obtained from \lemref{lem:s.e.s-2}. The second exact sequence together with \eqref{eqn:ample} imply $H^i(Y,i^*\Omega^q_X(\log E)(D-nY))=0$ for all $i+q \le N-2$.  Also, by induction hypothesis, $H^i(Y,A^{q-1}_{n+1})=0$ for all $i+q \le N-1$. Hence, from the top exact sequence of \eqref{eqn:s.e.s-5} we conclude $H^i(Y,A^q_n)=0$ for all $i+q \le N-2$.

\end{proof}

\begin{defn}\label{def:ample}
    A smooth hypersurface section $Y$ is called sufficiently ample with respect to $X$ (resp. sufficiently ample with respect to $(X,D)$) if the condition (1.b) (resp. (1.a) for every $D'$ such that $|D'| \le |D|$) of \lemref{lem:ample} is satisfied. 
\end{defn}

\begin{rem}\label{rem:ample}
    By virtue of \lemref{lem:ample}, a sufficiently ample smooth hypersurface section $Y$ w.r.t $X$ (resp. $(X, D)$) always exists. Moreover, we note in the proof of the lemma that condition (2.a) (resp. (2.b)) follows from (1.a) (resp. (1.b)).

\end{rem}
Now we prove the main theorem of this chapter. The following theorem also extends \cite[Thm~3.4]{Kerz-Saito-ANT}. 
\begin{thm}\label{thm:Lefschetz}
    Let $X$, $E$ be as in the beginning of \S ~\ref{chap:Lef} and $D \ge 0$ in $\Div_E(X)$. Let $Y$ be a smooth hypersurface section of $X$ and $D'=D \times_X Y$. If $Y$ is sufficiently ample w.r.t $X$(resp. $(X,D), (X,-D)$), then for all $m$ the group homomorphisms
    \begin{equation}
        H^i_\et(X,W_m\Omega^q_{X,\log}) \to H^i_\et(Y,W_m\Omega^q_{Y,\log}),
    \end{equation}
   \begin{equation}
     \H^i_\et(X,W_m\sF^{q,\bullet}_{D,U}) \to \H^i_\et(Y,W_m\sF^{q,\bullet}_{D',U'}),
      \end{equation} 
     \begin{equation}
      and  \ \   H^i_\et(X,W_m\Omega^q_{X|D,\log}) \to H^i_\et(Y,W_m\Omega^q_{Y|D',\log}) 
    \end{equation}

    are isomorphisms if $i+q \le N-2$ and injective if $i+q=N-1$.
\end{thm}
\begin{proof}
    We shall proceed by induction on $m$. By \propref{prop:usual}(1), \thmref{thm:Global-version}(11) and  \corref{cor:Complex-7}, it is enough to show, for all $j \ge 0$, the maps
    
     \begin{equation}\label{eqn:vanishing-0}
        H^i(X,\Omega^q_{X,\log}) \xrightarrow{} H^i(Y,\Omega^q_{Y,\log}),
    \end{equation}
    \begin{equation*}
        \H^i_\et(X,W_1\sF^{q,\bullet}_{D/p^j,U}) \xrightarrow{} \H^i_\et(Y,W_1\sF^{q,\bullet}_{D'/p^j,V}),
    \end{equation*}
    \begin{equation*}
        H^i_\et(X,\Omega^q_{X|\lceil D/p^j \rceil,\log}) \xrightarrow{} H^i_\et(Y,\Omega^q_{Y|\lceil D'/p^j \rceil,\log}) ,
    \end{equation*}
are bijective (resp. injective) if $i+q\le N-2$ (resp. $i+q=N-1$).

Since $D \ge 0$ in each of the above cases, we note that $D/p^j=D/p^{j+1}$ and $\lceil D/p^j \rceil = \lceil D/p^{j+1} \rceil$ for $j >> 0$.

In view of \corref{cor:ker-lef}, it is enough to show, for all $i+q \le N-1$ and $0 \le \ov D \le D$,
    \begin{equation}\label{eqn:vanishung-1}
        \H^i_\et(X,W_1\sF^{q,\bullet}_{-Y,V_1})=\H^i_\et(X,W_1\sF^{q,\bullet}_{\ov D-Y,V}) = \H^i_\et(X, \sN^{q,\bullet}_{\ov D-Y,U})= 0.
    \end{equation}

We shall show the left and the middle groups are zero. Proof for the right one is analogous to that of middle one.
For any $q\ge 1$, there exists a residue homomorphism (see the proof of  \cite[Lem~6.2]{Binda-Saito})
$$Res^q: \Omega^q_X(\log (E+Y)) \to i_*\Omega^{q-1}_Y(\log \ E')$$
which is characterized by 
\begin{itemize}
    \item $Res^q(\omega) =0$, if $\omega \in \Omega_X^q(\log \ E)$ and
    \item $Res^q(\omega \wedge \dlog(f))= \ov{\omega}$, if $\omega \in \Omega_X^{q-1}(\log \ E)$ is a lift of $\ov{\omega} \in \Omega_Y^{q-1}(\log \ E')$,
\end{itemize}
where $f$ is a local equation of $Y$.
Similarly, we have a residue homomorphism
$$Res^q: \Omega^q_X(\log Y) \to i_*\Omega^{q-1}_Y.$$
So we have exact sequences
\begin{equation}\label{eqn:res-1}
    0 \to \Omega_X^q(\log \ E) \to \Omega^q_X(\log (E+Y)) \xrightarrow{Res^q} i_*\Omega^{q-1}_Y(\log \ E') \to 0,
\end{equation}
\begin{equation}\label{eqn:res-2}
    0 \to \Omega^q_X \to \Omega^q_X(\log Y) \xrightarrow{Res^q} i_*\Omega^{q-1}_Y \to 0.
\end{equation}
Now applying $- \otimes_{\sO_X}\sO(\ov D-Y)$ in (\ref{eqn:res-1}) and $-\otimes_{\sO_X}\sO(-Y)$ in (\ref{eqn:res-2}), we get
\begin{equation}\label{eqn:res-3}
    0 \to \Omega_X^q(\log \ E)(\ov D-Y) \to \Omega^q_X(\log (E+Y))(\ov D-Y) \xrightarrow{Res^q} i_*\Omega^{q-1}_Y(\log \ E')(\ov D-Y) \to 0.
\end{equation}
\begin{equation}\label{eqn:res-4}
    0 \to \Omega^q_X(-Y) \to \Omega^q_X(\log Y)(-Y) \xrightarrow{Res^q} i_*\Omega^{q-1}_Y (-Y)\to 0.
\end{equation}

 Using the exact sequences (\ref{eqn:res-3}), (\ref{eqn:res-4}) and hypothesis on $Y$, we get 
\begin{equation}\label{eqn:s.e.s-4}
    H^i(X,\Omega^q_X(\log (E+Y))(\ov D-Y))=0 = H^i(X,\Omega^q_X(\log Y)(-Y)),
\end{equation}
for all $i \in [0, N-q-1]$ and $0 \le \ov D \le D$.

Finally, to prove (\ref{eqn:vanishung-1}), enough to show for all $i \in [0, N-q-1] \text{ and } 0 \le \ov D \le D$,
$$H^i(X, Z_1\Omega^q_X(\log (E+Y))(\ov D-Y))=0=H^i(X, Z_1\Omega^q_X(\log Y)(-Y)).$$
For $q=0$, this is clear by the hypothesis on $Y$ and noting that $Z_1\sO_X(\ov D-Y) \cong  \sO_X(\ov D/p-Y)$ (see \ref{lem:Complete-4}(2)).

For $q \ge 1$, consider the exact sequence (cf. \lemref{lem:Complete-9})
$$0 \to Z_1\Omega^{q-1}_X(\log (E+Y))(\ov D-Y)\to \Omega^{q-1}_X(\log (E+Y))(\ov D-Y) \xrightarrow{d} \hspace{1cm} $$
$$ \hspace{3cm} Z_1\Omega^q_X(\log (E+Y))(\ov D-Y) \xrightarrow{C} \Omega^q_X(\log (E+Y))(\ov D/p-Y) \to 0.$$
Now, by induction on $q$ and \eqref{eqn:s.e.s-4}, the claim follows and hence the theorem.

\end{proof}

\section{Weak Lefschetz for \texorpdfstring{$\ell$}{ell}-adic {\'e}tale cohomology}
Before we proceed further, we review some Weak Lefschetz theorems for $\ell$-adic {\'e}tale cohomology. Let $\ell$ be any prime number different from $p$.
\begin{thm}\label{thm:ell-lef}
    Let $k$ be either finite or 1-local or separably closed field of characteristic $p \ge 0$. Let $X$, $E$, $E'$, $U$, $U'$ and $Y$ be as in the beginning of \S ~\ref{chap:Lef}. Then, for all $j \in \Z$ we have the following.
    \begin{enumerate}
        \item The map $H^i_\et(X,\Z/\ell^m(j)) \to H^i_\et(Y,\Z/\ell^m(j))$ is isomorphism if $i \le N-2$ and injective if $i=N-1$.
        \item The map $H^i_\et(U,\Z/\ell^m(j)) \to H^i_\et(U',\Z/\ell^m(j))$ is isomorphism if $i \le N-2$ and injective if $i=N-1$.
        \item The map $H^i_E(X,\Z/\ell^m(j)) \to H^i_{E'}(Y,\Z/\ell^m(j))$ is isomorphism if $i \le N-1$ and injective if $i=N$.
    \end{enumerate}
\end{thm}

\begin{proof}
    First we prove item (1). Let's write $\Wedge=\Z/\ell^m$. Recall there is a perfect pairing of finite abelian groups (see \cite[Introduction]{JSZ}, \cite[Chapter VI,\S 11]{Milne-EC}, \cite[Thm~9.9]{GKR})
    \begin{equation}\label{eqn:duality}
        H^i_{c,\et}(X\setminus Y,\Wedge(j)) \times H^{2N+e-i}_\et(X\setminus Y,\Wedge(N+e'-j))\to H^{2N+e}_{c,\et}(X\setminus Y,\Wedge(N+e')) \cong \Wedge,
    \end{equation}
    where $e'=1$ if $k$ is an $1$-local field and $0$ otherwise; $e$ is the cohomological dimension of $k$. Note, $e=0$ if $k$ is separably closed, $e=1$ if $k$ is a finite field and $e=2$ if $k$ is an 1-local field.

    Now, for the smooth hypersurface section $Y$ in $X$, we look at the exact sequence
    $$\cdots \to H^i_{c,\et}(X\setminus Y, \Wedge(j)) \to H^i_\et(X, \Wedge(j)) \to H^i_\et(Y,\Wedge(j)) \to  H^{i+1}_{c,\et}(X\setminus Y, \Wedge(j)) \to \cdots.$$
    By \eqref{eqn:duality}, $H^i_{c,\et}(X\setminus Y, \Wedge(j)) \xrightarrow{\cong} (H^{2N+e-i}_\et(X\setminus Y,\Wedge(N+e'-j)))^\vee$. Since, $X\setminus Y$ is affine, its $\ell$-cohomological dimension is $\le N+e$ (\cite[VI, Thm~7.2]{Milne-EC}). Hence, $ H^i_{c,\et}(X\setminus Y, \Wedge(j))= (H^{2N+e-i}_\et(X\setminus Y,\Wedge(N-j)))^\vee=0$ if $i<N$. This proves item (1).

    Now we shall prove item (2) and (3) simultaneously using induction on the number of components of $E$. We write $E=\bigcup\limits_{t=1}^{r}E_t$, where $r$ is the number of irreducible components of $E$. Since $Y$ intersects $E$ transversally (by hypothesis on $Y$), we have $E'=\bigcup\limits_{t=1}^{r}E'_t$, where $E'_t=E_t \cap Y$. Now, we look at the following commutative diagram obtained by the localisation sequence. 
    \begin{equation}\label{eqn:loc}
        \xymatrix@C.8pc{
        \cdots \ar[r]&H^{i}_E(X) \ar[r] \ar[d]^-{\phi_i} & H^i_\et(X) \ar[r] \ar[d]^-{\theta_i}& H^i_\et(U) \ar[r] \ar[d]^-{\psi_i} & H^{i+1}_E(X) \ar[r] \ar[d]^-{\phi_{i+1}} &H^{i+1}_\et(X) \ar[r] \ar[d]^-{\theta_{i+1}}& \cdots \\
        \cdots \ar[r]&H^{i}_{E'}(Y) \ar[r] &H^i_\et(Y) \ar[r]& H^i_\et(U') \ar[r] & H^{i+1}_{E'}(Y) \ar[r]&H^{i+1}_\et(Y) \ar[r]& \cdots,
        }
    \end{equation}
    where $H^i_\et(-)$ (resp. $H^i_*(-)$) is the short hand notation for $H^i_\et(-,\Wedge(j))$ (resp. $H^i_*(-, \Wedge(j))$) and  all the vertical arrows are the canonical restriction maps. 
   
    To start the induction, let $r=1$ (i.e, $E$ and $E'$ are irreducible). Then, we claim that item (2) and (3) are true.
    Indeed, by purity \cite[VI, Thm~5.1]{Milne-EC}, we have $H_\et^{i-2}(E, \Wedge(j-1)) \xrightarrow{\cong} H^{i}_E(X, \Wedge(j))$ and a similar isomorphism for $E'$ and $Y$. These isomorphisms are compatible with the map $\phi_i$ (see \cite[A.2]{etale-motives} and \cite[Ch~10]{Mazza}). Since dim $E=N-1$, by item (1), we get that $\phi_i$ is bijective if $i \le N-1$ and injective if $i=N$. A diagram chase  in \eqref{eqn:loc} shows that $\psi_i$ is bijective if $i\le N-2$ and injective if $i =N-1$. This proves the claim.

    Now if $r>1$, we consider the following commutative diagram.
    \begin{equation}\label{eqn:loc-1}
        \xymatrix@C.6pc{
         \ar[r]& H^{i-1}_{E\setminus \ov{E}}(X\setminus \ov{E},\Wedge(j)) \ar[r] \ar[d]^-{\gamma_{i-1}} &H^{i}_{\ov E}(X,\Wedge(j)) \ar[r] \ar[d]^-{\alpha_i}& H^{i}_{E}(X,\Wedge(j)) \ar[r] \ar[d]^-{\beta_i}& H^i_{E\setminus \ov{E}}(X\setminus \ov{E}, \Wedge(j))\ar[r]\ar[d]^-{\gamma_i}&  \\
         \ar[r]&  H^{i-1}_{E'\setminus \ov{E}'}(Y\setminus \ov{E}',\Wedge(j))\ar[r] &H^{i}_{\ov E'}(Y,\Wedge(j)) \ar[r]& H^{i}_{ E'}(Y,\Wedge(j)) \ar[r]& H^{i}_{E'\setminus \ov E'}(Y\setminus \ov E',\Wedge(j)) \ar[r]&,
        }
    \end{equation}
    where $\ov E=\bigcup\limits_{t=2}^{r}E_t$ and $\ov E'= \ov E \cap Y$. We note that $E\setminus\ov E$ (resp. $E'\setminus \ov E'$) is a smooth divisor in $X\setminus \ov E$ (resp. $Y\setminus \ov E'$). So, using purity theorem again, we can canonically identify the map $\gamma_i$ with the map $H^{i-2}_\et(E\setminus \ov{E}, \Wedge(j-1)) \xrightarrow{}H^{i-2}_\et(E'\setminus \ov{E}', \Wedge(j-1))$.

    Since $E\setminus \ov E=E_1 \setminus (E_1 \cap \ov E)$ and $E_1 \cap \ov E$ is a simple normal crossing divisor in $E_1$ with $r-1$ irreducible components, our induction hypothesis implies that item (2) of the theorem is valid if we take $X=E_1$ and $U=E_1\setminus \ov E$. That is, $\gamma_i$ is bijective if $i-2 \le N-1-2=N-3$ (i.e, $i \le N-1$) and injective if $i=N$. Also, by induction hypothesis, item (3) is valid by taking $\ov E$ instead of $E$. Hence $\alpha_i$ is bijective if $i \le N-1$ and injective if $i=N$. An easy diagram chase in \eqref{eqn:loc-1} shows that $\beta_i$ is bijective if $i \le N-1$ and injective if $i=N$. That is, item (3) is true when number of irreducible components of $E$ is $r$. Moreover, note that $\beta_i=\phi_i$. Using item (1), a diagram chase in \eqref{eqn:loc} shows that item (2) is valid if number of irreducible components of $E$ is $r$. Hence, by induction, we have proved item (2) and (3) simultaneously.
\end{proof}

\section{Applications}\label{sec:Br-lef}
In this section, we shall prove a Lefschetz theorem for Brauer groups, as a application of \thmref{thm:Lefschetz}.

\subsection{Lefschetz for Picard groups}
 We need a version of Noether-Lefschetz theorem for Picard groups over an arbitrary field, which we are going to prove now. Before we proceed further, we let ${\bf Ab_p}$ be the category of $p$-primary torsion abelian groups. 
\begin{prop}\label{prop:N-lef}
    Let $k$ be any field. Let $X$ be a smooth geometrically connected projective variety over $k$ such that dim $X\ge 4$.  Let $Y$ be a smooth  hypersurface section which satisfies $q=0$ case of the conditions (1.b) and (2.b) of \lemref{lem:ample}   (cf. \defref{def:ample}). Then the canonical map $$\Pic(X) \xrightarrow{\tau} \Pic(Y)$$
    is injective and $\coker \ \tau \in {\bf Ab_p}$ (where $p=$char $k$).
\end{prop}
\begin{proof}
    \textbf{Case 1:} $k$ is algebraically closed. In this case $\tau$ is an isomorphism by \cite[Thm~IV.3.1]{Hartshorne-Ample}. One notes that the conditions (i) and (ii) of loc. cit. Thm IV.3.1 are always valid for any smooth hypersurface section $Y$ of dim $\ge 2$ (see Thm IV.1.5 of loc. cit. for the Leff condition). On the other hand, the condition (iii) of loc. cit. Thm IV.3.1  is also satisfied because of the $q=0$ case of the condition (2.b) of \lemref{lem:ample}.

    \textbf{Case 2:} $k$ is separably closed but not algebraically closed. Let $\ov k$ be an algebraic closure of $k$. Also, let  $\ov X=X\times_{\Spec k}  \Spec \ov k$ and $\ov Y= Y \times_{\Spec k} \Spec \ov k$. By flat base change theorem \cite[Thm~III.9.3]{Hartshorne-AG}, we see that $(\ov X,\ov Y)$ also satisfies the $q=0$ case of the conditions (1.b) and (2.b) of the \lemref{lem:ample}. Moreover, because $X$ is geometrically connected (equivalently, $H^0(X,\sO_X)=k$), the condition (1.b) of \lemref{lem:ample} implies $H^0(Y,\sO_Y)=k$, by looking at the exact sequence 
    $$0=H^0(X,\sO_X(-Y)) \to H^0(X,\sO_X) \xrightarrow{\cong} H^0(Y,\sO_Y) \to H^1(X,\sO_X(-Y))=0.$$
    In particular $\ov Y$ is connected.
    Now, case 1 implies $\Pic(\ov X) \xrightarrow{\ov \tau} \Pic(\ov Y)$ is an isomorphism. We note that the maps $\Pic(X) \to \Pic(\ov X)$ and $\Pic(Y) \to \Pic(\ov Y)$ are injective (see \cite{Stacks}) and their cokernels are in ${\bf Ab_p}$, if char $k=p>0$. To see the latter fact, we consider the exact sequence :
    $$0 \to \sO_X^\times \to \sO_{X_{k'}}^\times \to \sO_{X_{k'}}^\times/\sO_X^\times \to 0,$$
    which gives the exact sequence : $\Pic(X) \xrightarrow{\theta_{k'}} \Pic(X_{k'}) \to H^1_{\et}(X,\sO_{X_{k'}}^\times/\sO_X^\times )$, where $k'$ is any finite inseparable extension of $k$. Since $\sO_{X_{k'}}^\times/\sO_X^\times $ is annihilated by $p^M$ for some $M>0$, we get that $H^1_{\et}(X,\sO_{X_{k'}}^\times/\sO_X^\times )$ and hence $\coker \ \theta_{k'}$ is $p^M$-torsion. Now, by taking direct limit over $k'$, we get that the cokernel of $\theta$ and $\theta'$ in the following diagram are in ${\bf Ab_p}$.  
    Then, the following diagram 
    \begin{equation}
        \xymatrix@C.8pc{
        0 \ar[r]& \Pic(X) \ar[r]^-{\theta} \ar[d]^-{\tau}& \Pic(\ov X)\ar[r]\ar[d]^-{\ov \tau}_{\cong}& \coker \ \theta\ar[d]^-{\alpha}\ar[r]&0 \\
        0 \ar[r]& \Pic(Y)\ar[r]^-{\theta'}& \Pic(\ov Y)\ar[r]& \coker \ \theta'\ar[r]&0
        }
    \end{equation}
    implies that there is an exact sequence 
    \begin{equation}\label{eqn:N-Lef}
        0 \to \Pic(X) \xrightarrow{\tau} \Pic(Y) \to \Ker \alpha \to 0.
    \end{equation}
    In particular, $\coker \ \tau \in {\bf Ab_p}$. This proves our assertion in Case 2.

    \textbf{Case 3:} $k$ is arbitrary. Let $k^s$ be a separable closure of $k$ and $X^s=X\times_{\Spec k} \Spec k^s$, $Y^s=Y\times_{\Spec k} \Spec k^s$. By a similar argument given in case 2, we see that $X^s$ and $Y^s$ are geometrically connected and the pair $(X^s,Y^s)$ satisfies $q=0$ case of the conditions (1.b) and (2.b) of \lemref{lem:ample}. In particular, by case 2, we have an exact sequence of  $G=\Gal(k^s/k)$-modules
    $$ 0 \to \Pic(X^s) \xrightarrow{\tau^s} \Pic(Y^s) \to \coker \ \tau^s \to 0,$$
    where $\coker \ \tau^s \in {\bf Ab_p}$. As a result, we get an exact sequence
    $$0 \to \Pic(X^s)^G \xrightarrow{(\tau^s)^G} \Pic(Y^s)^G \to (\coker \ \tau^s)^G.$$ Thus $\coker ((\tau^s)^G)\in {\bf Ab_p}$. Now by \cite[Prop~5.4.2]{CTS}, we have a commutative diagram of exact sequences :
    \begin{equation*}
        \xymatrix@C.8pc{
        0 \ar[r]& \Pic(X) \ar[r] \ar[d]^-{\tau}& \Pic(X^s)^G \ar[r]\ar[d]^-{(\tau^s)^G}& \Br(k)\ar@{=}[d] \\
        0 \ar[r]& \Pic(Y)\ar[r]& \Pic(Y^s)^G\ar[r]& \Br(k).
        }
    \end{equation*}
    
    This implies $\tau$ is injective and $\coker \ \tau \in {\bf Ab_p}$. This proves the proposition.
\end{proof}
\subsection{Lefschetz for Brauer groups}
Now we prove the following Lefschetz theorem for $\Br(X)$ and $\Br^\divv(X|D)$.
\begin{thm}\label{thm:Br-lef}
    Let $k$ be either a finite field or 1-local field or algebraically closed field such that char $k=p>0$. Let $X$ be a smooth geometrically connected projective variety over $k$ such that dim $X\ge 4$.  Then, for a sufficiently ample (cf. \defref{def:ample})  hypersurface section $Y$ (w.r.t $X$ or ($X,D$)), we have 
    \begin{enumerate}
        \item \ $\Br(X) \xrightarrow{\cong} \Br(Y),$
        \item \ $\Br^\divv(X|D) \xrightarrow{\cong} \Br^\divv(Y|D')$, where $D$ and $D'$ are as in \thmref{thm:Lefschetz}.
    \end{enumerate}
\end{thm}
\begin{proof}
    For the item (1), we consider the following diagram
    \begin{equation}\label{eqn:lef-Br}
        \xymatrix@C.8pc{
        0 \ar[r] & \Pic(X)/p^m \ar[r] \ar[d] & H^1_\et(X,W_m\Omega^1_{X,\log}) \ar[r] \ar[d]^-{\cong} & \ _{p^m}\Br(X) \ar[r] \ar[d] & 0 \\
        0 \ar[r] & \Pic(Y)/p^m \ar[r] & H^1_\et(Y,W_m\Omega^1_{Y,\log}) \ar[r] & \ _{p^m}\Br(Y) \ar[r] & 0.
        }
    \end{equation}
    By \thmref{thm:Lefschetz}, the middle vertical arrow is isomorphism for all $m \ge 1$. Then, taking direct limit over $m$, the above diagram induces an exact sequence
    $$0 \to A \otimes_\Z\Q_p/\Z_p \to \Br(X)\{p\} \to \Br(Y)\{p\}\to 0,$$
    where $A:=\coker (\Pic(X) \to \Pic(Y))$.
    On the other hand, by \propref{prop:N-lef} and \remref{rem:ample}, we know that $A$ is a torsion group. Hence $A \otimes_\Z\Q_p/\Z_p=0$ and $\Br(X)\{p\}\xrightarrow{\cong}\Br(Y)\{p\}$.

    On the other hand, let $\ell$ be a prime different from $p$. Then, using \thmref{thm:ell-lef} instead of \thmref{thm:Lefschetz}, a similar argument as above shows $\Br(X)\{\ell\}\xrightarrow{\cong}\Br(Y)\{\ell\}$. Thus, we get $\Br(X)\xrightarrow{\cong}\Br(Y)$ (because $\Br(-)$ is torsion for regular scheme).

    Now we consider item (2). Recall, $U$ and $U'$ are the complements of $E$ and $E'$ in $X$ and $Y$, respectively (see \S ~\ref{chap:Lef}). Using the exact sequence (see \corref{cor:F^q})
    $$ 0 \to \Pic(U)/p^m \to \H^1_\et(X, W_m\sF^{1,\bullet}_{D}) \to \ _{p^m}\Br^\divv(X|D) \to 0,$$
    and \thmref{thm:Lefschetz},  the proof of $\Br^\divv(X|D)\{p\}\xrightarrow{\cong}\Br^\divv(Y|D')\{p\}$ is analogous to that of item (1). Here one also needs to use that $(\coker (\Pic(U) \to \Pic(U')))\otimes_\Z \Q_p/\Z_p=0$. But this follows from the above projective case because $\Pic(Y) \surj \Pic(U')$.   
    
    Now we consider the prime to $p$ case. Note that $\Br^\divv(X|D)\{\ell\}= \Br(U)\{\ell\}$ and $\Br^\divv(Y|D')\{\ell\}= \Br(U')\{\ell\}$, where $\ell \text{ is a prime different from } p$, by \lemref{lem:ell-torsion}. In this case, the theorem follows by a similar argument using \thmref{thm:ell-lef}.
\end{proof}
     \begin{rem}
         The proof of \thmref{thm:Br-lef} also shows that the isomorphisms $\Br(X)\{p\} \xrightarrow{\cong}\Br(Y)\{p\}$ and $\Br^\divv(X|D)\{p\} \xrightarrow{\cong}\Br^\divv(Y|D')\{p\}$ hold if we take any $F$-finite field $k$ in the above theorem. 
     \end{rem}
    
\subsection{Lefschetz in complete intersection case}
   In this section we shall prove a Lefschetz theorem for smooth complete intersection varieties. We recall some notations. For a field $k$, $k^s$ will be its separable closure and $\ov k$ will be an algebraic closure of $k$ containing $k^s$. For a $k$-scheme $X$, we shall write $X^s=X\times _{\Spec k}\Spec k^s$ and $\ov X=X \times_{\Spec k}  \Spec \ov k$.  
    \begin{defn}(cf. \cite[Def~5.5]{A-B-M})
        Let $X$ be a smooth $k$-scheme of dim $N$ and $\sL$ be an ample line bundle on $X$. Then $(X,\sL)$ is called a Kodaira pair if $H^i(X,\Omega^q_{X/k}\otimes_{\sO_X} \sL^{-n})=0$ for all $i+q \le N-1$ and $q \ge 0, n \ge 1$. 
    \end{defn}
    Here are few examples.\\
    \textbf{Examples :}
    (1) $(\P^N_k, \sO(n))$  is a Kodaira pair for $n >0$. \\
    (2) If $X$ is smooth projective over a characteristic $0$ field $k$ with an ample line bundle $\sL$, then $(X, \sL)$ is a Kodaira pair by \cite[Cor~2.11]{Deligne-Illusie}.

    \begin{lem}
        Let $X$ be smooth over an $F$-finite field $k$ (of char $p>0$) with an ample line bundle $\sL$. If ($X,\sL$) is a Kodaira pair then $H^i(X,\Omega^q_{X}\otimes_{\sO_X} \sL^{-n})=0$ for all $i+q \le N-1$ and $q \ge 0, n \ge 1$.
    \end{lem}
    \begin{proof}
        Since $k$ is $F$-finite, and $X$ is smooth over $k$, we have a short exact sequence of finite type locally free sheaves
        $$0 \to f^*\Omega^1_k \to \Omega^1_X \to \Omega^1_{X/k} \to 0,$$ where $f:X \to \Spec k$. Then by \cite[Ex~ II.5.16(d)]{Hartshorne-AG}, there is a decreasing filtration $\{F^t\}_{0 \le t\le q}$ of $\Omega^q_X \otimes_{\sO_X} \sL^{-n}$ such that $F^t/F^{t+1} \cong \Wedge^t(f^*\Omega^1_k) \otimes_{\sO_X} \Omega^{q-t}_{X/k}\otimes_{\sO_X} \sL^{-n} \cong  \Omega^t_k \otimes_k (\Omega^{q-t}_{X/k}\otimes_{\sO_X} \sL^{-n})$.  Now the hypothesis on $(X,\sL)$ implies $H^i(X, F^t/F^{t+1})=0$ for all $i+q \le N-1$ and $t \le q$. Hence we get the lemma. 
\end{proof}
    \begin{rem}\label{rem:Kodaira}
            The above lemma implies that the condition (1.b) (and hence (2.b)) of \lemref{lem:ample}  holds if $(X, \sO_X(Y))$ is Kodaira pair.
    \end{rem}

    \begin{lem}\label{lem:Kodaira}
        Let $X\inj \P^N_k$ be any smooth complete intersection subvariety. Let $Y$ be any hypersurface section of $X$. Then $(X,\sO_X(Y))$ is a Kodaira pair.
    \end{lem}
    \begin{proof}
        This follows from \cite[Cor~5.8]{A-B-M} in the following way. Since $(\P^N_k, \sO(1))$ is a Kodaira pair, loc.cit. implies $(X, \sO_{X}(1))$ is also a Kodaira pair. In particular, $(X, \sO_{X}(n))$ is also Kodaira pair for all $n \ge 1$. But $\sO_{X}(Y) \sim \sO_{X}(n)$ for some $n>0$ and we get the lemma. 
    \end{proof}
    \begin{prop}\label{prop:lef}
        Let $Y$ be any smooth complete intersection in $\P^N_k$. Let dim $Y=d$.Then we have the canonical isomorphism $H^i_\et(\P^N_k,W_m\Omega^q_{\P^N_k,\log}) \xrightarrow{\cong} H^i_\et(Y, W_m\Omega^q_{Y,\log})$ for all $i,q \ge 0$ such that $i+q \le d-1$. 
    \end{prop}
    \begin{proof}
               Since $Y$ is smooth and complete intersection, we can write $Y=Y_1\cap H$, where $Y_1$ is smooth complete intersection in $\P^N_k$ of dim $d+1$ and $H$ is a smooth hypersurface section. We have seen in \thmref{thm:Lefschetz} that the condition (1.b) of \lemref{lem:ample} (taking $X=Y_1$) implies the  isomorphism $H^i_\et(Y_1,W_m\Omega^q_{Y_1,\log}) \xrightarrow{\cong} H^i_\et(Y, W_m\Omega^q_{Y,\log})$, for all $i \le d-q-1$. Hence, to prove the isomorphism, it is enough to show that $(Y_1, \sO_{Y_1}(H))$ is a Kodaira pair (by \remref{rem:Kodaira}). But this follows from \lemref{lem:Kodaira}. Thus we have proven the isomorphism $H^i_\et(Y_1,W_m\Omega^q_{Y_1,\log}) \xrightarrow{\cong} H^i_\et(Y, W_m\Omega^q_{Y,\log})$, for all $i \le d-q-1$. Now, inductively, repeating the same argument for $Y_1$ we can prove the proposition.
    \end{proof}
The following corollary is a generalization of \cite[Cor.~5.5.4]{CTS}, where the result is proven for prime to p-part of the Brauer Groups.
    \begin{cor}\label{cor:Br-ci}
        Let $k$ be any field. If $Y$ is a smooth  complete intersection in $\P^N_k$ of dim $d \ge 3$, then $\Br(k) \cong \Br(Y)$.
    \end{cor}
    \begin{proof}
        By \cite[Cor~5.5.4]{CTS}, one knows $\Br(k) \cong \Br_1(Y)$, where $\Br_1(Y)=\Ker (\Br(Y) \to \Br(Y^s))$. Also it was proven in loc. cit. that $\Br(Y^s)\{\ell\}=0$ for all $\ell$ coprime to $p$. So, it is enough to show $\Br(Y^s)\{p\}=0$. 
        
          We claim that $\Br(Y^s) \inj \Br(\ov Y)$. By \cite[Thm~5.2.5]{CTS}, it is enough to show $Y^s$ is geometrically connected and $H^1(Y^s,\sO_{Y^s})=0$. This can be seen by using decreasing induction on dim $Y^s$. Note that $Y^s$ is a smooth complete intersection in $\P^N_{k^s}$ of dim $d$ and hence is geometrically connected. If $d=N$ (i.e, $Y^s=\P^N_{k^s}$), then the vanishing of $H^1(Y^s, \sO_{Y^s})$ is clear. If $d <N$, we can write $Y^s=Y_1 \cap H$ as in \propref{prop:lef}, where $Y_1$ is smooth complete intersection of dim $d+1$ in $\P^N_{k^s}$ and $H$ is a smooth hypersurface section. Then 
         $$0 \to \sO_{Y_1}(-H) \to \sO_{Y_1} \to \sO_{Y} \to 0 $$
         induces an exact sequence
         $$H^1(Y_1, \sO_{Y_1}) \to H^1(Y, \sO_Y) \to H^2(Y_1,\sO_{Y_1}(-H)).$$
         By induction hypothesis, $H^1(Y_1, \sO_{Y_1})=0$. Also, note that $H^2(Y_1,\sO_{Y_1}(-H))=0$, because $(Y_1,\sO_{Y_1}(H))$ is a Kodaira pair (by \lemref{lem:Kodaira}) and dim $Y_1=d+1 \ge 4$. Hence the claim is proved.

         Now, to prove the corollary, it is enough to show $\Br(\ov Y)\{p\}=0$. For that we consider the diagram \eqref{eqn:lef-Br}, replacing $X$ and $Y$ by $\P^N_{\ov k}$ and $\ov Y$, respectively.  By \cite[XII, Cor~3.7]{SGA2}, we have $\Pic(\P^N_{\ov k}) \xrightarrow{\cong} \Pic(\ov Y)$. Hence, using \propref{prop:lef}, we get $\Br(\P^N_{\ov k})\{p\} \cong \Br(\ov Y)\{p\}$. But $\Br(\P^N_{\ov k})=\Br(\ov k)=0$ by \cite[Thm~6.1.3]{CTS}. Hence the claim is proved.
        
    \end{proof}

\end{document}